\documentclass[11 pt,reqno]{amsart}

\usepackage[latin1]{inputenc}
\usepackage{amsmath,amsthm,amssymb,amscd,mathrsfs,mathtools,xcolor,esint}
\usepackage[margin=1in]{geometry}
\usepackage[shortlabels]{enumitem}
\definecolor{darkgreen}{rgb}{0.0, 0.6, 0.13}
\usepackage[perpage]{footmisc}
\usepackage{hyperref} 
\usepackage{macros_for_trees}

\usepackage{scalerel}

\usepackage{tikz}
\usetikzlibrary{fit,positioning}
\usepackage{pgfplots}
\pgfplotsset{width=7cm,compat=1.8}
\usepackage{pgfplotstable}

\usepackage{soul}
\definecolor{airforceblue}{rgb}{0.0, 0.30, 0.69}

\newtheorem{thm}{Theorem}[section]
 \newtheorem{prop}[thm]{Proposition}
 \newtheorem{lem}[thm]{Lemma}
\newtheorem{claim}[thm]{Claim}
 \theoremstyle{definition}
 \newtheorem{df}[thm]{Definition}
 \theoremstyle{remark}
 \newtheorem{rem}[thm]{Remark}
 
 \numberwithin{equation}{section}

\newcommand{\Sc}{\mathcal{S}}
\newcommand{\Nc}{\mathcal{N}}
\newcommand{\Lc}{\mathcal{L}}
\newcommand{\Ic}{\mathcal{I}}

\newcommand{\Cc}{\mathcal{C}}
\newcommand{\Dc}{\mathcal{D}}
\newcommand{\Ec}{\mathcal{E}}
\newcommand{\Fc}{\mathcal{F}}
\newcommand{\Pc}{\mathcal{P}}
\newcommand{\Qc}{\mathcal{Q}}
\newcommand{\Rc}{\mathcal{R}}
\newcommand{\Uc}{\mathcal{U}}
\newcommand{\Bc}{\mathcal{B}}
\newcommand{\Ac}{\mathcal{A}}
\newcommand{\Vc}{\mathcal{V}}
\newcommand{\Wc}{\mathcal{W}}
\renewcommand{\Mc}{\mathcal{M}} 
\newcommand{\Xc}{\mathcal{X}}
\newcommand{\Yc}{\mathcal{Y}}
\newcommand{\Zc}{\mathcal{Z}}
\newcommand{\Tb}{\mathbb{T}}
\newcommand{\Cb}{\mathbb{C}}
\newcommand{\Rb}{\mathbb{R}}
\newcommand{\Eb}{\mathbb{E}}
\newcommand{\Zb}{\mathbb{Z}}
\newcommand{\Pb}{\mathbb{P}}
\newcommand{\Ab}{\mathbb{A}}

\newcommand{\pf}{\mathfrak{p}}
\newcommand{\Hc}{\mathcal{H}}
\newcommand{\ff}{\mathfrak{f}}
\newcommand{\hf}{\mathfrak{h}}
\newcommand{\Xf}{\mathfrak{X}}
\newcommand{\Zf}{\mathfrak{Z}}
\newcommand{\Bf}{\mathfrak{B}}
\newcommand{\nf}{\mathfrak{n}}

\newcommand{\lf}{\mathfrak{l}}
\newcommand{\Nf}{\mathfrak{N}}
\newcommand{\mf}{\mathfrak{m}}

\newcommand{\Bs}{\mathscr{B}}
\newcommand{\Ms}{\mathscr{M}}
\newcommand{\Ns}{\mathscr{N}}

\newcommand{\Yf}{\mathfrak{Y}}
\newcommand{\Ps}{\mathscr{P}}
\newcommand{\Os}{\mathscr{O}}
\newcommand{\Ls}{\mathscr{L}}

\newcommand{\Rs}{\mathscr{R}}
\newcommand{\Vs}{\mathscr{V}}
\newcommand{\Ws}{\mathscr{W}}

\newcommand{\af}{\mathfrak{a}}
\newcommand{\bff}{\mathfrak{b}}
\newcommand{\cf}{\mathfrak{c}}
\newcommand{\dff}{\mathfrak{d}}
\newcommand{\ef}{\mathfrak{e}}
\newcommand{\gf}{\mathfrak{g}}
\newcommand{\iif}{\mathfrak{i}}
\newcommand{\jf}{\mathfrak{j}}
\newcommand{\kf}{\mathfrak{k}}
\newcommand{\of}{\mathfrak{o}}

 \makeatletter
\def\@tocline#1#2#3#4#5#6#7{\relax
  \ifnum #1>\c@tocdepth 
  \else
    \par \addpenalty\@secpenalty\addvspace{#2}
    \begingroup \hyphenpenalty\@M
    \@ifempty{#4}{
      \@tempdima\csname r@tocindent\number#1\endcsname\relax
    }{
      \@tempdima#4\relax
    }
    \parindent\z@ \leftskip#3\relax \advance\leftskip\@tempdima\relax
    \rightskip\@pnumwidth plus4em \parfillskip-\@pnumwidth
    #5\leavevmode\hskip-\@tempdima
      \ifcase #1
       \or\or \hskip 1em \or \hskip 2em \else \hskip 3em \fi
      #6\nobreak\relax
    \hfill\hbox to\@pnumwidth{\@tocpagenum{#7}}\par
    \nobreak
    \endgroup
  \fi}
\makeatother

\begin{document}
\author[Yu Deng]{Yu Deng$^1$}
\address{$^1$ Department of Mathematics, University of Southern California, Los Angeles,  CA 90089, USA }
\email{yudeng@usc.edu}
\thanks{$^1$Y. D. is funded in part by NSF DMS-1900251.}
\author[Andrea R. Nahmod]{Andrea R. Nahmod$^2$}
\address{$^2$ 
Department of Mathematics,  University of Massachusetts,  Amherst MA 01003}
\email{nahmod@math.umass.edu}
\thanks{$^2$A.N. is funded in part by NSF DMS-1800852 and the Simons Foundation Collaborations Grant on Wave Turbulence (Nahmod's Award ID 651469).}
\author[Haitian Yue]{Haitian Yue$^3$}
\address{$^3$Department of Mathematics, University of Southern California, Los Angeles,  CA 90089, USA}
\email{haitiany@usc.edu}
  \date{}
 \title{Random tensors, propagation of randomness, and nonlinear dispersive equations}
 \begin{abstract} The purpose of this paper is twofold. We introduce the theory of \emph{random tensors}, which naturally extends the method of \emph{random averaging operators} in our earlier work \cite{DNY}, to study the propagation of randomness under nonlinear \emph{dispersive} equations. By applying this theory we also solve Conjecture 1.7 in \cite{DNY}, and establish almost-sure local well-posedness for semilinear Schr\"{o}dinger equations in spaces that are subcritical in the \emph{probabilistic scaling}. The solution we find has an explicit expansion in terms of multilinear Gaussians with adapted random tensor coefficients.
 
In the random setting, the probabilistic scaling is the natural scaling for dispersive equations, and is \emph{different} from the natural scaling for parabolic equations. Our theory, which covers the full subcritical regime in the probabilistic scaling, can be viewed as the dispersive counterpart of the existing parabolic theories (regularity structure, para-controlled calculus and renormalization group techniques).
\end{abstract}
 \maketitle
 
\tableofcontents
 \section{Introduction}\label{intro} The study of partial differential equations with randomness has become an important and influential subject in the last few decades. In this work we will be concerned with a major topic of this subject, namely the local in time Cauchy problems with either random initial data or additive stochastic forcing.
 
 It is well known that in many situations, randomization or noise improves the behavior of solutions to PDEs. Usually this can be interpreted as \emph{generic} solutions being genuinely better than \emph{pathological} ones. This phenomenon, which has its roots in the various cancellation properties of independent random variables (e.g. Central Limit Theorem or Khintchine's inequality), has been extensively studied since the 70--80's. The key difficulty here is to analyze how the explicit randomness (given for example by a Wiener measure or Gaussian noise) \emph{propagates} under the flow of nonlinear PDEs.
 
In the past few years, there has been significant progress in the setting of singular parabolic stochastic equations (SPDEs): the development of the theory of \emph{regularity structures} of Hairer and the \emph{para-controlled calculus} of Gubinelli-Imkeller-Perkowski has led to tremendous success in local well-posedness theory, essentially completing the full picture in what is known as the \emph{subcritical} regime. Unfortunately, both theories rely crucially on the \emph{parabolic} nature of the equation, and have not achieved the same success in the other important class of PDEs, namely the \emph{dispersive} equations.

The purpose of this work is to develop a new theory, which we will call \emph{random tensors}, to fill this gap in the dispersive setting. This is a natural extension of the method of \emph{random averaging operators} in our earlier work \cite{DNY}, but is much more powerful. In fact, in the random setting dispersive equations have a natural scaling, which we call the probabilistic scaling (see Section \ref{heuristic}), that is \emph{different} from the parabolic one\footnote{See Remark \ref{remscale}, and the explanation in Section \ref{comparison}.}, and our method---just like the theory of regularity structures and the para-controlled calculus in the parabolic setting---provides the complete picture in the full subcritical regime with respect to this natural scaling.

In this work we will focus on the random data problem\footnote{Random data is a natural setting for dispersive equations (parallel to additive noise for parabolic ones) in view of the invariant measures. Of course one may also consider stochastic versions of (\ref{nls}), which are similar but correspond to different randomizations, see Section \ref{furthercomment}.} for the semilinear Schr\"{o}dinger equation, which is the most common (and most studied) nonlinear dispersive equation. Our method is general and can be applied to other dispersion relations (see Section \ref{furthercomment}).

The rest of this introduction is organized as follows. In Section \ref{setup} we describe the setup and state the main theorems. In Section \ref{heuristic} we present a heuristic scaling argument from \cite{DNY} to justify the notion of criticality in this work. In Section \ref{propagate} we briefly review the ideas in earlier works, and in Sections \ref{rao0} we discuss the method of random averaging operators in \cite{DNY}. Finally in Section \ref{randomtensors} we provide the motivation behind our theory of random tensors; the detailed explanation of this theory is left to Section \ref{intro2}.
\subsection{Setup and main results}\label{setup} Fix $d\geq 1$ and $p\geq 3$ odd, and assume $(d,p)\neq (1,3)$; in particular $d(p-1)\geq 4$. Consider the nonlinear Schr\"{o}dinger (NLS) equation on $\Rb_t\times\Tb_x^d$, where $\Tb^d=(\Rb/2\pi\Zb)^d$: \begin{equation}\label{nls}
\left\{
\begin{aligned}(i\partial_t+\Delta)u&=W^p(u),\\
u(0)&=f(\omega).
\end{aligned}
\right.
\end{equation} Here $f(\omega)$ is some choice of random initial data defined on an ambient probability space $(\Theta,\Bc,\Pb)$, $\omega\in\Theta$, and $W^p(u)$ is either $|u|^{p-1}u$ or its \emph{Wick ordering}, which will be defined precisely below. The Hamiltonian of (\ref{nls}) is linked to the $\Phi_d^{p+1}$ model in constructive quantum field theory.
\subsubsection{Almost-sure local well-posedness} In the context of almost-sure local well-posedness, the random initial data will be given by \begin{equation}\label{data0}f(\omega)=\sum_{k\in\Zb^d}\frac{g_k(\omega)}{\langle k\rangle^\alpha}e^{ik\cdot x},
 \end{equation} where $\{g_k(\omega)\}_{k\in\Zb^d}$ are i.i.d. centered normalized (complex) Gaussian random variables, and $\alpha$ is chosen as \begin{equation}\label{choicealpha}\alpha=s+\frac{d}{2},\qquad s>s_{pr}:=-\frac{1}{p-1}.\end{equation} This value $s_{pr}$ is the critical exponent for the \emph{probabilistic scaling}, which will be discussed in detail in Section \ref{heuristic}. It is always lower than $s_{cr}:=(d/2)-2/(p-1)$, which is the critical exponent for the usual (deterministic) scaling. The random data $f(\omega)$ defined by (\ref{data0}) almost surely belongs to $H^{s-}(\Tb^d):=\cap_{s'<s}H^{s'}(\Tb^d)$, but not $H^s(\Tb^d)$.
 
Our first main theorem, originally stated as Conjecture 1.7 in \cite{DNY}, proves that (\ref{nls}) is almost surely locally well-posed with random initial data (\ref{data0}), in the sense that canonical smooth approximations converge to a unique limit; here and throughout the paper, canonical smooth approximations always mean the ones described in Theorem \ref{main} and Remark \ref{strongrem} below. This result can be interpreted as almost sure local well-posedness in $H^{s-}(\Tb^d)$ with respect to the canonical Gaussian measure---the law of $f(\omega)$ defined by (\ref{data0})---for any $s>s_{pr}$, i.e. \emph{in the full probabilistically subcritical regime}.

To state the theorem, we need to define the canonical truncations and the associated Wick orderings. Given a dyadic number $N\geq 1$\footnote{We will assume $N\geq 1$ throughout, and only ``formally'' need to replace $N$ by $1/2$ in a few places.}, define the truncation operators $\Pi_N$ and $\Delta_N$ by
\begin{equation}\label{proj}(\Pi_Nu)_k=\Pi_Nu_k:=\mathbf{1}_{\langle k\rangle\leq N}\cdot u_k,\quad\Delta_N=\Pi_N-\Pi_{\frac{N}{2}},\end{equation} where $u_k$ represents the Fourier coefficient. For notational simplicity, we will identify $\Pi_N$ with the multiplier $\Pi_N(k)=\mathbf{1}_{\langle k\rangle\leq N}$, and similarly for $\Delta_N$. Define the expectation of truncated mass
\begin{equation}\label{truncmass}\sigma_N:=\Eb\fint_{\Tb^d}|\Pi_Nf(\omega)|^2=\sum_{\langle k\rangle\leq N}\frac{1}{\langle k\rangle^{2\alpha}},
\end{equation} and, for integers $r\geq 0$, the Wick-ordered monomials
\begin{equation}\label{wickpoly}
W_N^{2r}(u)=\sum_{j=0}^r(-1)^{r-j}{r\choose j}\frac{\sigma_N^{r-j}r!}{j!}|u|^{2j},\quad
W_N^{2r+1}(u)=\sum_{j=0}^r(-1)^{r-j}{{r+1}\choose {j+1}}\frac{\sigma_N^{r-j}r!}{j!}|u|^{2j}u,
\end{equation} where $\sigma_N$ is as in (\ref{truncmass}). The first main theorem is then stated as follows.
 \begin{thm}\label{main} Fix $s>s_{pr}$ and $\alpha$ as in (\ref{choicealpha}), and let $f(\omega)$ be as in (\ref{data0}). Let $u_N$ be the solution to the canonically truncated system \begin{equation}\label{nlstrunc}
\left\{
\begin{aligned}(i\partial_t+\Delta)u_N&=\Pi_NW_N^{p}(u_N),\\
u_N(0)&=\Pi_Nf(\omega).
\end{aligned}
\right.
\end{equation} Then, for $0<\tau\ll 1$, there exists a set $Z\subset\Theta$ with $\Pb(Z)\leq C_\theta e^{-\tau^{-\theta}}$, where $\theta$ is a small constant (ultimately determined by $(d,p,s)$, and independent of $\tau$) and $C_\theta$ is a constant determined by $\theta$, such that when $\omega\not\in Z$, the sequence $\{u_N\}$ converges, as $N\to\infty$, to a unique limit $u$ in $C_t^0H_x^{s-}[-\tau,\tau]$.

Moreover, for this $u$, the nonlinearity $W^p(u)$ in (\ref{nls}), which is the Wick ordering of $|u|^{p-1}u$, is well-defined as
\begin{equation}\label{wickpoly2}W^p(u)=\lim_{N\to\infty}W_N^p(\Pi_Nu)=\lim_{N\to\infty}\Pi_NW_N^p(\Pi_Nu)
\end{equation} in the sense of spacetime distributions (where both limits exist and are equal), and $u$ solves the equation (\ref{nls}) in the distributional sense. Finally this solution $u$ has an explicit expansion in terms of multilinear Gaussians with adapted random tensor coefficients; see (\ref{expansionofu}) for the precise form.
\end{thm}
  \begin{rem}\label{1dcubic} Theorem \ref{main} (and Theorem \ref{main2} below) can be shown for any rectangular torus $\Tb_\beta^d:=(\Rb/2\pi\beta_1\Zb)\times\cdots\times(\Rb/2\pi\beta_d\Zb)$, and for focusing nonlinearity, with almost the same proof.   
  
When $(d,p)=(1,3)$, instead of (\ref{wickpoly}) one should look at the completely non-resonant nonlinearity $(|u|^2-2\fint |u|^2)u$, since in this case $\|\Pi_Nu\|_{L^2}^2-\Eb\|\Pi_Nu\|_{L^2}^2$ does not converge as $N\to\infty$, see (\ref{devmass}). With this change, it is known, see \cite{GrH}, that (\ref{nls}) is deterministically locally well-posed in a Fourier-Lebesgue space which the data (\ref{data0}) almost surely belongs to (as in this case $s_{pr}=s_{cr}$), so Theorem \ref{main} remains true.
 \end{rem}
 \begin{rem}\label{needwick} The Wick ordering (\ref{wickpoly}) is crucial in Theorem \ref{main}, as our solution has infinite mass when $s_{pr}<s<0$; in this case the Wick-ordered NLS, rather than the original one, is the right equation to study. In fact, even in the simplest case $(d,p)=(1,3)$, the NLS without the renormalization as in Remark \ref{1dcubic} will have \emph{no solution} for \emph{any} infinite mass initial data; see \cite{GO}. As another example, for the dynamical $\Phi_2^4$ model (2D cubic heat equation with white noise forcing), the canonical smooth approximations will converge to a nontrivial limit only with Wick ordering, otherwise the limit would be identically zero for any initial condition (see \cite{ChW,HR+}).
\end{rem}
\begin{rem}\label{strongrem} Despite having low regularity $C_t^0H_x^{s-}$, the local solution constructed in Theorem \ref{main} is a \emph{strong} solution, as it is the unique limit of canonical smooth approximations. In fact, by slightly modifying our proof (which will not be done here for simplicity of presentation), we can obtain the following general convergence result:

Let $\varphi$ be a function on $\Rb^d$, $\varphi(0)=1$, and $\varphi$ is either Schwartz or equals the characteristic function $\mathbf{1}_{\mathbb{B}}$ of the unit ball $\mathbb{B}$ (the latter is the setting of Theorem \ref{main}). Let $\widetilde{\varphi}$ be either $1$ or $\mathbf{1}_{\mathbb{B}}$, in the latter case we assume $\varphi\equiv 0$ outside $\mathbb{B}$. Define $\mathtt{P}_\lambda(k)=\varphi(\lambda^{-1}k)$ and $\widetilde{\mathtt{P}}_\lambda(k)=\widetilde{\varphi}(\lambda^{-1}k)$. Let $W_{\lambda}^p$ be defined as in (\ref{wickpoly}) with $\sigma_N$ replaced by $\sigma_{\lambda}$, which is in turn defined as in (\ref{truncmass}) with $\Pi_N$ replaced by $\mathtt{P}_\lambda$. Consider the solution $u_{\lambda}$ to the system
\begin{equation}\label{nlstruncgen}
\left\{
\begin{aligned}(i\partial_t+\Delta)u_{\lambda}&=\widetilde{\mathtt{P}}_\lambda W_{\lambda}^{p}(u_{\lambda}),\\
u_{\lambda}(0)&=\mathtt{P}_\lambda f(\omega).
\end{aligned}
\right.
\end{equation}Then Theorem \ref{main} remains true with conclusion being the convergence of $\{u_{\lambda}\}$ as $\lambda\to\infty$. Here the exceptional set $Z$ and the limit $u$ \emph{do not depend on} the choice of $(\varphi,\widetilde{\varphi})$.

More precisely, there exists a random time $\mathtt{T}=\mathtt{T}(\omega)$ satisfying $\Pb(\mathtt{T}<\tau)\leq C_\theta e^{-\tau^{-\theta}}$ for any $\tau>0$, and a random function $u=u(t,x,\omega)$ defined for $|t|\leq \mathtt{T}(\omega)$, such that almost surely in $\omega$, we have $u_{\lambda}\to u$ in $C_t^0H_x^{s-}[-\mathtt{T},\mathtt{T}]$, as $\lambda\to\infty$, for \emph{any} choice of $(\varphi,\widetilde{\varphi})$.
 \end{rem}
\begin{rem} Although Theorem \ref{main} concerns singular (i.e. low regularity) data and short-time solutions, the fundamental issue here is to understand how the randomness structure\footnote{Such structure lives on high frequencies and fine scales. It becomes more explicit when considering low regularity solutions, and may be obscured by the dominant coarse scale profile in high regularity solutions.} propagates under the nonlinear Schr\"{o}dinger flow. With this understanding, we can easily obtain new results for regular data and long-time solutions, such as Theorem \ref{main2} below.
 \end{rem}
 \subsubsection{Long-time control for random homogeneous data} Consider the random homogeneous data, which is the random initial data given by
 \begin{equation}\label{data1}f_{\mathrm{ho}}(\omega)=N^{-\alpha}\sum_{k\in\Zb^d}\phi\big(\frac{k}{N}\big)g_k(\omega)e^{ik\cdot x};\qquad \alpha=s+\frac{d}{2},\,\,s>s_{pr}.
 \end{equation} Here $\phi$ is a fixed Schwartz function and $N$ is a \emph{fixed} large parameter. Compared to (\ref{data0}), which is a superposition of multiple scales, in (\ref{data1}) we have only one scale $N$ and the Fourier modes are uniformly distributed in the ball $\langle k\rangle\lesssim N$. Such random data in fact are, up to rescaling, the ones appearing in the derivation of \emph{wave kinetic equation} in weak turbulence problems \cite{BGHS,CG,DH}. Note that with high probability, $\|f_{\mathrm{ho}}(\omega)\|_{H^{s}}\sim 1$. The second main theorem is then stated as follows.
\begin{thm}\label{main2} Fix $(s,\alpha)$ and $f_{\mathrm{ho}}(\omega)$ as in (\ref{data1}), and $0<\nu<(p-1)(s-s_{pr})$. Let $u_{\mathrm{ho}}$ be the solution to the system \begin{equation}\label{nlstrunc2}
\left\{
\begin{aligned}(i\partial_T+\Delta)u_{\mathrm{ho}}&=|u_{\mathrm{ho}}|^{p-1}u_{\mathrm{ho}},\\
u_{\mathrm{ho}}(0)&=f_{\mathrm{ho}}(\omega).
\end{aligned}
\right.
\end{equation} Then there exists a set $Z\subset\Theta$ with $\Pb(Z)\leq C_\theta e^{-N^\theta}$, where $\theta$ is a small constant (ultimately determined by $(d,p,s)$, and independent of $N$) and $C_\theta$ is a constant determined by $\theta$, such that when $\omega\not\in Z$, the solution $u_{\mathrm{ho}}$ exists up to time $T=N^\nu$. Moreover, for some real valued gauge function $B(T)$, we have
\begin{equation}\label{longtimecon}\sup_{0\leq T\leq N^\nu}\|u_{\mathrm{ho}}(T)-e^{-iB(T)}e^{iT\Delta}u_{\mathrm{ho}}(0)\|_{H^s}\leq N^{-\theta}.
\end{equation}
\end{thm}
\begin{rem} When $s_{cr}>1$, i.e. (\ref{nls}) is $H^1$ supercritical in the usual (deterministic) sense, it is not a priorily known whether $u_{\mathrm{ho}}$ exists up to time $N^\nu$. Even in the $H^1$ subcritical case where $u_{\mathrm{ho}}$ exists for all time, Theorem \ref{main2} demonstrates the highly nontrivial fact that there is no significant energy shift (such as a cascade) in $u_{\mathrm{ho}}$, (almost) until this very long time $T=N^{(p-1)(s-s_{pr})}$.

In comparison, for deterministic data, say when $g_k(\omega)$ in (\ref{data1}) are replaced by $1$, one can describe the asymptotic behavior of $u_{\mathrm{ho}}$ only up to time\footnote{In fact there is energy shift in $u_{\mathrm{ho}}$ at time $N^{(p-1)(s-s_{cr})}$, dictated by the \emph{continuous resonance} (CR) equation; see \cite{BGHS0,FGH}. Note that, the approximation leading to the CR equation does not work in our case as randomization destroys the differentiability in rescaled Fourier space; in fact the solution $u_{\mathrm{ho}}$ in Theorem \ref{main2} has no energy shift at this time. It is currently unknown whether the solution $u_{\mathrm{ho}}$ in Theorem \ref{main2}, or the corresponding ensemble average, satisfies some effective equation at time $N^{(p-1)(s-s_{pr})}$.} $O(N^{(p-1)(s-s_{cr})})$ provided $s>s_{cr}$; see \cite{BGHS0,FGH} for the $p=3$ case. Therefore, the randomization (\ref{data1}) effectively \emph{extends the time of perturbative regime} for the given homogeneous initial data. In essence, this is the same as Theorem \ref{main}, where we keep the time of perturbative regime constant (namely $1$), and randomization allows us to increase the size of initial data at a given frequency (equivalent to reducing regularity).
\end{rem}
\begin{rem} Since we are on the square torus $\Tb^d$, the behavior of $u_{\mathrm{ho}}$ at long time is dominated by exact resonances. If $\Tb^d$ is replaced by a \emph{generic irrational} torus, then we may expect Theorem \ref{main2} to hold on even longer time intervals, conjecturally up to $N^{2(p-1)(s-s_{pr})}$, at least for some range of $s$. This, after rescaling, would correspond to justifying the wave kinetic equation up to the \emph{kinetic timescale} in the context of weak turbulence, which is still open at this point (despite the recent success in \cite{DH} of the first author with Z. Hani; see also \cite{CG}).
\end{rem}
\begin{rem} Unlike Theorem \ref{main}, in Theorem \ref{main2} we do not need the Wick ordering (\ref{wickpoly}). Indeed, the worst contributions in the context of Theorem \ref{main} (as well as the non-existence result of \cite{GO}), which can only be rescued by Wick ordering, are the high-low interactions where the high frequencies form a pairing and produce the mass term; for random homogeneous data there is no distinction between high and low frequencies, so such terms will not be a concern.
\end{rem}
 \subsection{The probabilistic scaling}\label{heuristic} The probabilistic criticality threshold $s_{pr}$ plays a key role in the local theory of (\ref{nls}) with random initial data. In this section we recall a heuristic justification of this fact, which originally appeared in \cite{DNY}.
 
Start with (\ref{nls}) on $\Rb\times\Tb^d$, for simplicity we will replace $W^p(u)$ by the artificial nonlinearity $\Nc_{\mathrm{np}}$ without Wick ordering, defined by\begin{equation}\label{nopairintro}(\Nc_{\mathrm{np}}(u))_k:=\sum_{k_1-\cdots +k_p=k}u_{k_1}\overline{u_{k_2}}\cdots u_{k_p},\end{equation} assuming there is no pairing (i.e. $k_j\not\in\{ k_{j'},k\}$ for any odd $j$ and even $j'$). Let the initial data $u(0)\in H^s$, to prove local well-posedness in $H^s$, one would like to control
 \begin{equation}\label{2iteration}u^{(1)}(t):=\int_0^te^{i(t-t')\Delta}\Nc_{\mathrm{np}}(e^{it'\Delta}u(0))\,\mathrm{d}t',\end{equation} namely the first nonlinear iteration, in $H^s$, at time $|t|\sim 1$. In the deterministic setting, if
 \begin{equation}\label{deterdata}u(0)=N^{-\alpha}\sum_{|k|\sim N}e^{ik\cdot x},\qquad \alpha=s+\frac{d}{2},
 \end{equation} then $\|u(0)\|_{H^s}\sim 1$. By (\ref{2iteration}) we can calculate the Fourier coefficients $u_k^{(1)}(t)$ of $u^{(1)}(t)$, where
 \begin{equation}\label{detersum}u_k^{(1)}(t)\sim N^{-p\alpha}\sum_{\substack{k_j\in\Zb^d,|k_j|\sim N\\k_1-\cdots+k_p=k}}\frac{1}{\langle\Omega\rangle},\quad \Omega:=|k|^2-|k_1|^2+\cdots-|k_p|^2
 \end{equation} for $|k|\sim N$ and $|t|\sim 1$. We may restrict to (say) $\Omega=0$ in (\ref{detersum}), and a dimension counting argument shows that the inner sum has size $N^{pd-d-2}$, hence 
 \begin{equation}\label{detersc}\|u^{(1)}(t)\|_{H^s}\sim N^{(p-1)(\frac{d}{2}-s)-2};\qquad \|u^{(1)}(t)\|_{H^s}\lesssim 1 \Leftrightarrow s\geq \frac{d}{2}-\frac{2}{p-1}:=s_{cr}.
 \end{equation} Indeed in the deterministic setting, (\ref{nls}) is locally well-posed for $s>s_{cr}$, and ill-posed for $s<s_{cr}$.
 
 Now, switching to the random setting, where instead of (\ref{deterdata}) we have
  \begin{equation}\label{randata}u(0)=N^{-\alpha}\sum_{|k|\sim N}g_k(\omega)e^{ik\cdot x},
   \end{equation} where $g_k(\omega)$ are i.i.d. centered normalized Gaussians, and instead of (\ref{detersum}) we have
  \begin{equation}\label{randsum}u_k^{(1)}(t)\sim N^{-p\alpha}\sum_{\substack{k_j\in\Zb^d,|k_j|\sim N\\k_1-\cdots+k_p=k}}\frac{1}{\langle\Omega\rangle}g_{k_1}\overline{g_{k_2}}\cdots g_{k_p}.
 \end{equation} Again we may restrict to $\Omega=0$; due to the square root cancellation in (\ref{randsum}), now with high probability the inner sum only has size $N^{(pd-d-2)/2}$, hence instead of (\ref{detersc}) we have
 \begin{equation}\label{randsc}\|u^{(1)}(t)\|_{H^s}\sim N^{-(p-1)s-1};\qquad \|u^{(1)}(t)\|_{H^s}\lesssim1\Leftrightarrow s\geq -\frac{1}{p-1}:=s_{pr}.
 \end{equation} This justifies the role of $s_{pr}$ in Theorem \ref{main}, and explains why almost-sure local well-posedness is plausible, in the \emph{probabilistically subcritical} regime $s>s_{pr}$.
 \begin{rem} Theorem \ref{main} establishes almost-sure local well-posedness when $s>s_{pr}$. In the probabilistically supercritical regime $s<s_{pr}$, we believe (\ref{nls}) is almost surely ill-posed, in the sense that almost surely, the approximations $u_N$ defined in (\ref{nlstrunc}) \emph{do not converge} in $C_t^0H_x^{s-}[-\tau,\tau]$ for \emph{any} $\tau>0$. This seems out of reach with the current methods, but some weaker results (for example failure of convergence\footnote{In comparison, when $s>s_{pr}$, the proof of Theorem \ref{main} easily implies the convergence of $u_N$ (including the $u_\lambda$ in Remark \ref{strongrem}) in $C_t^\iota H_x^{s-}$ for some $\iota>0$ ultimately determined by $(d,p,s)$.} in $C_t^\iota H_x^{s-}$ for any $\iota>0$) might be possible.
 \end{rem}
\begin{rem}\label{remscale} The notion of \emph{parabolic scaling} (also known as the \emph{super-renormalizability} scaling in quantum field theory, see \cite{CMW}) is central in the works on local theory of singular parabolic SPDEs, see for example \cite{Hairer,GIP}. For the analog of (\ref{nls}), namely 
 \begin{equation}\label{parab}(\partial_t-\Delta)u=\widetilde{W}^p(u)+\zeta,\qquad \zeta=\textrm{spacetime white (or colored) noise}
 \end{equation} (where $\widetilde{W}^p$ is some renormalization beyond Wick ordering, see (\ref{parab1}) below), this scaling has critical exponent $s_{pa}:=-2/(p-1)$. Note that $s_{pa}$ is strictly lower than $s_{pr}$, which reflects a fundamental difference between Schr\"{o}dinger and heat equations. See Section \ref{comparison} for more detailed discussions.
 \end{rem}
 \subsection{Propagating randomness: Earlier works}\label{propagate} The heuristics of Section \ref{heuristic} rely on the assumption that the Fourier coefficients of $u(0)$ are \emph{independent}; this is no longer satisfied by $u(t)$, as soon as $t>0$. Therefore, the key to the proof of Theorem \ref{main} is to propagate the randomness of initial data, for the anticipated amount of time, in such a way that the square cancellation in Section \ref{heuristic} remains valid.
 
 The idea of propagating randomness, interpreted in one way or another, has been central in all previous works concerning local well-posedness in the random setting. In this section we briefly review the existing approaches, especially those developed in the past few years.
\subsubsection{The method of Bourgain and Da Prato-Debussche}\label{BDDintro} The important early results in this direction are proved by Bourgain \cite{Bourgain} (for random initial data) and later by Da Prato-Debussche \cite{DaDe} (for additive noise). The idea is to propagate the random initial data (or the noise term) linearly, which preserves all the independence properties, and treat the nonlinearity as a perturbation.

For example, the equation studied in \cite{Bourgain} is (\ref{nls}) with $(d,p)=(2,3)$ and Gibbs measure initial data (i.e. $\alpha=1$ in (\ref{data0})), which is barely supercritical in the deterministic sense and subcritical in the probabilistic sense. In \cite{Bourgain} Bourgain constructed the solution as $u=u_{\mathrm{lin}}+w$, where $u_{\mathrm{lin}}=e^{it\Delta}u(0)$ is the linear evolution which enjoys the same randomness properties as the initial data $u(0)=f(\omega)$, and the remainder $w$ has improved regularity, say $C_t^0H_x^\sigma$ with $\sigma>0$, thus becoming subcritical in the deterministic sense. Then the classical fixed point analysis together with large deviation estimates apply to control the hybrid nonlinearity---of the difference equation that $w$ satisfies---containing interactions of $u_{\mathrm{lin}}$ with $w$. The situation in \cite{DaDe} is similar, except that Schr\"{o}dinger is replaced by Navier-Stokes, and $u_{\mathrm{lin}}$ is replaced by the linear evolution of the noise term. 

Until recent years, the methods of Bourgain and of Da Prato-Debussche has been the dominant strategy in the study of local well-posedness theory for random PDEs. The weakness of this approach is that the improved regularity of the nonlinear contribution $w$ may not be enough for the deterministic theory to be applicable, especially when one gets close to probabilistic criticality. One may try to enhance this by moving to higher-order variants and bringing in self-interactions of $u_{\mathrm{lin}}$, but in many situations there is an upper bound for all the regularity improvements obtained in this way, which may still fall short of the deterministic threshold. For an example of this see \cite{BOP}.
\subsubsection{The theory of regularity structures}\label{regstr} The theory of regularity structures was developed by Hairer \cite{Hairer, Hairer2} in the context of singular parabolic SPDEs, to provide a natural and mathematically rigorous notion of solutions to such equations and prove their local well-posedness. The theory is based on the \emph{local-in-space} properties of solutions at fine scales, hence it is well adapted to parabolic equations. It builds a general theory of distributions by means of an abstract generalization of local Taylor expansions of problem-dependent profiles (e.g. the spacetime white or colored noise, and self-interactions and Duhamel iterations thereof), in order to make sense of the equation---in particular the products of rough distributions emerging from the singular parabolic SPDE.
Furthermore, the solutions obtained can be approximated locally to arbitrarily high degree by linear combinations of this fixed family of problem-dependent profiles. These expansions in the context of singular parabolic SPDEs should be compared and contrasted to the ones that we obtain in Theorem \ref{main} (more precisely (\ref{expansionofu})).
 
 The theory of regularity structures establishes local well-posedness results in the sense of convergence of canonical smooth approximations. When taking limits of such approximations, a suitable renormalization in the form of divergent counterterms is usually needed. Sometimes (for example in the dynamical $\Phi_2^4$ model, see Remark \ref{needwick}) this is just the Wick ordering, but for more sophisticated equations further renormalizations become necessary. A nice feature of the regularity structures theory is that these renormalization constants can always be calculated using the profiles defined for the specific equation.

We illustrate this renormalization process following \cite{Hairer, Hairer3, Hairer4} where Hairer studies the dynamical $\Phi^4_3$ model. Here the canonical smooth approximations $u_\varepsilon$ satisfy the renormalized equations with the spacetime white noise $\zeta$ replaced by its regularization $\zeta_\varepsilon$, namely
\begin{equation}\label{parab1}
(\partial_t-\Delta)u_\varepsilon =3(C_1-3C_2) u_\varepsilon -u_\varepsilon^3 + \zeta_\varepsilon, 
\end{equation}
where $C_1\sim \varepsilon^{-1}$ corresponds to Wick ordering, and $C_2\sim\log \varepsilon$ is an additional renormalization constant. The problem-dependent profiles are in the following space 
\begin{equation}
T=\langle\bullet, \<3_black>, \<2_black>, \<2K*3_black>, \<1_black>, \<1K*3_black>, \<2K*2_black>, x_i \<2_black>, \mathbf{1}, \<K*3_black>, \<1K*2_black>, \cdots\rangle,
\end{equation}
with symbols ordered in increasing order of regularity and $\langle\, \cdot \, \rangle$ used to denote the linear span. The symbol $\bullet$ represents the noise, $x_i$ represent the coordinates of $x$, and $\<1_black> = (\partial_t-\Delta)^{-1} \bullet$, \, $\<3_black>= (\<1_black>)^3$, etc. The renormalization constants are then calculated from these profiles, such as $C_1= \mathbb{E}(\<1_black>)^2\sim \varepsilon^{-1}$, $C_2=\mathbb{E}(\<K*2_black>\cdot\<2_black>)\sim\log \varepsilon$, and convergence of $u_\varepsilon$ as $\varepsilon\to 0$ is proved for initial data in $C^\alpha$, $-\frac{2}{3}<\alpha<-\frac{1}{2}$.

Recently in a series of papers \cite{BCCH}\cite{BHZ}\cite{CMW}, the authors proved local well-posedness for the $\Phi^4_{4-\delta}$ model with initial data in $C^\alpha$, $\alpha>-1+\frac{\delta}{2}$, which covers as $\delta\to 0^+$ the whole subcritical regime in the sense of the parabolic scaling $s>s_{pa}$ (see Remark \ref{remscale}). 
\subsubsection{The para-controlled calculus and the renormalization group technique}\label{paracon} The theory of para-controlled calculus put forward by Gubinelli, Imkeller and Perkowski \cite{GIP} (see also Catellier and Chouk \cite{CCh}) and the work of  Kupiainen \cite{Kupia} based on renormalization group techniques provide alternative approaches to proving local well-posedness for singular parabolic SPDEs. The theory of para-controlled calculus, based on paradifferential calculus, leverages the key observation that the lack of regularity for $w$ in the approachs of Bourgain and of Da Prato-Debussche, as described at the end of Section \ref{BDDintro}, is only due to the high-low interactions where the high frequencies come from $u_{\mathrm{lin}}$ (or self-interactions and Duhamel evolutions thereof) and the low frequencies come from $w$. In this theory, such high-low interaction terms $X$ are \emph{para-controlled} by the high-frequency inputs (for example $u_{\mathrm{lin}}$), in the sense that\footnote{Here $\pi_>$ is the standard Bony paraproduct.} $X=\pi_>(u_{\mathrm{lin}},Y)+Z$ with $Z$ being smoother than $X$. Such terms, despite having insufficient regularity, are shown to enjoy similar randomness structures as $u_{\mathrm{lin}}$ itself, which allows for a fixed point argument, where $X$ is constructed in some low regularity space, and the remainder $Z$ is constructed in a higher regularity space. We refer the reader to \cite{CCh,GP, GP2, GP4, MouWe3, BaBe, BaBe2,GKO,Br} and references therein for nice expositions of these ideas and some other recent developments as well as a higher order variant of this method.
\subsection{Random averaging operators}\label{rao0} In view of these breakthrough works described in Sections \ref{regstr} and \ref{paracon} that deal with \emph{parabolic} equations, it would be natural to think that something similar can be done in the context of \emph{dispersive} equations. However, there are fundamental differences between dispersive (say Schr\"{o}dinger) and heat equations, preventing these methods from being applicable (for more comparisons see Section \ref{comparison}):

(a) The heat equation is compatible with local-in-space analysis, as is clear from the maximum principle or the off-diagonal exponential decay for the heat kernel. The Schr\"{o}dinger equation does not have these properties, which makes the application of the theory of regularity structures impossible, as the latter is based on pointwise Taylor expansions in physical space.

(b) Likewise, the heat equation is compatible with $C^\alpha$ (H\"{o}lder) spaces; in fact both the regularity structures theory and the para-controlled calculus rely heavily on such norms. On the other hand the Schr\"{o}dinger flow is unbounded on $C^\alpha$, and is bounded only on $H^s$ type spaces, which require a lot more derivatives to reach the same scaling as $C^\alpha$.

(c) The heat equation gains two spatial derivatives in terms of the fundamental solution $(\partial_t-\Delta)^{-1}$ (wave gains one), while the Schr\"{o}dinger equation gains none. The smoothing is seen only in terms of twisted temporal regularity by using $X^{s,b}$ type norms, which is clearly not compatible with either the regularity structures theory or the para-controlled calculus.

In our previous work \cite{DNY}, which studies (\ref{nls}) with $d=2$, arbitrary $p\geq 5$ and Gibbs measure initial data ($\alpha=1$), we introduced the method of \emph{random averaging operators}. The idea is to take the high-low interaction $X$ described in Section \ref{paracon}, but instead of putting it in a low regularity space (as is done in the para-controlled calculus), we write it as an \emph{operator} applied to the high frequency linear evolution $u_\mathrm{lin}$:
\begin{equation}\label{rao}X=\sum_N\sum_{L\ll N}\Pc_{NL}(\Delta_Nu_\mathrm{lin}),\end{equation} where $N$ and $L$ are the frequencies of the high and low inputs. The operator $\Pc_{NL}$, whose coefficients are independent with the modes of $\Delta_Nu_\mathrm{lin}$, contains all the randomness information of the low frequency components of $u$, which is carried by two operator norm estimates
\begin{equation}\label{rao1}\|\Pc_{NL}\|_{\mathrm{OP}}\leq L^{-\delta_0},\quad \|\Pc_{NL}\|_{\mathrm{HS}}\leq N^{\frac{1}{2}+\delta_1}L^{-\frac{1}{2}},\end{equation} where $\delta_1\ll\delta_0\ll 1$. Here $\Pc_{NL}$ is viewed as a linear operator between two Hilbert spaces that can be $L^2$ or $X^{s,b}$ depending on the context, $\|\cdot\|_{\mathrm{OP}}$ and $\|\cdot\|_{\mathrm{HS}}$ are respectively operator and Hilbert-Schmidt norms.

The method of random averaging operators, compared to the regularity structures theory and the para-controlled calculus, has three advantages relative to the three difficulties listed above, which makes it suitable for the study of Schr\"{o}dinger equations:

(a) The operator $\Pc_{NL}$ is a global-in-space object (in fact it is defined on the Fourier side), which is consistent with the non-local setting of Schr\"{o}dinger equations;

(b) The role of $C^\alpha$ norms is replaced by (essentially) the $L^2\to L^2$ operator norms, which is compatible with the well-established $L^2$ theory;

(c) The analysis for $\Pc_{NL}$ is performed in the category of $X^{s,b}$ spaces, which allows one to exploit the smoothing of the Schr\"{o}dinger fundamental solution.

By applying this method, we have been able to propagate the randomness of $\Pc_{NL}$ in terms of the above operator norm bounds, as well as control the remainder in a deterministically subcritical space, leading to the full resolution of the Gibbs measure problem in 2D. See \cite{DNY}.
 \subsection{Random tensors}\label{randomtensors} The core of this work is a broad extension of the method of random averaging operators, which we call the theory of \emph{random tensors}. A detailed introduction to this theory will be given in Section \ref{intro2}, here we will restrict our attention to only the motivation.
 
  Start by considering the random averaging operators; roughly speaking, the frequency $N$ piece of the high-low interaction $X$ in (\ref{rao}) is given by $\Delta_NX=\Pc_N(\Delta_Nu_\mathrm{lin})$ where $\Pc_N=\sum_L\Pc_{NL}$. In Fourier space it can be written as
 \[(\Delta_NX)_k(t)=e^{-i|k|^2t}\sum_{\langle k'\rangle\sim N}h_{kk'}(t)\frac{g_{k'}(\omega)}{\langle k'\rangle},\] where subscripts denote Fourier coefficients, and $h_{kk'}(t)$ is essentially the kernel of the operator $\Pc_N$. For fixed $t$ this is a random matrix, or $(1,1)$ tensor, that depends on the low frequency components of the solution.
 
 Now, to prove Theorem \ref{main} we will need higher order expansions. This naturally leads to the multilinear expressions (here we denote $(u^+,u^-)=(u,\overline{u})$ as usual)
 \begin{equation}\label{tensorintro0}\Psi_k=\sum_{k_1,\cdots,k_r}h_{kk_1\cdots k_r}\prod_{j=1}^r\langle k_j\rangle^{-\alpha}g_{k_j}^\pm(\omega),\end{equation} as well as the associated random $(r,1)$ tensors $h=h_{kk_1\cdots k_r}$, which depend on the low frequency components of the solution. For simplicity, in (\ref{tensorintro0}) we have omitted the dependence on $t$. We always assume there is no \emph{pairing}, i.e. $k_{j'}\neq k_j$ if the corresponding $\pm$ signs are the opposite.
 
Notice that, the product of $\Psi$ with another expression
 \[\Psi_{k'}'=\sum_{k_1',\cdots,k_s'}h_{k'k_1'\cdots k_s'}'\prod_{j=1}^s\langle k_j'\rangle^{-\alpha}g_{k_j'}^\pm(\omega)\] can be written as a linear combination of similar multilinear expressions, depending on the possible set of pairings between $\{k_1,\cdots,k_r\}$ and $\{k_1',\cdots,k_s'\}$. For example, if  \[\begin{aligned}\Psi_k&=\sum_{a,b,c,d}h_{kabcd}\cdot 
 \langle a\rangle^{-\alpha}g_a\cdot \langle b\rangle^{-\alpha}\overline{g_b}\cdot\langle c\rangle^{-\alpha}g_c\cdot\langle d\rangle^{-\alpha}\overline{g_d},\\
 \Psi_{k'}'&=\sum_{u,v,w}h_{k'uvw}'\cdot\langle  u\rangle^{-\alpha}\overline{g_u}\cdot\langle v\rangle^{-\alpha}g_v\cdot\langle w\rangle^{-\alpha}g_w,
 \end{aligned}\]
 and in the summation we assume $a=u$ and $b=w$, then we have one representative component of $\Psi\cdot\Psi'$ being
 \begin{equation}\label{examplemerge0}\begin{aligned}\Phi_m&=\sum_{c,d,v}H_{mcdv}\cdot \langle c\rangle^{-\alpha}g_c\cdot\langle d\rangle^{-\alpha}\overline{g_d}\cdot\langle v\rangle^{-\alpha}g_v,\\
 H_{mcdv}&\sim\sum_{k+k'=m}\sum_{a,b}\langle a\rangle^{-\alpha}\langle b\rangle^{-\alpha}h_{kabcd}\cdot\langle a\rangle^{-\alpha}\langle b\rangle^{-\alpha}h_{k'avb}',
 \end{aligned}\end{equation} where we have replaced $|g_a|^2$ and $|g_b|^2$ by $1$. The process of going from $(h,h')$ to $H$ as above will be called \emph{merging}, which gives the main algebraic structure of the tensors studied in this work.
 
 For purposes related to independence of Fourier coefficients (which will be explained in Section \ref{introtrim}), we also need another process called \emph{trimming}, which means contracting against free Gaussians, namely going from $h=h_{kk_1\cdots k_r}$ to
 \[h_{kk_1\cdots k_s}'=\sum_{k_{s+1},\cdots k_r}h_{kk_1\cdots k_r}\prod_{j=s+1}^r\langle k_j\rangle^{-\alpha}g_{k_j}^{\pm},\] where $1\leq s\leq r$. Note that $\Psi_k$ defined by (\ref{tensorintro0}) is formally invariant under trimming.
 
 Now, as in \cite{DNY}, the central objects in our work will be the tensors $h$ (instead of the multilinear expressions $\Psi$), as well as suitable $L^2\to L^2$ operator norms for these tensors. The theory of random tensors then provides a natural framework for studying such tensors, in particular proving estimates for such norms that are consistent with merging and trimming. In Section \ref{intro2} below we provide a more detailed discussion of this theory, and application to the proof of Theorem \ref{main}.
 \section{Overview of the theory}\label{intro2} This section contains an overview of the theory of random tensors in the context of the proof of Theorem \ref{main}. We start with the definition and norms in Section \ref{tensorintro}, then develop the algebraic structure and main tools in Section \ref{tmintro}. In Section \ref{simmodel} we introduce a simplified model, which is analyzed in Section \ref{simansatz} using tools from our theory. In Section \ref{generalintro} we explain the changes needed when moving to full generality, and finally in Section \ref{plan} we list an outline for the rest of the paper.
 \subsection{Tensors and tensor norms}\label{tensorintro} As discussed in Section \ref{randomtensors}, the central objects in this work are tensors and their $L^2\to L^2$ operator norms. We therefore start with the following definition.
\begin{df}\label{tensornorms} Let $A$ be a finite index set, we will denote $k_A=(k_j:j\in A)$. A \emph{tensor} $h=h_{k_A}$ is a function $(\Zb^d)^A\to \Cb$, with $k_A$ being the input variables. The \emph{support} of $h$ is the set of $k_A$ such that $h_{k_A}\neq 0$. These tensors are usually random (i.e. depend on $\omega$ which belongs to the ambient probability space $\Theta$, though we may omit this dependence), hence the name \emph{random tensors}.

A \emph{partition} of $A$ is a pair of sets $(B,C)$ such that $B\cup C=A$ and $B\cap C=\varnothing$. For such $(B,C)$ define the norm $\|\cdot\|_{k_B\to k_C}$ by
 \[\|h\|_{k_B\to k_C}^2=\sup\bigg\{\sum_{k_C}\bigg|\sum_{k_B}h_{k_A}\cdot z_{k_B}\bigg|^2:\sum_{k_B}|z_{k_B}|^2=1\bigg\}.\]By duality we have that 
 \begin{equation}\label{duality}\|h\|_{k_B\to k_C}=\sup\bigg\{\bigg|\sum_{k_B,k_C}h_{k_A}\cdot z_{k_B}\cdot y_{k_C}\bigg|:\sum_{k_B}|z_{k_B}|^2=\sum_{k_C}|y_{k_C}|^2=1\bigg\},
 \end{equation} hence $\|h\|_{k_B\to k_C}=\|h\|_{k_C\to k_B}=\|\overline{h}\|_{k_B\to k_C}$. If $B=\varnothing$ or $C=\varnothing$ we get the norm $\|\cdot\|_{k_A}$ defined by
 \[\|h\|_{k_A}^2=\sum_{k_A}|h_{k_A}|^2.\] Note that trivially $\|h\|_{k_B\to k_C}\leq \|h\|_{k_A}$.
 
 Finally, we define a \emph{subpartition} to be a pair $(B,C)$ such that $B\cup C\subset A$ and $B\cap C=\varnothing$. In such case let $E=A\backslash (B\cup C)$, then $(B,C)$ is a partition of $A\backslash E$ so we can define
 \begin{equation}\label{tensornorms:sup}
  \|h\|_{k_B\to k_C}=\sup_{k_E}\|h_{(k_E,\cdot)}\|_{k_B\to k_C}.
   \end{equation}
 \end{df}
 \begin{rem} In the main proof the tensors may depend on other parameters, such as $t$ or $(k_F,\lambda_F)$, where $\lambda_F=(\lambda_j:j\in F)$, for some set $F$; in such cases we will write respectively $h_{k_A}=h_{k_A}(t)$ or $h_{k_A}=h_{k_A}(k_F,\lambda_F)$. Moreover, the norm (\ref{tensornorms:sup}) is designed only to treat some degenerate cases, so it will not appear in the simplified model of this section.
 \end{rem}
 \subsection{Tensor algebra and basic tools}\label{tmintro} In this section we develop the algebra of random tensors given by merging and trimming as described in Section \ref{randomtensors}, and some important estimates which are the basic tools of our theory. The precise versions will be given in Sections \ref{defplant} and \ref{multi} below. First we record the definition of \emph{pairing}.
 \begin{df}\label{def:pairing}Define $u^\zeta$ for $\zeta\in\{\pm\}$ by $(u^+,u^-)=(u,\overline{u})$ (in doing algebra we may view such $\zeta$ as $\pm 1$). Given $k_i,k_j\in\Zb^d$ with associated signs $\zeta_i,\zeta_j\in\{\pm\}$, we say $(k_i,k_j)$ is a \emph{pairing} if $\zeta_i+\zeta_j=0$ and $k_i=k_j$. We say it is \emph{over-paired} (or an \emph{over-pairing}) if $k_i=k_j=k_l$ for some $l\not\in\{i,j\}$.
 \end{df}
 \subsubsection{Semi-products and merging}\label{semipro} As described in Section \ref{randomtensors}, our theory will focus on the  analysis of the tensors $h_{kk_A}$ and associated multilinear expressions\footnote{In practice we will use $\eta_k=|g_k|^{-1}g_k$, which is uniformly distributed on the unit circle, instead of $g_k$.}
 \begin{equation}\label{expressionintro}\Psi_k:=\sum_{k_A}h_{kk_A}\prod_{j\in A}\langle k_j\rangle^{-\alpha}g_{k_j}^{\pm},\end{equation} where $k_A=(k_j:j\in A)$ as in Definition \ref{tensornorms} and $k_{j'}\neq k_j$ when the signs are the opposite, as well as the merging and trimming operations loosely described in that section.
 
 In fact the merging operation can be viewed as a special case of taking what we may call \emph{semi-products} for two or more tensors, which means taking tensor products and contracting over the given set of repeated indices---note that, the repeated indices can be indices appearing in both tensors, or the result of specific pairings between the tensors. For example, if $h_{abcde}$ and $h_{uvawx}'$ are two tensors, then their semi-product, under the assumptions $b=u$ and $c=w$, will be\footnote{In practice we will also have a $\pm$ sign for each index of each tensor, for example a $+$ sign for the index $a$ in $h_{abcde}$ or a $-$ sign for the index $x$ in $h_{uvawx}'$. When precisely defining the merging operations, see Definition \ref{defmerge}, we will restrict to the cases where for each repeated or paired index, its signs in the two tensors that it appears are the opposite (for example if the sign of $b$ in $h_{abcde}$ is $+$ then the sign of $u$ in $h_{uvawx}'$ must be $-$, assuming $b=u$). In this section (including the simple case Definition \ref{defmerge0}) we will ignore this issue for simplicity.}
 \[H_{devx}=\sum_{a,b,c}h_{abcde}h_{bvacx}'.\] In general, suppose $h_{k_A}$ are $h_{k_B}'$ are two tensors, and we have a particular set of repeated indices (coming from $A\cap B$ or pairings between $k_A$ and $k_B$). We may then assume these repeated indices belong to both $A$ and $B$, and define the corresponding semi-product as
 \[H_{k_{A\Delta B}}=\sum_{k_{A\cap B}}h_{k_A}h_{k_B}'.\] This can easily be generalized to semi-products of more than two tensors, for example the semi-product of the three tensors $h_{abcd}$, $h_{aefg}'$ and $h_{cuv}''$ under the assumptions $b=f$ and $g=v$ is
 \[H_{deu}=\sum_{a,b,c,g}h_{abcd}h_{aebg}'h_{cug}''.\] Note that for simplicity we are not considering over-pairings where an index is repeated three or more times, but such situations do appear in the actual definition of merging and need special treatment (though they are only associated with degenerate situations which are always much easier). See Definition \ref{defmerge} for details.
 
 Now, with the notion of semi-products, we can define (in this simple case without over-pairing) the merging of finitely many tensors $h^{(1)}, \cdots,h^{(r)}$ via a \emph{base tensor} $h$ as follows:
 \begin{df}[Merging: simple version]\label{defmerge0} Let $h^{(j)}=h_{k_jk_{A_j}}^{(j)}$ be tensors, where $1\leq j\leq r$, $A_j$ are pairwise disjoint, and let $h=h_{kk_1\cdots k_r}$ be the base tensor. Also fix a set of pairings among the sets $A_1,\cdots,A_r$ (which creates paired indices that will be viewed as repeated indices; as before, assume there is no pairing within each $A_j$ and no pairing involving more than two indices), then the merging of $h^{(1)},\cdots h^{(r)}$ via $h$, assuming the given choice of pairings, is defined to be the semi-product of $\widetilde{h}^{(1)},\cdots,\widetilde{h}^{(r)}$ and $h$, where (i) each $\widetilde{h}^{(j)}$ is $h^{(j)}$ multiplied by the product of $\langle k_\lf\rangle^{-\alpha}$ over all $\lf\in A_j$ that appear in some pairing, and (ii) in addition to the paired indices, each $k_j\,(1\leq j\leq r)$ is also viewed as a repeated index and is summed over, as it appears in both $h^{(j)}$ and $h$.
 \end{df} For example, if $h^{(1)}=h_{k_1abcd}^{(1)}$, $h^{(2)}=h_{k_2efg}^{(2)}$, $h^{(3)}=h_{k_3uvw}^{(3)}$ and $h=h_{kk_1k_2k_3}$, then the merging of $h^{(1)}$, $h^{(2)}$ and $h^{(3)}$ via $h$, under the assumptions $a=w$, $b=f$ and $g=v$, will be
\[H_{kcdeu}=\sum_{k_1,k_2,k_3,a,b,g}h_{kk_1k_2k_3}\cdot\langle a\rangle^{-\alpha}\langle b\rangle^{-\alpha}h_{k_1abcd}^{(1)}\cdot\langle b\rangle^{-\alpha}\langle g\rangle^{-\alpha}h_{k_2ebg}^{(2)}\cdot\langle g\rangle^{-\alpha}\langle a\rangle^{-\alpha}h_{k_3uga}^{(3)}.\] Similarly, the example (\ref{examplemerge0}) which gives a component of the product $\Psi\cdot\Psi'$, is the merging of $h_{kabcd}$ and $h_{k'uvw}'$ via $\widetilde{h}_{mkk'}=\mathbf{1}_{m=k+k'}$, under the assumption $a=u$ and $b=w$.

The general version of Definition \ref{defmerge0}, which includes the signs of indices, dependence on other parameters and additional structures, as well as over-pairings, will be given in Definition \ref{defmerge}.
\subsubsection{Key bilinear and multilinear bounds} Our first basic tool is the following lemma (together with the multilinear version thereof), where the $\|\cdot\|_{k_B\to k_C}$ norms for semi-products of tensors, as defined in Definition \ref{tensornorms}, are estimated by the $\|\cdot\|_{k_B\to k_C}$ norms of the individual tensors.
\begin{prop}[A bilinear estimate]\label{bilineartensor0} Let $h_{k_A}$ and $h_{k_B}'$ be two tensors, assume that all repeated indices are already in $A\cap B$. Then for any partition $(X,Y)$ of $A\Delta B$, the semi-product $H$ of $h$ and $h'$ satisfies that
\[\|H\|_{k_X\to k_Y}\leq \|h\|_{k_{(X\cup B)\cap A}\to k_{Y\cap A}}\cdot\|h'\|_{k_{X\cap B}\to k_{(Y\cup A)\cap B}}.\]
\end{prop} For example, we have
\[\|H\|_{dv\to ex}\leq\|h\|_{abcd\to e}\cdot\|h'\|_{v\to xabc},\quad\mathrm{where}\quad H_{devx}=\sum_{a,b,c}h_{abcde}h_{bvacx}'.\] Note that in the setting $b=u$ and $c=w$ as in Section \ref{semipro}, we can identify $h_{bvacx}'$ with $h_{uvawx}'$ and the norm $\|\cdot\|_{v\to xabc}$ with $\|\cdot\|_{v\to xauw}$; the same comment applies to Lemma \ref{multibound0} below.

An equivalent form of Proposition \ref{bilineartensor0} will be stated and proved in Proposition \ref{bilineartensor}.
\begin{prop}[A multilinear estimate]\label{multibound0} Let $h_{k_{A_j}}^{(j)}\,(1\leq j\leq m)$ be tensors, assume all repeated indices coming from pairings between any $A_i$ and $A_j$ are already in $A_i\cap A_j$. Let $H=H_{A}$ be the semi-product of the $h^{(j)}$'s, where $A=A_1\Delta\cdots\Delta A_m$, then for any partition $(X,Y)$ of $A$ we have
\begin{equation}\label{multibound00}\|H\|_{k_X\to k_Y}\leq\prod_{j=1}^m\|h^{(j)}\|_{k_{(X\cap A_j)\cup B_j}\to k_{(Y\cap A_j)\cup C_j}},\end{equation} where
\[B_j=\bigcup_{\ell>j}(A_j\cap A_\ell),\quad C_j=\bigcup_{\ell <j}(A_j\cap A_\ell).\]
\end{prop} 
For example, we have
\[\|H\|_{e\to ud}\leq\|h\|_{abc\to d}\|h'\|_{eg\to ab}\|h''\|_{cug},\quad\mathrm{where}\quad H_{deu}=\sum_{a,b,c,g}h_{abcd}h_{aebg}'h_{cug}''.\]

An equivalent form of Proposition \ref{multibound0} will be stated and proved in Proposition \ref{multibound}.
  \subsubsection{Trimming}\label{introtrim} In the course the main proof, when considering the tensors $h_{kk_A}$ and associated multilinear expressions $\Psi_k$ as in (\ref{expressionintro}), we will always assume that the tensor $h$ is \emph{independent} with the Gaussians $g_{k_j}$, in order for the large deviation estimates (such as Lemma \ref{largedev0}) to be applicable. In practice, this is guaranteed by requiring that $\langle k_j\rangle\geq R$ in the support of $h_{k_A}$ for some $R$, and that $h_{k_A}$ is a Borel function of $\{g_k:\langle k\rangle<R\}$. A problem then occurs, say when merging two tensors $h=h_{k_A}$ and $h'=h_{k_B}'$ with cutoffs $R_1$ and $R_2$ as above, because the merged tensor $H$ is a Borel function of $\{g_k:\langle k\rangle<\max(R_1,R_2)\}$ and may not be independent with $g_{k_j}$ for $j\in A\Delta B$. Because of this, we will introduce the operation of \emph{trimming} as follows:
  \begin{df}[Trimming: simple version]\label{deftrim0} Let $h=h_{kk_A}$ be a tensor, assume for each $j\in A$ there is a dyadic $N_j$ such that $h$ is supported in $N_j/2<\langle k_j\rangle\leq N_j$. Then, for any $R$, the trimming of $h$ at frequency $R$ is defined to be the contraction against free Gaussians, namely
  \[h_{kk_{A'}}'=\sum_{k_{A\backslash A'}}h_{kk_A}\prod_{j\in A\backslash A'}\langle k_j\rangle^{-\alpha}g_{k_j}^\pm,\] where $A'=\{j\in A:N_j\geq R\}$. Note that those $g_{k_j}$ where $j\in A'$ are independent with those $g_{k_j}$ where $j\in A\backslash A'$. In particular we recover the expression $\Psi_k$ in (\ref{expressionintro}) if $A'=\varnothing$.
  \end{df} For example, if $h=h_{kabcd}$, where $N_1/2<\langle a\rangle\leq N_1$ etc., and assume $N_1\leq N_3<R\leq N_2\leq N_4$, then the trimming of $h$ at frequency $R$ will be
  \[h_{kbd}'=\sum_{a,c}h_{kabcd}\cdot\langle a\rangle^{-\alpha}g_a^\pm \cdot\langle c\rangle^{-\alpha}g_c^\pm.\]
  
The general version of Definition \ref{deftrim0}, which includes the signs of indices, as well as dependence on other parameters and additional structures, will be given in Definition \ref{deftrim}.
\subsubsection{Method of descent} Our second basic tool is the following lemma, where the $\|\cdot\|_{k_B\to k_C}$ norms for the contraction of a tensor against independent free Gaussians are estimated by the $\|\cdot\|_{k_B\to k_C}$ norms of this tensor. This inequality has an elegant form, and we believe it is of independent interest in the study of random matrices.
\begin{prop}\label{gausscont0} Let $h_{k_A}$ be a tensor, $A'$ be a subset of $A$ such that $\{g_{k_j}:j\in A\backslash A'\}$ is independent with $h_{k_A}$. Let $h'=h_{k_{A'}}'$ be the contraction of $h$ against the free Gaussians $\{g_{k_j}^\pm:j\in A\backslash A'\}$, namely
\[h_{k_{A'}}'=\sum_{k_{A\backslash A'}}h_{k_A}\prod_{j\in A\backslash A'}g_{k_j}^\pm,\] then for any partition $(X',Y')$ of $A'$, with high probability we have\footnote{In practice this will have a small power loss $M^\theta$ where $\theta$ is an arbitrary small number and $M$ is the size of $k_A\in(\Zb^d)^A$; see Propositions \ref{gausscont}--\ref{gausscont2}.}
\[\|h'\|_{k_{X'}\to k_{Y'}}\lesssim\sup_{(X,Y)}\|h\|_{h_X\to h_Y},\]where $(X,Y)$ runs over all partitions of $A$ such that $X'\subset X$ and $Y'\subset Y$.
\end{prop} For example, under the independence assumption, with high probability we have
\[\|h'\|_{b\to d}\lesssim\max(\|h\|_{abc\to d},\|h\|_{ab\to cd},\|h\|_{bc\to ad},\|h\|_{a\to bcd}),\quad \mathrm{where}\quad h_{bd}'=\sum_{a,c}h_{abcd}g_a^\pm g_c^\pm.\]

A more precise version of Proposition \ref{gausscont0} will be stated and proved in Proposition \ref{gausscont}. A slightly different version due to technical reasons will be stated and proved in Proposition \ref{gausscont2}.
 \subsection{A simple model}\label{simmodel} We now turn to the proof of Theorem \ref{main}. In this section we introduce a much simplified model for (\ref{nls}) and (\ref{nlstrunc}) that still preserves the main difficulties.
 
 First, we will replace the nonlinearities by the $\Nc_{\mathrm{np}}$ defined in (\ref{nopairintro}). This $\Nc_{\mathrm{np}}$ is essentially the result of Wick ordering and a suitable gauge transform, but avoids the complications linked to over-pairings and deviation of mass around its expected value.
 
Second, we will remove the time variable. Indeed, if we believe our solution is close to a linear solution, and thus restrict to functions $u$ whose spacetime Fourier transform looks like $
\widehat{u_k}(\xi)\sim u_k\cdot\psi(\xi+|k|^2)$ with some Schwartz function $\psi$ and some function $u_k$ of $k$ only, then by a formal calculation using Duhamel's formula, this $u_k$ will satisfy a fixed-point equation that essentially looks like \begin{equation}\label{modelintro0}u_k=\frac{g_k}{\langle k\rangle^{\alpha}}-i\sum_{k_1-\cdots +k_p=k;\,\Omega=0}u_{k_1}\overline{u_{k_2}}\cdots u_{k_p},\end{equation} where $\Omega:=|k|^2-|k_1|^2+\cdots-|k_p|^2$, and in (\ref{modelintro0}) we also assume no-pairing as in (\ref{nopairintro}).

Third, consistent with the setting of Section \ref{tmintro}, in analyzing (\ref{modelintro0}) we will disregard any over-pairings, and assume, when merging tensors, that no index is repeated more than twice.
 \subsection{The core ansatz}\label{simansatz} We now start the analysis of (\ref{modelintro0}). For simplicity we denote the terms on the right hand side of (\ref{modelintro0}) by \begin{equation}\label{deffintro}\langle k\rangle^{-\alpha}g_k=:f_k,\qquad -i\sum_{k_1-\cdots +k_p=k;\,\Omega=0}u_{k_1}\overline{u_{k_2}}\cdots u_{k_p}=:\Mc_{\mathrm{np}}(u,\cdots,u)_k,\end{equation} where $\Mc_{\mathrm{np}}$ is an $\Rb$-multilinear operator of degree $p$, so that (\ref{modelintro0}) reads \begin{equation}\label{modelintro1}
u_k = f_k + \Mc_{\mathrm{np}}(u,\cdots, u)_k,
\end{equation} We also introduce the canonical truncations of (\ref{modelintro1}), namely \begin{equation}\label{modelintro2}
(u_N)_k = \Pi_N f_k + \Mc_{\mathrm{np}}(u_N,\cdots, u_N)_k,
\end{equation} and define $y_N$  by 
\begin{equation}
y_N = u_N-u_{N/2};\quad u_N=\sum_{N'\leq N}y_{N'}.
\end{equation}Note that we do not put $\Pi_N$ before the nonlinearity in (\ref{modelintro2}); however in this model we still assume $y_N$ and $u_N$ are supported in $\langle k\rangle\leq N$. Under these assumptions, $y_N$ satisfies the equation
\begin{equation}\label{modelintro3}(y_N)_k=\Delta_N f_k+\sum_{\max(N_1,\cdots,N_p)=N}\Mc_{\mathrm{np}}(y_{N_1},\cdots,y_{N_p})_k.
\end{equation}

The core ansatz for $y_N$ will be constructed by induction (assuming $u_{N'}$ and $y_{N'}$ are already defined for $N'<N$). Recall that in \cite{DNY} the analogue ansatz for $y_N$ contains three parts: $\Delta_N f$ which corresponds to the linear evolution, the terms corresponding to the random averaging operators $\Pc_{NL}$, and a remainder $z_N$ of higher regularity. We start by a description of this simple case in Section \ref{raointro}. In order to prove Theorem \ref{main} which covers the full subcritical regime $s>s_{pr}$, in Section \ref{sec2.4:randomtensors}, we will further unravel the propagation of randomness from the remainder and make higher order expansions using the random tensors introduced in Sections \ref{tensorintro}--\ref{tmintro}.
\subsubsection{Random averaging operators}\label{raointro}
As stated in Section \ref{rao0}, the random averaging operator $\Pc_{NL}$ describes the high-low interaction, which in this simple model is defined by 
\begin{equation}\label{PNL}
\Pc_{NL} (y) = \Mc_{\mathrm{np}}(y, u_L, \cdots, u_L),
\end{equation} where\footnote{In \cite{DNY} we used $L=N^{1-\delta}$; here we need a smaller value of $L$ which works better in the general setting.} $L= N^{\delta}$ for some small $\delta$. Note that when we discuss the ansatz for $y_N$, $u_L$ has already been defined by the induction hypothesis. The ansatz for $y_N$ then includes the following term:
\begin{equation}\label{simpleansatz1}
(1 + \Pc_{NL}+  \Pc^2_{NL} + \Pc^3_{NL}+\cdots)\,\Delta_N f = (1-\Pc_{NL})^{-1} \Delta_N f,
\end{equation}
where convergence is guaranteed by the operator bound (\ref{rao1}) for $\Pc_{NL}$.  In the case $p=3$ for example, we may represent  the terms in (\ref{simpleansatz1}) by means of the following \emph{iteration trees}
\begin{equation}\label{symbol1}
\begin{aligned}
\bullet: \Delta_N f,\quad \circ: u_L,\quad \scaleto{\<dcc_black>}{0.35cm}: \Mc_{\mathrm{np}}(\Delta_N f, u_L, u_L)=\Pc_{NL} (\Delta_N f),\\ 
\scaleto{\<(dcc)cc_black>}{0.6cm}: \Pc^2_{NL} (\Delta_N f),\quad
\scaleto{\<((dcc)cc)cc_black>}{0.8cm}: \Pc^3_{NL} (\Delta_N f), \quad\textrm{etc.}
\end{aligned}\end{equation}
For each term in (\ref{symbol1}) we can define the associated random $(1,1)$ tensor (or equivalently random matrix). For example the random $(1,1)$ tensor associated to $\bullet=\Delta_N f$ is just the identity matrix, while for the iteration tree $\<dcc_black>$, the associated random $(1,1)$ tensor $h^{\scaleto{\<dcc_black>}{7pt}}$ is such that
\begin{equation}\label{tensorterm1}
(\scaleto{\<dcc_black>}{0.35cm})_k = \sum_{k_1}h^{\scaleto{\<dcc_black>}{7pt}}_{kk_1} \cdot\Delta_N f_{k_1},
\end{equation} where $f_{k_1}$ is as in (\ref{deffintro}). By (\ref{deffintro}) and (\ref{PNL}), we have the formula (where $\Omega=|k|^2-|k_1|^2+|k_2|^2-|k_3|^2$)
\begin{equation}\label{tensorterm11}
h^{\scaleto{\<dcc_black>}{7pt}}_{kk_1} = 
-i \sum_{\substack{k_1-k_2 +k_3=k;\,\Omega=0\\\langle k_2\rangle, \langle k_3\rangle\leq L,\,\langle k\rangle \leq N}}\overline{(u_L)_{k_2}}\cdot (u_L)_{k_3}.
\end{equation}
Similarly we can define the random $(1,1)$ tensors associated to other iteration trees in (\ref{symbol1}) such as $\<(dcc)cc_black>$. Since $u_L$ (represented by $\circ$) has less importance in the estimates (in fact they will be trimmed out, see the arguments below), all these random $(1,1)$ tensors can be treated in a similar way in our proof. Hence for simplicity we will denote them by the single notation $h^{\scaleto{\<d_black>}{7pt}}$.
\subsubsection{Random tensors}\label{sec2.4:randomtensors} With the random $(1,1)$ tensors $h^{\scaleto{\<d_black>}{7pt}}$ defined as above, we proceed to construct the random $(r,1)$ tensors in the ansatz for $y_N$ by induction, with $h^{\scaleto{\<d_black>}{7pt}}$ being the base case. These tensors arise from high order iterations of the equation (\ref{modelintro3}). We start with a simple case, namely the random $(2,1)$ tensor terms in the ansatz for $y_N$, by using the following iteration trees (assume $p=3$):
\begin{equation}\label{symbol2}\begin{aligned}
\scaleto{\<ddc_black>}{0.35cm}: \Mc_{\mathrm{np}}(\Delta_{N_{\lf_1}}f, \Delta_{N_{\lf_2}}f, u_L), \quad \scaleto{\<(ddc)cc_black>}{0.6cm}: \Pc_{NL} ( \<ddc_black>), \quad \scaleto{\<((ddc)cc)cc_black>}{0.8cm}: \Pc^2_{NL} ( \<ddc_black>),\\
\scaleto{\<(dcc)dc_black>}{0.6cm}:  \Mc_{\mathrm{np}}(\Mc_{\mathrm{np}}(\Delta_{N_{\lf_1}}f, u_L, u_L), \Delta_{N_{\lf_2}}f, u_L),\quad\mathrm{etc.}
\end{aligned}
\end{equation}
where $(\lf_1, \lf_2)$ is not a pairing (i.e. $k_{\lf_1}\neq k_{\lf_2}$ in (\ref{tensorterm2}) below, if the corresponding signs are the opposite), $N_{\lf_1}, N_{\lf_2}>L$ and $\max (N_{\lf_1}, N_{\lf_2}) = N$. These terms are similar to the high-low interactions in (\ref{simpleansatz1}) and hence will also be added to the ansatz for $y_N$.
  We can define the  random $(2, 1)$ tensors associated to terms in (\ref{symbol2}), proceeding similarly as in (\ref{tensorterm1}) and (\ref{tensorterm11}). Once again all these random $(2, 1)$ tensors can be treated in a similar way in our proof, hence for simplicity we also denote them by the single notation $h^{\scaleto{\<dd_black>}{7pt}}$, such that the $k$-th mode of terms in (\ref{symbol2}) are given by
 \begin{equation}\label{tensorterm2} \sum_{N_{\lf_i}/2<\langle k_{\lf_i}\rangle\leq N_{\lf_i},\ i=1,2} h^{\scaleto{\<dd_black>}{7pt}}_{kk_{\lf_1}k_{\lf_2}} (\Delta_{N_{\lf_1}}f_{k_{\lf_1}})^{\pm} (\Delta_{N_{\lf_2}}f_{k_{\lf_2}})^{\pm}.
 \end{equation}
 
Synthesizing the structures of these random $(1,1)$ and $(2,1)$ tensors, we describe the associated $(r,1)$ tensors in general. To that effect we introduce the {\emph{skeleton tree} $\Lc$ containing all solid leaves\footnote{In addition, we remove edges connecting a node to its \emph{only} child; they correspond to the random averaging operator in (\ref{PNL}) and can be dealt with as in Section \ref{raointro}.} $\bullet$ in the iteration trees associated to the $(r,1)$ tensors. 
 For each leaf $\lf\in \Lc$ we also attach a frequency $N_\lf$ (such as $N_{\lf_1}$ and $N_{\lf_2}$ in (\ref{symbol2}) above) and a sign $\zeta_\lf\in\{\pm\}$ (for simplicity we will not explicitly write $\zeta_\lf$ below). Define also the frequency of the skeleton tree to be $N$, which always equals the maximum of $N_\lf$; in particular if $\lf$ is the only leaf then $N=N_\lf$. For example, the term in (\ref{symbol2}) that has the iteration tree $\<(dcc)dc_black>$ will correspond to the skeleton tree $\<dd_black>$, or $\Lc =\{\lf_1, \lf_2\}$, with two leaves, no pairing, and $\max (N_{\lf_1}, N_{\lf_2}) = N$. In such terms we always assume $N_\lf>L$ so that the tensor $h^\Lc$, which is a Borel function of $u_L$, is \emph{independent} with the Gaussians $\Delta_{N_\lf}f$.

Let us now show with an example how the inductive definition of the $(r,1)$ tensors associated to the ansatz for $y_N$ proceeds using high order iterations, first in the no-pairing case. For $p=3$, consider the high order iteration term such as
\begin{equation}\label{tensorterm3}
\begin{aligned}
\scaleto{\<(ddc)d(dcc)_black>}{0.6cm}&= \Mc_{\mathrm{np}}(\scaleto{\<ddc_black>}{0.35cm}, \bullet, \scaleto{\<dcc_black>}{0.35cm});\\
\left(\scaleto{\<(ddc)d(dcc)_black>}{0.6cm}\right)_{k}&=\sum_{N_{\lf_i}/2<\langle k_{\lf_i}\rangle\leq N_{\lf_i};\,1\leq i\leq 4} \left(\sum_{k_1,k_2,k_3} h_{kk_1k_2k_3} h^{\Lc_1}_{k_1k_{\lf_1}k_{\lf_2}} h^{\Lc_2}_{k_2k_{\lf_3}} h^{\Lc_3}_{k_3k_{\lf_4}}\right) \prod_{i=1}^4 (\Delta_{N_{\lf_i}}f_{k_{\lf_i}})^{\pm},
\end{aligned}
\end{equation}
where $\Lc_1 = \{\lf_1, \lf_2\}$ corresponds to the iteration tree $\<ddc_black>$ and has frequency $N_1$, $\Lc_2 = \{\lf_3\}$ corresponds to $\bullet$ and has frequency $N_2$, and $\Lc_3 = \{\lf_4\}$ corresponds to the iteration tree $\<dcc_black>$ and has frequency $N_3$. Also note that $N= \max (N_1, N_2, N_3)$ by (\ref{modelintro3}). By the definition of $\Mc_{\mathrm{np}}$ in (\ref{deffintro}), the tensor $h_{kk_1k_2k_3}$ in (\ref{tensorterm3}) is (with the no-pairing restrictions which we omit)
\begin{equation}\label{formhintro}
h_{kk_1k_2k_3} = \mathbf{1}_{k=k_1-k_2+k_3}\cdot \mathbf{1}_{|k|^2=|k_1|^2-|k_2|^2+|k_3|^2}.
\end{equation}
Consider the case  when there is no pairing among $\{\lf_1,\lf_2,\lf_3,\lf_4\}$ (note that $(\lf_1, \lf_2)$ is already not a pairing in (\ref{symbol2})). 
By (\ref{tensorterm3}) and Definition \ref{defmerge0}, the random tensor $h^\Lc$ associated to  the iteration tree $\<(ddc)d(dcc)_black>$ is the merging of $h^{\Lc_1}$, $h^{\Lc_2}$, $h^{\Lc_3}$ via $h$, assuming there is no pairing; it has skeleton tree $\<(dd)dd_black>\,$, or $\Lc = \{\lf_1,\lf_2,\lf_3,\lf_4\}$.

 It becomes unnecessarily complex to keep track of the iteration or skeleton trees such as $\<(dd)dd_black>\,$ and $\<(ddc)d(dcc)_black>$, as we iterate further and increase the depth. It turns out, see Section \ref{sec:2.4.3}, that the desired estimates for the random $(r,1)$ tensors depend \emph{only} on the set of solid leaves and their corresponding frequencies, not on the tree structure, except for some minor corrections. Therefore it will suffice to consider structures that we will refer to as {\emph{flattened trees}} below, provided we keep the necessary information of the trees in a \emph{memory set} $\Yc$, which will yield the minor corrections alluded above. The pair $(\Lc, \Yc)$, where $\Lc$ is viewed as a \emph{set}, then plays the role of the trees (such as $\<(dd)dd_black>$ and $\<(ddc)d(dcc)_black>$). The process of viewing $\Lc$ as a set---forgetting its tree structure---and finding the set $\Yc$ associated to the trees, is then called the \emph{flattening} of trees. More precisely, every time we merge the tensors, the flattening of trees proceeds by putting an element $\pf$ into $\Yc$ and set $N_\pf$ to be the \emph{second maximum} among  all frequencies of the trees of the merged tensors. 
For example, in the situation of (\ref{tensorterm3}), the skeleton tree $\Lc = \{\lf_1,\lf_2,\lf_3,\lf_4\}$, which corresponds to the iteration tree $\<(ddc)d(dcc)_black>$, comes from merging $h^{\Lc_1}$, $h^{\Lc_2}$ and $h^{\Lc_3}$, so we have $\pf_1\in \Yc$ and $N_{\pf_1}$ equals the second maximum  among  $\{N_1, N_2, N_3\}$. Furthermore $\Lc_1$, which corresponds to the iteration tree $\<ddc_black>$, is constructed by merging the tensors $h^{\bullet}$ (from $\Delta_{N_{\lf_1}} f$), $h^{\bullet}$ (from $\Delta_{N_{\lf_2}} f$) and\footnote{Here $h^{\circ}_{k}=(u_L)_k$ can be understood as a $(0,1)$ tensor which has no input variable.} $h^{\circ}$ (from $u_L$), hence we have one more element $\pf_2$ in $\Yc$, and $N_{\pf_2}$  equals the second maximum  among  $\{N_{\lf_1}, N_{\lf_2}\}$. Since $h^{\Lc_2}$ and $h^{\Lc_3}$ are $(1,1)$ tensors which are defined directly without merging, we obtain the memory set $\Yc$ associated to the skeleton tree $\<(dd)dd_black>\,$, namely $\{\pf_1, \pf_2\}$. With this flattening process, we can forget the tree structures $\<(dd)dd_black>\,$ and $\<(ddc)d(dcc)_black>$, and replace it by the flattened tree $\<dddd_black>$ together with $\Yc$, so we may also denote $h^{\Lc}=h^{(\scaleto{\<dddd_black>}{0.3cm},\,\Yc)}$, where $\Yc=\{\pf_1,\pf_2\}$. 

We now need to trim\footnote{This corresponds to removing $\circ$'s and low-frequency $\bullet$'s from the iteration trees, or removing low-frequency leaves from the skeleton and flattened trees. In the main proof, in addition to this trimming after merging, we also need to trim the tensors \emph{before} merging; see (\ref{newpsi}) and (\ref{newH}).} the tensor $h^{(\scaleto{\<dddd_black>}{0.3cm},\,\Yc)}$ as above at frequency $2L$, in order to maintain the property $N_\lf>L$ and the independence between the Gaussians $\Delta_{N_\lf} f$ and the tensor which contains the low frequency components. When $N_\lf >L$ for all $\lf \in \Lc$, no trimming is needed and we get the same random $(4,1)$ tensor $h^{(\scaleto{\<dddd_black>}{0.3cm},\,\Yc)}$, with the $k$-th mode of the corresponding term in the ansatz being \begin{equation}\label{tensorterm5}
\sum_{N_{\lf_i}/2<\langle k_{\lf_i}\rangle\leq N_{\lf_i};\,1\leq i\leq 4} h^{(\scaleto{\<dddd_black>}{0.3cm},\,\Yc)}_{kk_{\lf_1}k_{\lf_2}k_{\lf_3}k_{\lf_4}}\cdot \prod_{i=1}^4 (\Delta_{N_{\lf_i}}f_{k_{\lf_i}})^\pm.
\end{equation}
Instead, if $N_{\lf_4} \leq L$, then by Definition \ref{deftrim0} the trimmed tensor is a $(3,1)$ tensor:
\begin{equation}
h^{(\scaleto{\<ddd_black>}{0.3cm},\,\Yc)}_{kk_{\lf_1}k_{\lf_2}k_{\lf_3}} = \sum_{k_{\lf_4}} h^{(\scaleto{\<dddd_black>}{0.3cm},\,\Yc)}_{kk_{\lf_1}k_{\lf_2}k_{\lf_3}k_{\lf_4}}\cdot  (\Delta_{N_{\lf_4}}f_{k_{\lf_4}})^\pm
\end{equation} and the 
$k$-th mode of the corresponding term in the ansatz should be \begin{equation}\label{tensorterm4}
\sum_{N_{\lf_i}/2<\langle k_{\lf_i}\rangle\leq N_{\lf_i};\,1\leq i\leq 3} h^{(\scaleto{\<ddd_black>}{0.3cm},\,\Yc)}_{kk_{\lf_1}k_{\lf_2}k_{\lf_3}}\cdot \prod_{i=1}^3  (\Delta_{N_{\lf_i}}f_{k_{\lf_i}})^\pm
\end{equation} which in fact is the same as (\ref{tensorterm5}). This means that the trimming process only changes the point of view by which we regard the terms in the ansatz, but not the terms themselves.

Finally, we consider the case when pairings (see Definition \ref{def:pairing}) occur in the merging process. In the above example, if we have a pairing $(\lf_2,\lf_4)$ in the merging process (\ref{tensorterm3}), then instead of the flattened tree $\<dddd_black>$ we will have a flattened tree \emph{with pairing}, namely\[\scaleto{\<pdddd_black>}{0.6cm};\] for simplicity we also assume $N_{\lf_i}>L$ for $1\leq i\leq 4$, i.e. no trimming is needed. The set of paired leaves is denoted by $\Pc =\{\lf_2, \lf_4\}$, and the set of  unpaired leaves is denoted by $\Uc =\{\lf_1,\lf_3\}$. In this case we still merge $h^{\Lc_1}$, $h^{\Lc_2}$ and $h^{\Lc_3}$ via $h$ using Definition \ref{defmerge0} as above, but assume now $(\lf_2, \lf_4)$ is a pairing, i.e. restricting $k_{\lf_2}=k_{\lf_4}$ in the sum (\ref{tensorterm3}). The merged tensor, denoted by $h^{({\scaleto{\<pdddd_black>}{0.45cm}},\,\Yc)}$, is in fact a random $(2,1)$ tensor as only $k_\lf$ for \emph{unpaired} leaves $\lf\in\Uc$ are input variables. The $k$-th mode of the corresponding term in the ansatz is then  \begin{equation}
\sum_{N_{\lf_i}/2<|k_{\lf_i}|\leq N_{\lf_i},\,i\in\{1,3\}}h^{({\scaleto{\<pdddd_black>}{0.45cm}},\,\Yc)}_{kk_{\lf_1}k_{\lf_3}}\cdot  (\Delta_{N_{\lf_1}}f_{k_{\lf_1}})^\pm (\Delta_{N_{\lf_3}}f_{k_{\lf_3}})^\pm.
\end{equation}

In summary, in order to construct a random $(r,1)$ tensor, we start with a high order iteration which can be understood as the process of merging several lower order tensors as in (\ref{tensorterm3}), and then trim the merged tensor at a given frequency $2L$. This trimmed tensor is $(r,1)$ tensor we seek for the ansatz for $y_N$.
\subsubsection{The core ansatz}\label{sec:2.4.3} Given a large parameter $D$, based on the above random $(r,1)$ tensors, we construct the ansatz for $y_N$ as follows:    
\begin{equation}\label{sec2.4:ansatz}
(y_N)_k = \sum_{(\Lc,\Yc)} h^{(\Lc,\Yc)}_{k k_\Uc}\cdot \prod_{\lf\in \Uc} (\Delta_{N_{\lf}}f_{k_{\lf}})^\pm   \, \, + \, \,  (z_N)_k,
\end{equation}
where $z_N$ is a smooth remainder, and the sum is taken over all flattened trees $\Lc$ with frequency $N$ and cardinality $|\Lc|\leq D$, and all possibilities of $\Yc$ that arise from the above inductive process. In (\ref{sec2.4:ansatz}), $\Uc$ is the set of unpaired leaves in $\Lc$, and denote by $\Pc =\Lc\backslash \Uc$ the set of paired leaves in $\Lc$.

The main a priori estimates contain the bounds for various operator norms for the tensors $h^\Lc$ (as well as a high-regularity bound for the remainder $z_N$, which we omit). Here we look at a simplified example\footnote{See Proposition \ref{mainprop} for the full detailed version. In particular there are distinctions between $h^{(*,0)}$ and $h^{(*,1)}$ tensors, which we omit here.}: for any partition $(B, C)$ of $\Uc$, we would like to show
\begin{equation}\label{sec24:operatornorm}
\left\| h^{(\Lc,\Yc)}_{kk_\Uc}\right\|_{kk_B\to k_C} \leq \prod_{\lf\in \Uc}N_\lf^{\beta}\cdot \big(\max_{\lf\in C}N_\lf\big)^{-\beta}\cdot \prod_{\lf\in \Pc} N_\lf^{-\varepsilon_1} \prod_{\pf\in \Yc} N_\pf^{-\delta_1},
\end{equation}
where $\beta$ is a constant which is a little bit smaller than $\alpha$, and $\varepsilon_1\ll 1$ and $\delta_1$ is small compared to $\varepsilon_1$. In particular, the factor $\prod_{\pf\in \Yc} N_\pf^{-\delta_1}$ shows the decay we gain from the tree structures (e.g. $\<(dd)dd_black>\,$), hence we only need to keep the flattened trees (e.g. $\<dddd_black>$) and the memory set $\Yc$ abstracted from the full tree structures.

We will prove (\ref{sec24:operatornorm}) by induction. The key step here is to show that if (\ref{sec24:operatornorm}) is true for some tensors $h^{(\Lc_j,\Yc_j)}=h_{k_jk_{\Uc_j}}^{(\Lc_j,\Yc_j)}$ where $1\leq j\leq p$, then it also holds for the tensor $h^{(\Lc,\Yc)}=h_{kk_\Uc}^{(\Lc,\Yc)}$ which is obtained by merging and trimming those tensors as in Section \ref{sec2.4:randomtensors}. This argument, which is the center of the whole paper, contains \textbf{three main ingredients}:

(1) The \textbf{inequalities} associated with the algebra of tensors, namely Propositions \ref{multibound0} and \ref{gausscont0} (and their precise versions in Section \ref{prelimtensor}). Note that these are problem non-specific and are not limited to Schr\"{o}dinger equations.

(2) The operator norm bounds for the base tensor $h$ that appears in the merging process. The form of $h$ is similar to (\ref{formhintro}), and operator norm bounds for $h$ follow from various \textbf{counting estimates} and Schur's Lemma. This is proved in Proposition \ref{final}.

(3) A particular \textbf{selection algorithm}. This is crucial when we apply Proposition \ref{multibound0}, since even though $H$ on the left hand side of (\ref{multibound00}) does not depend on the \emph{order} of the tensors $h^{(j)}$, the right hand side does. Therefore we have to follow a particular algorithm in order to go from bounds of $h^{(\Lc_j,\Yc_j)}$ to bounds of $h^{(\Lc,\Yc)}$. This algorithm is described in the proof of Proposition \ref{algorithm1}.

Finally, in addition to the operator norms, we need to control one more \emph{localization} norm for the tensor $h^{(\Lc,\Yc)}$, which localizes it around the hyperplane $k= \sum_{\lf\in \Uc} \zeta_\lf k_\lf$, where $\zeta_\lf\in\{\pm\}$ is the sign of $\lf$. This norm is essentially a weighted $L^2$ or Hilbert-Schmidt norm, and the corresponding estimates roughly look like
\begin{equation}\bigg\|\bigg(1+\frac{1}{L^2}\big|k-\sum_{\lf\in\Uc}\zeta_\lf k_\lf\big|\bigg)^\kappa h_{kk_\Uc}^{(\Lc,\Yc)}\bigg\|_{kk_\Uc}
\leq\prod_{\lf\in \Uc} N_\lf^{\beta} \prod_{\lf\in\Pc}N_\lf^{-\varepsilon_1}\prod_{\pf\in\Yc}N_\pf^{-\delta_1},
\end{equation}
where $\beta$, $\varepsilon_1$ and $\delta_1$ are the same as in (\ref{sec24:operatornorm}) and $\kappa$ is a large enough constant. Such localizations can be understood as our tensor $h^\Lc$ being close to a \emph{multilinear Fourier multiplier}. See Figure \ref{hyperplane} for an illustration of the regions around which the $(1,1)$ and $(2,1)$ tensors are localized.
 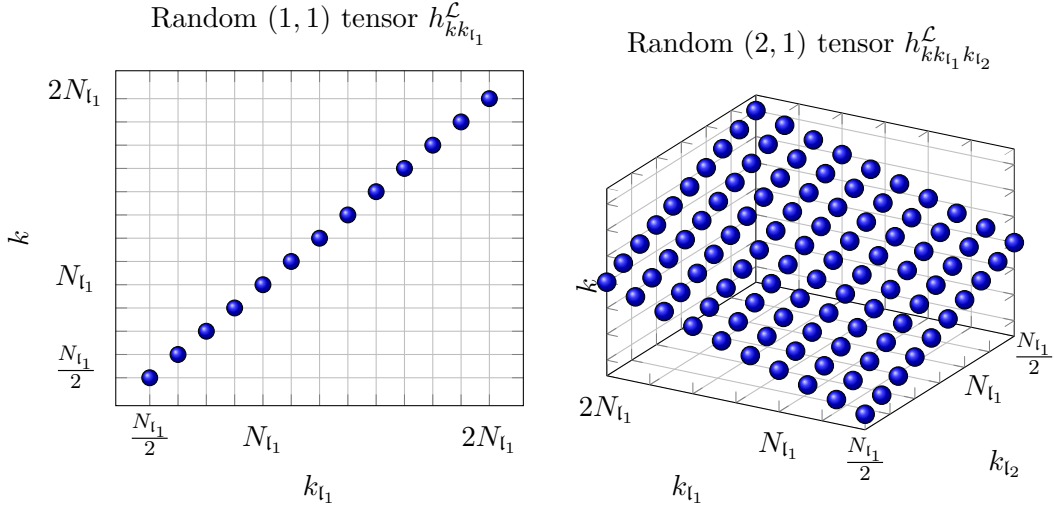
\begin{figure}
 \begin{minipage}[t]{0.45\textwidth}
\begin{tikzpicture}
\begin{axis}
[colormap={bw}{gray(0cm)=(0); gray(1cm)=(0)}, xlabel=$k_{\lf_1}$,ylabel=$k$,
title= {Random $(1,1)$ tensor $h^{\Lc}_{kk_{\lf_1}}$},
xtick={2,2.5,...,8},
xticklabels={$\frac{N_{\lf_1}}{2}$,,,,$N_{\lf_1}$,,,,,,,,$2N_{\lf_1}$},
xticklabel style={
anchor=base,
yshift=-\baselineskip
},
ytick={2,2.5,...,8},
yticklabels={$\frac{N_{\lf_1}}{2}$,,,,$N_{\lf_1}$,,,,,,,,$2N_{\lf_1}$},
yticklabel style={
anchor=base,
xshift=-\baselineskip
}, grid=both
]
\addplot+[scatter, only marks, mark=ball, mark size=3,
domain=2:8,samples=13]
{x};
\end{axis}
\end{tikzpicture}
\end{minipage}
\begin{minipage}[t]{0.45\textwidth}
\begin{tikzpicture}
	\begin{axis}[
	colormap={bw}{gray(0cm)=(0); gray(1cm)=(0)}, xlabel=$k_{\lf_1}$,ylabel=$k_{\lf_2}$,zlabel=$k$,
	title= {Random $(2,1)$ tensor $h^{\Lc}_{kk_{\lf_1}k_{\lf_2}}$},
xtick={2,3,...,8},
xticklabels={$\frac{N_{\lf_1}}{2}$,,$N_{\lf_1}$,,,,$2N_{\lf_1}$},
xticklabel style={
anchor=base,
yshift=-\baselineskip
},
ytick={2,3,...,8},
yticklabels={$\frac{N_{\lf_1}}{2}$,,$N_{\lf_1}$,,,,},
yticklabel style={
anchor=base,
yshift=-0.7\baselineskip,
xshift=0.5\baselineskip
},ztick={-8,-6,...,8},
zticklabels={,,,,,,,,,,},
zticklabel style={
anchor=base,
yshift=-0.7\baselineskip,
xshift=0.5\baselineskip
}
,grid= both,view={210}{30}]
	\addplot3+[mesh,scatter, only marks, mark=ball, mark size =3.5,samples=10,domain=2:8] 
		{x-y};
	\end{axis}
\end{tikzpicture}
\end{minipage}
\caption{Localization hyperplanes of random $(1,1)$ and $(2,1)$ tensors}\label{hyperplane}
\end{figure}

\subsection{The extended ansatz, and general case}\label{generalintro}
The ansatz (\ref{sec2.4:ansatz}) in Section \ref{simansatz} exhibits the main idea, however the full ansatz has extra layers of complexity. Some of them come from passing from the model (\ref{modelintro0}) to the full equation, such as the possibility of over-pairing (leading to the norms in (\ref{tensornorms:sup}) and the full Definition \ref{defmerge} of merging) and the role of time Fourier or modulation variables (leading to the spacetime norms defined in Section \ref{globalnorms} and allowing $\Omega\neq 0$ in (\ref{modelintro0})). The main one, however, arises already in the model (\ref{modelintro0}).

To demonstrate, suppose we plug the ansatz (\ref{sec2.4:ansatz}) into (\ref{modelintro3}). Consider the nonlinear term where one (or more) of the inputs is the remainder term $z_{N'}$ with $N'\leq N/2$, which is a part of $y_{N'}$ in (\ref{sec2.4:ansatz}). If $N^{1/2}\gg N'\gg N^\delta$, then this $N'$ is not high enough for the resulting term to have enough regularity to be put in the remainder $z_N$, and not low enough for the tensor arisen from the resulting term to be independent with the Gaussians.

To remedy this, we go back to the iteration trees in Section \ref{sec2.4:randomtensors} and introduce more random tensor terms by considering all possible configurations of iteration trees where we replace at least one $\circ$ (meaning $u_L$) by a diamond $\diamond$ (meaning $z_{N'}$ with $N'\leq N/2$). For example consider the term \begin{equation}\label{tensorterm6}
\scaleto{\<dsc_black>}{0.4cm}: \Mc_{\mathrm{np}}(\Delta_{N}f, z_{N'}, u_L),
\end{equation} where $N'\leq N/2$, whose $k$-th mode is given by \begin{equation}\label{tensortermV0}
\left(\scaleto{\<dsc_black>}{0.4cm}\right)_k
=\sum_{k_{\lf},k_{\ff}}  
h^{\scaleto{\<dsc_black>}{7.5pt}}_{kk_{\lf}}(k_{\ff})\cdot (\Delta_{N} f_{k_{\lf}})\cdot \overline{(z_{N'})_{k_{\ff}}},
\end{equation}
where
\begin{equation}\label{tensortermV}
h^{\scaleto{\<dsc_black>}{7.5pt}}_{kk_{\lf}}(k_{\ff}) = 
-i \sum_{\substack{k_{\lf}-k_{\ff} +k_3=k;\,\Omega=0\\\langle k_{\ff}\rangle\leq N', \langle k_3\rangle\leq L,\,\langle k\rangle \leq N}} (u_L)_{k_3}
\end{equation}
with $\Omega =|k|^2-|k_{\lf}|^2+ |k_{\ff}|^2-|k_3|^2$ and $u_{k_3}$ is the Fourier mode of $u_L$ in (\ref{tensorterm6}). The iteration term in (\ref{tensorterm6}) can be viewed, via (\ref{tensortermV0})--(\ref{tensortermV}), as a \emph{linear combination} of the tensor terms in Section \ref{sec2.4:randomtensors} with coefficients being the Fourier coefficients of $z_{N'}$, which are summable since the norm of $z_{N'}$ will be a large negative power of $N'$ (see part (4) of Proposition \ref{mainprop}), Moreover the tensors $h^{\scaleto{\<dsc_black>}{7.5pt}}_{kk_{\lf}}(k_{\ff})$ in (\ref{tensortermV}) do not involve $z_{N'}$ and thus retain independence.

Then we flatten these new iteration trees and repeat what we did in Sections \ref{sec2.4:randomtensors}--\ref{sec:2.4.3}, except that the pair $(\Lc,\Yc)$ alone is not sufficient to represent the new random tensor terms, and we need to introduce one more set $\Vc$ which contains all the $\diamond$'s in the new iteration trees. Hence in the full ansatz, the sum in (\ref{sec2.4:ansatz}) should be taken over all triples $(\Lc,\Vc,\Yc)$, which will be defined as \emph{plants}, see Definition \ref{defstr}.
\subsection{Outline of the paper}\label{plan} Sections \ref{prep}--\ref{prelimtensor} are mainly preparations, with definitions listed in Section \ref{prep} and lemmas proved in Section \ref{prelimtensor}. In Section \ref{ansatz} we introduce the main random tensor ansatz, thereby reducing Theorem \ref{main} to Proposition \ref{mainprop}, which is then proved in Sections \ref{mainproof}--\ref{mainproof1}. In Section \ref{proofmain} we finish the proof of Theorem \ref{main}, as well as the proof of Theorem \ref{main2}, which is a simplified version of the former. Finally in Section \ref{secfinrem} we make a few comments, including a comparison with parabolic equations. The structure of the proof is presented in Figure \ref{flowchart}.
    \begin{figure}
\centering
\begin{tikzpicture}
[
    node distance = 7mm and -3mm,
every node/.style = {draw=black, rectangle, fill=none, 
                     minimum width=2cm, minimum height=0.7cm,
                     align=center},
every path/.style = {draw, -latex}
                        ]

\node (411) {Prop \ref{bilineartensor}};
\node (412) [below = of 411]  {Prop \ref{multibound}};
\node (44) [left = 0.4cm of 412.west]{Lem \ref{largedev0}};

\node (45) [right=12mm of 411.east |- 411] {Lem \ref{cubic}--\ref{general2}};
\node (49)[below= of 45.south] {Prop \ref{final}};
\node (counting)[ rounded corners, fill=none, fit=(45) (49)] {};
\node (62)[below= of 49.south] {Prop \ref{algorithm1}--\ref{algorithm2}};

\node (414) [left = 4.4cm of 62.east |- 62] {Prop \ref{gausscont} (\ref{gausscont2})};
\node (randomtensor)[ rounded corners, fill=none, fit=(411) (414) (44)(412)] {};
\node (title1)[draw=white,fill=none, above= 0.01cm of randomtensor.north] {Random tensor theory};

\node (title2)[draw=white,fill=none, above= 0.01cm of counting.north] {Counting estimates};

\node (algorithm) [ rounded corners, right=12mm of 49.east |- 49] {Selection algorithm};

\node (61) [below = of 414.south] {Prop \ref{trimbd}};
\node (64) [below = of 62.south] {Prop \ref{overpair}--\ref{overpair3}};
\node (position) [draw=white,below = of 64.south]{};
\node (71) [left=0.2cm of 64.west |- position] {Prop \ref{linextra}--\ref{induct3}};
\node (51) [below = of 71.south]{Prop \ref{mainprop}};

\node (15) [below = 4.4cm of algorithm.south]{Thm \ref{main}};
\node (16) [below = 5.4cm of algorithm.south]{Thm \ref{main2}};
\node (mainthm) [ rounded corners, fit=(15)(16)]{};
\node (title3) [draw=white,fill=none, above= 0.01cm of mainthm.north] {Main theorems};

\draw (411)--(412);
\draw (412)--(414);
\draw (44)--(414);
\draw (414)--(61);
\draw (412)--(62);
\draw (45)--(49);
\draw (49)--(62);
\draw (algorithm)|-(62);
\draw (62)--(64);
\draw (61)--(71);
\draw (64)--(71);
\draw (71)--(51);
\draw (51)--(mainthm);

\end{tikzpicture}
\caption{Structure of the proof}\label{flowchart}
\end{figure}
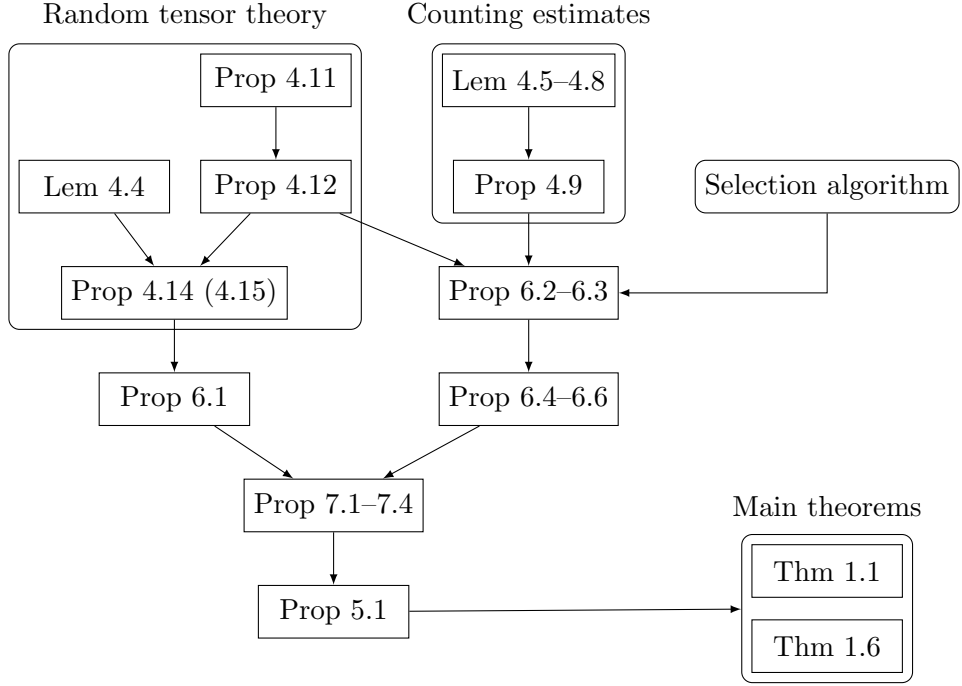 
 \section{Preliminaries I: Definitions}\label{prep} In this section we list the main definitions used in this work. In Section \ref{notations} we fix and collect the various notations and parameters, in Section \ref{defplant} we define the plant structure, associated tensors, and their operations. In Section \ref{globalnorms} we define the norms used in the main  proofs.
 \subsection{Choice of parameters and notations}\label{notations} We use $C$ to denote a generic large constant depending on $(d,p)$. Let $\alpha$ be fixed as in (\ref{choicealpha}). Define \begin{equation}\label{subcrit}\alpha_0:=\frac{d}{2}-\frac{1}{p-1};\quad\varepsilon:=(10^3dp)^{-1}(\alpha-\alpha_0)>0,\quad \beta:=\frac{\alpha+\alpha_0}{2},\quad \beta_1:=\frac{\beta+\alpha_0}{2}.\end{equation} We assume $0<\varepsilon\ll_C1$ and fix it throughout. We will use $C_\varepsilon$ to denote a generic large constant depending on $\varepsilon$; similarly $C_\delta$ etc. will depend on the small parameters $\delta$ etc. defined below. These constants, including $C$ and $(\theta,C_\theta)$ defined below, may be varying from line to line.
 
 Fix $(\delta,D,\kappa)$ such that\footnote{Roughly speaking $\delta=\varepsilon^{50}$, $D=\delta^{-50}$, $\kappa=D^{50}$ and $\theta\sim \kappa^{-50}$ should suffice.}
 \begin{equation}\label{defdelta} 0<\delta\ll_{C_\varepsilon}1;\quad D\gg_{C_\delta}1,\quad \kappa\gg_D1.
 \end{equation} Define also \begin{equation}\label{defb}D_1:=\delta^5D,\quad b:=\frac{1}{2}+8\kappa^{-1},\quad b^+:=\frac{1}{2}+16\kappa^{-1},\quad b_0=\frac{1}{2}+4\kappa^{-1}.\end{equation} Let $\theta$ denote a generic positive constant that is sufficiently small depending on $\kappa$, and (as above) $C_\theta$ a generic large constant depending on $\theta$. We also fix $\theta_0$ to be a specific positive constant that is sufficiently small depending on $\kappa$. Unless otherwise stated, the implicit constants in $\lesssim$ symbols will depend on $C_\theta$.  We fix $0<\tau\ll_{C_\theta}1$, and let $J=[-\tau,\tau]$. If an event on the ambient probability space $(\Theta,\Bc,\Pb)$ happens with probability $\geq 1-C_\theta e^{-A^\theta}$ for some quantity $A>0$, we say this event is $A$-certain. In the proof below many quantities will depend on $\omega\in\Theta$; we may include $\omega$ in the expressions for emphasis, or omit it for simplicity. We use $\mathbf{1}_E$ to denote the characteristic function of a statement $E$.
 
In the proof, the capital letters $N, M, L$ etc. will denote dyadic numbers $\geq 1$; when they (formally) take the value of $1/2$, we will understand that the corresponding quantities are $0$ (or the trivial case depending on the context). Define $N^{[\delta]}:=\max\{L:L<N^\delta\}$. The lower case letters $k,m$ etc. will denote integer vectors in $\Zb^d$ or Cartesian powers of $\Zb^d$. The letters $t,t'$ etc. will denote the time variable, and the letters $\lambda,\lambda',\lambda_j$ etc. will denote the Fourier dual of time (we call these \emph{modulation variables}). For $k\in\Zb^d$, let $\rho_k:=|g_k|$ and $\eta_k:=\rho_k^{-1}g_k$, which are independent random variables, such that each $\eta_k$ is uniformly distributed on the unit circle. For dyadic $N$, let $\Bc_N\subset\Bc$ be the $\sigma$-algebra generated by the random variables $\{\eta_k:\langle k\rangle\leq N\}$. The cardinality of a finite set $S$ is denoted by $|S|$ or $\#S$. Recall the notion of partition and subpartition in Definition \ref{tensornorms}, as well as the abbreviation $k_A=(k_j:j\in A)$; similarly we use $\lambda_A$ to denote $(\lambda_j:j\in A)$ and $\mathrm{d}\lambda_A$ to denote $\prod_{j\in A}\mathrm{d}\lambda_j$. Also recall the notion of $u^\zeta$, pairing and over-pairing in Definition \ref{def:pairing}.

We will use $u_k$ to denote Fourier coefficients of $u$, and the notation $\widehat{u}$ represents the time Fourier transform \emph{only} (maybe in multiple time dimensions, see (\ref{kerneltime})--(\ref{kernelfourier}) below). We will be loose about powers of $2\pi$, and may write formulas like 
\[\widehat{v}(\lambda)=\int_\Rb e^{-i\lambda t}v(t)\,\mathrm{d}t,\quad v(t)=\int_\Rb e^{it\lambda}\widehat{v}(\lambda)\,\mathrm{d}\lambda.\] If $\Ls$ is an $\Rb$-linear operator acting on spacetime functions, we can uniquely decompose it into the sum of a $\Cb$-linear operator, $\Ls^+$, and a $\Cb$-conjugate linear operator, $\Ls^-$. We will denote the kernel of $\Ls^\zeta$, where $\zeta\in\{\pm\}$, by $(\Ls^\zeta)_{kk'}(t,t')$, so that
\begin{equation}\label{kerneltime}(\Ls^\zeta w)_k(t)=\sum_{k'}\int\mathrm{d}t'\cdot (\Ls^\zeta)_{kk'}(t,t')w_{k'}^\zeta(t'),\end{equation} then on the time Fourier side we have 
\begin{equation}\label{kernelfourier}(\widehat{\Ls^\zeta w})_k(\lambda)=\sum_{k'}\int_\Rb (\widehat{\Ls^\zeta})_{kk'}(\lambda,-\zeta\lambda')(\widehat{w})_{k'}^\zeta(\lambda')\,\mathrm{d}\lambda'.\end{equation}Fix a smooth cutoff function $\chi(t)$ which equals 1 for $|t|\leq 1$ and equals 0 for $|t|\geq 2$. For $0<\tau\lesssim 1$ define $\chi_\tau(t):=\chi(\tau^{-1}t)$. Define the standard and truncated Duhamel operators
\begin{equation}\label{defi}\Ic v(t)=\int_0^tv(t')\,\mathrm{d}t',\quad \Ic_\chi v(t)=\chi(t)\int_0^t\chi(t')v(t')\,\mathrm{d}t'.
\end{equation} Note that these are not coming from the original Schr\"{o}dinger equation (\ref{nls}), but a variant of it after conjugating by the linear Schr\"{o}dinger flow (namely $v=e^{-it\Delta}u$). Finally we introduce the notion of \emph{simplicity} for real polynomials and $\Rb$-multilinear operators; in practice the Wick-ordered and suitably gauged power nonlinearities will be simple.
\begin{df}\label{def:simple} Consider a real polynomial (or $\Rb$-multilinear operator) $\Nc$ of degree $r$, given by
\begin{equation}\label{sec2.1:nonlinear}
\Nc(u)_k=\sum_{\zeta_1k_1+\cdots+\zeta_rk_r=k}c_{kk_1\cdots k_r}u_{k_1}^{\zeta_1}\cdots u_{k_r}^{\zeta_r}.
\end{equation}
We say it is \emph{simple} if the coefficients $c_{kk_1\cdots k_r}$ depend only on the \emph{set of pairings}\footnote{Here a pairing $(k,k_j)$ means $k=k_j$ and $\zeta_j=+$.} in $(k,k_1,\cdots,k_r)$, and $c_{kk_1\cdots k_r}=0$ unless each such pairing is over-paired. 
\end{df}
\subsection{Plants and plant tensors}\label{defplant} In this section we introduce the main structure---namely \emph{plants}---and the associated tensors, as well as two basic operations ($\mathtt{Trim}$ and $\mathtt{Merge}$) of these objects.
 \begin{df}[Plants]\label{defstr} A \emph{plant} $\Sc$ consists of the following objects:
 \begin{enumerate}
 \item Three disjoint finite sets $\Lc$ (called the \emph{tree}), $\Vc$ (called the \emph{blossom set}), and $\Yc$ (called the \emph{memory set}); elements of $\Lc$, $\Vc$ and $\Yc$ are called \emph{leaves}, \emph{blossoms} and \emph{pasts}, and are denoted by $\lf$, $\ff$ and $\pf$. An arbitrary element of $\Lc\cup\Vc\cup\Yc$ is denoted by $\nf$.
 \item A collection of pairwise disjoint $2$-element subsets of $\Lc$, which we refer to as \emph{pairings}; the set of paired leaves is denoted by $\Pc$, and the set of unpaired leaves is denoted by $\Uc:=\Lc\backslash\Pc$.
 \item A dyadic number $N=N(\Sc)$ (called the \emph{frequency} of $\Sc$), a \emph{sign} $\zeta_{\nf}\in\{\pm\}$ for each $\nf\in\Lc\cup\Vc$ (note that signs are not defined for pasts),  and a dyadic \emph{frequency} $N_\nf$ for each $\nf\in\Lc\cup\Vc\cup\Yc$. We require that $N_{\lf'}=N_{\lf}$ and $\zeta_{\lf'}=-\zeta_\lf$ for any pairing $(\lf,\lf')$ in $\Lc$; that $N_\nf\leq N$ for $\nf\in\Lc\cup\Yc$; and that $ N_\ff\leq N/2$ for $\ff\in\Vc$.
 \end{enumerate} 
 
We will denote a plant by $\Sc=(\Lc,\Vc,\Yc)$, and define $|\Sc|=|\Lc|+|\Vc|+|\Yc|$ to be the \emph{size} of the plant. Two plants will be identified if there is a bijection between them that preserves all these objects. We say a plant $\Sc$ is \emph{regular} if $N_\nf\geq N^\delta$ for any $\nf\in\Lc\cup\Vc\cup\Yc$, and \emph{plain} if $\Vc=\varnothing$ and $\sum_{\lf\in\Lc}\zeta_\lf=1$ (so in particular $|\Lc|$ is odd). We also define the \emph{mini plant} $\Sc_N^\zeta$, where $\zeta\in\{\pm\}$, to be the plant where $\Vc=\Yc=\varnothing$, $\Lc$ has only one element $\lf$ with sign $\zeta $ and frequency $N$, and $N(\Sc)=N$. This is regular, and is plain if $\zeta=+$. Finally we define the \emph{conjugate} of a plant $\Sc$ to be $\overline{\Sc}$, which is the same as $\Sc$ except that the signs of all elements in $\overline{\Sc}$ are the opposite to the signs in $\Sc$.
 \end{df} 
 \begin{rem}\label{notationrem} Note that the sets $\Uc,\,\Lc,\,\Vc$ etc. are associated with a plant $\Sc$. We will keep this correspondence throughout; for example whenever there is a plant $\Sc_j$ in some context, the set $\Uc_j$ will always be the one coming from $\Sc_j$. We may encounter sets $\Uc_j$ that do not come from any plant; in such case there will simply be \emph{no} plant called $\Sc_j$ that appears in the same context.
 \end{rem}
 \begin{df}[Plant tensors]\label{deftensor} Given a plant $\Sc=(\Lc,\Vc,\Yc)$, let $\Uc$ be as in Definition \ref{defstr}. We say a tensor\footnote{In this tensor $k_\Vc$ and $\lambda_\Vc$ appear as parameters. Note also that the definition does \emph{not} involve $\Pc$ or $\Yc$.} $h=h_{kk_\Uc}(k_\Vc,\lambda_\Vc)$ is \emph{an $\Sc$-tensor}, if $k$ and each $k_\nf\,(\nf\in\Uc\cup\Vc)$ are vectors in $\Zb^d$, and in the support of $h$ we have that 
 \begin{enumerate}
 \item $\langle k\rangle\leq N$, $\langle k_\ff\rangle\leq N_\ff$ and $|\lambda_\ff|\leq 2N^{\kappa^2}$ for each $\ff\in\Vc$, and $N_\lf/2<\langle k_\lf\rangle\leq N_\lf$ for each $\lf\in\Uc$;
 \item there is no pairing in $k_\Uc$, i.e. if $\lf,\lf'\in\Uc$ and $\zeta_{\lf'}=-\zeta_\lf$ then $k_{\lf'}\neq k_\lf$.
 \end{enumerate} Here $h$ may depend on other parameters like $t$, in which case we may write $h=h_{kk_\Uc}(t,k_\Vc,\lambda_\Vc)$.
  
 Suppose we have defined functions $f_{N'}=(f_{N'})_{k'}$ for any $N'$, and $z_{N'}=(z_{N'})_k(t)$ for any $N'<N$. Define $\Psi_k=\Psi_k[\Sc,h]$ by
 \begin{equation}\label{defpsi}\Psi_k=\sum_{k_\Uc,k_\Vc}\int\mathrm{d}\lambda_\Vc\cdot h_{kk_\Uc}(k_\Vc,\lambda_\Vc)\cdot\prod_{\lf\in\Uc}(f_{N_\lf})_{k_\lf}^{\zeta_\lf}\prod_{\ff\in\Vc}(\widehat{z_{N_\ff}})_{k_\ff}^{\zeta_\ff}(\lambda_\ff),
 \end{equation} which is an expression determined by the tensor $h$. Note also that an $\Sc$ tensor $h$ is also an $\overline{\Sc}$ tensor, and $\overline{\Psi_k[\Sc,h]}=\Psi_k[\overline{\Sc},\overline{h}]$.
 \end{df}
 \begin{df}[Trimming]\label{deftrim} Given a plant $\Sc=(\Lc,\Vc,\Yc)$ and $R\geq 1$, we can \emph{trim $\Sc$ at frequency $R$} to get $\Sc'=(\Lc',\Vc',\Yc')$, such that $\Lc'=\{\lf\in\Lc:N_\lf\geq R\}$ and $(\Vc',\Yc')$ are defined in the same way. The other objects (i.e. the frequency of $\Sc'$, the signs and frequencies of elements, and the pairings in $\Lc'$) are inherited from $\Sc$. Obviously, $\Sc'$ is regular if either $\Sc$ is regular or $R\geq N^\delta$.
 
 Now suppose we have defined functions $f_{N'}=(f_{N'})_{k'}$ for any $N'$, and $z_{N'}=(z_{N'})_k(t)$ for any $N'<R$. Then, given an $\Sc$-tensor $h=h_{kk_\Uc }(k_\Vc,\lambda_{\Vc})$, we can \emph{trim it at frequency $R$} to get an $\Sc'$-tensor $h'=(h')_{kk_{\Uc'}}(k_{\Vc'},\lambda_{\Vc'})$, which is defined by
 \begin{equation}\label{trim1}(h')_{kk_{\Uc'}}(k_{\Vc'},\lambda_{\Vc'})=\sum_{k_{\Uc\backslash\Uc'}}\sum_{k_{\Vc\backslash\Vc'}}\int\mathrm{d}\lambda_{\Vc\backslash\Vc'}\cdot h_{kk_\Uc}(k_\Vc,\lambda_{\Vc})\cdot\prod_{\lf\in\Uc\backslash\Uc'}(f_{N_\lf})_{k_\lf}^{\zeta_\lf}\prod_{\ff\in\Vc\backslash\Vc'}(\widehat{z_{N_\ff}})_{k_\ff}^{\zeta_\ff}(\lambda_\ff).
 \end{equation} We shall write the above definitions as $\Sc'=\mathtt{Trim}(\Sc,R)$ and $h'=\mathtt{Trim}(h,R)$. Note that the definition of $h'$ actually depends on the choices of $(f_{N'})$ and $(z_{N'})_{N'<R}$, but in practice these will be uniquely fixed whenever we apply $\mathtt{Trim}$ functions, so we will omit them from the list of parameters. If $N_\nf<R$ for all $\nf\in\Uc\cup\Vc$ then $\Uc'=\Vc'=\varnothing$, and $h'=(h')_k$ is just the $\Psi_k$ defined in (\ref{defpsi}).
 \end{df}
 \begin{df}[Merging]\label{defmerge} First, for any finite set $\Ac$ with a sign for each element, we will fix a maximal collection of pairwise disjoint two-element subsets of $\Ac$, such that each such subset contains two elements of opposite sign (i.e. pairings). Let $\Ps(\Ac)$ be this collection, and $\Qc(\Ac)$ be the union of the two-element subsets in $\Ps(\Ac)$.
 
 Now let $3\leq q\leq p$ be odd and $0\leq r\leq q$. Given dyadic numbers $N$ and $N_j$ and signs $\zeta_j\in\{\pm\}$ for $1\leq j\leq q$, so that $N_j\leq N$ for $1\leq j\leq r$ and $N_j\leq N/2$ for $r+1\leq j\leq q$, and $\sum_{j=1}^q\zeta_j=1$, denote the collection of these parameters by $\Bs$. Given pairwise disjoint plants $\Sc_j=(\Lc_j,\Vc_j,\Yc_j)$ with frequency $N_j$ for $1\leq j\leq r$, let $\Uc_j$ etc. be as in Definition \ref{defstr}. Let $\Lc:=\Lc_1\cup\cdots \cup\Lc_{r}$ and\footnote{If necessary we may replace the unions $\cup$ by the disjoint unions $\sqcup$ to avoid repetition of elements.} $\Vc:=\Vc_1\cup\cdots\cup\Vc_{r}\cup\{r+1,\cdots,q\}$, and for each $\nf\in\Lc_j\cup\Vc_j$ let the (new) sign of $\nf$ be $\zeta_\nf^*=\zeta_j\zeta_\nf$, where $\zeta_\nf$ is the sign of $\nf$ in $\Sc_j$. Let $\Os=\{\Ac_1,\cdots,\Ac_m\}$ be an arbitrary collection of disjoint subsets of $\Wc:=\Uc_1\cup\cdots\cup\Uc_{r}$, such that:
 \begin{enumerate}
 \item each $\Ac_i$ contains two elements of opposite $\zeta_\lf^*$ sign, but does not contain two elements of opposite $\zeta_\lf^*$ sign that belongs to \emph{the same} $\Uc_j$; \item the frequencies of $\lf\in\Ac_i$ are the same for each $1\leq i\leq m$.
 \end{enumerate} For each possible $\Os$, let 
 \begin{equation}\label{def3.6:Q}
 \Ps:=\Ps(\Ac_1)\cup\cdots\cup\Ps(\Ac_m)\quad \text{ and }\quad \Qc:=\Qc(\Ac_1)\cup\cdots\cup\Qc(\Ac_m)
  \end{equation}
  with $\Ps(\cdot)$ and $\Qc(\cdot)$ defined as above. We then \emph{merge $\Sc_j\,(1\leq j\leq r)$ via $\Bs$ and $\Os$}, to get a plant $\Sc=(\Lc,\Vc,\Yc)$ as follows. First $N(\Sc)=N$, $\Lc$ and $\Vc$ are as above, and the set of pairings in $\Lc$ is the union of the sets of pairings in each $\Lc_j$, together with $\Ps$ (the new pairings; in particular we have $\Uc=\Wc\backslash\Qc$). Second, the sign and frequency of $j$ are given by $\zeta_j$ and $N_j$ for $r+1\leq j\leq q$, and the sign and frequency of any $\nf\in\Lc_j\cup\Vc_j\,(1\leq j\leq r)$ is given by $\zeta_\nf^*$ and the frequency of $\nf$ in $\Sc_j$. Finally, $\Yc=\Yc_1\cup\cdots\cup\Yc_{r}\cup\{0\}$, with $N_0$ given by the \emph{second maximum} of all the $N_j\,(1\leq j\leq q)$; for any $\pf\in\Yc_j\,(1\leq j\leq r)$ the frequency of $\pf$ in $\Sc$ equals the frequency of $\pf$ in $\Sc_j$.
 
 Now suppose we have defined a tensor $h=h_{kk_1\cdots k_q}(\lambda_{r+1},\cdots,\lambda_{q})$, where $k,k_1,\cdots ,k_q$ are input variables and $\lambda_{r+1},\cdots,\lambda_{q}$ are parameters; assume in the support of $h$ that $\langle k\rangle\leq N$ and $\langle k_j\rangle\leq N_j$ for $1\leq j\leq q$, and that $|\lambda_j|\leq 2N^{\kappa^2}$ for $r+1\leq j\leq q$. Then, given $\Sc_j$-tensors $h^{(j)}=h_{k_jk_{\Uc_j}}^{(j)}(k_{\Vc_j},\lambda_{\Vc_j})$, where $1\leq j\leq r$, we shall \emph{merge these $h^{(j)}$ via $h$, $\Bs$ and $\Os$}, to form a new tensor $H=H_{kk_\Uc}(k_\Vc,\lambda_\Vc)$, namely
\begin{multline}\label{merge1}H_{kk_\Uc}(k_\Vc,\lambda_\Vc)=\prod_{\lf,\lf'}^{(1)}\mathbf{1}_{k_\lf=k_{\lf'}}\prod_{\lf,\lf'}^{(2)}\mathbf{1}_{k_\lf\neq k_{\lf'}}\cdot\sum_{(k_1,\cdots,k_r)}h_{kk_1\cdots k_q}(\lambda_{r+1},\cdots,\lambda_q)\\\times\sum_{k_{\Qc}}^{(3)}\prod_{\lf\in\Qc}\Delta_{N_\lf}\gamma_{k_\lf}\prod_{j=1}^{r}\big[h_{k_jk_{\Uc_j}}^{(j)}(k_{\Vc_j},\lambda_{\Vc_j})\big]^{\zeta_j}.
\end{multline}In the above expression, the product $\prod_{\lf,\lf'}^{(1)}$ is taken over all leaves $\lf,\lf'\in\Uc$ such that they belong to the same $\Ac_i$ (in particular $\zeta_{\lf'}^*=\zeta_{\lf}^*$), the product $\prod_{\lf,\lf'}^{(2)}$ is taken over all leaves $\lf,\lf'\in \Uc$ such that\footnote{Here we may also require $N_\lf=N_{\lf'}$; whether we do so will not affect the result of this product.} $\zeta_{\lf'}^*=-\zeta_{\lf}^*$ (so they do not belong to the same $\Ac_i$), and the summation $\sum_{k_{\Qc}}^{(3)}$ is taken over all possible $k_{\Qc}$ (with $\Qc$ defined above) such that for each $i$, all the $k_\lf$ for $\lf\in\Qc\cap\Ac_i$ are equal, and they equal $k_{\lf'}$ for $\lf'\in\Uc\cap\Ac_i$ (if such $\lf'$ exists). We can verify that $H$ is an $\Sc$-tensor.

We shall write the above definitions as \begin{equation}\label{merge2}\Sc=\mathtt{Merge}(\Sc_1,\cdots,\Sc_r,\Bs,\Os),\quad H=\mathtt{Merge}(h^{(1)},\cdots,h^{(r)},h,\Bs,\Os).\end{equation}
 \end{df}
 \begin{prop}\label{formulas} Assume we have fixed the choices of $f_{N'}$ and $z_{N'}$ as in Definitions \ref{deftensor} and \ref{deftrim}, and $h=h_{kk_1\cdots k_q}(\lambda_{r+1},\cdots,\lambda_{q})$, $\Bs$, $\Sc_j$ and $\Sc_j$-tensors $h^{(j)}$ for $1\leq j\leq r$ as in Definition \ref{defmerge}. In applying the $\mathtt{Merge}$ function below we will omit the parameters $h$ and $\Bs$. Then the following statements hold:
 
 (1) Recall the definition of $\Psi_k=\Psi_k[\cdot,\cdot]$ as in (\ref{defpsi}). Then for any $R$ we have $\Psi_k[\Sc_j,h^{(j)}]=\Psi_k[\mathtt{Trim}(\Sc_j,R),\mathtt{Trim}(h^{(j)},R)]$. Similarly, trimming at frequency $R_1$ and then $R_2$ is equivalent to trimming once at $\max(R_1,R_2)$.
 
 (2) Let $\Psi_{k_j}^{(j)}=\Psi_{k_j}[\Sc_j,h^{(j)}]$ be defined as in (\ref{defpsi}) from the $\Sc_j$-tensor $h^{(j)}$ for $1\leq j\leq r$. Then the quantity
 \begin{equation}\label{defphi}\Phi_k:=\sum_{(k_1,\cdots,k_q)}\int\mathrm{d}\lambda_{r+1}\cdots\mathrm{d}\lambda_q\cdot h_{kk_1\cdots k_q}(\lambda_{r+1},\cdots,\lambda_{q})\cdot\prod_{j=1}^r(\Psi_{k_j}^{(j)})^{\zeta_j}\prod_{j=r+1}^q(\widehat{z_{N_j}})_{k_j}^{\zeta_j}(\lambda_j)
 \end{equation} can be written as a linear combination\footnote{Here and below the phrase ``linear combination'' will refer to a linear combination with a fixed number of terms and fixed constant coefficients.} of $\Psi_k=\Psi_k[\Sc,H]$ (for different choices of $\Os$ as in Definition \ref{defmerge}), where
 \[\Sc=\mathtt{Merge}(\Sc_1,\cdots,\Sc_r,\Os),\quad H=\mathtt{Merge}(h^{(1)},\cdots,h^{(r)},\Os).\]
 
 (3) Let $\Sc_j'=(\Lc_j',\Vc_j',\Yc_j')=\mathtt{Trim}(\Sc_j,N^\delta)$ and $(h^{(j)})'=\mathtt{Trim}(h^{(j)},N^\delta)$ for $1\leq j\leq r$. For any $\Os$ as in Definition \ref{defmerge}, let $\Os'$ be the sub-collection of $\Os$ consisting of subsets that are contained in the union of $\Uc_j'$ for $1\leq j\leq r$. Let \begin{equation*}\begin{aligned}\Sc=(\Lc,\Vc,\Yc)&=\mathtt{Merge}(\Sc_1,\cdots,\Sc_r,\Os),&\Sc'&=(\Lc',\Vc',\Yc')=\mathtt{Trim}(\mathtt{Merge}(\Sc_1',\cdots,\Sc_r',\Os'),N^\delta),\\H&=\mathtt{Merge}(h^{(1)},\cdots,h^{(r)},\Os),&H'&= \mathtt{Trim}(\mathtt{Merge}((h')^{(1)},\cdots,(h')^{(r)},\Os'),N^\delta),\end{aligned}\end{equation*}then we have $\Sc'=\mathtt{Trim}(\Sc,N^\delta))$.
 Moreover, given any such $\Os'$, the tensor $H'$ can be written as a linear combination of tensors $\mathtt{Trim}(H,N^\delta)$ (for different choices of $\Os$ that are related to $\Os'$ as above).
 
 (4) Assume that some $N_j=N$, $\Sc_j$ is regular, and $N_{j'}\geq N^\delta$ for some $j'\neq j$. Then for the plant $\Sc=\mathtt{Trim}(\mathtt{Merge}(\mathtt{Trim}(\Sc_1,N^\delta),\cdots,\mathtt{Trim}(\Sc_r,N^\delta),\Os),N^\delta)$ where $\Os$ is as in Definition \ref{defmerge}, we have $|\Sc|>|\Sc_j|$.
 
 (5) Assume that each $\Sc_j$ is plain, and $r=q$, then the plant $\Sc=\mathtt{Merge}(\Sc_1,\cdots,\Sc_r,\Os)$, where $\Os$ is as in Definition \ref{defmerge}, is also plain.
 \end{prop}
 \begin{proof} First, (1) is obvious once we expand the $\Psi_k$ expressions using (\ref{defpsi}) and (\ref{trim1}); also (5) directly follows from definition, noticing also that\[\sum_{j=1}^q\sum_{\lf\in\Lc_j}\zeta_\lf^*=\sum_{j=1}^q\sum_{\lf\in\Lc_j}\zeta_j\zeta_\lf=\sum_{j=1}^q\zeta_j=1.\] Next, (4) is true because if $N_j=N$ and $\Sc_j$ is regular, then $\mathtt{Trim}(\Sc_j,N^\delta)=\Sc_j$, so by definition, if $\Sc=(\Lc,\Vc,\Yc)$ we have $\Lc\supset\Lc_j$ etc.; moreover as the second maximum of $N_1,\cdots,N_q$ is $\geq N^\delta$, by definition $\Yc$ will have at least one more element than $\Yc_j$, so $|\Sc|\geq |\Sc_j|+1$.
 
 Next consider (2). Recall $\Wc=\Uc_1\cup\cdots\cup\Uc_r$, using (\ref{defpsi}) and (\ref{defphi}), we can expand the expression $\Phi_k$ as a sum over the variables $k_{\Wc}$ and $(k_1,\cdots,k_q)$ and $k_{\Vc_j}$ for $1\leq j\leq r$, and an integration over the variables $(\lambda_{r+1},\cdots,\lambda_q)$ and $\lambda_{\Vc_j}$ for $1\leq j\leq r$, of the quantity
 \begin{equation}\label{quantity}h_{k_1\cdots k_q}(\lambda_{r+1},\cdots,\lambda_q)\cdot\prod_{j=1}^r \big[h_{k_jk_{\Uc_j}}^{(j)}(k_{\Vc_j},\lambda_{\Vc_j})\big]^{\zeta_j}\cdot\prod_{j=1}^r\prod_{\lf\in\Uc_j}(f_{N_\lf})_{k_\lf}^{\zeta_\lf^*}\cdot \prod_{j=1}^r\prod_{\ff\in\Vc_j}(\widehat{z_{N_\ff}})_{k_\ff}^{\zeta_\ff^*}(\lambda_\ff)\prod_{j=r+1}^q(\widehat{z_{N_j}})_{k_j}^{\zeta_j}(\lambda_j),
 \end{equation} where in the summation we do not impose any pairing or no-pairing condition (with respect to the $\zeta_\lf^*$ signs, same below) for the variables $k_{\Wc}$ (of course if there is a pairing within $k_{\Uc_j}$ then $h^{(j)}=0$ so the quantity (\ref{quantity}) is zero).
 
 On the other hand, using (\ref{defpsi}) and (\ref{merge1}), and noticing that $(f_N)_k\overline{(f_N)_k}=\Delta_N(\gamma_k)^2$, we can write $\Psi_k[\Sc,H]$, corresponding to a certain choice of $\Os=\{\Ac_1,\cdots,\Ac_m\}$ as in Definition \ref{defmerge}, as a sum and integration of the same quantity (\ref{quantity}) over the same set of variables as in $\Phi_k$, but with a set of additional pairing and no-pairing conditions on $k_\Wc$. Precisely, the extra conditions are (i) the $k_\lf$ are the same for $\lf$ in each $ \Ac_i$; (ii) there is no pairing in $k_{\Wc\backslash\Qc}$ where $\Qc$ is defined in Definition \ref{defmerge}. We denote this set of extra conditions by $(\Os,1)$. With these observations, it suffices to show that the sum $\sum_{k_{\Wc}}$ can be written as a linear combination of sums $\sum_{(\Os,1)}$ for different $\Os$. Now, by identifying the exact set of pairings among $k_{\Wc}$, we can write $\sum_{k_{\Wc}}$ as a linear combination of sums $\sum_{(\Os,2)}$ for different $\Os$'s , where $(\Os,2)$ represents a different set of extra conditions, namely (i) the $k_\lf$ are the same for $\lf$ in each $ \Ac_i$; (ii) the $k_\lf$ for different $\Ac_i$ are different, and are different from any $k_\lf$ for $\lf$ not in any $\Ac_i$; (iii) there is no pairing among the $k_\lf$ where $\lf$ is not in any $\Ac_i$. Note that we may assume these $\Os$'s are as in Definition \ref{defmerge}, i.e. each $\Ac_i$ contains two elements of opposite sign but does not contain two elements of opposite sign that belong to the same $\Uc_j$, and $N_\lf$ for $\lf$ in each $\Ac_i$ are all the same, since otherwise the summand (\ref{quantity}) would be zero by Definition \ref{deftensor}.
 
 Clearly the condition $(\Os,2)$ is stronger than $(\Os,1)$, and the difference $\sum_{(\Os,1)}-\sum_{(\Os,2)}$ can be written as a linear combination of sums $\sum_{(\Ws,2)}$ for different $\Ws$'s, where $\Ws$ has the same form as $\Os$, such that the condition $(\Ws,2)$ gives \emph{strictly} more pairings than $(\Os,2)$. Thus, we can inductively write $\sum_{k_{\Wc}}$ as a  linear combination of sums $\sum_{(\Os,1)}$ for different $\Os$. In this way we have written $\Phi_k$ as a linear combination of $\Psi_k=\Psi_k[\Sc,H]$ for different choices of $\Os$, which proves (2).
 
 Finally look at (3). As $\Os'$ consists of those subsets in $\Os$ that are contained in the union of $\Uc_j'$ (note that any $\Ac_i$ is either contained in the union of $\Uc_j'$, or contained in the union of $\Uc_j\backslash\Uc_j'$), we know that $\Ps'$ (defined from $\Os'$ as in Definition \ref{defmerge}) also consists of those subsets in $\Ps$ that are contained in the union of $\Uc_j'$. Then the equality $\Sc'=\mathtt{Trim}(\Sc,N^\delta)$ follows from Definitions \ref{deftrim} and \ref{defmerge}, and straightforward verification. Note that if $N_j\geq N^\delta$ for $r+1\leq j\leq s$, and $N_j<N^\delta $ for $s+1\leq j\leq q$ (which we may assume), then $\{r+1,\cdots,s\}\subset\Vc'$ and $\{s+1,\cdots,q\}\subset\Vc\backslash\Vc'$.
 
Now look at the tensor $H'$. Let $\Rc$ be the union of all the $\Uc_j\backslash\Uc_j'$ and $\Zc$ be the union of all the $\Vc_j\backslash\Vc_j'$, and let $\Qc'$ be defined as in Definition \ref{defmerge}, which occurs in the process of merging $\Sc_1',\cdots, \Sc_r'$ via $\Os'$, then using (\ref{trim1}) and (\ref{merge1}), we can expand $(H')_{kk_{\Uc'}}(k_{\Vc'},\lambda_{\Vc'})$ as a sum over the variables $(k_{s+1},\cdots,k_q)$ and $k_{\Qc'}$ and $k_\Rc$ and $k_{\Zc}$, and integration over the variables $(\lambda_{s+1},\cdots,\lambda_q)$ and $\lambda_{\Zc}$, of the quantity
 \begin{multline}\label{quantity2}
 \prod_{\lf,\lf'}^{(1)}\mathbf{1}_{k_\lf=k_{\lf'}}\prod_{\lf,\lf'}^{(2)}\mathbf{1}_{k_\lf\neq k_{\lf'}}\cdot \sum_{(k_1,\cdots,k_r)}h_{kk_1\cdots k_q}(\lambda_{r+1},\cdots,\lambda_q)\\\times\sum_{k_{\Qc'}}^{(3)}\prod_{\lf\in\Qc'}\Delta_{N_\lf}\gamma_{k_\lf}\prod_{j=1}^{r}\big[h_{k_jk_{\Uc_j}}^{(j)}(k_{\Vc_j},\lambda_{\Vc_j})\big]^{\zeta_j}\prod_{\lf\in\Rc}(f_{N_\lf})_{k_{\lf}}^{\zeta_\lf^*}\prod_{\ff\in\Zc}(\widehat{z_{N_\ff}})_{k_\ff}^{\zeta_\ff^*}(\lambda_\ff)\prod_{j=s+1}^q(\widehat{z_{N_j}})_{k_j}^{\zeta_j}(\lambda_j),
 \end{multline} where the sums and products $\prod_{\lf,\lf'}^{(1)}$, $\prod_{\lf,\lf'}^{(2)}$ and $\sum_{k_{\Qc'}}^{(3)}$ are defined as in Definition \ref{defmerge} in the process of merging $\Sc_1',\cdots, \Sc_r'$ via $\Os'$, and in this summation we do not impose any pairing or no-pairing condition for the variables $k_{\Rc}$. The signs $\zeta_{\mathfrak{n}}^*$ are also defined as in Definition \ref{defmerge}.
 
 On the other hand, if $\Os=\Os'\cup\Ws$, then using (\ref{trim1}) and (\ref{merge1}) again, we can expand the tensor $(\mathtt{Trim}(H,N^\delta))_{kk_{\Uc'}}(k_{\Vc'},\lambda_{\Vc'})$ as a sum and integration of the same quantity (\ref{quantity2}) over the same set of variables as in $(H')_{kk_{\Uc'}}(k_{\Vc'},\lambda_{\Vc'})$, but with the $k_\Rc$ variables satisfying a set of additional pairing and no-pairing conditions, given exactly by $(\Ws,1)$. Therefore, the same arguments as in the proof of part (2) above also imply that $H'$ can be written as a linear combination of $\mathtt{Trim}(H,N^\delta)$ for different choices of $\Os$. This completes the proof.
 \end{proof}
 \subsection{Working norms}\label{globalnorms} Based on the tensor norms of Definition \ref{tensornorms}, we can define the norms involving the modulation variables $\lambda,\lambda'$, etc., as well as some other parameters; these will be the norms used in the main proof.
 
Suppose $b_1,b_2\in[0,1]$, $h=h_{k_A}(t)$ depends on $t$, and let $\widehat{h}$ be the Fourier transform of $h$ in $t$. Let $(B,C)$ be a subpartition of $A$, we define
\begin{equation}\label{newnorm0}\|h\|_{X^{b_1}[k_B\to k_C]}^2=\int_\Rb\langle\lambda\rangle^{2b_1}\|\widehat{h}_{k_A}(\lambda)\|_{k_B\to k_C}^2\,\mathrm{d}\lambda.
\end{equation} If $h=h_{k_A}(k_F,\lambda_F)$ depends on some parameters $(k_F,\lambda_F)$, we define
 \begin{equation}\label{newnorm}\|h\|_{X_F^{-b_2}[k_B\to k_C]}^2=\sum_{k_F}\int\mathrm{d}\lambda_F\cdot\prod_{j\in F}\langle\lambda_j\rangle^{-2b_2}\|h_{k_A}(k_F,\lambda_F)\|_{k_B\to k_C}^2.
 \end{equation} If $h=h_{k_A}(t,k_F,\lambda_F)$ depends on both $t$ and $(k_F,\lambda_F)$, we define
  \begin{equation}\label{newnorm2}\|h\|_{X_F^{b_1,-b_2}[k_B\to k_C]}^2=\int_\Rb\langle\lambda\rangle^{2b_1}\sum_{k_F}\int\mathrm{d}\lambda_F\cdot\prod_{j\in F}\langle\lambda_j\rangle^{-2b_2}\|\widehat{h}_{k_A}(\lambda,k_F,\lambda_F)\|_{k_B\to k_C}^2\,\mathrm{d}\lambda,
 \end{equation} where $\widehat{h}$ is the Fourier transform of $h$ in $t$ only; note that
 \begin{multline}\label{relanorm}\|h\|_{X_F^{b_1,-b_2}[k_B\to k_C]}^2=\int_\Rb\langle \lambda\rangle^{2b_1}\|\widehat{h}_{k_A}(\lambda,\cdot,\cdot)\|_{X_F^{-b_2}[k_B\to k_C]}^2\,\mathrm{d}\lambda\\=\sum_{k_F}\int\mathrm{d}\lambda_F\cdot\prod_{j\in F}\langle\lambda_j\rangle^{-2b_2}\|h_{k_A}(\cdot,k_F,\lambda_F)\|_{X^{b_1}[k_B\to k_C]}^2.
 \end{multline}
 
Now, given a tensor $h=h_{k_A}(t,t')$, a subpartition $(B,C)$, and $b_1,b_2\in[0,1]$, we can similarly define
  \begin{equation}\label{newnorm3}\|h\|_{X^{b_1,-b_2}[k_B\to k_C]}^2=\int_{\Rb^2}\langle\lambda\rangle^{2b_1}\langle\lambda'\rangle^{-2b_2}\|\widehat{h}_{k_A}(\lambda,\lambda')\|_{k_B\to k_C}^2\,\mathrm{d}\lambda\mathrm{d}\lambda',
 \end{equation} where $\widehat{h}$ is the Fourier transform of $h$ in $(t,t')$. Finally, if $s_1\in\Rb$, the $X^{s_1,b_1}$ norm for a function $f=f_k(t)$ is defined by
 \begin{equation}\label{xbnorm}\|f\|_{X^{s_1,b_1}}^2=\int_\Rb\langle\lambda\rangle^{2b_1}\|\langle k\rangle^{s_1}\widehat{f}_k(\lambda)\|_{\ell_k^2}^2\,\mathrm{d}\lambda.
 \end{equation} When $s_1=0$ we simplify write $X^{b_1}$. 
\section{Preliminaries II: Estimates}\label{prelimtensor} In this section we collect the important core estimates. Sections \ref{section4.1}--\ref{section4.2} contain the basic linear and large deviation estimates. Section \ref{counting} contains counting estimates ultimately leading to Proposition \ref{final}, and Section \ref{multi} contains the main tensor norm estimates, Propositions \ref{bilineartensor}--\ref{multibound} and \ref{gausscont}--\ref{gausscont2}.
\subsection{Linear estimates}\label{section4.1} We record two estimates for Duhamel and time localization operators (recall from (\ref{defi}) the definition of $\Ic_\chi$), and another weighted estimate. For the proofs see \cite{DNY0,DNY}.
\begin{lem}\label{duhamelform0} We have the formula
\begin{equation}\label{duhamelform}\widehat{\Ic_\chi v}(\lambda)=\int_{\Rb}\Ic(\lambda,\lambda')\widehat{v}(\lambda')\,\mathrm{d}\lambda',
\end{equation} where the kernel $\Ic$ satisfies that
\begin{equation}\label{duhamleker}|\Ic|+|\partial_{\lambda,\lambda'}\Ic|\lesssim\bigg(\frac{1}{\langle \lambda\rangle^3}+\frac{1}{\langle\lambda-\lambda'\rangle^3}\bigg)\frac{1}{\langle\lambda'\rangle}\lesssim\frac{1}{\langle\lambda\rangle\langle\lambda-\lambda'\rangle}.
\end{equation}
\end{lem}
\begin{proof} See \cite{DNY0}, Lemma 3.1 whence by a similar proof, one can see that (\ref{duhamleker}) also holds for $|\partial_{\lambda,\lambda'}\Ic|$.
\end{proof}

\begin{lem}\label{localization} Let $v=v(t)$ be a function on $\mathbb{R}$ valued in some Banach function space. For $b_1\in[0,1]$ define the $Y^{b_1}$ norm by
\begin{equation}\|v\|_{Y^{b_1}}^2=\int_{\Rb}\langle \lambda\rangle^{2b_1}\|\widehat{v}(\lambda)\|^2\,\mathrm{d}\lambda,
\end{equation}where $\widehat{v}$ is the (vector-valued) Fourier transform of $u$. For $\tau\lesssim 1$ let $\chi_\tau(t)=\chi(\tau^{-1}t)$ be as in Section \ref{notations}, then for any $0<b_1\leq b_2<1/2$ and for any $v$, or for any $1/2<b_1\leq b_2<1$ and for any $v$ satisfying\footnote{In practice, the factor $\chi_\tau$ will always come with a $v$ which has the form $\Ic_\chi(\cdots)$, so we always have $v(0)=0$.} $v(0)=0$, we have
\begin{equation}\label{localize}\|\chi_\tau\cdot v\|_{Y^{b_1}}\lesssim \tau^{b_2-b_1}\|v\|_{Y^{b_2}}.
\end{equation}
\end{lem}\begin{proof} See \cite{DNY}, Proposition 2.7 (which proves the scalar case, but the proof directly carries over to vector valued cases).
\end{proof}
\begin{lem}\label{weightedbd} Fix $\kappa_1>0$. Let $\Ms=\Ms_{kk'}(\lambda,\lambda')$ be the kernel of an operator $\Ms$, namely
\[(\Ms w)_k(\lambda)=\sum_{k'}\int_{\Rb}\Ms_{kk'}(\lambda,\lambda')w_{k'}(\lambda')\,\mathrm{d}\lambda',\] and assume that $\Ms$ is supported in $|k-k'|\leq R$ for some dyadic $R$. Then uniformly in any $R$ and any $k^0\in\Zb^d$, we have
\begin{equation}\|(1+R^{-1}|k-k^0|)^{\kappa_1}(\Ms w)_k(\lambda)\|_{\ell_k^2L_\lambda^2}\lesssim\|\Ms\|_{\ell_{k'}^2L_{\lambda'}^2\to \ell_k^2L_\lambda^2}\cdot \|(1+R^{-1}|k'-k^0|)^{\kappa_1} w_{k'}(\lambda')\|_{\ell_{k'}^2L_{\lambda'}^2}.
\end{equation}
\end{lem}
\begin{proof} See \cite{DNY}, Proposition 2.5.
\end{proof}
 \subsection{Large deviation inequalities}\label{section4.2} We state a large deviation estimate that works for uniform distributions on the unit circle, see \cite{DNY}.
 \begin{lem}\label{largedev0} Let $E\subset \Zb^d$ be a finite set, $a=a_{k_1\cdots k_r}(\omega)$ be a random tensor such that the collection $\{a_{k_1\cdots k_r}\}$ is independent with the collection $\{\eta_k(\omega):k\in E\}$. Let $\zeta_j\in\{\pm\}$ and assume that in the support of $a_{k_1\cdots k_r}$ there is no pairing in $(k_1,\cdots,k_r)$ associated with the signs $\zeta_j$. Let the random variable
 \begin{equation}\label{largedev1}X(\omega):=\sum_{k_1,\cdots,k_r}a_{k_1\cdots k_r}(\omega)\prod_{j=1}^r\eta_j(\omega)^{\zeta_j},
 \end{equation} then for any $A>0$, we have $A$-certainly that
 \begin{equation}\label{largedev2}|X(\omega)|^2\leq A^\theta\cdot\sum_{k_1,\cdots,k_r}|a_{k_1\cdots k_r}(\omega)|^2.
 \end{equation}
 \end{lem}
 \begin{proof} This is a special case (i.e. no pairing) of \cite{DNY}, Lemma 4.1.
 \end{proof}
 \subsection{Lattice point counting bounds}\label{counting} Here we state and prove the various counting bounds that eventually lead to Proposition \ref{final}.
 \begin{lem}\label{cubic} Consider the set
 \begin{multline}S^{(3)}=\big\{(k_1,k_2,k_3)\in(\Zb^d)^3:\zeta_1k_1+\zeta_2k_2+\zeta_3k_3=m,\,\,\zeta_1|k_1|^2+\zeta_2|k_2|^2+\zeta_3|k_3|^2=\Gamma,
 \\|k_1-m_1|\leq M_1,\,\,|k_2-m_2|\leq M_2,\,\,\text{\rm{and there is no pairing in $(k_1,k_2,k_3)$}}\big\},
 \end{multline} where $\zeta_j\in\{\pm\}$, $(m,\Gamma)\in\Zb^d\times\Zb$, $m_j\in\Zb^d$ and $M_j$ dyadic are given. Then we have, uniformly in all parameters, that
 \begin{equation}\label{cubiccounting}\#S^{(3)}\lesssim (M_1M_2)^{d-1+\theta}.
 \end{equation}
 \end{lem}
 \begin{proof} We may assume $M_1\leq M_2$. If $d=1$, by simple algebra, we can reduce to a divisor counting problem in $\Zb[e^{2\pi i/3}]$ (if $\zeta_1=\zeta_2=\zeta_3$) or $\Zb$ (otherwise). Since each $k_j$ belongs to an interval of length $O(M_2)$, the estimate (\ref{cubiccounting}) follows from Lemma 3.4 of \cite{DNY0}.
 
Consider $d\geq 2$. Without loss of generality, we may assume (due to no-pairing) that either \emph{the first coordinates} of $(k_1,k_2,k_3)$ do not contain a pairing, or the $j$-th coordinates of $(k_1,k_2,k_3)$ contain a pairing for each $j$, and this pairing is \emph{not} from $(k_2,k_3)$ for $j=1$. In the former case the $j$-th coordinates of $(k_1,k_2,k_3)$ have at most $M_1M_2$ choices for each $2\leq j\leq d$, and then at most $(M_1M_2)^\theta$ choices for $j=1$ thanks to the $d=1$ case. In the latter case the $j$-th coordinates of $(k_1,k_2,k_3)$ have at most $M_2$ choices for $2\leq j\leq d$, and at most $M_1$ choices for $j=1$. In either case we get
\[\#S^{(3)}\lesssim\max((M_1M_2)^{d-1}(M_1M_2)^\theta,M_2^{d-1}M_1)\lesssim (M_1M_2)^{d-1+\theta}.\qedhere\]
 \end{proof}
 \begin{lem}\label{basic} Recall that $p\geq 3$ is odd. For $1\leq p_1\leq p$, consider a partition of a set $A\subset\{1,\cdots,p\}$, $|A|=p_1$, into pairwise disjoint nonempty subsets $B_1,\cdots,B_t$, say $B_u=\{i_u(1),\cdots,i_u(b_u)\}\,(1\leq u\leq t)$. Given $m_j\in\Zb^d$, $M_j$ dyadic, $\zeta_j\in\{\pm\}\,(j\in A)$ and $\Gamma\in\Zb$, consider the set $S$ consisting of vectors $k_A$ (where each $k_j\in\Zb^d$) that satisfy
 \begin{equation}\label{basicset}\sum_{j\in A}\zeta_j|k_j|^2=\Gamma;\quad \bigg|\sum_{z=1}^y\zeta_{i_u(z)}k_{i_u(z)}-m_{i_u(y)}\bigg|\leq M_{i_u(y)},\,\forall 1\leq u\leq t,1\leq y\leq b_u=|B_u|.
 \end{equation}
Assume $M_{i_{u}(b_u)}=1$ for $1\leq u\leq t$, and that there is no pairing in $k_A$. Then we have, uniformly in all parameters, that
 \begin{equation}\label{basiccounting1}\#S\lesssim \prod_{j\in A}(M_j)^{2\alpha_0+\theta},
 \end{equation} where $\alpha_0$ is as in (\ref{subcrit}).
 \end{lem}
 \begin{proof} We may assume $p_1=p$, since if $p_1<p$ we can add some elements to the sets $B_u$ and reduce to the $p_1=p$ case. We will prove (\ref{basiccounting1}) by induction. Suppose (\ref{basiccounting1}) is true for $p-2$, we will prove it for $p$. For simplicity we define
 \[w_{i_u(y)}:=\sum_{z=1}^y\zeta_{i_u(z)}k_{i_u(z)},\quad 1\leq u\leq t,\,1\leq y\leq b_u.\]
 
 Since $p$ is odd, at least one of $|B_u|$ (say $|B_1|$) must be odd. If $|B_1|=1$ then the value of $k_{i_1(1)}$ is fixed, so we only need to count the vector $k_{A\backslash B_1}$. If $|B_2|=1$ also then we may reduce to counting $k_{A\backslash (B_1\cup B_2)}$ and apply the induction hypothesis; otherwise $|B_2|\geq 2$ and we may add an element to $B_2$ at no cost and reduce to the case $|B_2|\geq 3$. Therefore in any case we may assume some $|B_j|\geq 3$, say $|B_1|=b_1\geq 3$.
 
Let $(i_1(b_1),i_1(b_1-1),i_1(b_1-2))=(n,n',n'')$. Since by (\ref{basicset}), each $w_{i_u(y)}$ belongs to a ball of radius $M_{i_u(y)}$, and $k_{i_u(y)}=\pm w_{i_u(y)}\pm w_{i_u(y-1)}$ for some choices of $\pm$ (same below), the number of choices for the vector $k_{B_1\backslash\{n,n',n''\}}$ is at most\[M_{i_1(1)}^d\cdots M_{i_1(b_1-3)}^d=(M_{n'}M_{n''})^{-d}\prod_{j\in B_1}M_j^d,\] noticing that $M_n=1$. Similarly for $2\leq u\leq t$, the number of choices for the vector $k_{B_u}$ is at most $\prod_{j\in B_u}M_j^d$. Once $k_{A\backslash\{n,n',n''\}}$ is fixed, $k_{n''}=\pm w_{n''}\pm w_{i_1(b_1-3)}$ belongs to a ball of radius $M_{n''}$, and $k_n=\pm w_n\pm w_{n'}$ belongs to a ball of radius $M_{n'}$. Then the number of choices for $(k_n,k_{n'},k_{n''})$ can be bounded by $(M_{n'}M_{n''})^{d-1+\theta}$ by Lemma \ref{cubic}, thus
\begin{equation}\label{interpolate1}\#S\lesssim (M_{n'}M_{n''})^{d-1+\theta}\prod_{j\in A\backslash\{n,n',n''\}}M_j^d.\end{equation} Note that this also settles the base case $p=3$.

On the other hand, for $p\geq 5$, since $M_n=1$, by (\ref{basicset}) we know that $k_n=\pm w_{n}\pm w_{n'}$ belongs to a ball of radius $M_{n'}$, and when $k_n$ is fixed, $k_{n'}=\pm w_n\pm k_n\pm w_{n''}$ belongs to a ball of radius $M_{n''}$. Thus the number of choices for $(k_n,k_{n'})$ is at most $(M_{n'}M_{n''})^d$. When $(k_n,k_{n'})$ is fixed, we only need to count the vectors $k_{A\backslash\{n,n'\}}$. Now $w_{i_1(y)}$ belongs to a ball of radius $M_{i_1(y)}$ if $y\leq b_1-3$, and to a ball of radius $1$ if $y=b_1-2$, so by the induction hypothesis we conclude that
\begin{equation}\label{interpolate2}\#S\lesssim (M_{n'}M_{n''})^{d}\prod_{j\in A\backslash\{n,n',n''\}}M_j^{d-\frac{2}{p-3}+\theta}.\end{equation} Interpolating (\ref{interpolate1}) and (\ref{interpolate2}) we get
\[\#S\lesssim\prod_{j\in A\backslash\{n\}}M_j^{d-\frac{2}{p-1}+\theta},\] which is just (\ref{basiccounting1}).
 \end{proof}
 \begin{lem}\label{general1}For $1\leq p_1\leq p$, consider a partition of a set $A\subset\{1,\cdots,p\}$, $|A|=p_1$, into pairwise disjoint nonempty subsets $A_1,\cdots,A_s$ and $B_1,\cdots,B_t$, say $A_v=\{\ell_v(1),\cdots,\ell_v(a_v)\}\,(1\leq v\leq s)$ and $B_u=\{i_u(1),\cdots,i_u(b_u)\}\,(1\leq u\leq t)$. Given $m_j\in\Zb^d$, $M_j$ dyadic, $\zeta_j\in\{\pm\}\,(j\in A)$ and $\Gamma,\Gamma_v\in\Zb\,(1\leq v\leq s)$, consider the set $S$ consisting of vectors $k_A$ (where each $k_j\in\Zb^d$) that satisfy
 \begin{equation}\label{generalset1}\sum_{j\in A}\zeta_j|k_j|^2=\Gamma,\quad \bigg|\sum_{z=1}^y\zeta_{i_u(z)}k_{i_u(z)}-m_{i_u(y)}\bigg|\leq M_{i_u(y)},\,\forall 1\leq u\leq t,1\leq y\leq b_u=|B_u|,
 \end{equation}
  \begin{equation}\label{generalset2}\sum_{j\in A_v}\zeta_j|k_j|^2=\Gamma_v,\quad\bigg|\sum_{z=1}^y\zeta_{\ell_v(z)}k_{\ell_v(z)}-m_{\ell_v(y)}\bigg|\leq M_{\ell_u(y)},\,\forall 1\leq v\leq s,1\leq y\leq a_v=|A_v|.
 \end{equation}
Assume that $M_{\ell_v(a_v)}=1$ and $|A_v|\leq p-2$ for each $1\leq v\leq s$, that $M_{i_u(b_u)}=1$ for each $1\leq u\leq t$, and that there is no pairing in $k_A$. Then we have, uniformly in all parameters, that
 \begin{equation}\label{generalcounting1}\#S\lesssim\prod_{j\in A}(M_j)^{2\alpha_0+\theta}\prod_{v=1}^s(\min_{1\leq y<a_v}M_{\ell_v(y)})^{2\alpha_0-d}.
 \end{equation}
\end{lem}
\begin{proof} This essentially follows from Lemma \ref{basic}. Let $B=B_1\cup\cdots\cup B_t$, by (\ref{generalset1}) and (\ref{generalset2}) we may count $k_B$ and each $k_{A_v}\,(1\leq v\leq s)$ separately. By Lemma \ref{basic}, the number of choices for $k_B$ is at most
\[\prod_{j\in B}(M_j)^{2\alpha_0+\theta};\] thus it suffices to prove that the number of choices for $k_{A_v}$ is at most
\[\prod_{j=1}^{a_v-1}(M_{\ell_v(y)})^{2\alpha_0+\theta}\cdot (\min_{1\leq y<a_v}M_{\ell_v(y)})^{2\alpha_0-d}.\] Let $|A_v|=a_v=n$, clearly we may assume $n\geq 2$. Lemma \ref{basic} then implies that the number of choices for $k_{A_v}$ is at most
\[\prod_{y=1}^{n-1}(M_{\ell_v(y)})^{d-\frac{2}{n'}+\theta},\] where $n'=n$ if $n$ is even, and $n'=n-1$ if $n$ is odd. Now the desired estimate follows, since
\[\prod_{y=1}^{n-1}(M_{\ell_v(y)})^{d-\frac{2}{n'}}\leq \prod_{y=1}^{n-1}(M_{\ell_v(y)})^{2\alpha_0}\cdot (\min_{1\leq y<a_v}M_{\ell_v(y)})^{2\alpha_0-d},\] due to the elementary inequality
\[(n-1)\big(2\alpha_0-d+\frac{2}{n'}\big)\geq d-2\alpha_0\] which can be verified for $2\leq n\leq p-2$.
\end{proof}
\begin{lem}\label{general2} Consider the same setting and set $S$ as in Lemma \ref{general1}, but instead of no pairing in $k_A$ we assume that (1) any pairing in $k_A$ must be over-paired, and (2) $d(p-1)\geq 8$. Then the bound (\ref{generalcounting1}) remains true.
\end{lem}
\begin{proof} As in the proof of Lemma \ref{basic}, we define
\[
\begin{aligned}w_{i_u(y)}&=\sum_{z=1}^y\zeta_{i_u(z)}k_{i_u(z)},&&1\leq u\leq t,\,1\leq y\leq b_u;\\
w_{\ell_v(y)}&=\sum_{z=1}^y\zeta_{\ell_v(z)}k_{\ell_v(z)},&&1\leq v\leq s,\,1\leq y\leq a_v.
\end{aligned}\] We also understand $M_{i_u(y)}=M_{\ell_v(y)}=1$ when $y=0$. It suffices to bound $\#S$ by 
\begin{equation}\Nf=\prod_{v=1}^s\Nf_v\prod_{u=1}^t\Nf_u^*,
\end{equation} where for $1\leq v\leq s$ and $1\leq u\leq t$,
\begin{equation}\label{factorseach}\Nf_v=\prod_{1\leq y<a_v}(M_{\ell_v(y)})^{2\alpha_0+\theta}\cdot(\min_{1\leq y<a_v}M_{\ell_v(y)})^{2\alpha_0-d},\quad\textrm{and}\quad \Nf_u^*=\prod_{1\leq y<b_u}(M_{i_u(y)})^{2\alpha_0+\theta}.
\end{equation}We proceed by induction. The no-pairing case is already known by Lemma \ref{general1}. Now suppose there is some over-pairing in $k_A$; we list all the different $j$'s in $A$ (there are at least three of them) such that $k_j$ are the same, and let this common value be $k$. We will fix $k$, count the remaining variables, and then sum in $k$. With $k$ fixed, by using induction hypothesis, the number of choices for the remaining variables will be bounded by some product
\[\Nf(k)=\prod_{v=1}^s\Nf_v(k)\prod_{u=1}^t\Nf_u^*(k),\] and we need to bound the quotients $\Nf_v(k)/\Nf_v$ and $\Nf_u^*(k)/\Nf_u^*$. We only need to consider those $v$ and $u$ such that at least one $k_j\,(j\in A_v\textrm{ or }B_u)$ equals $k$ (otherwise the quotient is 1). There will be several cases depending on how many $k_j$ equal $k$, and their positions. First, if $|A_v|=1$ (or $|B_u|=1$), then the value of $k$ will be uniquely determined, so there is no summation in $k$, and the result follows from the induction hypothesis. From now on we will assume $|A_v|\geq 2$ and $|B_u|\geq 2$. Similarly, if $|A_v|=3$ or $|B_u|=3$ then the elements in $\{k_j:j\in A_v\}$ (or $\{k_j:j\in B_u\}$) cannot all equal $k$, since otherwise $k$ would also be uniquely determined.

(1) If all (or all but one) elements in $\{k_j:j\in A_v\}$ (or in $\{k_j:j\in B_u\}$) equal $k$, then $\Nf_v(k)$ (or $\Nf_u^*(k)$) will be equal to 1. Considering $B_u$, since $k_{i_u(y)}=\pm w_{i_u(y)}\pm w_{i_u(y-1)}$, by (\ref{generalset1}) we see that if $k_{i_u(y)}=k$, then there exists a vector $m^*$ not depending on $k$, such that $\max(M_{i_u(y)},M_{i_u(y-1)})\gtrsim|k-m^*|$. This implies that
\begin{equation}\label{quotientest1}\frac{\Nf_u^*(k)}{\Nf_u^*}\lesssim\left\{
\begin{aligned}&|k-m^{*}|^{-2\alpha_0-\theta},&&\textrm{if $|B_u|\leq 3$};\\
&|k-m^*|^{-2\alpha_0-\theta}|k-m^{**}|^{-2\alpha_0-\theta},&&\textrm{if $|B_u|\geq 4$},
\end{aligned}
\right.
\end{equation} where $m^{**}$ is another vector not depending on $k$.

Similarly considering $A_v$, in view of the extra factor in (\ref{factorseach}), the estimates will be
\begin{equation}\label{quotientest2}\frac{\Nf_v(k)}{\Nf_v}\lesssim\left\{
\begin{aligned}&|k-m^*|^{d-4\alpha_0-\theta},&&\textrm{if $|A_v|=2$};\\
&|k-m^*|^{-2\alpha_0-\theta},&&\textrm{if $|A_v|\geq 3$};\\
&|k-m^*|^{-2\alpha_0-\theta}|k-m^{**}|^{d-4\alpha_0-\theta},&&\textrm{if $|A_v|\geq 4$}.
\end{aligned}
\right.
\end{equation}

(2) If at least two elements in $\{k_j:j\in A_v\}$ (or $\{k_j:j\in B_u\}$) do not equal $k$, then in particular $|A_v|\geq 3$ (or $|B_u|\geq 3$). In this case we only need to consider $B_u$, and the bounds for $A_v$ will be the same (if not better) due to the negative power of $\min_y M_{\ell_v(y)}$ in (\ref{factorseach}). If we fix $k_{i_u(y)}=k$, then $w_{i_u(y-1)}$ belongs to a ball of radius $\min(M_{i_u(y)},M_{i_u(y-1)})$, so we get
\begin{equation}\label{quotientest3}\frac{\Nf_u^*(k)}{\Nf_u^*}\lesssim\max(M_{i_u(y)},M_{i_u(y-1)})^{-2\alpha_0-\theta},\end{equation} which is bounded by $|k-m^*|^{-2\alpha_0-\theta}$, in the same way as in (1).

(3) By similar arguments as in (2), we know that if two non-adjacent elements in $B_u$ equal $k$ (similarly for $A_v$), say $k_{i_u(y)}=k_{i_u(y')}=k$, then $w_{i_u(y-1)}$ belongs to a ball of radius $\min(M_{i_u(y)},M_{i_u(y-1)})$ and $w_{i_u(y'-1)}$ belongs to a ball of radius $\min(M_{i_u(y')},M_{i_u(y'-1)})$, so we have
\[\frac{\Nf_u^*(k)}{\Nf_u^*}\lesssim \max(M_{i_u(y)},M_{i_u(y-1)})^{-2\alpha_0-\theta}\max(M_{i_u(y')},M_{i_u(y'-1)})^{-2\alpha_0-\theta}\lesssim |k-m^*|^{-2\alpha_0-\theta}|k-m^{**}|^{-2\alpha_0-\theta}.\]

In summary, since at least three elements in all the $A_v$'s and $B_u$'s equal $k$, we conclude that
\[\frac{\Nf(k)}{\Nf}=\prod_{v=1}^s\frac{\Nf_v(k)}{\Nf_v}\prod_{u=1}^t\frac{\Nf_u^*(k)}{\Nf_u^*}\lesssim |k-m^*|^{d-4\alpha_0-\theta}|k-m^{**}|^{d-4\alpha_0-\theta}.\] Since $2(4\alpha_0-d+\theta)>d$ because $d(p-1)\geq 8$, we conclude that
\[\sum_{k\in\Zb^d}|k-m^*|^{d-4\alpha_0-\theta}|k-m^{**}|^{d-4\alpha_0-\theta}\leq O(1)\Rightarrow\sum_{k\in\Zb^d}\Nf(k)\lesssim\Nf,\] which completes the proof.
\end{proof}
\begin{prop}\label{final} Partition $\{1,\cdots,p\}$ into disjoint nonempty subsets $A_1,\cdots,A_s$, $B_1,\cdots,B_t$ and $C$. Assume $A_v=\{\ell_v(1),\cdots ,\ell_v(a_v)\}$, $B_u=\{i_u(1),\cdots ,i_u(b_u)\}$, and $C=\{n_1,\cdots, n_c=1\}$. Consider a tensor $h=h_{kk_1\cdots k_p}$, where each $k,k_j\in\Zb^d$, which satisfies $|h_{kk_1\cdots k_p}|\lesssim 1$, and in the support of $h$ we have (with $m_j\in\Zb^d$ fixed)
  \begin{equation}\label{finalset1}\sum_{j\in A_v}\zeta_j|k_j|^2=\Gamma_v,\quad\bigg|\sum_{z=1}^y\zeta_{\ell_v(z)}k_{\ell_v(z)}-m_{\ell_v(y)}\bigg|\leq M_{\ell_u(y)},\,\forall 1\leq v\leq s,1\leq y\leq a_v=|A_v|.
 \end{equation}
   \begin{equation}\label{finalset2}\sum_{j=1}^p\zeta_j|k_j|^2-|k|^2=\Gamma,\quad \bigg|\sum_{z=1}^y\zeta_{i_u(z)}k_{i_u(z)}-m_{i_u(y)}\bigg|\leq M_{i_u(y)},\,\forall 1\leq u\leq t,1\leq y\leq b_u=|B_u|,
 \end{equation}
    \begin{equation}\label{finalset3}\bigg|\sum_{z=1}^y\zeta_{n_z}k_{n_z}-m_{n_y}\bigg|\leq M_{n_y},\,\forall 1\leq y\leq c-1=|C|-1.
 \end{equation} We assume that $M_{\ell_v(a_v)}=M_{i_u(b_u)}=1$ (denote also $M_{n_c}=1$), and that
 \begin{equation}\label{finalset4}\sum_{j\in A_v}\zeta_j=0\,(\forall 1\leq v\leq s),\,\,\sum_{j=1}^p\zeta_j=1;\quad\sum_{j=1}^p\zeta_jk_j-k=\sum_{j\in A_v}\zeta_jk_j=0.\end{equation} We also assume that any pairing in $(k,k_1,\cdots,k_p)$ is over-paired. Then, for any subset $P_0$ satisfying $P_0\subset C\backslash\{1\}$, let $\{1,\cdots,p\}\backslash P_0=Q_0$, then we have
 \begin{equation}\label{finalbound}\|h\|_{kk_{P_0}\to k_{Q_0}}\lesssim \prod_{j=2}^{p}(M_j)^{\alpha_0+\theta}\prod_{v=1}^s(\min_{1\leq y<a_v}M_{\ell_v(y)})^{\alpha_0-\frac{d}{2}},
 \end{equation} \emph{unless} $(d,p)=(1,7)$, and (up to permutation) that $|A_1|=|A_2|=2$, $k_{\ell_1(1)}=k_{\ell_2(1)}$.
 
Furthermore, if we do not assume that $h$ is supported in the set $\sum_{j=1}^p\zeta_j|k_j|^2-|k|^2=\Gamma$ as in (\ref{finalset2}), but instead assume 
\begin{equation}\label{extrabdh}|h_{k_1\cdots k_p}|\lesssim\frac{1}{\langle \Omega+\Gamma\rangle},\quad \Omega:=|k|^2-\sum_{j=1}^p\zeta_j|k_j|^2,\end{equation}then the same result holds. Finally, all the above results remain true if we replace $p$ by any odd $3\leq q\leq p$ (without changing the value of $\alpha_0$).
\end{prop}
\begin{proof} First assume $d(p-1)\geq 8$. If there is some pairing between $(k,k_{P_0})$ and $k_{Q_0}$, then using the simple fact that
\[\|h_{k_1k_Ak_2k_B}\cdot\mathbf{1}_{k_1=k_2}\|_{k_1k_A\to k_2k_B}\leq\sup_k\|h_{kk_Akk_B}\|_{k_A\to k_B}\] (this is proved in the same way as Lemma \ref{adjustment} below), we may fix the value $k$ of these paired variables with no summation (hence no cost of powers) and reduce to a problem involving a smaller set\footnote{Strictly speaking this reduction may not preserve (\ref{finalset4}), but (\ref{finalset4}) is only used to guarantee $|A_v|\leq p-2$ in order to apply Lemma \ref{general2}; after this (\ref{finalset4}) can be replaced by the more general versions where the linear combinations of $k_j$ and $k$ are fixed $\Zb^d$ vectors instead of $0$, which is preserved under the reduction.}. Therefore we may assume there is no pairing between $(k,k_{P_0})$ and $k_{Q_0}$. We may also assume $|A_v|\leq p-2$ for each $v$, since otherwise $(k,k_1)$ will be a pairing due to (\ref{finalset4}), which has to be over-paired with an element in $A_v$, and after removing these over-paired variables\footnote{These over-paired variables include a pairing between $(k,k_{P_0})$ and $k_{Q_0}$ as $1\in Q_0$, and thus can be treated using the argument in the beginning of the proof.}, the remaining set will satisfy $|A_v|\leq p-2$.

At this point we are ready to apply Lemma \ref{general2}. By Schur's Lemma, we have
\[\|h\|_{kk_{P_0}\to k_{Q_0}}\lesssim\bigg(\sup_{k_{Q_0}}\sum_{k,k_{P_0}}1\bigg)^{1/2}\bigg(\sup_{k,k_{P_0}}\sum_{k_{Q_0}}1\bigg)^{1/2},\] so we just need to count $(k,k_{P_0})$ with $k_{Q_0}$ given, and also count $k_{Q_0}$ with $(k,k_{P_0})$ given. Now Lemma \ref{general2} implies that
\begin{align}\label{schur1}\sup_{k_{Q_0}}\sum_{k,k_{P_0}}1&\lesssim\prod_{j\in P_0}(M_j)^{2\alpha_0+\theta},\\\label{schur2}\sup_{k,k_{P_0}}\sum_{k_{Q_0}}1&\lesssim\prod_{j\in Q_0}(M_j)^{2\alpha_0+\theta}\cdot\prod_{v=1}^s(\min_{1\leq y<a_v}M_{\ell_v(y)})^{2\alpha_0-d}.\end{align} This is because, for example, once $(k,k_{P_0})$ is fixed, for any $n_y\in C\cap Q_0$ we have
\[\sum_{1\leq z\leq y;\,n_z\in Q_0}\zeta_{n_z}k_{n_z}=\sum_{z=1}^y\zeta_{n_z}k_{n_z}+(\textrm{constant vector}),\] so the left hand side sum belongs to a ball of radius $M_{n_y}$. Using also (\ref{finalset1}) and (\ref{finalset2}), and noticing that any pairing in $k_{Q_0}$ must be over-paired, we can deduce (\ref{schur2}) from Lemma \ref{general2}, and similarly (\ref{schur1}). Combining (\ref{schur1}) and (\ref{schur2}) then gives (\ref{finalbound}).

Now consider the exceptional cases $d(p-1)\leq 6$. If $p=3$, then either there is no pairing at all and (\ref{finalbound}) follows from (\ref{schur1}) and (\ref{schur2}), which in turns follow from Lemma \ref{general1}, or there is an over-pairing and each $k_j\,(1\leq j\leq 3)$ is uniquely fixed, in which case (\ref{finalbound}) is immediate.

In the remaining cases we must have $d=1$, and $p\in\{5,7\}$. Again we may assume there is an over-pairing (otherwise (\ref{finalbound}) follows from Lemma \ref{general1}, which does require any condition on $(d,p)$); if $p=5$, then an over-pairing takes at least 3 variables while there are 6 in total ($k$ and each $k_j$), so there are only 3 variables remaining. By using Schur's Lemma and the one-dimensional version of Lemma \ref{cubic}, one can show that in such cases we always have $\|h\|_{kk_{Q_0}\to k_{P_0}}\lesssim (M_2\cdots M_5)^\theta$, which also implies (\ref{finalbound}). Finally if $p=7$, then we have
\[2(2\alpha_0+\theta)>(2\alpha_0+\theta)+(4\alpha_0-d+\theta)>d,\] so we can apply the same arguments in the proof of Lemma \ref{general2} and get the same result, unless we are in the worst case, namely the first line of (\ref{quotientest2}), which implies that (up to permutation) $|A_1|=|A_2|=2$, and $k_{\ell_1(1)}=k_{\ell_2(1)}$.

Finally, we look at the case where (\ref{extrabdh}) is assumed instead of the support condition (the result for $3\leq q\leq p$ follows from the same arguments, which we will not repeat). Here again we can reduce to the case where Lemma \ref{general2} is applicable, and by Schur's Lemma and (\ref{extrabdh}) we have
\[\|h\|_{kk_{P_0}\to k_{Q_0}}\lesssim\bigg(\sup_{k_{Q_0}}\sum_{k,k_{P_0}}\frac{1}{\langle \Omega+\Gamma\rangle}\bigg)^{1/2}\bigg(\sup_{k,k_{P_0}}\sum_{k_{Q_0}}\frac{1}{\langle \Omega+\Gamma\rangle}\bigg)^{1/2},\] so we just need to bound the sums on the right hand side. The idea is that, when $k_{Q_0}$ (or $(k,k_{P_0})$) is fixed, the number of choices for $(k,k_{P_0})$ (or $k_{Q_0}$) is at most $(M_2\cdots M_{p})^d$ due to (\ref{finalset1})--(\ref{finalset4}) (without using the equality for $\Omega$), so by H\"{o}lder
\[\sup_{k_{Q_0}}\sum_{k,k_{P_0}}\frac{1}{\langle \Omega+\Gamma\rangle}\lesssim (M_2\cdots M_{p})^{\theta}\bigg(\sup_{k_{Q_0}}\sum_{k,k_{P_0}}\frac{1}{\langle \Omega+\Gamma\rangle^{1+\theta}}\bigg)^{1/(1+\theta)}.\] Upon fixing the value of $\Omega+\Gamma$, the latter sum can be bounded by (\ref{schur1}); similarly the sum in $k_{Q_0}$ can be bounded by (\ref{schur2}) with a loss of $(M_2\cdots M_{p})^{\theta}$, which can always be incorporated into (\ref{finalbound}). This completes the proof.
\end{proof}
 \subsection{Tensor norm estimates}\label{multi} Here we prove the main estimates for tensor norms. Start with the following simple lemma.
 \begin{lem}\label{adjustment} Let $(B,C)$ be a subpartition of $A$, and let $E=A\backslash (B\cup C)$. Then the norm $\|h\|_{k_B\to k_C}$ increases by at most a constant multiple, under multiplication by:
 \begin{enumerate}[(1)]
 \item Any function $f(k_B,k_E)$, or any function $g(k_C,k_E)$, that is bounded;
 \item Any function of form $\varphi(L^{-1}[f(k_B,k_E)-g(k_C,k_E)])$, where $L>0$ is a real number, $\varphi$ is defined on some $\Rb^m$ such that $\widehat{\varphi}\in L^1$, and $f,g$ are arbitrary $\Rb^m$-valued functions;
 \item Any function of form $\mathbf{1}_{k_i=k_j}$ or $\mathbf{1}_{k_i\neq k_j}$, regardless whether $i$ or $j$ belong to $B$, $C$ or $E$.
 \end{enumerate}
 \end{lem}
 \begin{proof} (1) is obvious by definition, and (2) follows from (1) by writing
 \[\varphi(L^{-1}[f(k_B,k_E)-g(k_C,k_E)])=L^m\int_{\Rb^m}\widehat{\varphi}(L\xi)e^{i\xi\cdot f(k_B,k_E)}e^{-i\xi\cdot f(k_C,k_E)}\,\mathrm{d}\xi.\]
 To prove (3), we may assume $i\in B$ and $j\in C$ (otherwise $i,j\in B$ or $i,j\in C$ or one of them belongs to $E$, and the proof will be easier), and also assume $E=\varnothing$. Let $k_{B\backslash\{i\}}=m$ and $k_{C\backslash\{j\}}=n$, and let $k_i=k_j=k$ after multiplying by $\mathbf{1}_{k_i=k_j}$, then it suffices to prove that
 \[\sum_{k,n}\bigg|\sum_{m}h_{kmkn}\cdot z_{km}\bigg|^2\leq\|h_{k_imk_jn}\|_{k_im\to k_jn}^2\cdot\sum_{k,m}|z_{km}|^2.\] Clearly we may fix $k$, and consider the tensor $h_{kmkn}$, so the desired bound follows from the inequality
 \[\sup_k\|h_{kmkn}\|_{m\to n}\leq \|h_{k_imk_jn}\|_{k_im\to k_jn},\] which is obvious by definition. The result for $\mathbf{1}_{k_i\neq k_j}$ then follows, since $\mathbf{1}_{k_i\neq k_j}=1-\mathbf{1}_{k_i= k_j}$.
 \end{proof} Next we state and prove the bilinear semi-product estimate, equivalent to Proposition \ref{bilineartensor0}.
 \begin{prop}[Restatement of Proposition \ref{bilineartensor0}]\label{bilineartensor} Consider two tensors $h_{k_{A_1}}^{(1)}$ and $h_{k_{A_2}}^{(2)}$, where $A_1\cap A_2=C$. Let $A_1\Delta A_2=A$, define the semi-product
 \begin{equation}\label{combination}H_{k_A}=\sum_{k_C}h_{k_{A_1}}^{(1)}h_{k_{A_2}}^{(2)}.
 \end{equation} Then, for any partition $(X,Y)$ of $A$, let $X\cap A_1=X_1$, $Y\cap A_1 =Y_1$ etc., we have
 \begin{equation}\label{combinationbd}\|H\|_{k_X\to k_Y}\leq \|h^{(1)}\|_{k_{X_1\cup C}\to k_{Y_1}}\cdot\|h^{(2)}\|_{k_{X_2}\to k_{C\cup Y_2}}.
 \end{equation}
 \end{prop}
 \begin{proof} Note that $X_1$, $X_2$, $Y_1$, $Y_2$ and $C$ are five pairwise disjoint sets; let the vectors $x:=k_{X_1}$, $y:=k_{Y_1}$, $z=k_C$, $u:=k_{X_2}$ and $v:=k_{Y_2}$, then we can write
 \[h^{(1)}=h_{xyz}^{(1)},\quad h^{(2)}=h_{uvz}^{(2)};\quad H=H_{xyuv}=\sum_z h_{xyz}^{(1)}h_{uvz}^{(2)}.\] Now for any $\alpha=\alpha_{xu}$, we can bound
 \[
 \begin{aligned}\sum_{y,v}\bigg|\sum_{x,u}H_{xyuv}\alpha_{xu}\bigg|^2&=\sum_{y,v}\bigg|\sum_{x,u,z}h_{xyz}^{(1)}h_{uvz}^{(2)}\alpha_{xu}\bigg|^2=\sum_{v}\sum_y\bigg|\sum_{x,z}h_{xyz}^{(1)}\bigg(\sum_u h_{uvz}^{(2)}\alpha_{xu}\bigg)\bigg|^2\\
 &\leq \|h^{(1)}\|_{xz\to y}^2\cdot\sum_{x,z,v}\bigg|\sum_u h_{uvz}^{(2)}\alpha_{xu}\bigg|^2= \|h^{(1)}\|_{xz\to y}^2\cdot\sum_{x}\sum_{v,z}\bigg|\sum_u h_{uvz}^{(2)}\alpha_{xu}\bigg|^2\\
 &\leq   \|h^{(1)}\|_{xz\to y}^2\cdot \|h^{(2)}\|_{u\to vz}^2\cdot\sum_{x,u}|\alpha_{xu}|^2,
 \end{aligned}
 \] so by definition, $\|H\|_{xu\to yv}\leq \|h^{(1)}\|_{xz\to y}\cdot \|h^{(2)}\|_{u\to vz}$, as desired.
 \end{proof} A corollary is the following multilinear semi-product estimate, equivalent to Proposition \ref{multibound0}.
 \begin{prop}[Restatement of Proposition \ref{multibound0}]\label{multibound} Let $A_j\,(1\leq j\leq m)$ be index sets, such that any index appears in at most two $A_j$'s, and let $h^{(j)}=h_{k_{A_j}}^{(j)}$ be tensors. Let $A=A_1\Delta\cdots\Delta A_m$ be the set of indices that belong to only one $A_j$, and $C=(A_1\cup\cdots \cup A_m)\backslash A$ be the set of indices that belong to two different $A_j$'s. Define the semi-product
 \begin{equation}\label{combination2}H_{k_A}=\sum_{k_C}\prod_{j=1}^mh_{k_{A_j}}^{(j)}.
 \end{equation} Let $(X,Y)$ be a partition of $A$. For $1\leq j\leq m$ let $X_j=X\cap A_j$ and $Y_j=Y\cap A_j$, and define
 \begin{equation}\label{intermedsets}B_j:=\bigcup_{\ell>j}(A_j\cap A_\ell),\quad C_j=\bigcup_{\ell<j}(A_j\cap A_\ell),
 \end{equation} then we have
 \begin{equation}\label{combinationbd2}\|H\|_{k_X\to k_Y}\leq\prod_{j=1}^m\|h^{(j)}\|_{k_{X_j\cup B_j}\to k_{Y_j\cup C_j}}.
 \end{equation} 
 \end{prop}
 \begin{proof} We induct in $m$. When $m=2$, (\ref{combinationbd2}) is just (\ref{combinationbd}); suppose (\ref{combinationbd2}) holds for $m-1$. Then, define $F=A_2\Delta\cdots\Delta A_m$, $E=(A_2\cup\cdots \cup A_m)\backslash F$ and
 \[Y_{k_F}:=\sum_{k_E}\prod_{j=2}^mh_{k_{A_j}}^{(j)},\] then we have $A =A_1\Delta F$, and
 \[H_{k_A}=\sum_{k_G}h_{k_{A_1}}^{(1)}Y_{k_F},\quad G:=A_1\cap F=\bigcup_{\ell>1}(A_1\cap A_\ell)=B_1.\] Applying (\ref{combinationbd}) we get
 \[\|H\|_{k_X\to k_Y}\leq\|h^{(1)}\|_{k_{X_1\cup B_1}\to k_{Y_1}}\cdot\|Y\|_{k_{X\cap F}\to k_{(Y\cap F)\cup B_1}}.\] Note that $X':=X\cap F$ and $Y':=(Y\cap F)\cup B_1$ form a partition of $F$, by induction hypothesis we have
 \[\|Y\|_{k_{X'}\to k_{Y'}}\leq \prod_{j=2}^m\|h^{(j)}\|_{k_{(X'\cap A_j)\cup B_j}\to k_{(Y'\cap A_j)\cup (C_j\backslash B_1)}}.\] Finally, note that $(X'\cap A_j)\cup B_j= X_j\cup B_j$ and $(Y'\cap A_j)\cup (C_j\backslash B_1)=Y_j\cup C_j$, this completes the inductive proof.
 \end{proof}
 \begin{rem} Note that, if we fix $(X,Y)$ and rearrange the tensors $h^{(j)}$, then the expression (\ref{combination2}) will not change, but the norms appearing on the right hand side of (\ref{combinationbd2}) will. We may take advantage of this and arrange these tensors in some order using a particular algorithm so that (\ref{combinationbd2}) gives the desired bound. This will be the key to the proof of Proposition \ref{algorithm1} below.
 \end{rem}
 Finally we state and prove the precise form of Proposition \ref{gausscont0}, and a similar variant. The proofs rely on higher order versions of Bourgain's $\mathscr{G}\mathscr{G}^*$ argument in \cite{Bourgain}.
  \begin{prop}[Precise form of Proposition \ref{gausscont0}]\label{gausscont} Let $A$ be a finite set and $h_{bck_A}=h_{bck_A}(\omega)$ be a tensor, where each $k_j\in \Zb^d$ and $(b,c)\in (\Zb^d)^q$ for some integer $q\geq 2$. Given signs $\zeta_j\in\{\pm\}$, we also assume that $\langle b\rangle,\langle c\rangle\lesssim M$ and $\langle k_j\rangle\lesssim M$ for all $j\in A$, where $M$ is a dyadic number, and that in the support of $h_{bck_A}$ there is no pairing in $k_A$. Define the tensor
  \begin{equation}\label{contract}
  H_{bc}=\sum_{k_{A}}h_{bck_A}\prod_{j\in A}\eta_{k_j}^{\zeta_j},
  \end{equation}where we restrict $k_j\in E$ in (\ref{contract}), $E$ being a finite set such that $\{h_{bck_A}\}$ is independent with $\{\eta_k:k\in E\}$. Then $\tau^{-1}M$-certainly, we have
  \begin{equation}\label{contbd}\|H_{bc}\|_{b\to c}\lesssim\tau^{-\theta} M^\theta\cdot\max_{(B,C)}\|h\|_{bk_B\to ck_C},
  \end{equation} where $(B,C)$ runs over all partitions of $A$.
 \end{prop}
 \begin{proof} By conditioning on $\{h_{bck_A}(\omega)\}$, we may assume $h_{bck_A}$ are constant tensors. View $H$ as a linear operator that maps functions of $c$ to functions of $b$, and consider the kernel of $(HH^*)^m$ for a large positive integer $m$. 
 
Define $R_n=(HH^*)^m$ if $n=2m$, and $R_n=(HH^*)^mH$ if $n=2m+1$. By induction in $n$, we will prove that the kernel of $R_n$ can be written as a linear combination of terms $\Rc_n$ which has the form
 \begin{equation}\label{kernelp}
 \left\{
 \begin{aligned}
 (\Rc_n)_{bb'}&=\sum_{k_Z}y_{bb'k_Z}\prod_{j\in Z}\eta_{k_j}^{\zeta_j},&& n\textrm{ even};\\
  (\Rc_n)_{bc}&=\sum_{k_Z}y_{bck_Z}\prod_{j\in Z}\eta_{k_j}^{\zeta_j},&& n\textrm{ odd},\\
 \end{aligned}
 \right.
 \end{equation} where $Z$ is a finite set, $\zeta_j\in\{\pm\}$, $y_{bb'k_Z}$ (or $y_{bck_Z}$) is a tensor such that in its support, there is no pairing in $k_Z$, and satisfies the bound
 \begin{equation}\label{inductbdy}\|y\|_{bb'k_Z}\,\textrm{(or }\|y\|_{bck_Z})\lesssim \big(\sup_{(B,C)}\|h\|_{bk_B\to ck_C}\big)^{n-1}\|h\|_{bck_A}.
 \end{equation}In fact, when $n=1$ this is obvious (with $Z=A$). Suppose (\ref{kernelp}) and (\ref{inductbdy}) are true for $n-1$, where $n$ is odd, then since $R_n=R_{n-1}H$ it suffices to consider the kernel (note that by relabeling we may assume $Z\cap A=\varnothing$)
  \begin{equation}\label{kernelr}(\Rc_n)_{bc}=\sum_{b'}(\Rc_{n-1})_{bb'}H_{b'c}=\sum_{b'}\sum_{k_Z,k_A}y_{bb'k_Z}h_{b'ck_A}\prod_{j\in Z}\eta_{k_j}^{\zeta_j}\prod_{j\in A}\eta_{k_j}^{\zeta_j}.\end{equation}
 
 Now, by repeating the arguments in the proof of Proposition \ref{formulas}, we can write (\ref{kernelr}) as a linear combination of sums (for different choices of $\Os$), which have the same summand as (\ref{kernelr}) and are taken over the same set of variables ($b'$ and $k_{Z\cup A}$), but with a set of additional pairing and no-pairing conditions for the variables $k_{Z\cup A}$ given by $(\Os,1)$. More precisely, here $\Os=\{\Ac_1,\cdots,\Ac_m\}$ where $\Ac_i$ are pairwise disjoint subsets of $Z\cup A$ such that each subset contains two elements of $Z\cup A$ with opposite $\zeta_j$ sign, but does not contain two elements of $Z$ or two elements of $A$ with opposite $\zeta_j$ sign, and the set of conditions $(\Os,1)$ is defined by (i) the $k_j$ are the same for $j$ in each $\Ac_i$, and (ii) there is no pairing in $k_{(Z\cup A)\backslash\Qc}$ where $\Qc=\Qc(\Ac_1)\cup\cdots\Qc(\Ac_m)$ (see (\ref{def3.6:Q})), as in the proof of Proposition \ref{formulas}.
 
 Since $|\eta_j|^2\equiv 1$, we may recast the sum corresponding to $\Ps$ defined above as  \begin{equation}\label{descent1}\Rc_{bc}=\sum_{k_Y}w_{bck_Y}\prod_{j\in Y}\eta_{k_j}^{\zeta_j},\quad w_{bck_Y}=\prod_{(j,j')}^{(1)}\mathbf{1}_{k_j\neq k_{j'}}\cdot\sum_{b'}\sum_{k_{\Qc}}^{(2)}\widetilde{y}_{bb'k_Z}\widetilde{h}_{b'ck_A}.\end{equation} Here $Y=(Z\cup A)\backslash \Qc$, the product $\prod_{(j,j')}^{(1)}$ is taken over all $j,j'\in Y$ such that $\zeta_{j'}=-\zeta_{j}$, the sum $\sum_{k_{\Qc}}^{(2)}$ is taken over the variables $k_{\Qc}$ such that $k_j=k_{j'}$ whenever $\{j,j'\}$ is one of the opposite-sign 2-element subsets (pairings) selected when obtaining $\Qc$ as in Definition \ref{defmerge}, and
 \begin{equation}\label{descent2}\widetilde{y}_{bb'k_Z}:=y_{bb'k_Z}\cdot\prod_{(j,j')}^{(3)}\mathbf{1}_{k_j=k_{j'}},\quad \widetilde{h}_{b'ck_A}:=h_{b'ck_A}\cdot\prod_{(j,j')}^{(3)}\mathbf{1}_{k_j=k_{j'}},\end{equation} where the products $\prod_{(j,j')}^{(3)}$ are taken over all $j,j'\in Z$ (for $y$, or $j,j'\in A$ for $h$) such that they belong to the same $\Ac_i$. We shall apply Proposition \ref{bilineartensor} to estimate $\|w_{bck_Y}\|_{bck_Y}$; in order to do so we need to make an adjustment in notations. Namely, for any pairing $\{j,j'\}$, as we always require $k_j=k_{j'}$ in the sum $\sum_{k_{\Qc}}^{(2)}$, we may combine them into a single element and include this element in both $Z$ and $A$. In this way we are changing pairings between $Z$ and $A$ to intersections of $Z$ and $A$, which is the setting of Proposition \ref{bilineartensor}.

With (\ref{descent1}), (\ref{descent2}) and these adjustments, by Lemma \ref{adjustment} and Proposition \ref{bilineartensor}, we conclude that
 \[\|w_{bck_Y}\|_{bck_Y}\lesssim\|y_{bb'k_Z}\|_{bb'k_Z}\cdot\|h_{b'ck_A}\|_{ck_C\to b'k_B},\] where $B=\Qc\cap A$, and $C=A\backslash \Qc$. This completes the inductive proof of (\ref{kernelp}) and (\ref{inductbdy}) when $n$ is odd. When $n$ is even noticing that since $R_n=R_{n-1}H^*$, we have
   \begin{equation}\label{kernelr2}(\Rc_n)_{bb'}=\sum_{c}(\Rc_{n-1})_{bc}\overline{H_{b'c}}=\sum_{c}\sum_{k_Z,k_A}y_{bck_Z}\overline{h_{b'ck_A}}\prod_{j\in Z}\eta_{k_j}^{\zeta_j}\prod_{j\in A}\eta_{k_j}^{-\zeta_j}\end{equation} instead of (\ref{kernelr}), and the rest of proof goes analogously.
 
Now consider the product $(HH^*)^m$ with $n=2m$. Using (\ref{inductbdy}), Lemma \ref{largedev0} and noticing that the number of choices for $(b,b')$ is at most $M^{O(1)}$, we conclude that $\tau^{-1}M$-certainly, we have
  \begin{multline*}\|H_{bc}\|_{b\to c}^{4m}=\|(HH^*)^m\|_{\mathrm{OP}}^2\lesssim\sum_{b,b'}|(\Rc_n)_{bb'}|^2\lesssim(\tau^{-1}M)^\theta\|y\|_{bb'k_Z}^2\\\lesssim(\tau^{-1}M)^\theta \big(\sup_{(B,C)}\|h\|_{bk_B\to ck_C}\big)^{4m-2}\|h\|_{bck_A}^2,\end{multline*} and hence
  \[\|H_{bc}\|_{b\to c}\lesssim (\tau^{-1}M)^\theta\big(\sup_{(B,C)}\|h\|_{bk_B\to ck_C}\big)^{1-\frac{1}{2m}}\|h\|_{bck_A}^{\frac{1}{2m}}.\] Fixing $m$ large enough, and noticing that by the support condition, 
  \[\|h\|_{bck_A}\lesssim M^{\frac{q+|A|}{2}}\sup_{b,c,k_A}|h_{bck_A}|\lesssim M^{\frac{q+|A|}{2}}\sup_{(B,C)}\|h\|_{bk_B\to ck_C},\]we deduce (\ref{contbd}).
  \end{proof}
\begin{prop}[A variant of Proposition \ref{gausscont}]\label{gausscont2} Consider the same setting as in Proposition \ref{gausscont}, with the following differences: (1) we only restrict $\langle k_j\rangle\lesssim M$ but do not impose any condition on $\langle b\rangle$ or $\langle c\rangle$; (2) we assume $b,c\in\Zb^d$ also, and assume that in the support of $h_{bck_A}$ we have $|b-\zeta c|\lesssim M$ where $\zeta\in\{\pm\}$; (3) the tensor $h_{bck_A}$ only depends on $b-\zeta c$, $|b|^2-\zeta |c|^2$ and $k_A$, and is supported in the set where $||b|^2-\zeta|c|^2|\leq M^{\kappa^3}$. The other assumptions are the same. Then $\tau^{-1}M$-certainly we have
  \begin{equation*}\|H_{bc}\|_{b\to c}\lesssim \tau^{-\theta}M^\theta\cdot\sup_{(B,C)}\|h\|_{bk_B\to ck_C}.
  \end{equation*}
\end{prop}
\begin{proof} We may assume $\zeta=+$, since the other case is much easier. Since $h_{bck_A}$ is supported in the set where $|b-c|\lesssim M$, by an orthogonality argument we may modify $h$ by restricting it to the set where $|b-f|\lesssim M$ and $|c-f|\lesssim M$, and to bound $\|H_{bc}\|_{b\to c}$ it suffices to bound these restricted operators \emph{uniformly} in $f\in\Zb^d$.

For any $f$, let $x=b-f$ and $y=c-f$, then $x$ and $y$ are both assumed to have size $\lesssim M$, and it suffices to estimate the norm
\[\|\widetilde{H}_{f;xy}\|_{x\to y},\quad\mathrm{where}\quad \widetilde{H}_{f;xy}=\sum_{k_A}\widetilde{h}_{f;xyk_A}\prod_{j\in A}\eta_{k_j}^{\zeta_j},\quad\mathrm{and}\quad\widetilde{h}_{f;xyk_A}:=h_{x+f,y+f,k_A}\cdot\mathbf{1}_{|x|,|y|\lesssim M}.\] For any \emph{fixed} value of $f$, we may apply Proposition \ref{gausscont} to conclude that $\tau^{-1}M$-certainly we have
\begin{equation}\label{thisbound}\|\widetilde{H}_{f;xy}\|_{x\to y}\lesssim \tau^{-\theta}M^\theta \cdot \sup_{(B,C)}\|\widetilde{h}_{f;xyk_A}\|_{xk_B\to yk_C}\leq \tau^{-\theta}M^\theta\cdot \sup_{(B,C)}\|h\|_{bk_B\to ck_C},\end{equation} so it suffices to establish (\ref{thisbound}) \emph{uniformly} in $f$. Note that by assumption, $\widetilde{h}_{f;xyk_A}$ is in fact a function of $(x,y,k_A)$ and $|x+f|^2-|y+f|^2=2f\cdot(x-y)+(|x|^2-|y|^2)$. The desired uniform bound (\ref{thisbound}), and hence the proof of Proposition \ref{gausscont2}, will follow from the following statement:
\begin{claim}\label{uniformclaim} There exist finitely many integer-valued functions $g_j(z)$ (defined on a subset $E_j\subset\{z:\langle z\rangle\lesssim M\}$), where $1\leq j\leq K\leq M^{\kappa^4}$, such that for \emph{any} integer vector $f\in \Zb^d$, there exists $1\leq j\leq A$, such that for any $z$ satisfying $\langle z\rangle\lesssim M$, we have $|f\cdot z|\leq M^{2\kappa^3}$ if and only if $z\in E_j$, and in such case we have $f\cdot z=g_j(z)$.
\end{claim}
\begin{proof}[Proof of claim \ref{uniformclaim}] The proof is a slight modification of (a special case of) the proof of \cite{DH}, Claim 3.7, so we refer the reader to that paper.
\end{proof}\phantom\qedhere
\end{proof}
 \section{The random tensor ansatz}\label{ansatz} In this section we begin the main proof. We make several reductions to the equation (\ref{nlstrunc}) in Section \ref{reduct}, write down the central random tensor ansatz in Section \ref{constructansatz}, and state the key a priori estimates, Proposition \ref{mainprop}, in Section \ref{section5.3}.
 \subsection{First reductions}\label{reduct} We start by analyzing (\ref{nlstrunc}). The first step is to reduce it to a more suitable form. This is done by using a gauge transform, conditioning on the norms of Fourier modes of (\ref{data0}), and conjugating by the linear Schr\"{o}dinger flow.
\subsubsection{The gauge transform}\label{gauging} Define the \emph{gauge transform}
\begin{equation}\label{gauge}\widetilde{u_N}(t)=u_N(t)\cdot\exp\bigg(\frac{(p+1)i}{2}\int_0^t\fint_{\Tb^d}W_N^{p-1}(u_N)\,\mathrm{d}t'\bigg),
\end{equation} which has inverse
\begin{equation}\label{gaugeinv}u_N(t)=\widetilde{u_N}(t)\cdot\exp\bigg(\frac{-(p+1)i}{2}\int_0^t\fint_{\Tb^d}W_N^{p-1}(\widetilde{u_N})\,\mathrm{d}t'\bigg),
\end{equation} then $u_N$ satisfies (\ref{nlstrunc}) if and only if $\widetilde{u_N}$ satisfies the gauged equation
\begin{equation}\label{nlstruncgauged}
\left\{
\begin{aligned}(i\partial_t+\Delta)\widetilde{u_N}&=\Pi_N\bigg(W_N^{p}(\widetilde{u_N})-\frac{p+1}{2}\fint_{\Tb^d}W_N^{p-1}(\widetilde{u_N})\cdot \widetilde{u_N}\bigg),\\
\widetilde{u_N}(0)&=\Pi_Nf(\omega).
\end{aligned}
\right.
\end{equation} The nonlinearity in the big parenthesis of (\ref{nlstruncgauged}) can be recast in the following form:
\begin{equation}\label{simpleterm}W_N^{p}(\widetilde{u_N})-\frac{p+1}{2}\fint_{\Tb^d}W_N^{p-1}(\widetilde{u_N})\cdot \widetilde{u_N}=\sum_{3\leq q\leq p}a_{pq}(m_N-\sigma_N)^{(p-q)/2}\Nc_q(\widetilde{u_N}),\end{equation} where $q$ runs over odd integers, $a_{pq}$ are constants, $m_N$ denotes the conserved mass of $\widetilde{u_N}$ (and $u_N$),
\[m_N=\fint_{\Tb^d}|\widetilde{u_N}|^2=\sum_{\langle k\rangle\leq N}\frac{|g_k|^2}{\langle k\rangle^{2\alpha}},\] $\sigma_N$ is as in (\ref{truncmass}), and $\Nc_q$ is a degree $q$ real polynomial (regarded also as a $\Rb$-multilinear expression of degree $q$) that is simple in the sense of Definition \ref{def:simple}. For the derivation of (\ref{simpleterm}), see \cite{DNY}, Proposition 2.2.
\subsubsection{Conditioning and conjugating}\label{condition1}
Note that each $m_N$ is a function of the norms $\rho_k=|g_k|$, moreover let $m_N^*:=m_N-\sigma_N$, then by standard large deviation estimates,
\begin{equation}\label{devmass}\big|m_N^*-m_{\frac{N}{2}}^*\big|^2=\bigg|\sum_{N/2<\langle k\rangle\leq N}\frac{\rho_k^2-1}{\langle k\rangle^{2\alpha}}\bigg|^2\leq\tau^{-\theta}N^\theta\sum_{N/2<\langle k\rangle\leq N}\frac{1}{\langle k\rangle^{4\alpha}}\leq \tau^{-\theta} N^{d-4\alpha+\theta}\leq \tau^{-\theta}N^{-40\varepsilon}\end{equation} holds $\tau^{-1}N$-certainly, as $4\alpha-d>80\varepsilon$ by our assumptions.

Now, by excluding a set of probability $\leq C_\theta e^{-\tau^{-\theta}}$ and conditioning on $\{\rho_k\}$, we may fix the values of $\rho_k$ and hence $m_N^*$, so that $\widetilde{u_N}$ solves the equation
\begin{equation}\label{nlstruncgauged2}
\left\{
\begin{aligned}(i\partial_t+\Delta)\widetilde{u_N}&=\sum_{3\leq q\leq p}a_{pq}(m_N^*)^{(p-q)/2}\cdot\Pi_N\Nc_q(\widetilde{u_N}),\\
\widetilde{u_N}(0)&=\sum_{k\in\Zb^d}\Pi_N\gamma_k\cdot\eta_k(\omega)e^{ik\cdot x},
\end{aligned}
\right.
\end{equation} where $\gamma_k=\langle k\rangle^{-\alpha}\rho_k$, and they and the $m_N^*$ are constants that satisfy
\begin{equation}\label{constants}|\gamma_k|\leq \tau^{-\theta}\langle k\rangle^{-\alpha+\theta},\quad |m_N^*|\leq \tau^{-\theta},\quad \big|m_N^*-m_{\frac{N}{2}}^*\big|\leq \tau^{-\theta}N^{-40\varepsilon}.
\end{equation} We may also assume, due to (\ref{devmass}), that
\begin{equation}\label{constants2}\sum_{N/2<\langle k\rangle\leq N}\gamma_k^2\leq\tau^{-\theta}N^{d-2\alpha}.
\end{equation} Finally, define $v_N$ by $(v_N)_k(t):=e^{it|k|^2}(\widetilde{u_N})_k(t)$, then $v_N$ satisfies the following equation
\begin{equation}\label{reducedeqn}(v_N)_k(t)=(F_N)_k-i\sum_{3\leq q\leq p}a_{pq}(m_N^*)^{(p-q)/2}\int_0^t\Pi_N\Mc_q(v_N,\cdots,v_N)_k(t')\,\mathrm{d}t',
\end{equation} where the initial data $F_N$ is defined by 
\begin{equation}\label{deffnfn}(F_N)_k:=\Pi_N\gamma_k\cdot\eta_k(\omega)=\sum_{N'\leq N}(f_{N'})_k,\quad \mathrm{where}\quad(f_{N})_k:=\Delta_N\gamma_k\cdot\eta_k(\omega)
\end{equation} (this $f_N$ will be the one appearing in Definitions \ref{deftensor} and \ref{deftrim} and in particular the $\mathtt{Trim}$ functions), $\Mc_q$ is a conjugated version of $\Nc_q$, and is given by
\begin{equation}\label{multilin}\Mc_q(v^{(1)},\cdots, v^{(q)})_k(t')=\sum_{\zeta_1k_1+\cdots +\zeta_qk_q=k}c_{kk_1\cdots k_q}\cdot e^{it'\Omega}\prod_{j=1}^q(v^{(j)})_{k_j}^{\zeta_j}(t').
\end{equation} In (\ref{multilin}), the signs $(\zeta_1,\cdots,\zeta_q)=(+,-,\cdots,+)$, and the coefficients $c_{kk_1\cdots k_q}$ satisfy the simplicity condition in the sense of Definition \ref{def:simple}. Finally $\Omega$ is defined by
\begin{equation}\label{defomega}\Omega=|k|^2-|k_1|^2+\cdots -|k_q|^2=|k|^2-\sum_{j=1}^q\zeta_j|k_j|^2.
\end{equation} Below we will focus on the system (\ref{reducedeqn})--(\ref{multilin}) on $J=[-\tau,\tau]$, with the parameters satisfying (\ref{constants})--(\ref{constants2}). Using (\ref{defi}), (\ref{reducedeqn}) can be rewritten as
 \begin{equation}\label{reducedeqn2}
 (v_N)_k(t)=(F_N)_k(t)-i\sum_{3\leq q\leq p}a_{pq}(m_N^*)^{(p-q)/2}\cdot\Ic\Pi_N\Mc_q(v_N,\cdots,v_N)_k(t),
 \end{equation} where $\Ic$ is as in (\ref{defi}). In order to use the global-in-time norms defined in Section \ref{globalnorms}, we need to construct functions $v_N^\dagger$ that are well-controlled for all time, and agree with $v_N$ on $J$. The strategy is to construct $v_N^\dagger$ by the time truncated system
  \begin{equation}\label{reducedeqn3}
 (v_N^\dagger)_k(t)=\chi(t)\cdot(F_N)_k-i\sum_{3\leq q\leq p}a_{pq}(m_N^*)^{(p-q)/2}\chi_\tau(t)\cdot\Ic_\chi\Pi_N\Mc_q(v_N,\cdots,v_N)_k(t),
 \end{equation} where $\Ic_\chi$ is as in (\ref{defi}). Clearly if $v_N^\dagger$ solves (\ref{reducedeqn3}) then they must agree with the solution $v_N$ to (\ref{reducedeqn2}) on $J$. Unlike $v_N$, which always solve (\ref{reducedeqn2}), the $v_N^\dagger$ we construct are solutions to (\ref{reducedeqn3}) only $\tau^{-1}$-certainly, i.e. apart from a set of probability $\leq C_\theta e^{-\tau^{-\theta}}$.
 \subsection{Construction of tensors}\label{constructansatz} In this section we present the random tensor ansatz. The reader may recall that the core idea of this ansatz was presented for a simpler model in Section \ref{simansatz}. Suppose $v_N^\dagger$ solves (\ref{reducedeqn3}), and let $y_N$ be defined by
  \begin{equation}\label{defyndagger}y_N=v_N^\dagger-v_{\frac{N}{2}}^\dagger;\quad v_N^\dagger=\sum_{N'\leq N}y_{N'},\end{equation} then $y_N$ solves the system
   \begin{equation}\label{eqnyn}
 \begin{split}(y_N)_k(t)&=\chi(t)\cdot (f_N)_k-i\sum_{3\leq q\leq p}a_{pq}(m_N^*)^{(p-q)/2}\chi_\tau(t)\cdot\Ic_\chi\Delta_N\Mc_q(v_{\frac{N}{2}}^\dagger,\cdots v_{\frac{N}{2}}^\dagger)_k(t)\\
 &-i\sum_{3\leq q\leq p}a_{pq}(m_N^*)^{(p-q)/2}\chi_\tau(t)\cdot\Ic_\chi\Pi_N\big[\Mc_q(v_N^\dagger,\cdots,v_N^\dagger)-\Mc_q(v_{\frac{N}{2}}^\dagger,\cdots v_{\frac{N}{2}}^\dagger)\big]_k(t)\\
 &-i\sum_{3\leq q\leq p}a_{pq}\big[(m_N^*)^{(p-q)/2}-(m_{\frac{N}{2}}^*)^{(p-q)/2}\big]\chi_\tau(t)\cdot\Ic_\chi\Pi_{\frac{N}{2}}\Mc_q(v_{\frac{N}{2}^\dagger},\cdots v_{\frac{N}{2}}^\dagger)_k(t),
 \end{split}
 \end{equation} where $\Mc_q$ is as in (\ref{multilin}). Conversely if $y_N$ solves (\ref{eqnyn}), then $v_N^\dagger$ solves (\ref{reducedeqn3}) where we understand that $v_{1/2}^\dagger=0$. So it suffices to construct $y_N$.
 
 We shall construct $y_N$ by an ansatz which involves $\Sc$-tensors $h^{(\Sc,n)}=h_{kk_\Uc}^{(\Sc,n)}(t,k_\Vc,\lambda_\Vc)$  for $n\in\{0,1\}$ and \emph{regular} plants $\Sc$ having frequency $N(\Sc)=N$ and size $|\Sc|\leq D$, as well as a remainder term $z_N$. Here $D$ is as in (\ref{defdelta}). This construction will be inductive, first in $N$ and then in $|\Sc|$ with fixed $N$. As the base case we understand that all these quantities are $0$ when (formally) $N$ is $1/2$.
 
\emph{Step 1: the induction hypothesis.} Now, given dyadic $M\geq 1$, assume that we have already defined the $\Sc$-tensors $h^{(\Sc,n)}$ for all $n\in\{0,1\}$, all regular plants $\Sc$ with $N(\Sc)<M$ and $|\Sc|\leq D$, as well as $z_{N'}=(z_{N'})_k(t)$ for $N'<M$. For $N<M$, define
\begin{equation}\label{ansatzsum}
\begin{aligned}(y_N)_k(t)&=\sum_{n\in\{0,1\}}\sum_{\substack{\Sc:N(\Sc)=N\\|\Sc|\leq D}}\sum_{k_\Uc,k_\Vc}\int\mathrm{d}\lambda_\Vc\cdot h_{kk_\Uc}^{(\Sc,n)}(t,k_\Vc,\lambda_\Vc)\cdot\prod_{\lf\in\Uc}(f_{N_\lf})_{k_\lf}^{\zeta_\lf}\prod_{\ff\in\Vc}(\widehat{z_{N_\ff}})_{k_\ff}^{\zeta_\ff}(\lambda_\ff)+(z_N)_k(t),\\
(v_N^\dagger)_k(t)&=\sum_{N'\leq N}(y_{N'})_k(t);
\end{aligned}
\end{equation} note that the first equation in (\ref{ansatzsum}) can also be written as
\begin{equation}\label{ansatzalt}(y_N)_k(t)=\sum_{n\in\{0,1\}}\sum_{\substack{\Sc:N(\Sc)=N\\|\Sc|\leq D}}\Psi_k[\Sc,h^{(\Sc,n)}(t)]+(z_N)_k(t)
\end{equation} in view of (\ref{defpsi}). Here and throughout the proof, the $f_{N'}$ in (\ref{defpsi}) will be fixed as in (\ref{deffnfn}), and $(z_{N'})_{N'<M}$ will be fixed as above. Moreover, define the $\Rb$-linear operator (which plays the role of $\Pc_{NL}$ in Section \ref{raointro})
\begin{equation}\label{deflin}
(\Ls^{M}w)_k(t)=-i\sum_{3\leq q\leq p}a_{pq}(m_M^*)^{(p-q)/2} \chi_\tau(t)\cdot\Ic_\chi\Pi_M\sum_{\mathrm{sym}}\Mc_q(w,v_{M^{[\delta]}}^\dagger,\cdots,v_{M^{[\delta]}}^\dagger)_k(t),
\end{equation} where $\sum_{\mathrm{sym}}$ represents the sum over all possible permutations, for example 
\[\sum_{\mathrm{sym}}\Mc(w,v,v):=\Mc(w,v,v)+\Mc(v,w,v)+\Mc(v,v,w).\] Let the components $\Ls^{M,\zeta}$, as well as the kernels $(\Ls^{M,\zeta})_{kk'}(t,t')$, be defined as in Section \ref{notations}. Let the $\Rb$-linear operator $\Rs^{M}=(1-\Ls^{M})^{-1}$, which is bounded from $X^{b_0}$ to itself, if $\|\Ls^{M}\|_{X^{b_0}\to X^{b_0}}< 1/2$ (with $b_0$ in (\ref{defb})); otherwise let $\Rs^{M}=\mathrm{Id}$. Define also $\Vs^M=\Rs^M-1$. The goal is to define the $\Sc$-tensors $h^{(\Sc,n)}$ for $n\in\{0,1\}$ and regular plants $\Sc$ having $N(\Sc)=M$ and $|\Sc|\leq D$, and the remainder term $z_M$, such that $y_M$ defined by (\ref{ansatzsum}) with $N$ replaced by $M$ solves (\ref{eqnyn}) with high probability.

\emph{Step 2: paralinearization.} If we assume $\mathscr{R}^M=(1-\mathscr{L}^M)^{-1}$, then using the operator $\Ls^M$, we can paralinearize (\ref{eqnyn}) and rewrite it as
\begin{equation}\label{workingeqn}(y_M)_k(t)=\chi(t)\cdot (f_M)_k+(\Ls^My_M)_k(t)+\sum_{3\leq q \leq p}\sum_{N_1,\cdots,N_q\leq M}\Upsilon\cdot\chi_\tau(t)\big[\Ic_\chi\Pi\Mc_q(y_{N_1},\cdots y_{N_q})\big]_k(t).
\end{equation} In the above summation $q$ is odd, $\Pi$ is one of the projections $\Pi_M$, $\Delta_M$ or $\Pi_{\frac{M}{2}}$, and we require that if $N_j=M$ for some $j$, then there must be another $j'\neq j$ such that $N_{j'}\geq M^\delta$ (otherwise the term is contained in the second term $(\Ls^My_M)_k(t)$). The coefficient $\Upsilon$ depends only on $q$ and the $N_j$'s, and satisfies $|\Upsilon|\leq \tau^{-\theta}$; moreover if $N_j\leq (50dp)^{-1}M$ for each $1\leq j\leq q$, we have the stronger bound $|\Upsilon|\leq\tau^{-\theta}M^{-40dp\varepsilon}$.

Using the operators $\Rs^M$, $\Vs^M$ and their kernels defined by (\ref{kerneltime})--(\ref{kernelfourier}), we can rewrite (\ref{workingeqn}) as
\begin{equation}\label{workingeqn2}
\begin{aligned}(y_M)_k(t)&=\chi(t)\cdot (f_M)_k+\sum_{\zeta\in\{\pm\}}\sum_{k'}\int\mathrm{d}t'\cdot(\Vs^{M,\zeta})_{kk'}(t,t')\chi(t')\cdot (f_M)_{k'}^\zeta\\
&+\sum_{\zeta\in\{\pm\}}\sum_{k'}\int\mathrm{d}t'\cdot(\Rs^{M,\zeta})_{kk'}(t,t')\sum_{3\leq q \leq p}\sum_{N_1,\cdots,N_q\leq M}\Upsilon\cdot\chi_\tau(t')\big[\Ic_\chi\Pi\Mc_q(y_{N_1},\cdots y_{N_q})\big]_{k'}^{\zeta}(t').
\end{aligned}
\end{equation} The strategy is to construct the tensors $h^{(\Sc,n)}$ with $N(\Sc)=M$ inductively in $|\Sc|$, such that when we plug (\ref{ansatzsum}) into (\ref{workingeqn2}) allowing $N=M$, the terms on the left and right sides cancel to sufficiently high order so that the remainders can be put in $z_M$.

\emph{Step 3: definition of $h^{(\Sc,n)}$.} Expanding the right hand side of (\ref{workingeqn2}) using (\ref{ansatzsum}) and allowing $N=M$, we obtain a sum of terms of the form (omitting $\Rs^{M,\zeta}$ and other factors for the moment)
\begin{equation}\label{sums1}\sum_{\mathrm{sym}}\Mc_q(\Psi_{k_1}^{(\Sc_1,n_1)},\cdots,\Psi_{k_r}^{(\Sc_r,n_r)},z_{N_{r+1}},\cdots,z_{N_q})_{k'}(t'),
\end{equation} where $\Psi_{k_j}^{(\Sc_j,n_j)}=\Psi_{k_j}^{(\Sc_j,n_j)}(t')=\Psi_{k_j}[\Sc_j,h^{(\Sc_j,n_j)}(t')]$.

Let $\Bs=(M,q,r,\zeta_1,\cdots,\zeta_q,N_1,\cdots,N_q)$, note that $\sum_{j=1}^q\zeta_j=1$. By Proposition \ref{formulas} (1) and (2), if $N_j\leq M/2$ for each $r+1\leq j\leq q$, and each $z_{N_j}$ in (\ref{sums1}) is replaced by its low-modulation cutoff\footnote{This matches Definition \ref{deftensor}.} $z_{N_j}^{\mathrm{lo}}$ defined by \[(\widehat{z_{N_j}^{\mathrm{lo}}})_{k_j}(\lambda_j)=(\widehat{z_{N_j}})_{k_j}(\lambda_j)\cdot\chi(M^{-\kappa^2}\lambda_j)\] (see Section \ref{notations} for the definition of $\chi$), then (\ref{sums1}) can be recast as a linear combination of $\widetilde{\Psi}_k^{(\Sc)}=\widetilde{\Psi}_k^{(\Sc)}(t')=\Psi_k[\Sc,H]$ (for different choices of $\Os$ as in Definition \ref{defmerge}), where
\begin{equation}\label{newpsi}\Sc=\mathtt{Trim}(\mathtt{Merge}(\mathtt{Trim}(\Sc_1,M^\delta),\cdots,\mathtt{Trim}(\Sc_r,M^\delta),\Bs,\Os),M^\delta)
\end{equation}
\begin{equation}\label{newH}H=\mathtt{Trim}(\mathtt{Merge}(\mathtt{Trim}(h^{(\Sc_1,n_1)},M^\delta),\cdots,\mathtt{Trim}(h^{(\Sc_r,n_r)},M^\delta),h,\Bs,\Os),M^\delta).
\end{equation} In (\ref{newH}) the tensor $h=h_{kk_1\cdots k_q}(t',\lambda_{r+1},\cdots,\lambda_q)$ is given by
\begin{multline}\label{basetensor}h_{kk_1\cdots k_q}(t',\lambda_{r+1},\cdots,\lambda_q)=\mathbf{1}_{k=\zeta_1k_1+\cdots +\zeta_qk_q}\cdot\mathbf{1}_{\langle k\rangle\leq M}\prod_{j=1}^q\mathbf{1}_{\langle k_j\rangle\leq  N_j}\\\times\prod _{j=r+1}^q\chi(M^{-\kappa^2}\lambda_j)c_{kk_1\cdots k_q}e^{it'(\Omega+\zeta_{r+1}\lambda_{r+1}+\cdots +\zeta_q\lambda_q)}
\end{multline}with $c_{kk_1\cdots k_q}$ as in (\ref{multilin}) and $\Omega$ as in (\ref{defomega}). Here and throughout the proof, when applying $\mathtt{Trim}$ functions, we always fix $f_{N'}$ as in (\ref{deffnfn}), and $(z_{N'})_{N'<M}$ as in the beginning of \emph{Step 1}.

We define the $h^{(\Sc,n)}$ tensors inductively in $|\Sc|$, by the equations
\begin{equation}\label{deftensor1}h_{kk_\Uc}^{(\Sc,0)}(t,k_\Vc,\lambda_\Vc)=\chi(t)\mathbf{1}_{\Sc=\Sc_M^+}\cdot\mathbf{1}_{k=k_\lf}\mathbf{1}_{M/2<\langle k\rangle\leq M}+\sum_{\mathrm{sym}}\sum_{(a)}\Upsilon\cdot\chi_\tau(t)\big[\Ic_\chi\Pi H_{kk_\Uc}\big](t,k_\Vc,\lambda_\Vc),
\begin{aligned}
\end{aligned}
\end{equation}
\begin{equation}\label{deftensor2}\begin{aligned}h_{kk_\Uc}^{(\Sc,1)}(t,k_\Vc,\lambda_\Vc)&=\sum_{\zeta\in\{\pm\}}\mathbf{1}_{\Sc=\Sc_M^\zeta}\int\mathrm{d}t'\cdot\mathbf{1}_{M/2<\langle k_\lf\rangle\leq M}\cdot\Vs_{kk_\lf}^{M,\zeta}(t,t')\chi(t')\\
&+\sum_{\mathrm{sym}}\sum_{(b)}\Upsilon\cdot\chi_\tau(t)\big[\Ic_\chi\Pi H_{kk_\Uc}\big](t,k_\Vc,\lambda_\Vc)\\
&+\sum_{\zeta\in\{\pm\}}\sum_{k'}\int\mathrm{d}t'\cdot (\Vs^{M,\zeta})_{kk'}(t,t')\sum_{\mathrm{sym}}\sum_{(c[\zeta])}\Upsilon\cdot\chi_\tau(t')\big[\Ic_\chi\Pi H_{k'k_\Uc}\big](t',k_\Vc,\lambda_\Vc)^\zeta.
\end{aligned}
\end{equation} 
Here $\Sc_M^\zeta$ are the mini plants defined in Definition \ref{defstr}. The summation $\sum_{(c[\zeta])}$ is taken over $\Bs$, $n_j\in\{0,1\}$ and regular plants $\Sc_j$ with frequency $N_j\leq M$ and size $|\Sc_j|\leq D$ for $1\leq j\leq r$, and\footnote{Strictly speaking the sum over $\Os$ should carry the coefficients in the linear combination of $\widetilde{\Psi}_k^{(\Sc)}$ that gives (\ref{sums1}) as above; these are constants, and for simplicity we will treat them as $1$.} $\Os$, such that
\begin{enumerate}[(i)]
\item if $N_j=M$ for some $1\leq j\leq q$ then there is $q\geq j'\neq j$ with $N_{j'}\geq M^\delta$;
\item $N_j\leq M/2$ for $r+1\leq j\leq q$;
\item if $\zeta=+$ then (\ref{newpsi}) is true with the given $\Sc$, and if $\zeta=-$ then (\ref{newpsi}) is true with the left hand side replaced by $\overline{\Sc}$.
\end{enumerate} The term $H_{k'k_\Uc}(t',k_\Vc,\lambda_\Vc)$ that appears in the summand is defined in (\ref{newH}) with $h^{(\Sc_j,n_j)}$ given by the induction hypothesis. The summation $\sum_{(a)}$ is taken over the same set of variables as $\sum_{(c[+])}$ but with the restrictions (in addition to those in $\sum_{(c[\zeta])}$) that $q=r$, $n_j=0$ and $N_{\lf}\geq M^\delta$ for each $j$ and each $\lf\in\Lc_j$ (where $\Lc_j$ is the set of leaves of $\Sc_j$, see Remark \ref{notationrem}), and $\sum_{(b)}=\sum_{(c[+])}-\sum_{(a)}$.

The above is a valid inductive definition, i.e. the tensors $h^{(\Sc_j,n_j)}$ in (\ref{newH}) are already defined when we use them to define $h^{(\Sc,n)}$ via (\ref{deftensor1})--(\ref{deftensor2}), thanks to Proposition \ref{formulas} (4). Note that the first term on the right hand side of (\ref{deftensor1}) and the first line in (\ref{deftensor2}) are precisely the random $(1,1)$ tensors described in Section \ref{raointro}; the rest come from higher order iterations in Section \ref{sec2.4:randomtensors}.

\emph{Step 4: definition of $z_M$.} Now we have finished the inductive definition of $\Sc$-tensors $h^{(\Sc,n)}$ for $n\in\{0,1\}$ and regular plants $\Sc$ of frequency $N(\Sc)\leq M$ and $|\Sc|\leq D$. Using (\ref{defpsi}) this gives definition to terms $\Psi_k^{(\Sc,n)}=\Psi_k^{(\Sc,n)}(t')=\Psi_k[\Sc,h^{(\Sc,n)}]$ for such $n$ and $\Sc$. Finally we shall construct $z_M$ to complete the inductive definition; this is simply defined to be the solution to the equation 
\begin{multline}\label{eqnzm}(z_M)_k(t)=\sum_{\zeta\in\{\pm\}}\sum_{k'}\int\mathrm{d}t'\cdot(\Rs^{M,\zeta})_{kk'}(t,t')\sum_{\mathrm{sym}}\sum_{(d)}\Upsilon\cdot\chi_\tau(t')\\\times\big[\Ic_\chi\Pi\Mc_q(\Psi_{k_1}^{(\Sc_1,n_1)},\cdots,\Psi_{k_r}^{(\Sc_r,n_r)},z_{N_{r+1}}^*,\cdots,z_{N_q}^*)\big]_{k'}^\zeta(t'),
\end{multline} where $z_{N_j}^*\,(r+1\leq j\leq q)$ is either $z_{N_j}$ or $z_{N_j}^{\mathrm{lo}}$ or the high-modulation cutoff $z_{N_j}^{\mathrm{hi}}:=z_{N_j}-z_{N_j}^{\mathrm{lo}}$. Here in (\ref{eqnzm}), the sum $\sum_{(d)}$ is taken over $\Bs$, $n_j\in\{0,1\}$, regular plants $\Sc_j$ with frequency $N_j$ and size $|\Sc_j|\leq D$ for $1\leq j\leq r$, and choices of $z_{N_j}^*$, under the restrictions that (i) if $N_j=M$ for some $1\leq j\leq q$ then there is $q\geq j'\neq j$ with $N_{j'}\geq M^\delta$, (ii) either $N_j=M$ for at least one $r+1\leq j\leq q$ and $z_{N_j}^*=z_{N_j}$ for all $r+1\leq j\leq q$, or $N_j\leq M/2$ for all $r+1\leq j\leq q$ and $z_{N_j}^*=z_{N_j}^{\mathrm{hi}}$ for at least one $r+1\leq j\leq q$, or $(N_j\leq M/2)\wedge(z_{N_j}^*=z_{N_j}^{\mathrm{lo}})$ for all $r+1\leq j\leq q$ and the plant
\begin{equation}\label{defnewS}\Sc=\mathtt{Trim}(\mathtt{Merge}(\mathtt{Trim}(\Sc_1,M^\delta),\cdots,\mathtt{Trim}(\Sc_r,M^\delta),\Bs,\Os),M^\delta)\end{equation} has size\footnote{Note that $\Os$ actually does not appear in the summation (\ref{eqnzm}). But this is fine, since it can easily be checked that the \emph{size} of $\Sc$ defined by (\ref{defnewS}) does not depend on $\Os$.} $|\Sc|>D$.

Note that in (\ref{eqnzm}), all terms $\Psi_{k_j}^{(\Sc_j,n_j)}=\Psi_{k_j}[\Sc_j,h^{(\Sc_j,n_j)}]$ are already defined for $n_j\in\{0,1\}$ and regular plants $\Sc_j$ with $N(\Sc_j)\leq M$ and $|\Sc_j|\leq D$, so (\ref{eqnzm}) can be viewed as an equation for the function $z_M=(z_M)_k(t)$. If the mapping defined by the right hand side of (\ref{eqnzm}) is a contraction mapping from the set $\{z_M:\|z_M\|_{X^{b_0}}\leq M^{-D_1}\}$ to itself, we define $z_M$ to be the unique fixed point of this mapping; otherwise define $z_M=0$.

This finishes the inductive definition of the tensors $h^{(\Sc,n)}$ and the remainder $z_N$, which give the definition of $y_N$ by the ansatz (\ref{ansatzsum}).
 \subsection{The a priori estimates}\label{section5.3} With the complete definition of $h^{(\Sc,n)}$ tensors and $z_N$, we can now state the main a priori estimates for these terms. First notice that they satisfy the following simple properties, which are easily verified (using Definitions \ref{deftrim} and \ref{defmerge}, and Proposition \ref{formulas}) during the construction process:
 \begin{itemize}
 \item The tensors $h^{(\Sc,0)}$ are constant (i.e. do not depend on $\omega$) and are nonzero only when $\Sc$ is plain, the tensors $h^{(\Sc,1)}$ are $\Bc_{N^{[\delta]}}$ measurable (recall Section \ref{notations} for definition) for any $\Sc$ with $N(\Sc)=N$ and $|\Sc|\leq D$, and the remainder $z_N$ is $\Bc_N$ measurable;
 \item All these terms are supported in  $|t|\leq 1$, and $(z_N)_k(t)$ is supported in $\langle k\rangle\leq N$. In the support of $h_{kk_\Uc}^{(\Sc,n)}(t,k_{\Vc},\lambda_\Vc)$ we have $\langle k\rangle\leq N$, that $N_\lf/2<\langle k_\lf\rangle\leq N_\lf$ for each $\lf\in\Uc$, $\langle k_\ff\rangle\leq N_\ff$ and $|\lambda_\ff|\leq 2N^{\kappa^2} $ for each $\ff\in\Vc$, and that there is no pairing in $k_{\Uc}$.
 \end{itemize} The main a priori estimates are listed in the following proposition.
 \begin{prop}\label{mainprop} Given a dyadic $M$, consider the following set of statements (viewed as an event for $\omega$), which we shall refer to as $\mathtt{Local}(M)$ below: 

(1) For any plain regular plant $\Sc=(\Lc,\varnothing,\Yc)$ with $N(\Sc)=N< M$ and $|\Sc|\leq D$, we have that
\begin{equation}\label{support01} h^{(\Sc,0)}=h_{kk_\Uc}^{(\Sc,0)}(t)\mathrm{\ is\ supported\ in\ the\ set\ where\ }k=\sum_{\lf\in\Uc}\zeta_\lf k_\lf.
\end{equation}For any $\Gamma\in\Zb$, let $h^{(\Sc,0,\Gamma)}$ be the restriction of $h^{(\Sc,0)}$ to the set where
 \begin{equation}\label{support02} |k|^2-\sum_{\lf\in\Uc}\zeta_\lf|k_\lf|^2=\Gamma,
 \end{equation} obtained by multiplying by the indicator function of this set. Let $(B,C)$ be a subpartition of $\Uc$ and let $E=\Uc\backslash(B\cup C)$. Then we have
 \begin{equation}\label{estimate01}\int_\Rb\langle \lambda\rangle^{2b}\bigg(\sum_{\Gamma\in\Zb}\|\widehat{h_{kk_\Uc}^{(\Sc,0,\Gamma)}}(\lambda)\|_{kk_B\to k_C}\bigg)^2\,\mathrm{d}\lambda\leq\bigg(\prod_{\lf\in B\cup C}N_\lf^{\beta_1}\prod_{\lf\in\Pc}N_\lf^{-8\varepsilon}\prod_{\lf\in E}N_\lf^{4\varepsilon}\prod_{\pf\in\Yc}N_\pf^{-\delta^3}\cdot\Xc_0\Xc_1\bigg)^2,
 \end{equation} with the quantities
 \begin{equation}\label{estimate02}\Xc_0=\left\{
  \begin{aligned}&\big(\max_{\lf\in C}N_\lf\big)^{-\beta_1},&&\mathrm{if\ }C\neq\varnothing;\\
    & \big(\min_{\lf\in\Lc}N_\lf\big)^{\frac{d}{2}-\beta_1},&&\mathrm{if\ }C=E=\varnothing;\\
   & N^{-\varepsilon\delta},&&\mathrm{if\ }B=C=\varnothing;\\
    &1,&&\mathrm{otherwise},
 \end{aligned}
 \right.\qquad\Xc_1=\left\{
 \begin{aligned}&1,&{\rm{if }}\max_{\lf\in \Lc}N_\lf&\geq (10^3dp)^{-|\Lc|}N;\\
&N^{-4\varepsilon},&{\rm{if }}\max_{\lf\in \Lc}N_\lf&< (10^3dp)^{-|\Lc|}N.
 \end{aligned}
 \right.
 \end{equation}
 
 (2) For any regular plant $\Sc=(\Lc,\Vc,\Yc)$ with $N(\Sc)=N< M$ and $|\Sc|\leq D$, consider the tensor $h^{(\Sc,1)}=h_{kk_\Uc}^{(\Sc,1)}(t,k_\Vc,\lambda_\Vc)$. Let $(B,C)$ be a subpartition of $\Uc$ and let $E=\Uc\backslash(B\cup C)$. Then, recall the norms defined in (\ref{newnorm})--(\ref{newnorm2}), for $C\neq\varnothing$ we have \begin{equation}\label{estimate11}\|h_{kk_\Uc}^{(\Sc,1)}(t,k_\Vc,\lambda_\Vc)\|_{X_\Vc^{1-b,-b_0}[kk_B\to k_C]}\leq\prod_{\lf\in B\cup C}N_\lf^\beta\prod_{\lf\in\Pc}N_\lf^{-4\varepsilon}\prod_{\lf\in E}N_\lf^{8\varepsilon}\prod_{\pf\in\Yc}N_\pf^{-\delta^3}\prod_{\ff \in \Vc} N_\ff^d
 \cdot 
\big(\max_{\lf\in C}N_\lf\big)^{-\beta},
 \end{equation} while for $C=\varnothing$ we have
 \begin{equation}
 \label{estimate12}\|h_{kk_\Uc}^{(\Sc,1)}(t,k_\Vc,\lambda_\Vc)\|_{X_\Vc^{\widetilde{b},-b_0}[kk_B]}\leq\prod_{\lf\in B}N_\lf^\beta\prod_{\lf\in\Pc}N_\lf^{-4\varepsilon}\prod_{\lf\in E}N_\lf^{8\varepsilon}\prod_{\pf\in\Yc}N_\pf^{-\delta^3}\prod_{\ff \in \Vc} N_\ff^d\cdot N^{-\varepsilon},
 \end{equation} where $\widetilde{b}$ equals $1-b$ if $E\neq \varnothing$ and equals $b$ if $E=\varnothing$. We also have a localization bound (for $C=\varnothing$)
 \begin{equation} \label{estimate13}\bigg\|\bigg(1+\frac{1}{N^{2\delta}}\big|k-\sum_{\lf\in\Uc}\zeta_\lf k_\lf-\ell\big|\bigg)^\kappa h_{kk_\Uc}^{(\Sc,1)}(t,k_\Vc,\lambda_\Vc)\bigg\|_{X_\Vc^{\widetilde{b},-b_0}[kk_B]}\leq\prod_{\lf\in B}N_\lf^\beta\prod_{\lf\in\Pc}N_\lf^{-4\varepsilon}\prod_{\lf\in E}N_\lf^{8\varepsilon}\prod_{\pf\in\Yc}N_\pf^{-\delta^3}\prod_{\ff \in \Vc} N_\ff^d,
 \end{equation} where $\ell=\sum_{\ff\in\Vc}\zeta_\ff k_\ff$, and $\widetilde{b}$ is the one in (\ref{estimate12}).
 
 Finally, we have an auxiliary bound for the $\lambda_\Vc$-derivative\footnote{Note that, once we have (\ref{estimate14}), we automatically have the same bound for the norm with the weight in (\ref{estimate13}), with the right hand side multiplied by $N^\kappa$, which is negligible as the right hand side is super-polynomial.} of $h_{kk_\Uc}^{(\Sc,1)}(t,k_\Vc,\lambda_\Vc)$ (for $C=E=\varnothing$),
  \begin{equation} \label{estimate14}\|\partial_{\lambda_\Vc}h_{kk_\Uc}^{(\Sc,1)}(t,k_\Vc,\lambda_\Vc)\|_{X_\Vc^{b,-b_0}[kk_\Uc]}\leq \exp[(\log N)^5+|\Sc|(\log N)^3].
 \end{equation}
 
 (3) For $n\in\{0,1\}$ and regular plant $\Sc$ with $N(\Sc)=N<M$ and $|\Sc|\leq D$, let the expression $\Psi_{k}^{(\Sc,n)}=\Psi_k[\Sc,h^{(\Sc,n)}]$ be defined as in (\ref{defpsi}), then we have
 \begin{equation}\label{psibd}\|\Psi^{(\Sc,n)}\|_{X^{s',b_0}}\leq\tau^{-\theta_0}N^{s'-s}\prod_{\nf\in\Lc\cup\Vc\cup\Yc}N_\nf^{-\delta^3}
 \end{equation} for any $s-\delta^2<s'<s$, where $(s,b_0,\theta_0)$ are defined in (\ref{choicealpha}) and Section \ref{notations}.
 
  (4) For all $N<M$, the mapping that defines $z_N$ (namely the right hand side of (\ref{eqnzm}) but with $M$ replaced by $N$ and then $z_N$ replaced by an independent variable $z$) is a contraction mapping from $\{z:\|z\|_{X^{b_0}}\leq N^{-D_1}\}$ to itself with $D_1$ as in (\ref{defb}). In particular, we have $\|z_N\|_{X^{b_0}}\leq N^{-D_1}$ for all $N<M$.
 
 (5) Let $y_{N}$ and $v_{N}^\dagger$ be defined as in (\ref{ansatzsum}) for $N<M$, then they solve the system (\ref{eqnyn}) for $N<M$. Moreover, for any $N_2,\cdots,N_{q}< M$ and any $N$ (which may be $\geq M$), consider the operator $\Ls^\zeta$ (which is complex linear if $\zeta=+$, and conjugate complex linear if $\zeta=-$) defined by
 \begin{equation}\label{deflinoper}(\Ls^\zeta w)_k(t)=\chi_\tau(t)\cdot\Ic_\chi\Pi_N\Mc_q(y_{N_2}^*,\cdots,w,\cdots, y_{N_{q}}^*)_k(t),
 \end{equation} and the corresponding kernel $(\Ls^\zeta)_{kk'}(t,t')$, where each $y_{N_j}^*$ is either $y_{N_j}$ or one of its components, namely $z_{N_j}$ (possibly with Fourier truncations similar to the ones in $z_N^{\mathrm{hi}}$ and $z_N^{\mathrm{lo}}$ defined in Section \ref{constructansatz}) or $\Psi[\Sc_j,h^{(\Sc_j,n_j)}]$ for some $n_j\in\{0,1\}$ and regular plant $\Sc_j$ with $N(\Sc_j)=N_j$ and $|\Sc_j|\leq D$ (see (\ref{defpsi}) and (\ref{ansatzsum})), then they satisfy
 \begin{equation}\label{linbd}\|\Ls^\zeta\|_{X^{1-b,-b}[k\to k']}\leq \tau^{(5\kappa)^{-1}}\big(\max_{2\leq j\leq q}N_j\big)^{-4\varepsilon\delta}.
 \end{equation}
 
Now, with the above definition of $\mathtt{Local}(M)$, we have that \emph{the probability that $\mathtt{Local}(M)$ holds but $\mathtt{Local}(2M)$ does not hold} is $\leq C_\theta e^{-(\tau^{-1}M)^\theta}$. In particular, $\tau^{-1}$-certainly, $\mathtt{Local}(M)$ holds for all $M$.
  \end{prop}
   \section{Trimming and merging estimates}\label{mainproof} In this and in the next section we prove Proposition \ref{mainprop}. This section is devoted to the proof of some important estimates on trimming and merging of tensors, which will be crucial to the proof we will give in Section \ref{mainproof1}. These are: trimming bounds in Section \ref{6.1}, no-over-pairing merging bounds in Section \ref{mainproof2} and general merging bounds in Section \ref{6.3}. Throughout this and the next section we will fix a dyadic $M$ (we may always assume $M\gg_{C_\theta}1$, since otherwise the relevant bounds become trivial as $\tau\ll_{C_\theta}1$), and assume that the statement $\mathtt{Local}(M)$, defined in Proposition \ref{mainprop}, is already true. If results in this and the next section rely on more assumptions (such as parts of $\mathtt{Local}(2M)$ that have been established in preceding proofs), we will explicitly point this out.
  
All the quantities (functions, tensors, etc.) in this and the next section that depend on $t$ (or $t'$ etc.) will be supported in $|t|\leq1$ (or $|t'|\leq 1$ etc.). This implies that their derivatives in the time-Fourier variables ($\lambda$, $\lambda_j$ etc.) automatically satisfy the same bounds as they do; these derivative bounds will be useful in application of \emph{meshing arguments} (see the proof of Proposition \ref{trimbd} for details) below. Moreover due to $\mathtt{Local}(M)$ part (4), the functions $z_{N'}=(z_{N'})_k(t)$ for $N'<M$ are already defined and satisfy that $\|z_{N'}\|_{X^{b_0}}\leq (N')^{-D_1}$. When applying $\mathtt{Trim}$ functions below we always fix these $z_{N'}$, and the $f_{N'}$ defined in (\ref{deffnfn}).
  \subsection{Trimming estimates}\label{6.1} We first prove the trimming estimates.
  \begin{prop}[Trimmed tensor bounds]\label{trimbd} Let $\Sc$ be a regular plant, $N(\Sc)=N\leq M$ and $|\Sc|\leq D$. Let $h=h_{kk_\Uc}(k_\Vc,\lambda_\Vc)$ be an $\Sc$-tensor which is $\Bc_{N^{[\delta]}}$ measurable. For $N^\delta\leq R\leq M^\delta$, consider the trimmed plant $\Sc'=(\Lc',\Vc',\Yc')=\mathrm{Trim}(\Sc,R)$ and the trimmed tensor $h'=(h')_{kk_{\Uc'}}(k_{\Vc'},\lambda_{\Vc'})=\mathtt{Trim}(h,R)$. Let $(B,C)$ be a subpartition of $\Uc$ with $E=\Uc\backslash(B\cup C)$ and $(B',C')$ be a subpartition of $\Uc'$ with $E'=\Uc'\backslash(B'\cup C')$.
  
  (1) Assume $h$ satisfies that, for any $(B,C)$ with $C\neq\varnothing$,
  \begin{equation}\label{trim-bd1}\|h_{kk_\Uc}(k_\Vc,\lambda_\Vc)\|_{X_\Vc^{-b_0}[kk_B\to k_C]}\lesssim\Xf\cdot\prod_{\lf\in B\cup C}N_\lf^\beta\prod_{\lf\in\Pc}N_\lf^{-4\varepsilon}\prod_{\lf\in E}N_\lf^{8\varepsilon}\prod_{\pf\in\Yc}N_\pf^{-\delta^3}\prod_{\ff \in \Vc} N_\ff^d\cdot \big(\max_{\lf\in C}N_\lf\big)^{-\beta},
  \end{equation} and for $C=\varnothing$,
  \begin{equation}
 \label{trim-bd2}\|h_{kk_\Uc}(k_\Vc,\lambda_\Vc)\|_{X_\Vc^{-b_0}[kk_B]}\lesssim\Xf\cdot\prod_{\lf\in B}N_\lf^\beta\prod_{\lf\in\Pc}N_\lf^{-4\varepsilon}\prod_{\lf\in E}N_\lf^{8\varepsilon}\prod_{\pf\in\Yc}N_\pf^{-\delta^3}\prod_{\ff \in \Vc} N_\ff^d\cdot N^{-\varepsilon},
 \end{equation}
 \begin{equation} \label{trim-bd3}\bigg\|\bigg(1+\frac{1}{N^{2\delta}}\big|k-\sum_{\lf\in\Uc}\zeta_\lf k_\lf-\ell\big|\bigg)^\kappa h_{kk_\Uc}(k_\Vc,\lambda_\Vc)\bigg\|_{X_\Vc^{-b_0}[kk_B]}\lesssim\Xf\cdot\prod_{\lf\in B}N_\lf^\beta\prod_{\lf\in\Pc}N_\lf^{-4\varepsilon}\prod_{\lf\in E}N_\lf^{8\varepsilon}\prod_{\pf\in\Yc}N_\pf^{-\delta^3}\prod_{\ff \in \Vc} N_\ff^d,
 \end{equation} where $\ell=\sum_{\ff\in\Vc}\zeta_\ff k_\ff$, and also the auxiliary bound for $C=E=\varnothing$,
 \begin{equation}\label{trim-bd4}\|\partial_{\lambda_\Vc}h_{kk_\Uc}(k_\Vc,\lambda_\Vc)\|_{X_\Vc^{-b_0}[kk_\Uc]}\leq\Xf\cdot\exp[(\log N)^5+|\Sc|(\log N)^3].
 \end{equation} Then, $\tau^{-1}M$-certainly, the estimates (\ref{trim-bd1})--(\ref{trim-bd3}) hold for $h=h_{kk_\Uc}(k_\Vc,\lambda_\Vc)$ replaced by $h'=h_{kk_{\Uc'}}'(k_{\Vc'},\lambda_{\Vc'})$, the sets $B,C,E,\Pc,\Vc$ etc. replaced by $B',C',E',\Pc',\Vc'$ etc., the fraction $1/N^{2\delta}$ in (\ref{trim-bd3}) replaced by $1/\max(N^{2\delta},R)$, and the factor $\Xf$ replaced by $\Xf\cdot\tau^{-\theta}M^\theta$.
 
 (2) Assume $\Vc=\varnothing$, $h$ is supported in the set $k=\sum_{\lf\in\Uc}\zeta_\lf k_\lf$ and satisfies that, for any $(B,C)$,
 \begin{equation}
 \label{trim-bd5}\|h_{kk_\Uc}\|_{kk_B\to k_C}\leq\Xf\cdot \prod_{\lf\in B\cup C}N_\lf^{\beta_1}\prod_{\lf\in\Pc}N_\lf^{-8\varepsilon}\prod_{\lf\in E}N_\lf^{4\varepsilon}\prod_{\pf\in\Yc}N_\pf^{-\delta^3}\cdot\Xc_0\Xc_1,
 \end{equation} where $\Xc_0$ and $\Xc_1$ are defined as in (\ref{estimate02}). Assume also that either $C'\cup E'\neq\varnothing$, or $\Lc\neq\Lc'$ (i.e. $N_\lf<R$ for at least one $\lf\in\Lc$), then $\tau^{-1}M$-certainly, the estimates (\ref{trim-bd1})--(\ref{trim-bd3}) hold for $h'=h_{kk_{\Uc'}}'$, with the sets $B,C,E,\Pc$ etc. replaced by $B',C',E',\Pc'$ etc., the fraction $1/N^{2\delta}$ in (\ref{trim-bd3}) replaced by $1/R$, and the factor $\Xf$ replaced by $\Xf\cdot \tau^{-\theta}M^\theta $ for (\ref{trim-bd1}), and by $\Xf\cdot \tau^{-\theta}M^\theta (1+N^{-3\varepsilon}R^{d/2-\beta_1})$ for (\ref{trim-bd2})--(\ref{trim-bd3}).
  \end{prop}
  \begin{proof} (1) By definition we have
\begin{equation}\label{trimexpand}
\begin{split} h'=h_{kk_{\Uc'}}'(k_{\Vc'},\lambda_{\Vc'})&=\sum_{k_{\Vc\backslash\Vc'}}\int\mathrm{d}\lambda_{\Vc\backslash\Vc'}\cdot \widetilde{h}_{kk_{\Uc'}}(k_{\Vc},\lambda_{\Vc})\prod_{\ff\in\Vc\backslash\Vc'}(\widehat{z_{N_\ff}})_{k_\ff}^{\zeta_\ff}(\lambda_\ff),\\
\widetilde{h}_{kk_{\Uc'}}(k_{\Vc},\lambda_{\Vc})&:=\sum_{k_{\Uc\backslash\Uc'}}h_{kk_\Uc}(k_\Vc,\lambda_\Vc)\prod_{\lf\in\Uc\backslash\Uc'}(f_{N_\lf})_{k_\lf}^{\zeta_\lf}.
\end{split}\end{equation} By Cauchy-Schwartz we have 
\begin{equation}\label{blostep0}\|h'\|_{X_{\Vc'}^{-b_0}[\cdots]}\lesssim\|\widetilde{h}\|_{X_\Vc^{-b_0}[\cdots]}\cdot\prod_{\ff\in\Vc\backslash\Vc'}\|z_{N_\ff}\|_{X^{b_0}}\lesssim \|\widetilde{h}\|_{X_\Vc^{-b_0}[\cdots]}\cdot\prod_{\ff\in\Vc\backslash\Vc'} N_\ff^{-D_1},\end{equation} where $[\cdots]$ represents any $kk_{B'}\to k_{C'}$ or $kk_{B'}$ or weighted norm as in (\ref{trim-bd1})--(\ref{trim-bd3}).

It thus suffices to bound the corresponding norms for $\widetilde{h}$. Note that if we \emph{fix the values of} $k_\Vc$ and $\lambda_\Vc$, the tensor $\widetilde{h}_{kk_{\Uc'}}(k_\Vc,\lambda_\Vc)$ can be estimated $\tau^{-1}M$-certainly, using Proposition \ref{gausscont}; in order to transform this to a bound for the $X_\Vc^{-b_0}[\cdots]$ norms, we need to make this bound \emph{uniform in all choices of $k_\Vc$ and $\lambda_\Vc$}. There is no problem in doing so for $k_\Vc$ since\footnote{The same applies to the $k_E$ variables when measuring $kk_B\to k_C$ norms, where $E=\Uc\backslash(B\cup C)$.} the number of choices for $k_\Vc$ is at most $M^{\kappa}$. To deal with $\lambda_\Vc$, we will employ the following argument, which will be referred to as the \emph{meshing argument} (see the proof of Lemma 4.2 in \cite{DNY}), and will be used frequently in the proofs below. First note that $|\lambda_\ff|\lesssim N^{\kappa^2}$ for each $\ff\in\Vc$, then we divide the big box $\{\lambda_\Vc:|\lambda_\ff|\lesssim N^{\kappa^2},\forall\,\ff\in\Vc\}$ into small boxes $\Bf$ of size $\nu:=\exp(-(\log N)^6)$. Now by taking averages on these small boxes and using Poincar\'{e} inequality, there exists a tensor $h_{\mathrm{avg}}=(h_{\mathrm{avg}})_{kk_\Uc}(k_\Vc,\lambda_\Vc)$ such that $h_{\mathrm{avg}}$ satisfies the same measurability condition as $h$, that $h_{\mathrm{avg}}$ is supported in the big box and constant when $\lambda_\Vc$ moves within each small box (and other parameters are fixed), and that
\begin{equation}\label{approx}\|h-h_{\mathrm{avg}}\|_{X_\Vc^{-b_0}[kk_\Uc]}\lesssim\nu\cdot(\|h\|_{X_\Vc^{-b_0}[kk_\Uc]}+\|\partial_{\lambda_\Vc}h\|_{X_\Vc^{-b_0}[kk_\Uc]}).\end{equation} Let $\widetilde{h}_{\mathrm{avg}}$ be defined from $h_{\mathrm{avg}}$, in the same way as $\widetilde{h}$ is defined from $h$, then for \emph{fixed values of} $\lambda_\Vc$, the tensor $(\widetilde{h}_{\mathrm{avg}})_{kk_\Uc'}(k_\Vc,\lambda_\Vc)$ can be estimated $\tau^{-1}M$-certainly, using Proposition \ref{gausscont}, in the same way that $\widetilde{h}$ is bounded in terms of $h$. Since $\lambda_{\Vc}$ has at most $\exp(\kappa(\log N)^6)$ different choices in studying $h_{\mathrm{avg}}$, and $N\leq M$, we know that the estimate for $\widetilde{h}_{\mathrm{avg}}$ is \emph{uniform in all choices of} $\lambda_\Vc$, after removing an exceptional set whose probability is still $\leq C_\theta e^{-(\tau^{-1}M)^\theta}$. This gives the $X_\Vc^{-b_0}[\cdots]$ norm bounds for $\widetilde{h}_{\mathrm{avg}}$; but by (\ref{trim-bd4}), (\ref{approx}) and our choice for $\nu$, the $X_\Vc^{-b_0}[\cdots]$ norm of the difference $\widetilde{h}-\widetilde{h}_{\mathrm{avg}}$ is negligible, so we get the desired $X_\Vc^{-b_0}[\cdots]$ norm bounds for $\widetilde{h}$.

Armed with the meshing argument, we can now apply Proposition \ref{gausscont} to control the $X_\Vc^{-b_0}[\cdots]$ norms of $\widetilde{h}$. Given any subpartition $(B',C')$ of $\Uc'$, $C'\neq\varnothing$, let $E=E'=\Uc'\backslash(B'\cup C')$; since $h$ is $\Bc_{N^{[\delta]}}$ measurable and $N_\lf>N^{[\delta]}$ for all $\lf\in\Uc$, we can apply Proposition \ref{gausscont} (note that $(f_{N_\lf})_{k_\lf}=\Delta_{N_\lf}\gamma_{k_\lf}\cdot\eta_{k_\lf}(\omega)$, where $\Delta_{N_\lf}\gamma_{k_\lf}$ can be replaced by $\tau^{-\theta}N_\lf^{-\alpha+\theta}$ due to (\ref{constants}) and Lemma \ref{adjustment}) to get
\[\|\widetilde{h}_{kk_{\Uc'}}(k_\Vc,\lambda_\Vc)\|_{X_\Vc^{-b_0}[kk_{B'}\to k_{C'}]}\lesssim\tau^{-\theta}M^\theta\prod_{\lf\in\Uc\backslash\Uc'}N_\lf^{-\alpha+\theta}\cdot\sup_{(F,G)}\|h_{kk_{\Uc}}(k_\Vc,\lambda_\Vc)\|_{X_\Vc^{-b_0}[kk_{B}\to k_{C}]},\] where $(F,G)$ is any partition of $\Uc\backslash\Uc'$, and $B=B'\cup F$, $C=C'\cup G$. The conclusion about the $X_{\Vc'}^{-b_0}[kk_{B'}\to k_{C'}]$ norms then follows by combining (\ref{trim-bd1}) and (\ref{blostep0}), and noticing that $\Pc'\subset\Pc$, $\Yc'\subset\Yc$, and $\max_{\lf\in C'}N_\lf=\max_{\lf\in C}N_\lf$. When $C'=\varnothing$, the proofs for the $X_{\Vc'}^b[kk_{B'}]$ norms are completely analogous (simply choose $G=\varnothing$ so $C=\varnothing$), and so is the weighted norm bound, where for the latter we notice that
\[1+\frac{1}{\max(N^{2\delta},R)}\bigg|k-\sum_{\lf\in\Uc'}\zeta_\lf k_\lf-\ell'\bigg|\lesssim 1+\frac{1}{N^{2\delta}}\bigg|k-\sum_{\lf\in\Uc}\zeta_\lf k_\lf-\ell\bigg|\] with $\ell'=\sum_{\ff\in\Vc'}\zeta_\ff k_\ff$ and $\ell=\sum_{\ff\in\Vc}\zeta_\ff k_\ff$, since $\langle k_\ff\rangle\leq R$ for any $\ff\in\Vc\backslash\Vc'$ (hence $|\ell-\ell'|\lesssim R$) and $\langle k_\lf\rangle\leq R$ for any $\lf\in\Uc\backslash\Uc'$.

(2) The proof is similar to (1) but much easier since there is no blossom $\ff\in\Vc$ (hence no meshing argument) involved. Since $|k-\sum_{\lf\in\Uc'}\zeta_\lf k_\lf|\lesssim R$ because $k-\sum_{\lf\in\Uc}\zeta_\lf k_\lf=0$ and $\langle k_\lf\rangle\leq R$ for $\lf\in\Uc\backslash\Uc'$, we know that the weighted norm bound (\ref{trim-bd3}) follows from the unweighted norm bound (\ref{trim-bd2}). For the $kk_{B'}\to k_{C'}$ norms in (\ref{trim-bd1})--(\ref{trim-bd2}), we simply apply Proposition \ref{gausscont}; when $C'\neq\varnothing$ we readily get (\ref{trim-bd1}) with the indicated changes. When $C'=\varnothing$ (and $C=\varnothing$) we make two observations. First, when $E=E'=\varnothing$ there is an extra factor $\Xc_0$ in (\ref{trim-bd5}), but since $\min_{\lf\in\Lc}N_\lf\leq R$ due to the assumption $\Lc\neq\Lc'$, we have $\Xc_0\lesssim R^{d/2-\beta_1}$ by definition (\ref{estimate02}), which gives rise to the factor $R^{d/2-\beta_1}$ in the desired estimate. Second, thanks to the different powers between (\ref{trim-bd1})--(\ref{trim-bd3}) and (\ref{trim-bd5}), in the process of using (\ref{trim-bd5}) for $h$ to deduce (\ref{trim-bd1})--(\ref{trim-bd2}) for $h'$, we will be gaining a factor
\[\prod_{\lf\in \Uc\backslash E'}N_\lf^{\beta_1-\beta}\prod_{\lf\in E'\cup\Pc}N_\lf^{-4\varepsilon}\leq\big(\max_{\lf\in\Lc}N_\lf\big)^{-4\varepsilon},\] which is at most $N^{-4\varepsilon}$ if $\max_{\lf\in\Lc}N_\lf\geq (10^3dp)^{-|\Lc|}N$. If $\max_{\lf\in\Lc}N_\lf< (10^3dp)^{-|\Lc|}N$ we gain exactly the same $N^{-4\varepsilon}$ from the factor $\Xc_1$ in (\ref{estimate02}). In any case, this gain will contribute the $N^{-3\varepsilon}$ in the desired estimate, after providing the $N^{-\varepsilon}$ factor for (\ref{trim-bd2}).
  \end{proof}
  \subsection{No-over-pairing merging estimates}\label{mainproof2} Next we will prove two merging estimates in the no-over-pairing case by introducing a \emph{selection algorithm}. Note that in Propositions \ref{algorithm1} and \ref{algorithm2} below, the sets $\Uc_j$ are just sets by themselves and are \emph{not} coming from any plant; nevertheless in applications, they do occur as suitable \emph{subsets} of the sets coming from some plants, see the proof of Propositions \ref{overpair} and \ref{overpair2}.
  \begin{prop}[Selection algorithm: Case I]\label{algorithm1} Let $\Uc_2,\cdots,\Uc_{p}$ be pairwise disjoint finite index sets (could be empty), $|\Uc_j|\leq D$. Given $\zeta_j\in\{\pm\}$ for $1\leq j\leq p$ and $\zeta_\lf\in\{\pm\}$ for any $\lf\in\Uc_j$, and $N_j$ for $2\leq j\leq p$ and $N_\lf$ for any $\lf\in\Uc_j$, let $N_*=\max(N_2,\cdots,N_{p})$, and define $\zeta_\lf^*=\zeta_\lf\zeta_j$ for $\lf\in\Uc_j$. Assume that \[\sum_{j=1}^p\zeta_j=1,\quad N_j^\delta\leq N_\lf\leq N_j\,(\forall \lf\in \Uc_j,\,2\leq j\leq p).\]  Assume there are some pairings in $\Wc:=\Uc_2\cup\cdots\cup\Uc_{p}$ (i.e. a collection of pairwise disjoint two-element subsets of $\Wc$, each containing two elements from two different $\Uc_j$), such that for any pair $(\lf,\lf')$ we have $\zeta_{\lf'}^*=-\zeta_\lf^*$ and $N_{\lf'}=N_\lf$. Let $h^{(j)}=h_{k_jk_{\Uc_j}}^{(j)}$, where $2\leq j\leq p$, be tensors, and $h=h_{kk_1\cdots k_p}$ be a tensor supported in $\langle k_j\rangle\leq N_j$ for $2\leq j\leq p$. Let the set of paired elements in $\Wc$ be $\Qc$ and the set of unpaired elements be $\Uc$, define the semi-product
\begin{equation}\label{defH}H=H_{kk_1k_\Uc}=\sum_{k_2,\cdots,k_{p}}\sum_{k_\Qc}h_{kk_1\cdots k_p}\prod_{j=2}^{p}\big(h_{k_jk_{\Uc_j}}^{(j)}\big)^{\zeta_j},\end{equation} where the sum is taken over all $(k_2,\cdots,k_{p})$ and $k_\Qc$ that satisfy $k_{\lf'}=k_\lf$ for any pairing $(\lf,\lf')$.

For each $2\leq j\leq p$, in the support of $h^{(j)}$ we assume that $k_\lf\in\Zb^d$, $\langle k_j\rangle\leq N_j$ and $N_\lf/2<\langle k_\lf\rangle\leq N_\lf$ for each $\lf\in\Uc_j$. Moreover, $h^{(j)}$ has one of the following three types:
\begin{enumerate}
\item Type $R0$: where we assume, in the support of $h^{(j)}$, that
\begin{equation}\label{type1}\sum_{\lf\in \Uc_j}\zeta_\lf=1,\quad k_j=\sum_{\lf\in\Uc_j}\zeta_\lf k_\lf,\quad |k_j|^2-\sum_{\lf\in\Uc_j}\zeta_\lf |k_\lf|^2=\Gamma_j.
\end{equation} Moreover, for any partition $(P_j,Q_j)$ of $\Uc_j$, we assume
\begin{equation}\label{type1bd1}\|h^{(j)}\|_{k_jk_{P_j}\to k_{Q_j}}\leq \Xf_j\cdot \prod_{\lf\in \Uc_j}N_\lf^{\beta_1}\cdot\Zc_{0,j}\Zc_{1,j},
\end{equation}where $\Zc_{0,j}$ and $\Zc_{1,j}$ are defined similarly as in (\ref{estimate02}) but with the following modifications,
 \begin{equation}\label{type1bd2}\Zc_{0,j}=\left\{
  \begin{aligned}&\big(\max_{\lf\in Q_j}N_\lf\big)^{-\beta_1},&&\mathrm{if\ }Q_j\neq\varnothing;\\
    & \big(\min_{\lf\in\Uc_j}N_\lf\big)^{\frac{d}{2}-\beta_1},&&\mathrm{if\ }Q_j=\varnothing,
 \end{aligned}
 \right.\qquad\Zc_{1,j}=\left\{
 \begin{aligned}&1,&{\rm{if }}\max_{\lf\in \Uc_j}N_\lf&\geq (10^3dp)^{-D}N_j;\\
&N_j^{-4\varepsilon},&{\rm{if }}\max_{\lf\in \Uc_j}N_\lf&< (10^3dp)^{-D}N_j.
 \end{aligned}
 \right.
 \end{equation}

 \item Type $R0^+$: similar to type $R0$, but instead of (\ref{type1}), in the support of $h^{(j)}$ we only assume that 
\begin{equation}\label{type1+}
 k_j-\sum_{\lf\in\Uc_j}k_\lf=m_j\end{equation} for some given $m_j\in \Zb^d$. Moreover the bounds (\ref{type1bd1})--(\ref{type1bd2}) also hold (in particular $\Zc_{1,j}=N_j^{-4\varepsilon}$ if $\Uc_j=\varnothing$), except that in (\ref{type1bd2}), when $Q_j=\varnothing$, we have $\Zc_{0,j}=1$ instead of $(\min_{\lf\in\Uc_j}N_\lf)^{d/2-\beta_1}$.

\item Type $R1$: where we assume, in the support of $h^{(j)}$, that
\begin{equation}\label{type2}\bigg|k_j-\sum_{\lf\in\Uc_j}\zeta_\lf k_\lf-m_j\bigg|\leq (N_*)^{3\delta}
\end{equation} for some given $m_j\in \Zb^d$. Moreover, for any partition $(P_j,Q_j)$ of $\Uc_j$ we have
\begin{equation}\label{type2bd2}\|h^{(j)}\|_{k_jk_{P_j}\to k_{Q_j}}\leq
\left\{
\begin{aligned}&\Xf_j\cdot \prod_{\lf\in \Uc_j}N_\lf^\beta\cdot \big(\max_{\lf\in Q_j}N_\lf\big)^{-\beta},&Q_j&\neq\varnothing;\\
&\Xf_j\cdot \prod_{\lf\in \Uc_j}N_\lf^\beta\cdot N_j^{-\varepsilon},&Q_j&=\varnothing.
\end{aligned}
\right. 
\end{equation}
\end{enumerate} Regarding the tensor $h$, we assume that $|h_{kk_1\cdots k_p}|\lesssim 1$, and that $h=h_{kk_1\cdots k_p}$ is supported in the set
\begin{equation}
\label{tensorsupport}k=\sum_{j=1}^{p}\zeta_jk_j,\quad |k|^2-\sum_{j=1}^p\zeta_j|k_j|^2=\Gamma,
\end{equation} and that any pairing in $(k,k_1,\cdots,k_p)$ must be over-paired.

Then, for any partition $(P,Q)$ of $\Uc$, we have the bound
\begin{equation}\label{outputbd}\prod_{\lf\in\Qc}N_\lf^{-\alpha+4\varepsilon+\theta}\cdot\|H\|_{kk_P\to k_1k_Q}\lesssim\prod_{j=2}^{p}\Xf_j\cdot\prod_{\lf\in\Uc}N_\lf^\beta\cdot\prod_{j}^{(0,0^+)}N_j^{-2\varepsilon}\cdot(N_*)^{-\varepsilon^3},
\end{equation} where the product $\prod_{j}^{(0,0^+)}$ is taken over all $2\leq j\leq p$ such that $h^{(j)}$ is of type $R0$ or type $R0^+$. The result is uniform in all parameters $\Gamma, \Gamma_j,m_j$ etc., and remains true if we replace $p$ by an odd $3\leq q\leq p$.  It also remains true if instead of the second equation in the support condition in (\ref{tensorsupport}), we assume that $h$ satisfies (\ref{extrabdh}).
  \end{prop}
  \begin{proof} 
  The two key ingredients in the proof are Proposition \ref{multibound} and Proposition \ref{final}.
  
\emph{Step 1: first reductions.} We start by making an adjustment in notations, just like the one in the proof of Proposition \ref{gausscont}, which will allow us to apply Proposition \ref{multibound} below. For each pairing  $(\lf,\lf')$, where $\lf\in \Uc_j$ and $\lf'\in\Uc_{j'}$, since in the sum (\ref{defH}) we are always assuming $k_\lf=k_{\lf'}$, we may combine them into a single element and include this element in both $\Uc_j$ and $\Uc_{j'}$. In this way we are changing pairings between different $\Uc_j$'s to intersections of different $\Uc_j$'s, which is the setting of Proposition \ref{multibound}. Then $\Uc$ will be the set of elements that occurs once in all the $\Uc_j$'s, and $\Qc$ is the set of elements that occur twice.
  
  Next, we will identify all subsets $A\subset\{2,\cdots,p\}$ such that each and every element in the union of $\Uc_j\,(j\in A)$ occurs twice in these sets $\Uc_j\,(j\in A)$. We only need to consider \emph{minimal} subsets $A$ that satisfy this, which will be pairwise disjoint. Let them be $A_v\,(1\leq v\leq s)$ and $B_u\,(1\leq u\leq t)$, where for each $v$, the tensor $h^{(j)}$ is of type $R0$ for each $j\in A_v$, while for each $u$, there is at least one $j\in B_u$ such that the tensor $h^{(j)}$ is of type $R0^+$ or $R1$. Let $\Ec$ be the union of all these sets $A_v$ and $B_u$ and $\Dc=\{2,\cdots,p\}\backslash \Ec$; See Figure \ref{graph6.2} for an illustration of the above  subsets of $\{2,\cdots,p\}$. We know that $\Uc_j\subset\Qc$ for $j\in \Ec$.
   \begin{figure}
\centering
\includegraphics[width=0.65\textwidth]{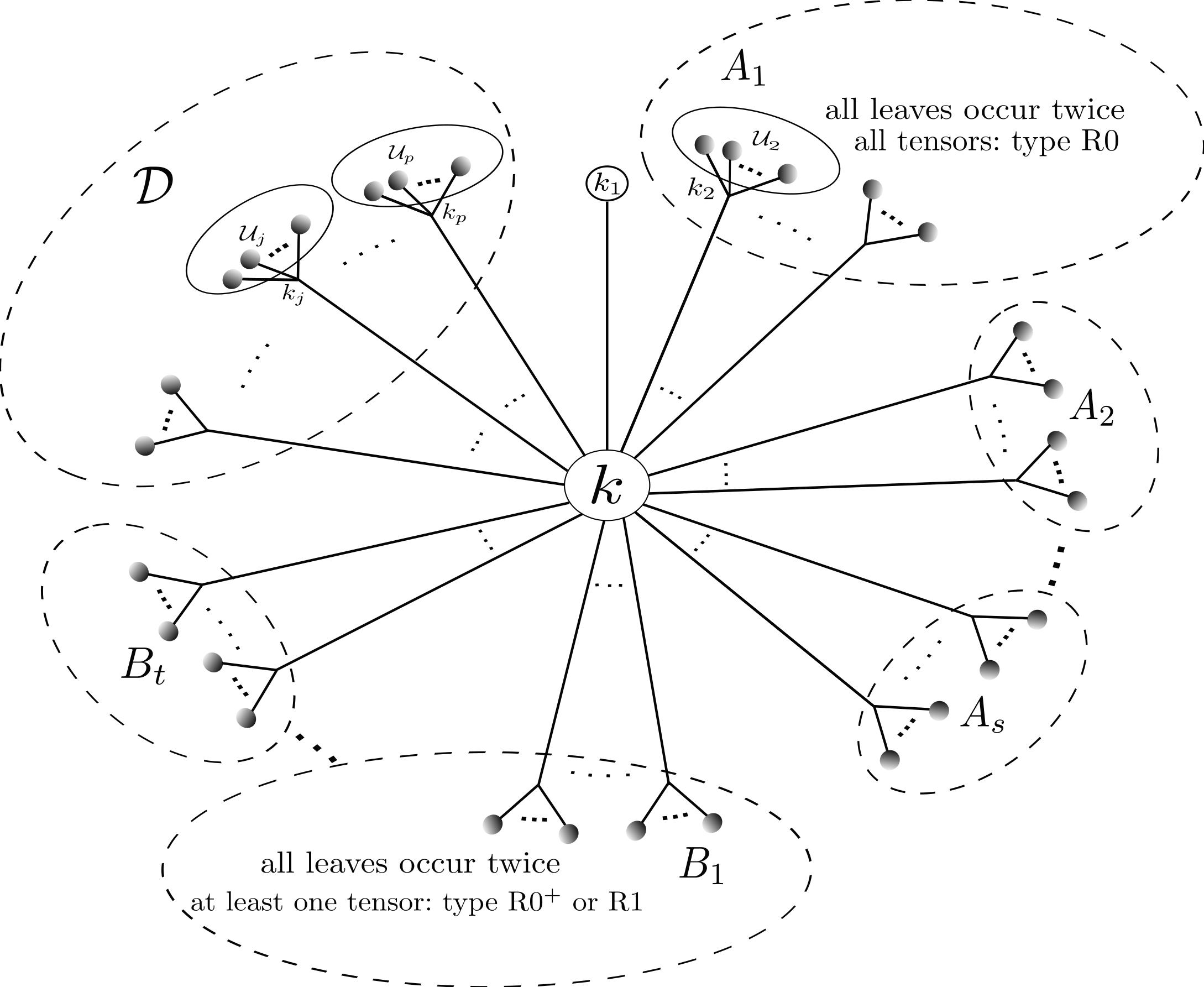}
\caption{Classification of $\Uc_j$, $j\in \{2,\cdots,p\}$}\label{graph6.2}
\end{figure}
  Let $\Rc\subset \Qc$ be the set of elements that occur twice in the sets $\Uc_j (j\in\Dc)$. Then we have
  \begin{equation}\label{altexp1}H=H_{kk_1k_\Uc}=\sum_{(k_{A_1},\cdots k_{A_s},k_{B_1},\cdots k_{B_t})}\prod_{v=1}^s H_{k_{A_v}}^{[v]}\prod_{u=1}^tR_{k_{B_u}}^{[u]}\cdot H^{\mathrm{sg}}_{kk_1k_{\Ec}k_\Uc},
  \end{equation} where
    \begin{equation}\label{altexp2}H_{k_{A_v}}^{[v]}:=\sum_{(k_{\Uc_j}:j\in A_v)}\prod_{j\in A_v}h_{k_jk_{\Uc_j}}^{(j)},\qquad R_{k_{B_u}}^{[u]}:=\sum_{(k_{\Uc_j}:j\in B_u)}\prod_{j\in B_u}h_{k_jk_{\Uc_j}}^{(j)},
  \end{equation} 
  \begin{equation}\label{altexp3}H^{\mathrm{sg}}_{kk_1k_{\Ec}k_\Uc}:=\sum_{(k_{\Dc},k_\Rc)}h_{kk_1\cdots k_p}\cdot\prod_{j\in \Dc}h_{k_jk_{\Uc_j}}^{(j)}.
  \end{equation} By (\ref{altexp1}) and Proposition \ref{multibound} we have
  \begin{equation}\label{mainest1}\|H\|_{kk_P\to k_1k_Q}\leq \prod_{q=1}^s\|H^{[v]}\|_{k_{A_v}}\prod_{u=1}^t\|R^{[u]}\|_{k_{B_u}}\cdot\|H^{\mathrm{sg}}\|_{kk_P\to k_1k_{\Ec}k_Q},\end{equation} so it suffices to bound the norm $\|H^{\mathrm{sg}}\|_{kk_P\to k_1k_{\Ec} k_Q}$ and the other norms $\|H^{[v]}\|_{k_{A_v}}$ and $\|R^{[u]}\|_{k_{B_u}}$.
  
  By (\ref{altexp3}), the tensor $H^{\mathrm{sg}}$ is a semi-product the tensors $h$ and $h^{(j)}\,(j\in \Dc)$ in the sense of (\ref{combination2}) in Proposition \ref{multibound}, so our strategy is to select these tensors in some particular order and apply Proposition \ref{multibound}. The selection algorithm is described as follows.

\emph{Step 2: the selection algorithm for $H^{\mathrm{sg}}$.} First, by our choice of the set $\Dc$, there will be at least one element in the union of the sets $\Uc_j\,(j\in \Dc)$ that appears only once in these sets. Consider such an element $\lf$ with $N_\lf$ being the biggest. Denote this $\lf$ by $\lf_{c-1}$ and  the $j\in \Dc$ such that $\lf_{c-1}\in \Uc_j$ by $n_{c-1}$, where $c-1=|\Cc|-1=|\Dc|$ and $\Cc=\Dc\cup\{1\}$. Next, there will be at least one element in the union of the sets $\Uc_j\,(j\in \Dc\backslash\{n_{c-1}\})$ that appears only once in these sets. Consider such an element $\lf=\lf_{c-2}$ such that $N_\lf$ is the biggest, and suppose such $\lf_{c-2}\in \Uc_{j'}$, where $j'\in \Dc\backslash\{n_{c-1}\}$; we shall denote $n_{c-2}=j'$. Next there will be at least one element in the union of the sets $\Uc_j\,(j\in \Dc\backslash\{n_{c-1},n_{c-2}\})$ that appears only once in these sets, and so on. Repeating this process, we can label the elements of $\Dc$ as $n_1, n_2, \cdots, n_{c-1}$. Notice that by (\ref{type1}) and (\ref{type2}) and our selection algorithm, for each $1\leq y\leq c-1$ we must have 
  \begin{equation}\label{cancellation1}\bigg|\sum_{z=1}^y\zeta_{n_z}k_{n_z}-m_{n_y}'\bigg|\lesssim \max((N_*)^{3\delta}, N_{\lf_y}):=M_{n_y},
  \end{equation} where $m_{n_y}'\in\Zb^d$ is some fixed vector, otherwise the product in (\ref{altexp3}) will be zero. Let us provide some details to explain (\ref{cancellation1}). When $y=1$, $\lf_1\in \Uc_{n_1}$ and $N_{\lf_1}$ is the biggest among all $\lf$ that appear only once in $\Uc_{n_1}$ (here these $\lf$'s are just all elements $\lf\in \Uc_{n_1}$). By (\ref{type1}), (\ref{type1+}) and (\ref{type2}), setting $m_{n_1}'=\zeta_{n_1}m_{n_1}$, we then have\begin{equation}\label{cancellation1'}
  |\zeta_{n_1}k_{n_1} - m'_{n_1}| \leq |\Uc_{n_1}| N_{\lf_1} +(N_*)^{3\delta},
  \end{equation}
  where we understand that $m_{n_1}=0$ if $h^{(n_1)}$ has type $R0$ (same below), which implies (\ref{cancellation1}) since $|\Uc_{n_1}|\leq  D$.
  When $y=2$, $\lf_2\in \Uc_{n_2}$ and $N_{\lf_2}$ is the biggest among all $\lf$ that appear only once in $\Uc_{n_1}\cup \Uc_{n_2}$ (i.e. all $\lf\in\Uc_{n_1}\Delta \Uc_{n_2}$), then by
  (\ref{type1}), (\ref{type1+}) and (\ref{type2}), we have $|\zeta_{n_1}k_{n_1}+\zeta_{n_2}k_{n_2}-m'_{n_2}|\lesssim  \max((N_*)^{3\delta}, N_{\lf_2})$ with $m'_{n_2}=\zeta_{n_1} m_{n_1}+\zeta_{n_2} m_{n_2}$, since the $k_\lf$ terms for $\lf\in\Uc_{n_1}\cap \Uc_{n_2}$ always cancel themselves thanks to our assumption about signs of paired elements. For the other $y$'s, (\ref{cancellation1}) is obtained similarly. The above selection for $n_{c-1},\cdots, n_1$ is designed to fit the hypothesis (\ref{finalset3}) in Proposition \ref{final} via (\ref{cancellation1}).  Proposition \ref{final} will be applied to $h$ later in \emph{Step 4}.
  
  Next, we divide these $n_y$ into two classes: first $n_{c-1}$ will be class $P$ (or $Q$) if $\lf_{c-1}$ which belongs to $\Uc$, is in $P$ (or $Q$). Next, if $\lf_{c-2}\in \Uc$, then $n_{c-2}$ will be class $P$ (or $Q$) if $\lf_{c-2}$ is in $P$ (or $Q$); if $\lf_{c-2}\in\Qc$, then $\lf_{c-2}$ also belongs to $\Uc_j$ for some previously selected $j$ (here $j=n_{c-1}$), then $n_{c-2}$ will be same class as $j$. Then consider $\lf_{c-3}$, and so on. Repeating this process we can assign a class $P$ or $Q$ to each $n_y\,(1\leq y\leq c-1)$. Now we can apply Proposition \ref{multibound} by arranging the tensors $h$ and $h^{(j)}\,(j\in \Dc)$ in a particular order (which will correspond to the superscripts in the tensors in Proposition \ref{multibound}): first list all the $h^{(n_y)}$, where $n_y$ has class $Q$, in the \emph{decreasing} order for $y$, then list $h$, then list all the $h^{(n_y)}$, where $n_y$ has class $P$, in the \emph{increasing} order for $y$. By applying Proposition \ref{multibound} to (\ref{altexp3}) with $(k_A, k_C, k_X, k_Y) = (kk_1k_\Ec k_\Uc, k_\Dc k_\Rc, kk_P, k_1 k_\Ec k_Q)$ in the above order we then have
\begin{equation}  \label{estproof}
  \|H^{\mathrm{sg}}\|_{kk_P\to k_1k_\Ec k_Q}\lesssim\|h\|_{kk_{P_0}\to k_{Q_0}}\prod_{j\text{\,of\,class\,}Q}\|h^{(j)}\|_{k_jk_{P_j}\to k_{Q_j}} \prod_{j\text{\,of\,class\,}P}\|h^{(j)}\|_{k_{Q_j}\to k_j k_{P_j}},
  \end{equation} where $(P_0,Q_0)$ and $(P_j,Q_j)\,(j\in\Dc)$ are sets defined according to Proposition \ref{multibound}, which are explained next. Next we analyze the individual factors on the right hand side of (\ref{estproof}).
  
  (1) For $j=n_y$ of class $Q$, by Proposition \ref{multibound} and our algorithm, the set $Q_j$ will consist of all $\lf\in \Uc_{n_y}$ such that either $\lf\in Q$, or $\lf$ belongs to $\Uc_{n_{y'}}$ for some $y'>y$ with $n_{y'}$ of class $Q$. Furthermore $P_j = \Uc_j\backslash Q_j$. By the definition of classes, this implies that $\lf_{y}\in Q_j$.  By (\ref{type1bd1}) and (\ref{type2bd2}) we then have
  \begin{equation}\label{actbd1}\|h^{(j)}\|_{k_jk_{P_j}\to k_{Q_j}}\lesssim \Xf_j\cdot\prod_{\lf\in\Uc_j} N_\lf^\beta\cdot N_{\lf_y}^{-\beta}\lesssim \Xf_j\cdot\prod_{\lf\in\Uc_j} N_\lf^\beta\cdot
  (N_*)^{C\delta}M_{n_y}^{-\beta};
  \end{equation} moreover if $h^{(j)}$ has type $R0$ or $R0^+$, we gain an extra factor namely $\Zc_{1,j}$ from (\ref{type1bd2}).
  
  (2) Similarly, for $j=n_y$ of class $P$, by Proposition \ref{multibound} and our algorithm, the set $Q_j$ will consist of all $\lf\in\Uc_{n_y}$ such that either $\lf\in P$, or $\lf$ belongs to $\Uc_{n_{y'}}$ for some $y'>y$ with $n_{y'}$ of class $P$. Furthermore $P_j = \Uc_j\backslash Q_j$. By the definition of class, we also have $\lf_{y}\in Q_j$. 
Now by (\ref{type1bd1}), (\ref{type2bd2}), and using the duality of the operator norm $\|h^{(j)}\|_{k_{Q_j}\to k_jk_{P_j}}=\|h^{(j)}\|_{k_jk_{P_j}\to k_{Q_j}}$, we know that (\ref{actbd1}) is true (with the gain $\Zc_{1,j}$ for types $R0$ and $R0^+$) also in this case.
  
  (3) For the tensor $h$, by our algorithm we have that $P_0$ consists of all $j\in \Dc$ of class $P$, and $Q_0$ consists of all $j\in \Dc$ of class $Q$, as well as $1$ and all $j\in \Ec$.  
  
We illustrate the above algorithm with an explicit example. Suppose $p=7$, $\Dc =\{2, 3, 4, 5\}$, $\Ec=\{6,7\}$ and $\Uc_2 =\{\af, \bff\}$, $\Uc_3 =\{\cf, \dff, \ef\}$, $\Uc_4 =\{\ef, \ff, \gf\}$, $\Uc_5 =\{\af, \ff, \hf\}$ with $N_\af\geq N_{\bff}\geq N_{\cf}\geq N_{\dff}\geq N_{\ef}\geq N_{\ff}\geq N_{\gf}\geq N_{\hf}$, then $\Uc=\{\bff,\cf,\dff,\gf,\hf\}$, $\Rc =\{\af,\ef,\ff\}$ and
\begin{equation*}H^{\mathrm{sg}}_{kk_1k_6k_7k_\Uc}=\sum_{(k_2,\cdots,k_5)}\sum_{(k_\af,k_\ef,k_\ff)}h_{kk_1\cdots k_7}\cdot h_{k_2k_\af k_\bff}^{(2)} \cdot h_{k_3k_\cf k_\dff k_\ef}^{(3)}\cdot h_{k_4k_\ef k_\ff k_\gf}^{(4)} \cdot h_{k_5k_\af k_\ff k_\hf}^{(5)}.
\end{equation*} Suppose $P=\{\bff, \dff\}$ and $Q=\{\cf, \gf, \hf\}$ is a partition of all unpaired leaves. Then, by our algorithm we have $n_4 = 2$ since $N_\bff = \max_{\lf\in \{\bff, \cf, \dff, \gf, \hf\}} (N_\lf)$ and $\bff \in \Uc_2$.  Then $n_3 = 5$ since $N_\af=\max_{\lf\in\{\af,\cf,\dff,\gf,\hf\}} N_\lf$ and $\af\in \Uc_5$. Similarly we have $n_2 = 3$ and $n_1 = 4$. Next, by definition $n_4=2$ and $n_3=5$ will have class $P$, while $n_2=3$ and $n_1=4$ will have class $Q$. Then we can apply Proposition \ref{multibound} to (\ref{altexp3}) in the following order: $h^{(3)}\to h^{(4)}\to h\to h^{(5)}\to h^{(2)}$ and obtain that
\begin{multline*}
 \|H^{\mathrm{sg}}\|_{kk_\bff k_\dff\to k_1k_6k_7k_\cf k_\gf k_\hf}\lesssim\|h^{(3)}\|_{k_{3}k_\dff k_\ef \to k_\cf} \cdot\|h^{(4)}\|_{k_{4}k_{\ff}\to k_\ef k_\gf} \\\times\|h\|_{kk_2k_5\to k_1k_3k_4k_6k_7}\cdot\|h^{(5)}\|_{k_{\af}\to k_{5} k_\ff k_\hf}\cdot\|h^{(2)}\|_{k_{\bff}\to k_{2} k_\af}.
\end{multline*} Moreover we have the following inequalities (assume $m_j=0$, and up to error $(N_*)^{3\delta}$):
\[|k_2\pm k_3\pm k_4\pm k_5|\lesssim N_\bff,\quad |k_3\pm k_4\pm k_5|\lesssim N_\af,\quad |k_3\pm k_4|\lesssim N_\cf,\quad |k_4|\lesssim N_\ef.\]
 
  \emph{Step 3: the selection algorithms for $H^{[v]}$ and $R^{[u]}$.} Now we discuss the estimates for the norms $\|H^{[v]}\|_{k_{A_v}}$ and $\|R^{[u]}\|_{k_{B_u}}$. The basic idea is the same as before, but as each element in the sets $\Uc_j\,(j\in A_v\textrm{ or }B_u)$ occurs twice in these sets, we have to make some small adjustments. Let us first look at $A_v$. Let $a_v=|A_v|$, and first choose $\lf=\lf_{a_v}$ in the sets $\Uc_j\,(j\in A_v)$ such that $N_\lf$ is the smallest over all $\lf$ in the sets $\Uc_j\,(j\in A_v)$. We denote the $j$ such that $\lf_{a_v}\in\Uc_j$ by $\ell_v(a_v)$. Next, as in \emph{Step 2} there will be at least one element in the union of the sets $\Uc_j\,(j\in A_v\backslash\{\ell_v(a_v)\})$ that appears only once in these sets. 
Consider such an element $\lf$ with $N_\lf$ being the biggest. We shall denote this $\lf$ by $\lf_{a_v-1}$ and  the $j\in A_v\backslash \{\ell_v(a_v)\}$ such that $\lf_{a_v-1}\in \Uc_j$ by $\ell_v (a_v-1)$.
  Then we can repeat the process in \emph{Step 2} and label the elements of $A_v$ as $\ell_v(1),\cdots ,\ell_v(a_v)$. Recall that $h^{(j)}$ has type $R0$ for all $j\in A_v$.
  By (\ref{type1}), in the same way as in \emph{Step 2}, for $1\leq y\leq a_v-1$ we have
  \begin{equation}\label{cancellation2}\bigg|\sum_{z=1}^y\zeta_{\ell_v(z)}k_{\ell_v(z)}\bigg|\lesssim N_{\lf_y}:=M_{\ell_v(y)},
  \end{equation} as well as 
  \begin{equation}\label{cancellation3}\sum_{z=1}^{a_v}\zeta_{\ell_v(z)}k_{\ell_v(z)}=0,\quad \sum_{z=1}^{a_v}\zeta_{\ell_v(z)}|k_{\ell_v(z)}|^2=\widetilde{\Gamma},
  \end{equation}
where $\widetilde{\Gamma} = \sum_{z=1}^{a_v}\zeta_{\ell_v(z)}\Gamma_{\ell_v(z)}$.
  Now we apply Proposition \ref{multibound} to $H^{[v]}$ in (\ref{altexp2}) with $(A, C, X, Y)=(A_v, \bigcup\{\Uc_j:j\in A_v\}, A_v, \varnothing)$ by arranging the tensors $h^{(\ell_v(y))}$ in the decreasing order for $y$, hence
  \begin{equation}\label{estproof2}\|H^{[v]}\|_{k_{A_v}}\lesssim\prod_{j\in A_v}\|h^{(j)}\|_{k_jk_{P_j}\to k_{Q_j}},
  \end{equation} where $Q_j=\varnothing$ for $j=\ell_v(a_v)$, and for $j=\ell_v(y)$ with $y<a_v$, $Q_j$ consists of all $\lf\in \Uc_{\ell_v(y)}$ that belongs to $\Uc_{\ell_v(y')}$ for some $y'>y$, so in particular $\lf_y\in Q_j$. Furthermore $P_j = \Uc_j\backslash Q_j$ for all $j\in A_v$. By (\ref{type1bd1}) and (\ref{type1bd2}) we then get
  \begin{equation}\label{actbd2}\|h^{(j)}\|_{k_jk_{P_j}\to k_{Q_j}}\lesssim \Xf_j\cdot\prod_{\lf\in \Uc_j}N_\lf^{\beta_1}\cdot M_{\ell_v(y)}^{-\beta_1}\cdot\Zc_{1,j},\quad \mathrm{if\ }j=\ell_v(y)\quad (y<a_v),
  \end{equation}
    \begin{equation}\label{actbd3} \|h^{(j)}\|_{k_jk_{P_j}\to k_{Q_j}}\lesssim \Xf_j\cdot\prod_{\lf\in \Uc_j}N_\lf^{\beta_1}\cdot (N_{\lf_{a_v}})^{\frac{d}{2}-\beta_1}\cdot\Zc_{1,j},\quad  \mathrm{if\ }j=\ell_v(a_v),
  \end{equation}noticing also that $N_{\lf_{a_v}}\leq M_{\ell_v(y)}$ for any $1\leq y\leq a_v-1$ by our choice.
  
  As for $B_u$, the argument is essentially the same as above. We choose $i_u(b_u)$ such that $h^{(j)}$ has type $R0^+$ or $R1$ for $j=i_u(b_u)$; then $i_u(y)$ for $y<b_u$ are chosen in the same manner as $\ell_{v}(y)$ above.
 In this case,  (\ref{cancellation2}) holds after translating by some fixed vector $m_{i_u(y)}'$, and with a loss of $(N_*)^{C\delta}$ due to the weaker bound (\ref{type2}). For (\ref{cancellation3}) in this case we don't have the equation for $\sum_{z\leq b_u}\zeta_{i_u(z)}|k_{i_u(z)}|^2$, and the sum $\sum_{z\leq b_u}\zeta_{i_u(z)}k_{i_u(z)}$ only belongs to a ball of radius $(N_*)^{3\delta}$. However by losing a factor $(N_*)^{C\delta}$ in the operator bound for the tensor $h$ we may assume $\sum_{z\leq b_u}\zeta_{i_u(z)}k_{i_u(z)}$ is constant. Therefore, for the norms appearing in (\ref{estproof2}), instead of (\ref{actbd2}) and (\ref{actbd3}), we now have\begin{equation}\label{actbd4}\|h^{(j)}\|_{k_jk_{P_j}\to k_{Q_j}}\lesssim \Xf_j\cdot(N_*)^{C\delta}\prod_{\lf\in \Uc_j}N_\lf^{\beta}\cdot M_{i_u(y)}^{-\beta},\quad \mathrm{if\ }j=i_u(y)\quad(y<b_u),
  \end{equation}
    \begin{equation}\label{actbd5} \|h^{(j)}\|_{k_jk_{P_j}\to k_{Q_j}}\lesssim \Xf_j\cdot(N_*)^{C\delta}\prod_{\lf\in \Uc_j}N_\lf^{\beta},\quad  \mathrm{if\ }j=i_u(b_u);
  \end{equation} moreover we gain an extra factor $\Zc_{1,j}$ in (\ref{actbd4}) and (\ref{actbd5}) if $h^{(j)}$ has type $R0^+$.
  
  \emph{Step 4: putting together.} We now come back to the estimate for $\|h\|_{kk_{P_0}\to k_{Q_0}}$ that appears in (\ref{estproof}). By all the previous discussions, we may assume that in the support of $h=h_{kk_1\cdots k_p}$ we have the equalities and inequalities (\ref{tensorsupport}), (\ref{cancellation1}), (\ref{cancellation2}), (\ref{cancellation3}), as well as the variants of (\ref{cancellation2}) and (\ref{cancellation3}) for $B_u$ (see \emph{Step 3}), and that any pairing in $(k,k_1,\cdots,k_p)$ must be over-paired.  All these allow us to apply Proposition \ref{final} (unless we are in the exceptional case, namely $(d,p)=(1,7)$, and up to permutation $|A_v|=2$ for $v=1, 2$, $k_{\ell_1(1)}=k_{\ell_2(1)}$) to obtain the bound
\begin{equation}
\label{topbound}\|h\|_{kk_{P_0}\to k_{Q_0}}\lesssim (N_*)^\theta\prod_{j=2}^{p}M_j^{\alpha_0}\prod_{v=1}^s(\min_{1\leq y<a_v}M_{\ell_v(y)})^{\alpha_0-\frac{d}{2}}.
\end{equation} Now combining (\ref{mainest1}), (\ref{estproof}), (\ref{actbd1}) and (\ref{estproof2})--(\ref{topbound}), we conclude that
\begin{equation}\label{lastbd}\prod_{\lf\in\Qc}N_\lf^{-\alpha+4\varepsilon+\theta}\cdot\|H\|_{kk_P\to k_1k_Q}\lesssim \prod_{j=2}^{p}\Xf_j\prod_{\lf\in\Uc}N_\lf^\beta\cdot (N_*)^{C\delta}(\max_{\lf\in\Qc}N_\lf)^{-4p\varepsilon}(\max_{2\leq j\leq p}M_j)^{-4p\varepsilon}\prod_j^{(0,0^+)}\Zc_{1,j},
\end{equation} 
where $\prod_j^{(0,0^+)}$ is defined as in (\ref{outputbd}). By definition of $M_j$ we have \[
\max_{2\leq j\leq p}M_j\geq\max_{\lf\in\Uc}N_\lf;\]
using also (\ref{lastbd}) and the fact that $\Uc \cup \Qc =\Uc_2\cup\cdots\cup \Uc_p$, we then have 
\begin{equation}\label{lastbd'}\prod_{\lf\in\Qc}N_\lf^{-\alpha+4\varepsilon+\theta}\cdot\|H\|_{kk_P\to k_1k_Q}\lesssim \prod_{j=2}^{p}\Xf_j\prod_{\lf\in\Uc}N_\lf^\beta\cdot (N_*)^{C\delta}(\max_{\lf\in\Uc_j, 2\leq j\leq p}N_\lf)^{-4p\varepsilon}\prod_j^{(0,0^+)}\Zc_{1,j}.
\end{equation} 
The factor $\prod_{j}^{(0,0^+)}N_j^{-2\varepsilon}$ in (\ref{outputbd}) will be provided by separating a square root of  the last two factors in the right hand side of (\ref{lastbd'}). 
To gain the other factor $(N_*)^{-\varepsilon^3}$ in (\ref{outputbd}), we consider two cases.
If $\max\{N_\lf:\lf\in\Uc_j,2\leq j\leq p\}\geq (N_*)^{\varepsilon^2}$, then the factor $(\max_{\lf\in\Uc_j, 2\leq j\leq p}N_\lf)^{-2p\varepsilon}$ in the other square root is bounded by $(N_*)^{-C\delta- \varepsilon^3}$, hence (\ref{outputbd}) is proved. Otherwise, let $2\leq j\leq p$ be such that $N_j=N_*$, then $N_\lf\ll N_j^{\varepsilon^2}$ for all $\lf\in\Uc_j$. If $h^{(j)}$ has type $R0$ or $R0^+$, the factor $\Zc_{1,j}^{1/2}=N_j^{-2\varepsilon}$ in the other square root  is bounded by $(N_*)^{-C\delta- \varepsilon^3}$. If $h^{(j)}$ has type $R1$, by using the bound of $\|h^{(j)}\|_{k_jk_{\Uc_j}}$ (the second bound in (\ref{type2bd2})) and that $N_j^{-\varepsilon/2}\leq (\max_{\lf\in Q_j}N_\lf)^{-\beta}$, we have 
\begin{equation}\label{type2bd2'}
\|h^{(j)}\|_{k_jk_{P_j}\to k_{Q_j}}\leq \Xf_j\cdot N_j^{-\varepsilon/2}\prod_{\lf\in\Uc_j}N_\lf^\beta\cdot\big(\max_{\lf\in Q_j}N_\lf\big)^{-\beta}\end{equation} 
for any partition $(P_j, Q_j)$ of $\Uc_j$. Then we gain an extra factor $N_j^{-\varepsilon/2}$ (which is less than $(N_*)^{-C\delta- \varepsilon^3}$) by using (\ref{type2bd2'}) instead of (\ref{type2bd2}) in (\ref{actbd1}), (\ref{actbd4}) or (\ref{actbd5}),  and hence (\ref{outputbd}) is proved.

Finally, in the exceptional case mentioned above, we may assume (up to permutation) that $A_1=\{2,3\}$ and $A_2=\{4,5\}$, so $(d,p)=(1,7)$ and $k_2=k_3=k_4=k_5:=k_*$ by the setting of $A_v$ in \emph{Step 1}. Here we may fix and sum in $k_*$, while for fixed $k_*$ the corresponding part of the tensor $H$ can be bounded as above, with $h^{(j)}\,(2\leq j\leq 5)$ measured in the norm
\[\sup_{k_j}\|h^{(j)}\|_{k_{\Uc_j}}\leq \|h^{(j)}\|_{k_j\to k_{\Uc_j}}.\] One can check that the power gain coming from using these norms is enough to cancel the summation in $k$, and the rest of the proof goes just like above. The cases when $p$ is replaced by odd $3\leq q\leq p$, or when $h$ satisfies (\ref{extrabdh}) can be proved in the same way, since Proposition \ref{final} works equally well.
\end{proof}
\begin{prop}[Selection algorithm: Case II]\label{algorithm2} Consider the same setting as in Proposition \ref{algorithm1}. Here we assume that each $h^{(j)}$ has type $R0$ or $R0^+$ (in the sense of Proposition \ref{algorithm1}), but in both cases the factor $\Zc_{1,j}$ in (\ref{type1bd1}) is replaced by 1 (the factor $\Zc_{0,j}$ remains the same). Then we have
\begin{equation}\label{outputbd2}\prod_{\lf\in\Qc}N_\lf^{-\alpha+8\varepsilon+\theta}\cdot\|H\|_{kk_P\to k_1k_Q}\lesssim\prod_{j=2}^{p}\Xf_j\cdot\prod_{\lf\in\Uc}N_\lf^{\beta_1}\cdot \big(\max_{2\leq j\leq p}\max_{\lf\in\Uc_j}N_\lf\big)^{-4p\varepsilon}.
\end{equation} The same holds if $p$ is replaced by odd $3\leq q\leq p$, without changing the power $4p\varepsilon$ in (\ref{outputbd2}).
\end{prop}
\begin{proof} The proof is the same as Proposition \ref{algorithm1}, with the following adjustments. First due to the absence of type $R1$ tensors, we will not lose any $(N_*)^{C\delta}$ factor in the proof process; second, we do not gain any extra factor as in the proof of Proposition \ref{algorithm1}, since $\Zc_{1,j}$ has been replaced by 1. With these, we are still able to gain a factor
\[\big(\max_{\lf\in\Qc}N_\lf\big)^{-4p\varepsilon}\big(\max_{2\leq j\leq p}M_j\big)^{-4p\varepsilon}\] as in (\ref{lastbd}), which implies (\ref{outputbd2}) in the same way as in the proof of Proposition \ref{algorithm1}.
\end{proof}
\subsection{Merging estimates}\label{6.3} Finally we prove the general merging bounds, Propositions \ref{overpair}--\ref{overpair3}.
\begin{prop}[Merged tensor bounds: Case I]\label{overpair} Let $3\leq q\leq p$ be odd, $0\leq r\leq q$, and let $\Sc_j\,(1\leq j\leq r)$ be regular plants with frequency $N(\Sc_j)=N_j\leq M$ and $|\Sc_j|\leq D$. Fix  $N_j\leq M/2$ for $r+1\leq j\leq q$ and $\zeta_j\in\{\pm\}$ for $1\leq j\leq q$, and assume that $\sum_{j=1}^q\zeta_j=1$. Denote $\Bs=(M,q,r,\zeta_1,\cdots,\zeta_q,N_1,\cdots,N_q)$. Let $h=h_{kk_1\cdots k_q}(\lambda_{r+1},\cdots,\lambda_q)$ be a constant tensor\footnote{i.e. which does not depend on $\omega$.} supported in the set \begin{equation}\label{mergesupp}
\begin{aligned}\langle k\rangle&\leq M,&\langle k_j\rangle&\leq N_j\,(1\leq j\leq q),\quad\langle\lambda_j\rangle\leq 2M^{\kappa^2}\,(r+1\leq j\leq q),\\
k&=\sum_{j=1}^q\zeta_jk_j,&\bigg||k|^2&-\sum_{j=1}^q\zeta_j|k_j|^2+\sum_{j=r+1}^q\zeta_j\lambda_j+\widetilde{\Gamma}\bigg|\lesssim 1,
\end{aligned}\end{equation} where $\widetilde{\Gamma}\in\Zb$ is fixed. Assume that
\begin{equation}\label{mergetensor}|h|+|\partial_{\lambda_j}h|\lesssim \tau^{-\theta},\quad r+1\leq j\leq q,
\end{equation} 
and that any pairing in $(k,k_1,\cdots,k_q)$ must be over-paired.
Now let $h^{(j)}=h_{k_jk_{\Uc_j}}^{(j)}(k_{\Vc_j},\lambda_{\Vc_j})$ be an $\Sc_j$-tensor for $1\leq j\leq r$, $\Os$ be as in Definition \ref{defmerge}, and let
\begin{equation}\label{mergetrim}
\begin{aligned}\Sc&=(\Lc,\Vc,\Yc)=\mathtt{Trim}(\mathtt{Merge}(\Sc_1,\cdots,\Sc_r,\Bs,\Os),M^\delta),\\
H&=\mathtt{Trim}(\mathtt{Merge}(h^{(1)},\cdots,h^{(r)},h,\Bs,\Os),M^\delta).
\end{aligned}
\end{equation} 
Let $N_*:=\max(N_2,\cdots,N_{q})$, assume $N_*\geq M^\delta$, and $\Upsilon$ be a factor such that
\begin{equation}\label{prop5.4:Upsilon}
\Upsilon\leq \tau^{-\theta};\qquad \Upsilon\leq \tau^{-\theta}M^{-40dp\varepsilon}\,\,\,\mathrm{if}\,\,\,\max_{1\leq j\leq q}N_j\leq (50dp)^{-1}M.
\end{equation}
We assume that the tensor $h^{(1)}$ is $\Bc_{M^{[\delta]}}$ measurable, and $N_\lf\geq M^\delta$ for $\lf\in\Lc_1$; for $2\leq j\leq r$, the tensor $h^{(j)}$ is $\Bc_{(N_*)^{[\delta]}}$ measurable, and $N_\lf\geq (N_*)^\delta$ for $\lf\in\Lc_j$.
Furthermore we assume that for $1\leq j\leq r$, $\Sc_j$ and $h^{(j)}$ have one of the following two types:

(1) Type 0: where $\Vc_j=\varnothing$, and in the support of $h^{(j)}$ we have
\begin{equation}\label{supportmerge}\sum_{\lf\in\Uc_j}\zeta_\lf=1,\quad k_j=\sum_{\lf\in\Uc_j}\zeta_\lf k_\lf,\quad |k_j|^2-\sum_{\lf\in\Uc_j}\zeta_\lf|k_\lf|^2=\Gamma_j,
\end{equation} where $\Gamma_j\in\Zb$ is fixed, and $h^{(j)}$ satisfies the bound
\begin{equation}\label{boundmerge1}\|h^{(j)}\|_{k_jk_{B_j}\to k_{C_j}}\lesssim\Xf_j\cdot\tau^{-\theta}\prod_{\lf\in B_j\cup C_j} N_\lf^{\beta_1}\prod_{\lf\in\Pc_j}N_\lf^{-8\varepsilon}\prod_{\lf\in E_j}N_\lf^{4\varepsilon}\prod_{\pf\in\Yc_j}N_\pf^{-\delta^3}\cdot\Xc_{0,j}\Xc_{1,j},
\end{equation} for any subpartition $(B_j,C_j)$ of $\Uc_j$, where $E_j=\Uc_j\backslash(B_j\cup C_j)$, and $\Xc_{0,j}$ and $\Xc_{1,j}$ are defined as in (\ref{estimate02}) but are associated with $\Sc_j$ and $(B_j,C_j)$ instead.

(2) Type 1: where $h^{(j)}$ satisfies the bounds (with $B_j,C_j,E_j$ same as (1)) \begin{equation}\label{boundmerge2}\|h^{(j)}\|_{X_{\Vc_j}^{-b_0}[k_jk_{B_j}\to k_{C_j}]}\lesssim\Xf_j\cdot\tau^{-\theta}\prod_{\lf\in B_j\cup C_j}N_\lf^\beta\prod_{\lf\in\Pc_j}N_\lf^{-4\varepsilon}\prod_{\lf\in E_j}N_\lf^{8\varepsilon}\prod_{\pf\in\Yc_j}N_\pf^{-\delta^3} \prod_{\ff \in \Vc_j} N_\ff^d\cdot \big(\max_{\lf\in C_j}N_\lf\big)^{-\beta},\quad \mathrm{if\ }C_j\neq\varnothing,
 \end{equation} 
 \begin{equation}
 \label{boundmerge3}\|h^{(j)}\|_{X_{\Vc_j}^{-b_0}[k_jk_{B_j}]}\lesssim\Xf_j\cdot\tau^{-\theta}\prod_{\lf\in B_j}N_\lf^\beta\prod_{\lf\in\Pc_j}N_\lf^{-4\varepsilon}\prod_{\lf\in E_j}N_\lf^{8\varepsilon}\prod_{\pf\in\Yc_j}N_\pf^{-\delta^3}\prod_{\ff \in \Vc_j} N_\ff^d\cdot N_j^{-\varepsilon},\quad\mathrm{if\ }C_j=\varnothing;
 \end{equation} and we also assume for $C_j=\varnothing$ the localization bound
 \begin{equation} \label{boundmerge5}\bigg\|\bigg(1+\frac{1}{R_j}\big|k_j-\sum_{\lf\in\Uc_j}\zeta_\lf k_\lf-\ell_j\big|\bigg)^\kappa h^{(j)}\bigg\|_{X^{-b_0}_{\Vc_j}[k_jk_{B_j}]}\lesssim\Xf_j\cdot \tau^{-\theta}\prod_{\lf\in B_j}N_\lf^\beta\prod_{\lf\in\Pc_j}N_\lf^{-4\varepsilon}\prod_{\lf\in E_j}N_\lf^{8\varepsilon}\prod_{\pf\in\Yc_j}N_\pf^{-\delta^3}\prod_{\ff \in \Vc_j} N_\ff^d,
 \end{equation} where $R_1=M^{2\delta}$, $R_j=(N_*)^{2\delta}$ for $j\geq 2$ and $\ell_j=\sum_{\ff\in\Vc_j}\zeta_\ff k_\ff$, and the $\lambda_{\Vc_j}$-derivative bound (for $C_j=E_j=\varnothing$),
  \begin{equation} \label{boundmerge6}\|\partial_{\lambda_{\Vc_j}}h^{(j)}\|_{X_{\Vc_j}^{-b_0}[k_jk_{\Uc_j}]}\lesssim \Xf_j\cdot\tau^{-\theta}\exp[(\log N_j)^5+|\Sc_j|(\log N_j)^3].
 \end{equation}

 Finally fix a subpartition $(B,C)$ of $\Uc$, and let $E=\Uc\backslash(B\cup C)$.
Then, under all of the above assumptions, we have the following results, where we denote
 \begin{equation}\label{boundmerge10}\Yf=\prod_{j=1}^r\Xf_j\cdot\tau^{-\theta}M^{\theta}(N_*)^{-2\varepsilon^4}:
 \end{equation}

(i) If $C\neq\varnothing$, and assume that $\max\{N_\lf:\lf\in C\cap\Uc_1\}\sim\max\{N_\lf:\lf\in C\}$, then $\tau^{-1}M$-certainly, the tensor $H=H_{kk_\Uc}(k_\Vc,\lambda_\Vc)$ satisfies the bounds \begin{equation}\label{boundmerge7}\sqrt{\Upsilon}\cdot\|H\|_{X_{\Vc}^{-b_0}[kk_{B}\to k_{C}]}\lesssim\Yf\cdot\tau^{-\theta}\prod_{\lf\in B\cup C}N_\lf^\beta\prod_{\lf\in\Pc}N_\lf^{-4\varepsilon}\prod_{\lf\in E}N_\lf^{8\varepsilon}\prod_{\pf\in\Yc}N_\pf^{-\delta^3} \prod_{\ff \in \Vc} N_\ff^d\cdot \big(\max_{\lf\in C}N_\lf\big)^{-\beta}.
 \end{equation}

(ii) If $C=\varnothing$, and assume that $h^{(1)}$ has type 1, and $N_1\gtrsim N_j$ for all $1\leq j\leq r$ such that $h^{(j)}$ has type 1, then $\tau^{-1}M$-certainly we have  \begin{equation}  \label{boundmerge8}\sqrt{\Upsilon}\cdot\|H\|_{X_{\Vc}^{-b_0}[kk_{B}]}\lesssim\Yf\cdot\tau^{-\theta}\prod_{\lf\in B}N_\lf^\beta\prod_{\lf\in\Pc}N_\lf^{-4\varepsilon}\prod_{\lf\in E}N_\lf^{8\varepsilon}\prod_{\pf\in\Yc}N_\pf^{-\delta^3}\prod_{\ff \in \Vc} N_\ff^d\cdot M^{-\varepsilon}.
 \end{equation}
 
 (iii) If $C=\varnothing$, and assume that $h^{(1)}$ has type 1. Moreover, assume we restrict the tensors $H_{kk_\Uc}(k_\Vc,\lambda_\Vc)$ and $h_{k_1k_{\Uc_1}}^{(1)}(k_{\Vc_1},\lambda_{\Vc_1})$ to the sets
 \begin{equation}\label{Prop5.4:localsupp}
 1+\frac{1}{M^{2\delta}}\bigg|k-\sum_{\lf\in\Uc}\zeta_\lf^* k_\lf-\ell\bigg|\sim K,\quad 1+\frac{1}{M^{2\delta}}\bigg|k_1-\sum_{\lf\in\Uc_1}\zeta_\lf k_\lf-\ell_1\bigg|\sim K_1,\end{equation} where $K\lesssim K_1$ are two dyadic numbers, $\zeta_\nf^*$ represents the sign of $\nf$ in $\Sc$, $\ell=\sum_{\ff\in\Vc}\zeta_\ff^* k_\ff$, and $\ell_1=\sum_{\ff\in\Vc_1}\zeta_\ff k_\ff$ as above. Then $\tau^{-1}M$-certainly we have  \begin{equation} \label{boundmerge9}\sqrt{\Upsilon}\cdot\bigg\|\bigg(1+\frac{1}{M^{2\delta}}\big|k-\sum_{\lf\in\Uc}\zeta_\lf^* k_\lf-\ell\big|\bigg)^\kappa H\bigg\|_{X_\Vc^{-b_0}[kk_{B}]}\lesssim\Yf\cdot \tau^{-\theta}\prod_{\lf\in B}N_\lf^\beta\prod_{\lf\in\Pc}N_\lf^{-4\varepsilon}\prod_{\lf\in E}N_\lf^{8\varepsilon}\prod_{\pf\in\Yc}N_\pf^{-\delta^3} \prod_{\ff \in \Vc} N_\ff^d.
 \end{equation}
\end{prop}
\begin{proof}
We would like to apply Proposition \ref{algorithm1}.
The technical difficulty before doing so is two-folded. On the one hand, we must separate the tensor $h^{(1)}$ from the tensor $H$ defined in (\ref{mergetrim}).
On the other hand, since there are no over-pairings in the statement of Proposition \ref{algorithm1}
we must remove these from $H$. Once we have dealt with this technical difficulty, the heart of the matter lies in implementing Proposition \ref{algorithm1} to obtain the desired bounds above.

 In the proof below we will mainly focus on (\ref{boundmerge7}). The proof of (\ref{boundmerge8}) will be analogous, and we will only point out the necessary changes in the course of the proof.  
 Moreover, in proving (\ref{boundmerge8}) we will only use the bound (\ref{boundmerge3}) for the tensor $h^{(1)}$, so (\ref{boundmerge9}) will follow from the same arguments as (\ref{boundmerge8}) once we use (\ref{boundmerge5}) instead of (\ref{boundmerge3}), in view of (\ref{Prop5.4:localsupp}) and $K\lesssim K_1$, $N_1\leq M$.
 
The proof will proceed in four steps. In \emph{Step 1}, we reduce the desired estimates for the tensor $H$ to those for the tensor $H^\circ$ defined in (\ref{prop5.4:H0}) below. In \emph{Step 2}, we remove and estimate the over-pairings, and reduce the desired estimates for $H^\circ$ to those for $(H^\circ)^\dagger$ defined in (\ref{reducedtrim2}). In \emph{Step 3}, we first single out $h^{(1)}$ in $(H^\circ)^\dagger$ and in turn apply Propositions \ref{bilineartensor} and \ref{gausscont}, then remove and estimate the over-pairings as in \emph{Step 2}, to reduce the desired estimates for $(H^\circ)^\dagger$ to those for $\Hc^\dagger$ defined in (\ref{reducedmerge4}). Finally in \emph{Step 4}, we implement Proposition \ref{algorithm1} and conclude the proof.

\emph{Step 1: pre-processing.} Define \[\widetilde{\Sc}=(\widetilde{\Lc},\widetilde{\Vc},\widetilde{\Yc})=\mathtt{Merge}(\Sc_1,\cdots,\Sc_r,\Bs,\Os),\]\[\widetilde{H}=\widetilde{H}_{kk_{\widetilde{\Uc}}}(k_{\widetilde{\Vc}},\lambda_{\widetilde{\Vc}})=\mathtt{Merge}(h^{(1)},\cdots, h^{(r)},h,\Bs,\Os),\] noticing that $\widetilde{\Vc}=\Vc_1\cup\cdots\cup\Vc_r\cup\{r+1,\cdots,q\}$, and define
\begin{equation}\label{prop5.4:H0}
H^\circ=(H^\circ)_{kk_\Uc}(k_{\widetilde{\Vc}},\lambda_{\widetilde{\Vc}})=\sum_{k_{\widetilde{\Uc}\backslash\Uc}}\widetilde{H}_{kk_{\widetilde{\Uc}}}(k_{\widetilde{\Vc}},\lambda_{\widetilde{\Vc}})\prod_{\lf\in \widetilde{\Uc}\backslash\Uc}(f_{N_\lf})_{k_\lf}^{\zeta_\lf^*},\end{equation} then we have
\[H_{kk_\Uc}(k_\Vc,\lambda_\Vc)=\sum_{k_{\widetilde{\Vc}\backslash\Vc}}\int\mathrm{d}\lambda_{\widetilde{\Vc}\backslash\Vc}\cdot(H^\circ)_{kk_\Uc}(k_{\widetilde{\Vc}},\lambda_{\widetilde{\Vc}})\prod_{\ff\in\widetilde{\Vc}\backslash\Vc}(\widehat{z_{N_\ff}})_{k_\ff}^{\zeta_\ff^*}(\lambda_\ff).\] By the same proof as part (1) of Proposition \ref{trimbd}, for any $X_\Vc^{-b_0}[\cdots]$ norm we have
\begin{equation}\label{gainofv}\|H\|_{X_\Vc^{-b_0}[\cdots]}\lesssim\|H^\circ\|_{X_{\widetilde{\Vc}}^{-b_0}[\cdots]}\cdot\prod_{\ff\in\widetilde{\Vc}\backslash\Vc}N_\ff^{-D_1},\end{equation} as well as for the weighted norm in (\ref{boundmerge9}). Therefore, it suffices to estimate the $X_{\widetilde{\Vc}}^{-b_0}[\cdots]$ norms (as well as the weighted ones) for $H^\circ$. 

For each $1\leq j\leq r$, if $h^{(j)}$ has type 1, we can define $\Xf_j^*=\Xf_j^*(k_{\Vc_j},\lambda_{\Vc_j})$ to be the smallest positive number such that the bounds (\ref{boundmerge2})--(\ref{boundmerge6}) are true for this choice of $(k_{\Vc_j},\lambda_{\Vc_j})$ with $X_{\Vc_j}^{-b_0}$ in the norms removed (for example $X_{\Vc_j}^{-b_0}[k_jk_{B_j}\to k_{C_j}]$ replaced by $k_jk_{B_j}\to k_{C_j}$), and with $\Xf_j$ replaced by $\Xf_j^*(k_{\Vc_j},\lambda_{\Vc_j})$; for example one of the inequalities satisfied by $\Xf_j^*=\Xf_j^*(k_{\Vc_j},\lambda_{\Vc_j})$, which corresponds to (\ref{boundmerge2}), would be
\begin{equation}\label{eq:modified6.40}
\|h^{(j)}\|_{k_jk_{B_j}\to k_{C_j}}\lesssim\Xf_j^*\cdot\tau^{-\theta}\prod_{\lf\in B_j\cup C_j}N_\lf^\beta\prod_{\lf\in\Pc_j}N_\lf^{-4\varepsilon}\prod_{\lf\in E_j}N_\lf^{8\varepsilon}\prod_{\pf\in\Yc_j}N_\pf^{-\delta^3} \prod_{\ff \in \Vc_j} N_\ff^d\cdot \big(\max_{\lf\in C_j}N_\lf\big)^{-\beta}
\end{equation}
 with $C_j\neq\varnothing$, for fixed $(k_{\Vc_j},\lambda_{\Vc_j})$. If $h^{(j)}$ has type 0, we simply define $\Xf_j^*=\Xf_j$. Then we have
\begin{equation}\label{reducednorms}\sum_{k_{\Vc_j}}\int\mathrm{d}\lambda_{\Vc_j}\cdot\prod_{\ff\in\Vc_j}\langle \lambda_\ff\rangle^{-2b_0}\cdot\Xf_j^*(k_{\Vc_j},\lambda_{\Vc_j})^2\lesssim \Xf_j^2
\end{equation} for type 1 tensors. When $(k_{\widetilde{\Vc}},\lambda_{\widetilde{\Vc}})$ is fixed, which means $(k_{\Vc_j},\lambda_{\Vc_j})$ are fixed for $1\leq j\leq r$ and $(k_j,\lambda_j)$ are fixed for $r+1\leq j\leq q$, we shall view $h=h_{kk_1\cdots k_r}$ as a tensor depending on $(k,k_1,\cdots,k_r)$; for $1\leq j\leq r$ we shall view $h^{(j)}=h_{k_jk_{\Uc_j}}^{(j)}$ as a tensor depending on $(k_j,k_{\Uc_j})$.

With these reductions, we can view $\widetilde{H}=\widetilde{H}_{kk_{\widetilde{\Uc}}}$ as a tensor depending on $(k,k_{\widetilde{\Uc}})$ and $H^\circ=H_{kk_\Uc}^\circ$ depending on $(k,k_\Uc)$, namely
\begin{equation}\label{reducedmerge}\widetilde{H}_{kk_{\widetilde{\Uc}}}=\prod_{\lf,\lf'}^{(1)}\mathbf{1}_{k_\lf=k_{\lf'}}\prod_{\lf,\lf'}^{(2)}\mathbf{1}_{k_\lf\neq k_{\lf'}}\cdot\sum_{(k_1,\cdots,k_r)}h_{kk_1\cdots k_r}\sum_{k_{\Qc}}^{(3)}\prod_{\lf\in\Qc}\Delta_{N_\lf}\gamma_{k_\lf}\prod_{j=1}^{r}\big[h_{k_jk_{\Uc_j}}^{(j)}\big]^{\zeta_j},
\end{equation}
\begin{equation}\label{reducedtrim}(H^\circ)_{kk_\Uc}=\sum_{k_{\widetilde{\Uc}\backslash\Uc}}\widetilde{H}_{kk_{\widetilde{\Uc}}}\prod_{\lf\in \widetilde{\Uc}\backslash\Uc}(f_{N_\lf})_{k_\lf}^{\zeta_\lf^*},
\end{equation} where in (\ref{reducedmerge}), the sum and products are taken in the same way as (\ref{merge1}) when merging the tensors $(h^{(1)},\cdots,h^{(r)})$ via $(h,\Bs,\Os)$. Similarly $\Qc$ is defined as in Definition \ref{defmerge}. Also note that
$\Uc =\{\lf\in\widetilde{\Uc}:N_\lf\geq M^\delta\}$.

Our goal is to prove that, for fixed values of $(k_{\widetilde{\Vc}},\lambda_{\widetilde{\Vc}})$ defined in (\ref{prop5.4:H0}), the tensor $H^\circ=H_{kk_\Uc}^\circ$ satisfies (\ref{boundmerge7}), (\ref{boundmerge8}) and (\ref{boundmerge9}) $\tau^{-1}M$-certainly, but with the following three adjustments: (a) we remove the $X_\Vc^{-b_0}$ parts in the norms
 (for example $X_\Vc^{-b_0}[kk_B\to k_C]$ is replaced by $kk_B\to k_C$), and multiply the left hand sides by the extra factor $\prod_{j=r+1}^q N_j^{-d/2}$; (b) the set $\Vc$ in the factors $\prod_{\ff\in\Vc}N_\ff^d$ on the right hand sides is replaced by $\Vc_1\cup\cdots\cup\Vc_r$ and $\Yc$ is replaced by $\widetilde{\Yc}$; and (c) the $\Xf_j$ in the definition (\ref{boundmerge10}) of $\Yf$ is replaced by $\Xf_j^*$. 
 For example, the analogue of (\ref{boundmerge7}) with the above adjustments (a)--(c) amounts to showing:
\begin{equation}\label{prop5.4:5.45'}
\begin{aligned}
\sqrt{\Upsilon}\cdot\prod_{j=r+1}^q N_j^{-d/2}\,\|H^\circ\|_{kk_{B}\to k_{C}}\lesssim&\tau^{-\theta}\prod_{\lf\in B\cup C}N_\lf^\beta\prod_{\lf\in\Pc}N_\lf^{-4\varepsilon}\prod_{\lf\in E}N_\lf^{8\varepsilon}\big(\max_{\lf\in C}N_\lf\big)^{-\beta}\\
&\times \prod_{\pf\in\widetilde{\Yc}}N_\pf^{-\delta^3}\, \prod_{\ff \in \Vc_1\cup\cdots\cup\Vc_r} N_\ff^d
\cdot\left(
\prod_{j=1}^r\Xf^*_j\cdot\tau^{-\theta}M^{\theta}(N_*)^{-2\varepsilon^4}\right),
\end{aligned}
\end{equation}
where   $\prod_{j=1}^r\Xf^*_j\cdot\tau^{-\theta}M^{\theta}(N_*)^{-2\varepsilon^4}$ is $\Yf$ in  (\ref{boundmerge10}) with replacements of $\Xf_j$ by $\Xf_j^*$. The corresponding analogues of (\ref{boundmerge8}) and (\ref{boundmerge9}) with the adjustments (a)--(c) are similar.

If we can prove (\ref{prop5.4:5.45'}) for a fixed choice of $(k_{\widetilde{\Vc}}, \lambda_{\widetilde{\Vc}})$ in $H^{\circ}$, then by applying the \emph{meshing argument} in the same way as we did in the proof of Proposition \ref{trimbd}, using (\ref{mergetensor}) and (\ref{boundmerge6}), we can reduce to the case of at most $\exp(\kappa(\log M)^6)$ choices for $(k_{\widetilde{\Vc}},\lambda_{\widetilde{\Vc}})$, hence $\tau^{-1}M$-certainly we may assume that (\ref{prop5.4:5.45'}) holds for $H_{kk_\Uc}^\circ$ for all choices of $(k_{\widetilde{\Vc}},\lambda_{\widetilde{\Vc}})$. Since $\widetilde{\Vc}=\Vc_1\cup\cdots\cup\Vc_r\cup\{r+1,\cdots,q\}$ and the definition of $X_{\widetilde{\Vc}}^{-b_0}[\cdots]$ involves summing and integrating over $(k_{\widetilde{\Vc}}, \lambda_{\widetilde{\Vc}})$, then $\tau^{-1}M$-certainly we have the following estimate 
\begin{equation}\label{prop5.4:5.45"}
\sqrt{\Upsilon}\cdot\|H^\circ\|_{X_{\widetilde{\Vc}}^{-b_0}[kk_{B}\to k_{C}]}\lesssim\Yf\cdot\tau^{-\theta}\prod_{\lf\in B\cup C}N_\lf^\beta\prod_{\lf\in\Pc}N_\lf^{-4\varepsilon}\prod_{\lf\in E}N_\lf^{8\varepsilon}\prod_{\pf\in\widetilde{\Yc}}N_\pf^{-\delta^3} \prod_{\ff \in \widetilde{\Vc}} N_\ff^d\cdot \big(\max_{\lf\in C}N_\lf\big)^{-\beta},
\end{equation}
which follows from (\ref{prop5.4:5.45'}) and
\begin{equation}\label{powersave}\bigg(\sum_{k_{\widetilde{\Vc}}}\int\mathrm{d}\lambda_{\widetilde{\Vc}}\cdot\prod_{\ff\in\widetilde{\Vc}}\langle\lambda_\ff\rangle^{-2b_0}\prod_{j=1}^r\Xf_j^*(k_{\Vc_j},\lambda_{\Vc_j})^2\bigg)^{1/2}\lesssim\prod_{j=1}^r\Xf_j\prod_{j=r+1}^q N_j^{d/2}.\end{equation} Note that (\ref{powersave}) follows from taking the tensor product of (\ref{reducednorms}) for $1\leq j\leq r$ and summing and integrating over $(k_j,\lambda_j)$ for $r+1\leq j\leq q$.
Finally, the desired bound (\ref{boundmerge7}) for $H_{kk_\Uc}(k_\Vc,\lambda_\Vc)$ follows from (\ref{prop5.4:5.45"}) and (\ref{gainofv}) in view of the $N_\ff^{-D_1}$ powers on the right hand side of (\ref{gainofv}). 
The desired bounds (\ref{boundmerge8}) and (\ref{boundmerge9}) for $H_{kk_\Uc}(k_\Vc,\lambda_\Vc)$ can be obtained in a similar way.

\emph{Step 2: removing over-pairings.} From now on we will fix the value of $(k_{\widetilde{\Vc}},\lambda_{\widetilde{\Vc}})$ and reduce to the setting of the tensor $\widetilde{H}$ in (\ref{reducedmerge}) and the tensor $H^\circ$ in (\ref{reducedtrim}).

We will first focus on (\ref{boundmerge7}), fix a subpartition $(B,C)$ of $\Uc$ with $C\neq\varnothing$, and denote $E=\Uc\backslash(B\cup C)$. Recall that $\Os=\{\Ac_1,\cdots,\Ac_m\}$ is the collection of all pairings and over-pairings (see Definition \ref{defmerge}), and notice that the frequencies $N_\lf$ are the same for $\lf\in\Ac_i$. Without loss of generality,  we may assume the frequency  $N_\lf$ for $\lf\in\Ac_i$ is $\geq M^\delta$ for all $1\leq i\leq v$ (where $1\leq v\leq m$), and $\leq M^{[\delta]}$ for $v+1\leq i\leq m$. In particular, we have that $\Ac_i\cap\Uc=\varnothing$ and that $\Ac_i\cap \Uc_1=\varnothing$ for $i\geq v+1$ (since $N_\lf\geq M^\delta$ for $\lf\in\Uc$ and $\lf\in\Uc_1$).

Next, in (\ref{reducedmerge}), the product $\prod_{\lf,\lf'}^{(2)}$ has two parts, namely $\prod_{\lf,\lf'}^{(2,\geq)}$ containing $(\lf,\lf')$ such that $N_\lf=N_{\lf'}\geq M^\delta$, and $\prod_{\lf,\lf'}^{(2,<)}$ containing $(\lf,\lf')$ such that $N_\lf=N_{\lf'}< M^\delta$. In the proof below we will slightly modify (\ref{reducedmerge}) by changing $\prod_{\lf,\lf'}^{(2)} \mathbf{1}_{k_\lf\neq k_{\lf'}}$ into $\prod_{\lf,\lf'}^{(2,<)} \mathbf{1}_{k_\lf\neq k_{\lf'}}$. This will be necessary in order to separate the tensor $h^{(1)}$ from the rest later in the proof, and will not cause a problem, since the original $\widetilde{H}$ equals the modified $\widetilde{H}$ multiplied by $\prod_{\lf,\lf'}^{(2,\geq)}\mathbf{1}_{k_\lf\neq k_{\lf'}}$, and the original $H^\circ$ equals the modified $H^\circ$ multiplied by $\prod_{\lf,\lf'}^{(2,\geq)}\mathbf{1}_{k_\lf\neq k_{\lf'}}$, which is a bounded operation due to Lemma \ref{adjustment}.
The reason we need to keep the factors in the product $\prod_{\lf,\lf'}^{(2,<)}\mathbf{1}_{k_\lf\neq k_{\lf'}}$ is to guarantee the no-pairing assumption required to apply Proposition \ref{gausscont} later in \emph{Step 3}.

Recall that when $C\neq\varnothing$ we have $\max\{N_\lf:\lf\in C\cap\Uc_1\}\sim\max\{N_\lf:\lf\in C\}$  (see part (i) in the assumption). Denote the particular $\lf\in C\cap\Uc_1$  where the maximum is attained by $\lf_{\mathrm{top}}$. In this step we shall fix the values of $k_\lf$ for $\lf\in E$ and for $\lf\in \Ac_i$,  where $1\leq i\leq v$ and $|\Ac_i|\geq 3$. These $\lf$'s are divided into groups according to the pairing and over-pairing relations, and there are four possible cases:

\emph{Case 1}: 
For each $\lf\in E$ that does not belong to any $\Ac_i$, we form a group with only one element $\lf$. Each of these $\lf$ belongs to a unique $\Uc_j$ for some $1\leq j\leq r$.

\emph{Case 2}:
For each $i$ such that $\Ac_i\cap E\neq\varnothing$, we form a group containing all elements of this $\Ac_i$.
Define\footnote{Recall that $\Qc$ is defined in (\ref{def3.6:Q}). Since $\Qc$ contains the two-element pairings in $\Ac_i$, $|\Ac_i\cap\Qc|$ is even.} $y_i$ and $z_i$ such that $|\Ac_i\cap\Qc|=2y_i$ and $|\Ac_i\cap\widetilde{\Uc}|=z_i$. We then have $|\Ac_i|=2y_i+z_i$.

\emph{Case 3}:
For each $i$ such that $\Ac_i\cap E=\varnothing$ and $\Ac_i\cap\widetilde{\Uc}\neq\varnothing$, we form a group containing all elements of this $\Ac_i$. Define $y_i$ and $z_i$ such that $|\Ac_i\cap\Qc|=2y_i$ and $|\Ac_i\cap\widetilde{\Uc}|=z_i$, then $|\Ac_i|=2y_i+z_i$.

\emph{Case 4}: For each $i$ such that $\Ac_i\cap\widetilde{\Uc}=\varnothing$, we form a group containing all elements of this $\Ac_i$. Note that in this case $|\Ac_i|\geq 4$. Define $y_i$ such that $|\Ac_i|=2y_i$.

Note that in \emph{Cases 1--3}, some $\lf$ in the group belong to $\widetilde{\Uc}$ (and hence $\Uc$) and thus appear as a variable of $\widetilde{H}$ and $H^\circ$, while some $\lf$ in the group may not. In \emph{Case 4}, no $\lf$ in the group appear as a variable of $\widetilde{H}$ or $H^\circ$, and they only appear in the summation $\sum_{k_\Qc}^{(3)}$ in (\ref{reducedmerge}).

Now let $G$ be the union of all groups in \emph{Cases 1--4}. Define $\widetilde{\Uc}^\dagger=\widetilde{\Uc}\backslash G$ and similarly $\Uc^\dagger=\Uc\backslash G$, and similarly $\Qc^\dagger$, $B^\dagger$ and $C^\dagger$. Note that $\widetilde{\Uc}^\dagger\backslash\Uc^\dagger=\widetilde{\Uc}\backslash\Uc$.  Let $\Uc_j^\dagger=\Uc_j\backslash G$ for $1\leq j\leq r$, and define\footnote{Note this definition is only for $j=1$ and not for $2\leq j\leq r$, which we will define later.} $G_1=\Uc_1\backslash\Uc_1^\dagger=\Uc_1\cap G$. Let $\Os^\dagger$ be $\Os$ after removing the $\Ac_i$'s involved in \emph{Cases 2--4} above. Thus, for any $\Ac_i\in\Os^\dagger$, if $1\leq i\leq v$ (in particular if $\Ac_i\cap \Uc_1^\dagger\neq\varnothing$), then $|\Ac_i|=2$.

Once we fix all the variables $k_\lf$ for $\lf \in G$ described above, we can view $\widetilde{H}$ as a tensor depending on $(k,k_{\widetilde{\Uc}^\dagger})$, and $H^\circ$ as a tensor depending on $(k,k_{\Uc^\dagger})$. In the same way $h^{(j)}$ can be viewed as a tensor depending on $(k_j,k_{\Uc_j^\dagger})$. More precisely, we define $h^{(j,\dagger)}=h_{k_jk_{\Uc_j^{\dagger}}}^{(j,\dagger)}$ to be $h_{k_jk_{\Uc_j}}^{(j)}$ with the values of $k_\lf$ for $\lf\in\Uc_j\backslash \Uc_j^{\dagger} = G \cap \Uc_j$ fixed.

If we use the triangle inequality, as well as the simple fact  that
\[\|h_{k_Xk_Yk_Zk_W}\|_{k_Xk_Y\to k_Zk_W}\lesssim\prod_{\lf\in X\cup Z}N_\lf^{d/2}\cdot\sup_{k_X,k_Z}\|h_{k_Xk_Yk_Zk_W}\|_{k_Y\to k_W}\] (where $X,Y,Z,W$ are arbitrary sets) under the assumption that $\langle k_\lf\rangle\lesssim N_\lf$ for $\lf\in X\cup Z$, then we can deduce that
\begin{equation}\label{reducedest}\|H^\circ\|_{kk_B\to k_C}\lesssim\prod_{i}^{(\Os,2,\geq)}N_{\lf_i}^{-2y_i(\alpha-\theta)}\prod_i^{(\Os,3,\geq)}N_{\lf_i}^{(d/2)-2y_i(\alpha-\theta)}\prod_i^{(\Os,4,\geq)}N_{\lf_i}^{d-2y_i(\alpha-\theta)}\cdot\sup_{(k_\lf)}\|(H^\circ)^\dagger\|_{kk_{B^\dagger}\to k_{C^\dagger}},
\end{equation} where $\prod_i^{(\Os,n,\geq)}$ is taken over all groups $\Ac_i$ of \emph{Case $n$} for $2\leq n\leq 4$, $\lf_i$ is any element of $\Ac_i$, and $\sup_{(k_\lf)}$ is taken over all choices of the $k_\lf$'s  with $\lf\in G$. The tensor $(H^\circ)^\dagger$ is defined by
\begin{equation}\label{reducedtrim2}(H^\circ)^\dagger=(H^\circ)_{kk_{\Uc^\dagger}}^\dagger=\sum_{k_{\widetilde{\Uc}^\dagger\backslash\Uc^\dagger}}(\widetilde{H})_{kk_{\widetilde{\Uc}^\dagger}}^\dagger\prod_{\lf\in \widetilde{\Uc}^\dagger\backslash\Uc^\dagger}(f_{N_\lf})_{k_\lf}^{\zeta_\lf^*},
\end{equation}
\begin{equation}\label{reducedmerge2}(\widetilde{H})_{kk_{\widetilde{\Uc}^\dagger}}^\dagger=\prod_{\lf,\lf'}^{(1,\dagger)}\mathbf{1}_{k_\lf=k_{\lf'}}\prod_{\lf,\lf'}^{(2,<)}\mathbf{1}_{k_\lf\neq k_{\lf'}}\cdot\sum_{(k_1,\cdots,k_r)}h_{kk_1\cdots k_r}\sum_{k_{\Qc^\dagger}}^{(3,\dagger)}\prod_{\lf\in\Qc^\dagger}\Delta_{N_\lf}\gamma_{k_\lf}\prod_{j=1}^{r}\big[h_{k_jk_{\Uc_j^\dagger}}^{(j,\dagger)}\big]^{\zeta_j}.
\end{equation} Here the product $\prod_{\lf,\lf'}^{(2,<)}$ is the same as the one defined above in the third paragraph of \emph{Step 2}, and the product $\prod_{\lf,\lf'}^{(1,\dagger)}$ is the same as $\prod_{(\lf,\lf')}^{(1)}$ defined in (\ref{reducedmerge}) (which is taken from (\ref{merge1})), except that the product here is only taken over $\lf,\lf'\in\Ac_i$ for $\Ac_i\in\Os^\dagger$. Similarly the sum $\sum_{k_{\Qc^\dagger}}^{(3,\dagger)}$ is the same as $\sum_{k_\Qc}^{(3)}$ defined  in (\ref{reducedmerge}) (which is taken from (\ref{merge1})), except that the sum here does not involve the variables $k_{\Qc\backslash\Qc^\dagger}$. Also note that, the product $\prod_{\lf,\lf'}^{(1,\dagger)}$ and sum $\sum_{k_{\Qc^\dagger}}^{(3,\dagger)}$ are exactly the same as the ones defined in (\ref{merge1}) when merging the tensors $(h^{(1,\dagger)},\cdots, h^{(r,\dagger)})$ via $(h,\Bs,\Os^\dagger)$.

Let us illustrate the above \emph{Step 2} with an explicit example, which is an extension of the example in \emph{Step 2} of the proof of Proposition \ref{algorithm1}; for simplicity, assume there is no blossoms (i.e. $\Vc_j=\varnothing$) and all frequencies are $\geq M^\delta$, so no trimming is needed throughout the process. Suppose $q=r=7$, \[(\Uc_1,\Uc_2,\cdots,\Uc_7)= (\{\lf_{\mathrm{top}}, \cf, \jf, \kf, \mf\},\,\{\af, \bff, \kf,\of\},\,\{\cf, \dff, \ef, \kf\},\,\{\ef, \ff, \gf\},\,\{\af, \ff, \hf, \jf, \mf\},\,\{\iif, \jf, \kf,\mf\},\,\{\kf,\mf\}),\] and the corresponding frequencies satisfy \[N_{\lf_{\mathrm{top}}}\geq N_\af\geq N_{\bff}\geq N_{\cf}\geq N_{\dff}\geq N_{\ef}\geq N_{\ff}\geq N_{\gf}\geq N_{\hf} \geq N_{\iif}\geq N_{\jf}\geq N_{\kf}\geq N_{\mf}\geq N_{\of}\geq M^\delta.\] Like in the proof of Proposition \ref{algorithm1}, a leaf occurring in at least two sets represents a pairing or over-pairing; for these leaves $\lf$, we also use $\lf_j$ to indicate the copy of $\lf$ in $\Uc_j$. Then $\Uc =\widetilde{\Uc} =\{\lf_{\mathrm{top}},\bff,\dff, \gf, \hf,\iif,\jf,\kf,\of\}$ and $\widetilde{H}=H^{\circ}$, and
\begin{multline*}(H^\circ)_{kk_\Uc}=\sum_{(k_1,\cdots,k_7)}\sum_{(k_\af,k_\cf,k_\ef,k_\ff,k_\mf)}h_{kk_1\cdots k_7}\cdot h^{(1)}_{k_1k_{\lf_{\mathrm{top}}}k_\cf k_\jf k_\kf k_\mf}\cdot h^{(2)}_{k_2k_\af k_\bff k_\kf k_\of}\\\times h^{(3)}_{k_3k_\cf k_\dff k_\ef k_\kf}\cdot h^{(4)}_{k_4k_\ef k_\ff k_\gf}\cdot h^{(5)}_{k_5k_\af k_\ff k_\hf k_\jf k_\mf}\cdot h^{(6)}_{k_6 k_\iif k_\jf k_\kf k_\mf}\cdot h^{(7)}_{k_7k_\kf k_\mf},
\end{multline*} where for simplicity,  we have omitted the various $\Delta_{N_\lf}\gamma_{k_\lf}$ factors.

Suppose $B=\{\bff,\dff\}$ and $C=\{\lf_{\mathrm{top}},\gf, \hf, \kf\}$ is a subpartition of $\Uc$ and $E=\{\iif, \jf, \of\}$.
Our goal is to estimate $\|H^\circ\|_{kk_B\to k_C}$. The groups in \emph{Cases 1--4} in the above process are then:
\begin{itemize}
\item\emph{Case 1}: $\{\iif\}$ and $\{\of\}$;
\item\emph{Case 2}: $\Ac_1 =\{\jf_1, \jf_5, \jf_6\}$;
\item\emph{Case 3}: $\Ac_2 =\{\kf_1, \kf_2, \kf_3,\kf_6,\kf_7\}$;
\item\emph{Case 4}: $\Ac_3 = \{\mf_1, \mf_5,\mf_6,\mf_7\}$.
\end{itemize}
After fixing the values of $(k_\iif,k_\of,k_\jf,k_\kf,k_\mf)$, which are considered in \emph{Cases 1--4} above, we can reduce $\|H^\circ\|_{kk_B\to k_C}$ to $\|(H^\circ)^\dagger\|_{kk_{B^\dagger}\to k_{C^\dagger}}$ by the above argument, where the relevant sets
\begin{equation}\label{prop6.4:example}
B^\dagger =\{\bff,\dff\},\,\, C^{\dagger} =\{\lf_{\mathrm{top}}, \gf, \hf\},\,\, \Uc^\dagger = \{\lf_{\mathrm{top}},\bff,\dff, \gf, \hf\},\,\, \Qc^{\dagger} =\{\af_2, \af_5, \cf_1,\cf_3, \ef_3, \ef_4, \ff_4, \ff_5\},
\end{equation}
\[(\Uc_1^\dagger,\cdots, \Uc_7^\dagger)=(\{\lf_{\mathrm{top}}, \cf\},\,\{\af, \bff\},\,\{\cf, \dff, \ef\},\,\{\ef, \ff, \gf\},\,\{\af, \ff, \hf\},\,\varnothing,\,\varnothing),\] and the tensors
\[(H^\circ)_{kk_{\Uc^\dagger}}^\dagger=\sum_{(k_1,\cdots,k_7)}\sum_{(k_\af,k_\cf,k_\ef,k_\ff)}h_{kk_1\cdots k_7}\cdot h^{(1,\dagger)}_{k_1k_{\lf_{\mathrm{top}}}k_\cf }\cdot h^{(2,\dagger)}_{k_2k_\af k_\bff }\cdot h^{(3,\dagger)}_{k_3k_\cf k_\dff k_\ef }\cdot h^{(4,\dagger)}_{k_4k_\ef k_\ff k_\gf}\cdot h^{(5,\dagger)}_{k_5k_\af k_\ff k_\hf }\cdot h^{(6,\dagger)}_{k_6 }\cdot h^{(7,\dagger)}_{k_7},\] where $h_{k_jk_{\Uc_j^\dagger}}^{(j,\dagger)}$ is $h_{k_jk_{\Uc_j}}^{(j)}$ after fixing the $k_{\Uc_j\backslash\Uc_j^\dagger}$ for $1\leq j\leq 7$.

\emph{Step 3: the method of descent.} Now we need to estimate $(H^\circ)^\dagger$. A key step is to implement Proposition \ref{algorithm1} by singling out $h^{1,\dagger}$. To that effect, consider those $\Ac_i\in\Os^\dagger$ such that $\Ac_i\cap \Uc_1^\dagger\neq\varnothing$; we know that each such $\Ac_i$ contains exactly one pair. Let $(\Os^{\dagger})'$ be $\Os^\dagger$ after removing these $\Ac_i$'s, and $\Dc$ be the union of these $\Ac_i$'s. Define $\Fc=(\Uc^\dagger\cup\Dc)\backslash\Uc_1^\dagger$ and similarly define $\widetilde{\Fc}=(\widetilde{\Uc}^\dagger\cup\Dc)\backslash\Uc_1^\dagger$, so that $\widetilde{\Uc}^\dagger\backslash \Uc^\dagger=\widetilde{\Fc}\backslash\Fc$.  Then by (\ref{reducedtrim2})--(\ref{reducedmerge2}) we have
\begin{equation}\label{separate}
(H^\circ)_{kk_{\Uc^\dagger}}^\dagger=\sum_{k_1}\sum_{k_\Dc}\prod_{\lf\in\Dc}\Delta_{N_\lf}\gamma_{k_\lf}\cdot\big[h_{k_1k_{\Uc_1^\dagger}}^{(1,\dagger)}\big]^{\zeta_1}\cdot (H^*)_{kk_1k_\Fc},
\end{equation} where $\sum_{k_\Dc}$ is taken over all $k_\Dc$ such that $k_\lf=k_{\lf'}$ for any pairing $\Ac_i=\{\lf,\lf'\}\subset\Dc$, and the tensor $(H^*)_{kk_1k_\Fc}$ is defined by
\begin{equation}\label{reducedtrim3}(H^*)_{kk_1k_\Fc}=\sum_{k_{\widetilde{\Fc}\backslash\Fc}}\Hc_{kk_1k_{\widetilde{\Fc}}}\prod_{\lf\in\widetilde{\Fc}\backslash\Fc}(f_{N_\lf})_{k_\lf}^{\zeta_\lf^*},
\end{equation}
\begin{equation}\label{reducedmerge3}\Hc_{kk_1k_{\widetilde{\Fc}}}=\prod_{\lf,\lf'}^{(1',\dagger)}\mathbf{1}_{k_\lf=k_{\lf'}}\prod_{\lf,\lf'}^{(2,<)}\mathbf{1}_{k_\lf\neq k_{\lf'}}\cdot\sum_{(k_2,\cdots,k_r)}h_{kk_1\cdots k_r}\sum_{k_{\Rc}}^{(3',\dagger)}\prod_{\lf\in\Rc}\Delta_{N_\lf}\gamma_{k_\lf}\prod_{j=2}^{r}\big[h_{k_jk_{\Uc_j^\dagger}}^{(j,\dagger)}\big]^{\zeta_j}.
\end{equation}  Here $\Rc=\Qc^\dagger\backslash\Dc$; in (\ref{reducedmerge3}) the product $\prod_{\lf,\lf'}^{(2,<)}$ is the same as the one defined above in \emph{Step 2}, and the product $\prod_{\lf,\lf'}^{(1',\dagger)}$ is the same as $\prod_{\lf,\lf'}^{(1,\dagger)}$ defined in (\ref{reducedmerge2}). The sum $\sum_{k_{\Rc}}^{(3',\dagger)}$ is the same as $\sum_{k_{\Qc^\dagger}}^{(3,\dagger)}$ defined {in (\ref{reducedmerge2})}, except that the sum here does not involve the variables $k_\Dc$. 
The tensor $H^*$ can be understood as a ``partial trimming'' of $\Hc$ at frequency $M^\delta$, i.e. only the tree part $\widetilde{\Fc}$ is trimmed.

Note that $N_\lf^{\alpha-\theta}\cdot\Delta_{N_\lf}\gamma_{k_\lf}$ is a bounded function of $k_\lf$ and can be absorbed into one of the tensors $h^{(1,\dagger)}$  or $H^*$.
Applying Proposition \ref{bilineartensor} to (\ref{separate})
 with $(k_X,k_Y,k_{A_1},k_{A_2})=(kk_{B^\dagger},k_{C^\dagger},k_1k_{\Uc^\dagger_1},kk_1k_{\Fc})$, where we combine any two elements in $\Dc$ that form a pairing into a single element as we did in the proof of Proposition \ref{gausscont}, we obtain
\begin{equation}\label{reducedest2}\|(H^\circ)^\dagger\|_{kk_{B^\dagger}\to k_{C^\dagger}}\lesssim\prod_{\lf\in\Dc}N_\lf^{-\alpha+\theta}\cdot\|h_{k_1k_{\Uc_1^\dagger}}^{(1,\dagger)}\|_{k_1k_{(B^\dagger\cup\Dc)\cap\Uc_1^\dagger}\to k_{C^\dagger\cap \Uc_1^\dagger}}\cdot\|(H^*)_{kk_1k_\Fc}\|_{kk_{B^\dagger\cap \Fc}\to k_1k_{(C^\dagger\cup\Dc)\cap\Fc}}.\end{equation}

To estimate $H^*$ in (\ref{reducedest2}) we shall apply Proposition \ref{gausscont}. First recall
 $N_*=\max(N_2,\cdots,N_q)$, and that since the expression (\ref{reducedmerge3}) for $\Hc$ involves only $h$ and $h^{(j,\dagger)}$ for $2\leq j\leq r$, by assumption this $\Hc$ \emph{is $\Bc_{(N_*)^{[\delta]}}$ measurable}. Next note that by assumption $N_\lf\geq (N_*)^\delta$ for each $\lf\in\widetilde{\Fc}\backslash\Fc$ and that there is no pairing among $k_{\widetilde{\Fc}\backslash\Fc}$ in (\ref{reducedtrim3}) in view of the product $\prod_{\lf,\lf'}^{(2,<)} \mathbf{1}_{k_\lf\neq k_{\lf'}}$ in the definition (\ref{reducedmerge3}) of $\Hc$. Now we apply Proposition \ref{gausscont} to (\ref{reducedtrim3}), setting $(b,c)=(kk_{B^\dagger\cap \Fc},\,k_1k_{(C^\dagger\cup\Dc)\cap\Fc})$, then for some partition $(B_0,C_0)$ of $\widetilde{\Fc}\backslash\Fc$ we have $\tau^{-1}M$-certainly that 
\begin{equation}\label{step3}\|(H^*)_{kk_1k_\Fc}\|_{kk_{B^\dagger\cap \Fc}\to k_1k_{(C^\dagger\cup\Dc)\cap\Fc}}\lesssim(\tau^{-1}M)^\theta\prod_{\lf\in\widetilde{\Fc}\backslash\Fc}N_\lf^{-\alpha+\theta}\cdot\|\Hc\|_{kk_S\to k_1k_T},
\end{equation}
where $S=B_0\cup (B^\dagger\cap \Fc)$ and $T=C_0\cup((C^\dagger\cup\Dc)\cap\Fc)$. Note that $(S,T)$ form a partition of $\widetilde{\Fc}$ such that $S\supset B^\dagger\cap \Fc$ and $T\supset (C^\dagger\cup\Dc)\cap\Fc$.

It now remains to estimate $\Hc$. By applying Lemma \ref{adjustment} we may remove the product $\prod_{\lf,\lf'}^{(2,<)} \mathbf{1}_{k_\lf\neq k_{\lf'}}$ from the definition (\ref{reducedmerge3}). Then, we can repeat the arguments in \emph{Step 2} above, and fix the values of $k_\lf$ for $\lf\in\Ac_i$ with $\Ac_i\in(\Os^\dagger)'$ and $|\Ac_i|\geq 3$ (where necessarily $i\geq v+1$). Here we only have two cases:

\emph{Case 3}:
 For each $i$ such that $\Ac_i\cap\widetilde{\Fc}\neq\varnothing$, we form a group containing all elements of this $\Ac_i$. Define $y_i$ and $z_i$ such that $|\Ac_i\cap\Rc|=2y_i$ and $|\Ac_i\cap\widetilde{\Fc}|=z_i$, then $|\Ac_i|=2y_i+z_i$.

\emph{Case 4}: 
For each $i$ such that $\Ac_i\cap\widetilde{\Fc}=\varnothing$, we form a group containing all elements of this $\Ac_i$. Note that in this case $|\Ac_i|\geq 4$. Define $y_i$ such that $|\Ac_i|=2y_i$.

As in \emph{Step 2}, we define $G^\dagger$ to be the union of all groups in \emph{Cases 3--4} defined above. Define $(\Os^{\dagger\dagger})'$ to be $(\Os^{\dagger})'$ after removing the $\Ac_i$'s involved in \emph{Cases 3--4} above, $\widetilde{\Fc}^\dagger=\widetilde{\Fc}\backslash G^\dagger$, and similarly define $S^\dagger$, $T^\dagger$ and $\Rc^\dagger$. We also define $\Uc_j^{\dagger\dagger}=\Uc_j^\dagger\backslash G^\dagger$. With these variables $k_\lf$ for $\lf\in G^\dagger$ fixed, we can view $\Hc$ as a tensor depending on $(k,k_1,k_{\widetilde{\Fc}^\dagger})$, and view $h^{(j,\dagger)}$ as a tensor depending on $(k_j,k_{\Uc_j^{\dagger\dagger}})$. We define $h^{(j,\dagger\dagger)}=h_{kk_{\Uc_j^{\dagger\dagger}}}^{(j,\dagger\dagger)}$ to be $h_{kk_{\Uc_j^{\dagger}}}^{(j,\dagger)}$ with the values of $k_\lf$ for $\lf\in\Uc_j^\dagger\backslash \Uc_j^{\dagger\dagger}=G^\dagger\cap \Uc_j^\dagger$ fixed. By the same arguments as in \emph{Step 2}, we deduce that
\begin{equation}\label{reducedest3}\|\Hc\|_{kk_S\to k_1k_T}\lesssim\prod_i^{(\Os,3,<)}N_{\lf_i}^{(d/2)-2y_i(\alpha-\theta)}\prod_i^{(\Os,4,<)}N_{\lf_i}^{d-2y_i(\alpha-\theta)}\cdot\sup_{(k_\lf)}\|(\Hc^\dagger)_{kk_1k_{\widetilde{\Fc}^\dagger}}\|_{kk_{S^\dagger}\to k_1k_{T^\dagger}},
\end{equation} where the products $\prod_i^{(\Os,n,<)}$ and the supremum $\sup_{(k_\lf)}$ are defined in the same way as in \emph{Step 2} above, and
\begin{equation}\label{reducedmerge4}(\Hc^\dagger)_{kk_1k_{\widetilde{\Fc}^\dagger}}=\prod_{\lf,\lf'}^{(1',\dagger\dagger)}\mathbf{1}_{k_\lf=k_{\lf'}}\sum_{(k_2,\cdots,k_r)}h_{kk_1\cdots k_r}\sum_{k_{\Rc^\dagger}}^{(3',\dagger\dagger)}\prod_{\lf\in\Rc^\dagger}\Delta_{N_\lf}\gamma_{k_\lf}\prod_{j=2}^{r}\big[h_{k_jk_{\Uc_j^{\dagger\dagger}}}^{(j,\dagger\dagger)}\big]^{\zeta_j}.
\end{equation} In (\ref{reducedmerge4}) the product $\prod_{\lf,\lf'}^{(1',\dagger\dagger)}$ is the same as $\prod_{\lf,\lf'}^{(1',\dagger)}$ defined  in (\ref{reducedmerge3}), except that the product here is only taken over $\lf,\lf'\in\Ac_i$ for some $\Ac_i\in(\Os^{\dagger\dagger})'$; similarly, the sum $\sum_{k_{\Rc^\dagger}}^{(3',\dagger\dagger)}$ is the same as $\sum_{k_{\Rc}}^{(3',\dagger)}$ defined in (\ref{reducedmerge3}), except that the sum here does not involve the variables $k_{\Rc\backslash\Rc^\dagger}$. Note that $|\Ac_i|=2$ and $\Ac_i\cap \widetilde{\Fc}^\dagger=\varnothing$ for any $\Ac_i\in (\Os^{\dagger\dagger})'$, so the product $\prod_{\lf,\lf'}^{(1',\dagger\dagger)}$ vacuously equals $1$, and the condition of the sum $\sum_{k_{\Rc^\dagger}}^{(3',\dagger\dagger)}$ is just $k_\lf=k_{\lf'}$ for each pairing $\{\lf,\lf'\}\in(\Os^{\dagger\dagger})'$ (recall that $\Rc^\dagger = \Rc\backslash G^\dagger = \Qc\backslash (\Dc\cup G\cup G^\dagger)$).

Now recall the example in \emph{Step 2}. To estimate $\|(H^\circ)^\dagger\|_{kk_{B^\dagger}\to k_{C^\dagger}}$, we separate $h^{(1)}$ from the others as in (\ref{separate})--(\ref{reducedmerge3}), where  $\Dc = \{\cf_1, \cf_3\}$, $\Uc_1^\dagger=\{\lf_{\mathrm{top}}, \cf\}$, $\Fc =\widetilde{\Fc} =\{\bff, \cf, \dff, \gf, \hf\}$ and $\Rc = \Qc^{\dagger}\backslash \Dc=\{\af_2,\af_5,\ef_3,\ef_4,\ff_4,\ff_5\}$, namely
\[(H^\circ)_{kk_{\Uc^\dagger}}^\dagger=\sum_{k_1,k_\cf}h^{(1,\dagger)}_{k_1k_{\lf_{\mathrm{top}}}k_\cf}(H^*)_{kk_1k_\Fc},\]
\[(H^*)_{kk_1k_\Fc}=\sum_{(k_2,\cdots k_7)}\sum_{(k_\af,k_\ef,k_\ff)}h_{kk_1\cdots k_7}\cdot h_{k_2k_\af k_\bff}^{(2,\dagger)} \cdot h_{k_3k_\cf k_\dff k_\ef}^{(3,\dagger)}\cdot h_{k_4k_\ef k_\ff k_\gf}^{(4,\dagger)} \cdot h_{k_5k_\af k_\ff k_\hf}^{(5,\dagger)}\cdot h_{k_6}^{(6,\dagger)}\cdot h_{k_7}^{(7,\dagger)},\]
where for simplicity, we have again omitted the various $\Delta_{N_\lf}\gamma_{k_\lf}$ factors (and also the power factors below). Note that $H^* = \Hc = \Hc^{\dagger}$ as all frequencies are $\geq M^\delta$, in particular no partial trimming or \emph{Cases 3--4} in \emph{Step 3} above is involved. By (\ref{reducedest2}) we have
\[
\|(H^\circ)_{kk_{\Uc^\dagger}}^\dagger\|_{kk_\bff k_\dff\to k_{\lf_{\mathrm{top}}}k_\gf k_\hf}\leq\|h^{(1,\dagger)}_{k_1k_{\lf_{\mathrm{top}}}k_{\cf}}\|_{k_1k_\cf\to k_{\lf_{\mathrm{top}}}}\cdot
\|(H^*)_{kk_1k_\Fc}\|_{kk_\bff k_{\dff}\to k_1k_{\cf}k_{\gf}k_{\hf}}.
\] The norm $\|h^{(1,\dagger)}\|_{k_1k_\cf\to k_{\lf_{\mathrm{top}}}}$ is then controlled using (\ref{boundmerge1})--(\ref{boundmerge3}), and $\|H^*\|_{kk_\bff k_{\dff}\to k_1k_{\cf}k_{\gf}k_{\hf}}$ is controlled using Proposition \ref{algorithm1}; note that here $\Ec=\{6,7\}$ as in the example in \emph{Step 2} of the proof of Proposition \ref{algorithm1}, and the corresponding $H^{\mathrm{sg}}$ is the same as the one in that example, but with $h^{(j)}$ replaced by $h^{(j,\dagger)}$. After putting these two bounds together and calculating the various powers involved (see \emph{Step 4}), we can obtain the desired estimate (\ref{boundmerge7}) for this example.

\emph{Step 4: putting together.} We now need to estimate $\Hc^\dagger$. As in \emph{Step 3}, we may replace $\Delta_{N_\lf}\gamma_{k_\lf}$ by $N_\lf^{-\alpha+\theta}$ and absorb the resulting factor into one of the tensors $h^{(j,\dagger\dagger)}$ using Lemma \ref{adjustment}, so instead of $\Hc^\dagger$ we only need to consider
\begin{equation}\label{reducedmerge5}\Mc_{kk_1k_{\widetilde{\Fc}^\dagger}}:=\prod_{\lf\in\Rc^\dagger}N_\lf^{-\alpha+\theta}\sum_{(k_2,\cdots,k_r)}h_{kk_1\cdots k_r}\sum_{k_{\Rc^\dagger}}\prod_{j=2}^{r}\big[h_{k_jk_{\Uc_j^{\dagger\dagger}}}^{(j,\dagger\dagger)}\big]^{\zeta_j},
\end{equation} where $\sum_{k_{\Rc^\dagger}}$ is the sum such that $k_\lf=k_{\lf'}$ for each pairing $\{\lf,\lf'\}\in (\Os^{\dagger\dagger})'$  which is just $\sum_{k_{\Rc^\dagger}}^{(3',\dagger\dagger)}$ in (\ref{reducedmerge4}). We shall apply Proposition \ref{algorithm1} to estimate (\ref{reducedmerge5}), but we first need to make a few adjustments to fit the framework of (\ref{defH})--(\ref{outputbd}).

First, for $r+1\leq j\leq q$, we may define $\Uc_j^{\dagger\dagger}=\varnothing$ and $h^{(j,\dagger\dagger)}=h_{k_j}^{(j,\dagger\dagger)}$ be supported at a single point $k_j$ (the one fixed in \emph{Step 1}).  Moreover, in view of the extra factor  $\prod_{j=r+1}^q N_j^{-d/2}$ described in \emph{Step 1} when stating the norm bounds we want to prove for $H_{kk_\Uc}^\circ$, we may assume $|h_{k_j}^{(j,\dagger\dagger)}|\lesssim N_j^{-d/2}$ for $r+1\leq j\leq q$, so it satisfies the type $R0^+$ conditions\footnote{The support condition (\ref{type1+}) for type $R0^+$ can be immediately verified, since $\Uc_j^{\dagger\dagger} = \varnothing$.} of Proposition \ref{algorithm1} with $\Xf_j$ replaced by $1$.
Second, we also view $h_{kk_1\cdots k_r}$ as a tensor depending on $(k,k_1,\cdots,k_q$), in the support of which $k_j\,(r+1\leq j\leq q)$ takes a single value, and denote it by $h_{kk_1\cdots k_q}$. Moreover, the tensor $h=h_{kk_1\cdots k_q}$ satisfies the support condition (\ref{tensorsupport}) for some choice of $\Gamma$, as well as other conditions of $h$ listed in Proposition \ref{algorithm1}. 

Next we check that $h_{k_jk_{\Uc_j^{\dagger\dagger}}}^{(j,\dagger\dagger)}$ for $2\leq j\leq r$ satisfy the conditions in Proposition \ref{algorithm1}.
Note that $\Uc_j^{\dagger\dagger}$ is a subset of $\Uc_j$, and $h^{(j,\dagger\dagger)}$ is formed from $h^{(j)}$ by fixing the values of $k_{\Uc_j\backslash\Uc_j^{\dagger\dagger}}$ (for simplicity we denote it by $k_{G_j}$, where $G_j = \Uc_j\backslash\Uc_j^{\dagger\dagger}= (G\cap \Uc_j)\cup (G^\dagger\cap \Uc_j)$). We consider the following scenarios.

\emph{Scenario 1}: when $h^{(j)}$ has type 1. Then $h^{(j,\dagger\dagger)}$ will have type $R1$ in the sense of Proposition \ref{algorithm1}. Indeed by (\ref{boundmerge5}) (actually the modified version of it\footnote{In the analogous sense that (\ref{eq:modified6.40}) is the modified version of (\ref{boundmerge2}).} in \emph{Step 1}), we may assume
\[\bigg|k_j-\sum_{\lf\in\Uc_j}\zeta_\lf k_\lf-\ell_j\bigg|\leq (N_*)^{3\delta},\quad \ell_j=\sum_{\ff\in\Vc_j}\zeta_\ff k_\ff,\] otherwise we gain a huge power of $M$ from (\ref{boundmerge5}) that would suffice (note that $\kappa\gg_{C_\delta}1$ and $N_*\geq M^\delta$). When $k_{G_j}$ is fixed, the above implies that $h^{(j,\dagger\dagger)}$ satisfies (\ref{type2}) with $\Uc_j$ replaced by $\Uc_j^{\dagger\dagger}$ and some fixed $m_j= \ell_j + \sum_{\lf\in G_j}\zeta_\lf k_\lf$. Moreover, by (\ref{boundmerge2})--(\ref{boundmerge3}) we deduce that $h^{(j,\dagger\dagger)}$ satisfies (\ref{type2bd2}) with $\Uc_j$ replaced by $\Uc_j^{\dagger\dagger}$, and $\Xf_j$ replaced by
\[\Xf_j^*\cdot\prod_{\lf\in\Pc_j}N_\lf^{-4\varepsilon}\prod_{\lf\in G_j}N_\lf^{8\varepsilon}\prod_{\pf\in\Yc_j}N_\pf^{-\delta^3}\prod_{\ff\in\Vc_j}N_\ff^d.\]

\emph{Scenario 2}: when $h^{(j)}$ has type 0, and $G_j=\varnothing$. Then in this case, $h^{(j,\dagger\dagger)}=h^{(j)}$  and it will have type $R0$ in the sense of Proposition \ref{algorithm1}. It satisfies (\ref{type1bd1})--(\ref{type1bd2}) with $\Xf_j$ replaced by
\begin{equation}\label{newadd1}\Xf_j^*\cdot\prod_{\lf\in\Pc_j}N_\lf^{-4\varepsilon}\prod_{\pf\in\Yc_j}N_\pf^{-\delta^3}.\end{equation}

\emph{Scenario 3}: when $h^{(j)}$ has type 0, and $G_j\neq\varnothing$. Then $h^{(j,\dagger\dagger)}$ will have type $R0^+$ in the sense of Proposition \ref{algorithm1}. It satisfies the modified versions of (\ref{type1bd1})--(\ref{type1bd2}) with $\Uc_j$ replaced by $\Uc_j^{\dagger\dagger}$, and $\Xf_j$ replaced by
\begin{equation}\label{newadd2}\Xf_j^*\cdot\prod_{\lf\in\Pc_j}N_\lf^{-4\varepsilon}\prod_{\lf\in G_j}N_\lf^{8\varepsilon}\prod_{\pf\in\Yc_j}N_\pf^{-\delta^3}.\end{equation} In \emph{Scenarios 2--3}, note the different powers of $N_\lf$ for $\lf\in \Pc_j\cup G_j$ between (\ref{newadd1})--(\ref{newadd2}) and (\ref{boundmerge1}) for type 0 tensors, which allow us to bridge from $\Xc_{1,j}$ in (\ref{boundmerge1}) and (\ref{estimate02}) to $\Zc_{1,j}$ in (\ref{type1bd1})--(\ref{type1bd2}).

Therefore, by applying Proposition \ref{algorithm1}  with $P$, $Q$, $\Uc$ and $\Qc$ replaced by $S^\dagger$, $T^\dagger$, $\widetilde\Fc^\dagger$ and $\Rc^\dagger$ respectively, we obtain that
\begin{multline}\label{prop5.4:curlM}\prod_{j=r+1}^qN_j^{-d/2}\cdot\|\Mc\|_{kk_{S^\dagger}\to k_1k_{T^\dagger}}\lesssim\prod_{j=2}^r\Xf_j^*\cdot\prod_{\lf\in\Pc_2\cup\cdots\cup\Pc_r\cup\Rc^\dagger}N_\lf^{-4\varepsilon}\cdot\prod_{\lf\in G_2\cup\cdots\cup G_r}N_\lf^{8\varepsilon}\cdot\prod_{\pf\in\Yc_2\cup\cdots\cup\Yc_r}N_\lf^{-\delta^3}\\\times\prod_{\lf\in\Vc_2\cup\cdots\cup\Vc_r}N_\lf^{d}\cdot\prod_{\lf\in\widetilde{\Fc}^\dagger}N_\lf^\beta\cdot\prod_j ^{(0)}N_j^{-2\varepsilon}\cdot (N_*)^{-\varepsilon^3},
\end{multline} where $\prod_j^{(0)}$ is taken over all $j$ such that $r+1\leq j\leq q$, or $1\leq j\leq r$ and $h^{(j)}$ has type 0. 

This implies the same bound for $\Hc^\dagger$. Next we can control the norm \begin{equation}\label{prop5.4:h1norm}
\|h_{k_1k_{\Uc_1^\dagger}}^{(1,\dagger)}\|_{k_1k_{(B^\dagger\cup\Dc)\cap\Uc_1^\dagger}\to k_{C^\dagger\cap \Uc_1^\dagger}}
\end{equation} using  the modified versions of (\ref{boundmerge1})--(\ref{boundmerge3}) in \emph{Step 1}. In fact, if $h^{(1)}$ has type 1, the norm (\ref{prop5.4:h1norm}) is bounded as follows
\begin{equation}\label{prop5.4:h(1)1}
(\ref{prop5.4:h1norm})\lesssim \Xf_1^*\cdot\tau^{-\theta}\prod_{\lf\in  \Uc^\dagger_1}N_\lf^\beta\prod_{\lf\in\Pc_1}N_\lf^{-4\varepsilon}\prod_{\lf\in G_1}N_\lf^{8\varepsilon}\prod_{\pf\in\Yc_1}N_\pf^{-\delta^3} \prod_{\ff \in \Vc_1} N_\ff^d\cdot \Xc_1^{\textrm{type1}},
\end{equation}
by the modified versions of (\ref{boundmerge2}) and (\ref{boundmerge3}), where $\Xc_1^{\textrm{type1}}$ equals $(\max_{\lf\in C^\dagger\cap \Uc_1^\dagger}N_\lf)^{-\beta}$ if $C^\dagger\cap\Uc_1^\dagger \neq \varnothing$, and equals $N_1^{-\varepsilon}$ if $C^\dagger\cap\Uc_1^\dagger = \varnothing$. If $h^{(1)}$ has type 0, the norm (\ref{prop5.4:h1norm}) is bounded as follows
\begin{equation}\label{prop5.4:h(1)2}
(\ref{prop5.4:h1norm})\lesssim \Xf_1^*\cdot\tau^{-\theta}\prod_{\lf\in \Uc^\dagger_1}N_\lf^{\beta_1}\prod_{\lf\in\Pc_1}N_\lf^{-8\varepsilon}\prod_{\lf\in G_1}N_\lf^{4\varepsilon}
\prod_{\pf\in\Yc_1}N_\pf^{-\delta^3}\cdot\Xc_{0,1}\Xc_{1,1}
\end{equation} by (\ref{boundmerge1}), where $\Xc_{0,1}$ and $\Xc_{1,1}$ are as in (\ref{boundmerge1}) (which is taken from (\ref{estimate02})), corresponding to $B_1=(B^\dagger\cup\Dc)\cap\Uc_1^\dagger$ and $C_1=C^\dagger\cap\Uc_1^\dagger$ (and $E_1=G_1$).

Now, by plugging the bound for $\Hc^\dagger$ (which follows from (\ref{prop5.4:curlM})) back into (\ref{reducedest3}), (\ref{step3}), (\ref{reducedest2}) and (\ref{reducedest}), using the bounds for (\ref{prop5.4:h1norm}) described above, and exploiting the relations between the various sets defined before, we can eventually obtain the bound for $H^\circ$ that implies (\ref{prop5.4:5.45'}), and hence (\ref{boundmerge7}). In fact, by separating the two cases according to whether $\lf_{\mathrm{top}}\in \Uc_1^\dagger$ or not (recall that $\lf_{\mathrm{top}}\in C\cap \Uc_1$ is such that $N_{\lf_{\mathrm{top}}}\sim\max\{N_\lf:\lf\in C\}$), and by using (\ref{prop5.4:h(1)1}) if $h^{(1)}$ has type 1 or (\ref{prop5.4:h(1)2}) if\footnote{In fact in the cases we consider the bound in (\ref{prop5.4:h(1)2}) is always better than that in (\ref{prop5.4:h(1)1}) due to the $\Xc_{1,1}$ factor in (\ref{prop5.4:h(1)2}), so below we will always assume that we are using (\ref{prop5.4:h(1)1}).} $h^{(1)}$ has type 0, we can obtain in both two cases that
\begin{multline}\label{boundmergefinal}\prod_{j=r+1}^q N_j^{-d/2}\cdot\|H^\circ\|_{kk_{B}\to k_{C}}\lesssim(\tau^{-1}M)^\theta
 (N_*)^{-\varepsilon^3/2}
\prod_{j=1}^r\Xf_j^*\cdot
\prod_{(\lf\in B\cup C)\backslash\{\lf_{\mathrm{top}}\}}N_\lf^\beta
\cdot\prod_{\lf\in\Pc}N_\lf^{-4\varepsilon}\\\times\prod_{\lf\in E}N_\lf^{8\varepsilon}\cdot\prod_{\pf\in\widetilde{\Yc}}N_\pf^{-\delta^3}\prod_{\ff \in \Vc_1\cup\cdots\cup\Vc_r} N_\ff^d\cdot\prod_{j}^{(0)}N_j^{-2\varepsilon}\cdot\Zf,
 \end{multline} where $\prod_{j}^{(0)}$ is defined in (\ref{prop5.4:curlM}) and $\Zf$ is a quantity such that
 \begin{equation}\label{defZ}\Zf\leq\prod_i^{(\Os,2)}N_{\lf_i}^{-2y_i(\alpha-\theta)}\cdot\prod_i^{(\Os,3)}N_{\lf_i}^{(d/2)-2y_i(\alpha-\theta)}\cdot\prod_i^{(\Os,4)}N_{\lf_i}^{d-2y_i(\alpha-\theta)}\cdot\prod_{\lf}N_\lf^{50\varepsilon}\cdot\big(\max_\lf N_\lf\big)^{50\varepsilon}.\end{equation} In (\ref{defZ}), the product $\prod_i^{(\Os,n)}$, for $2\leq n\leq 4$, is the product of $\prod_i^{(\Os,n,\geq)}$ defined in \emph{Step 2} and $\prod_i^{(\Os,n,<)}$ defined in \emph{Step 3} (there is no $\prod_i^{(\Os,2,<)}$, which is replaced by 1). In the last two factors in (\ref{defZ}), the product is taken over all $\lf\in (G\cup G^\dagger)\cap\Qc$, and the maximum is taken over all $\lf\in G\cup G^\dagger$ involved in \emph{Cases 2--4} in \emph{Step 2} and \emph{Cases 3--4} in \emph{Step 3}.
 
 We make a few remarks regarding the calculations leading to (\ref{boundmergefinal})--(\ref{defZ}):
 \begin{enumerate}[(I)]
 \item The factor $\prod_{\pf\in\widetilde{\Yc}}N_\pf^{-\delta^3}$ in (\ref{boundmergefinal}) is obtained from $N_*^{-\varepsilon^3/2}$ and $\prod_{\lf\in\Yc_1\cup\cdots\cup\Yc_r}N_\lf^{-\delta^3}$ coming from (\ref{prop5.4:h(1)1}) and (\ref{prop5.4:curlM}), given that  $\widetilde{\Yc}$ arises from the merging process  as in Definition \ref{defmerge}.
 \item By combining all terms involving $N_\lf^{\beta}$ or $N_\lf^{-\alpha+\theta}$ in (\ref{prop5.4:h(1)1}), (\ref{prop5.4:curlM}), (\ref{step3}) and (\ref{reducedest2}), we obtain $\prod_{(\lf\in B\cup C)\backslash\{\lf_{\mathrm{top}}\}}N_\lf^\beta$ in (\ref{boundmergefinal}) with extra decays $\prod_{\lf\in ((B\cup C)\cap G)\backslash\{\lf_{\mathrm{top}}\}}N_\lf^{-\beta}$ and $\prod_{\lf\in \widetilde{\Fc}\cap G^\dagger}N_\lf^{-\beta}$ and $\prod_{\lf\in\Dc} N_\lf^{-(\alpha-\beta)+\theta}$.
 \item By comparing $\prod_{\lf\in E}N_\lf^{8\varepsilon}$ in (\ref{boundmergefinal}) with $\prod_{\lf\in G_1\cup\cdots\cup G_r}N_\lf^{8\varepsilon}$ coming from (\ref{prop5.4:h(1)1}) and (\ref{prop5.4:curlM}), we observe that the quotient between these two products arises from the $\lf$'s in $(G\cup G^\dagger)\backslash E$. This quotient, multiplied by the extra decays $\prod_{\lf\in ((B\cup C)\cap G)\backslash\{\lf_{\mathrm{top}}\}}N_\lf^{-\beta}$ and $\prod_{\lf\in \widetilde{\Fc}\cap G^\dagger}N_\lf^{-\beta}$ obtained in (II), accounts for one square root of the factor $\prod_{\lf}N_\lf^{50\varepsilon}\cdot\big(\max_\lf N_\lf\big)^{50\varepsilon}$ in (\ref{defZ}) (the other square root will be used in (IV)).
 \item Recall from Definition \ref{defmerge} that $\Pc\subset\widetilde{\Pc}= \Pc_1\cup\cdots\cup \Pc_r \cup \Qc$, also $\Rc^\dagger = \Qc\backslash (\Dc\cup G\cup G^\dagger)$. Considering
 $\prod_{\lf\in\Pc}N_\lf^{-4\varepsilon}$ in (\ref{boundmergefinal}), we know that $\prod_{\lf\in\Pc_1\cup\cdots\cup\Pc_r\cup\Rc^\dagger}N_\lf^{-4\varepsilon}$ coming from (\ref{prop5.4:h(1)1}) and (\ref{prop5.4:curlM}), multiplied by the extra decay $\prod_{\lf\in\Dc} N_\lf^{-(\alpha-\beta)+\theta}$ obtained in (II), can be bounded by $\prod_{\lf\in\Pc}N_\lf^{-4\varepsilon}$  multiplied by the other square root of $\prod_{\lf}N_\lf^{50\varepsilon}\cdot\big(\max_\lf N_\lf\big)^{50\varepsilon}$ in (\ref{defZ}).
 \end{enumerate}
 
 Now we have verified (\ref{boundmergefinal})--(\ref{defZ}). Since $y_i\geq 1$ in \emph{Cases 2--3} and $y_i\geq 2$ in \emph{Case 4} (this holds in both \emph{Step 2} and \emph{Step 3}), it is easy to verify that $\Zf$ is a product of negative powers of $N_{\lf_i}$ in (\ref{defZ}), thus $\Zf\leq 1$, hence (\ref{boundmergefinal}) implies the desired estimate (\ref{prop5.4:5.45'}) for $H_{kk_\Uc}^\circ$. This finishes the proof of (\ref{boundmerge7}).
 
 Finally we prove (\ref{boundmerge8}) ((\ref{boundmerge9}) follows from the same arguments as explained in \emph{Step 1}). When $C=\varnothing$, we know that $h^{(1)}$ has type 1 and $C^\dagger\cap \Uc_1^\dagger=\varnothing$, thus we only need to use the modified version of the bound (\ref{boundmerge3}) for $h^{(1)}$. The same arguments as above yield, instead of (\ref{boundmergefinal}), that
 \begin{multline}\label{boundmergefinal3}\prod_{j=r+1}^q N_j^{-d/2}\cdot\|H^\circ\|_{kk_{B}}\lesssim(\tau^{-1}M)^\theta (N_*)^{-\varepsilon^3/2}
\prod_{j=1}^r\Xf_j^*\cdot\prod_{\lf\in B}N_\lf^\beta\cdot\prod_{\lf\in\Pc}N_\lf^{-4\varepsilon}\\\times\prod_{\lf\in E}N_\lf^{8\varepsilon}\cdot\prod_{\pf\in\widetilde{\Yc}}N_\pf^{-\delta^3}\prod_{\ff \in \Vc_1\cup\cdots\cup\Vc_r} N_\ff^d\cdot \bigg(N_1^{-\varepsilon}\prod_{j}^{(0)}N_j^{-2\varepsilon}\bigg)\cdot\Zf.
 \end{multline} Therefore, it suffices to prove that
 \[\sqrt{\Upsilon}\cdot N_1^{-\varepsilon}\cdot \prod_{j}^{(0)}N_j^{-2\varepsilon}\lesssim M^{-\varepsilon},\] which easily follows from the definition of $\Upsilon$, the product $\prod_j^{(0)}$, and the assumption that $N_1\gtrsim N_j$ for all $2\leq j\leq r$ such that $h^{(j)}$ has type 1.
\end{proof}
\begin{prop}[Merged tensor bounds: Case II]\label{overpair2} Consider the same setting as Proposition \ref{overpair}, but assume $q=r$ and each $\Sc_j$ and $h^{(j)}$ have type 0. Moreover, assume that $N_\lf\geq M^\delta$ for each $\lf\in\Lc_j$ and $1\leq j\leq q$. Then, $H$ satisfies the bound 
\begin{equation}\label{boundmergenew1}\sqrt{\Upsilon}\cdot\|H\|_{kk_{B}\to k_{C}}\lesssim\Yf\cdot\tau^{-\theta}\prod_{\lf\in B\cup C} N_\lf^{\beta_1}\prod_{\lf\in\Pc}N_\lf^{-8\varepsilon}\prod_{\lf\in E}N_\lf^{4\varepsilon}\prod_{\pf\in\Yc}N_\pf^{-\delta^3}\cdot\Xc_{0}\Xc_{1},
\end{equation} where $\Upsilon$ is defined in (\ref{prop5.4:Upsilon}), $\Xc_0, \Xc_1$ are defined as in (\ref{estimate02}) but with $N$ replaced by $M$, and
\begin{equation}\label{boundmergenew2}\Yf=\prod_{j=1}^q\Xf_j\cdot \tau^{-\theta}M^{-\delta^3}.
\end{equation}
 Note that unlike Proposition \ref{overpair}, this bound (\ref{boundmergenew1}) is deterministic, i.e. we do not need to remove any exceptional set.
\end{prop}

\begin{proof}The proof will be similar to the proof of Proposition \ref{overpair}, so we mainly focus on the parts where the two proofs are different.

Since $q=r$ and each $\Sc_j$ and $h^{(j)}$ have type 0, we know that $\Vc =\varnothing$. Since also $N_\lf\geq M^\delta$ for each $\lf$, the possible trimming process in this proof will only affect the $\Yc$ set and we may omit it; in particular we will not need Propositions \ref{gausscont}, and can obtain (\ref{boundmergenew1}) deterministically. Next, it can be easily verified that \[\Upsilon^{1/4}\cdot\prod_{j=1}^q\Xc_{1,j}\lesssim \Xc_1\] where 
$\Xc_1$ and $\Xc_{1,j}$ are defined as in (\ref{estimate02}) relative to $\Sc$ and $\Sc_j$ respectively, so in the proof below we will replace $\sqrt{\Upsilon}$ by $\Upsilon^{1/4}$, and replace $\Xc_1$ and $\Xc_{1,j}$ for all $j$ by $1$.

By rearranging the tensors, we may assume that (i) if $C\neq\varnothing$, then $\max\{N_\lf:\lf\in C\cap\Uc_1\}\sim\max\{N_\lf:\lf\in C\}$, denote the particular $\lf\in C\cap\Uc_1$  where the maximum is attained by $\lf_{\mathrm{top}}$; (ii) if $C=\varnothing$ and $E\neq\varnothing$ then $E\cap\Uc_1\neq\varnothing$; and (iii) if $C=E=\varnothing$ then $\min\{N_\lf:\lf\in\Lc_1\}\sim\min\{N_\lf:\lf\in\Lc\}$, denote the particular $\lf\in C\cap\Lc_1$  where the minimum is attained by $\lf_{\mathrm{bot}}$. We shall repeat the same arguments in \emph{Steps 2--3} of the proof of Proposition \ref{overpair}. Namely, we first remove the over-pairings as in \emph{Step 2} (in particular, the \emph{Cases 1--4} are the same as in \emph{Step 2}, except that $\widetilde{\Uc}$ there is replaced by $\Uc$, and we do not require $i\leq v$ since there is no $v$ in the current case), then separate $h^{(1)}$ from the others as in \emph{Step 3} up to (\ref{reducedest2}). After these, we shall apply Proposition \ref{algorithm2} (instead of Proposition \ref{algorithm1} as in \emph{Step 4} of the proof of Proposition \ref{overpair}) to get that
\begin{multline}\label{Prop5.5:bound}\|H\|_{kk_B\to k_C}\lesssim\tau^{-\theta}\prod_{j=1}^q\Xf_j\cdot\prod_i^{(\Os,2)}N_{\lf_i}^{-2y_i(\alpha-\theta)}\prod_i^{(\Os,3)}N_{\lf_i}^{(d/2)-2y_i(\alpha-\theta)}\prod_i^{(\Os,4)}N_{\lf_i}^{d-2y_i(\alpha-\theta)}\\\times \prod_{\lf\in\Dc}N_\lf^{-\alpha+\theta}\prod_{\lf\in\Uc_1^\dagger}N_\lf^{\beta_1}\prod_{\lf\in\Fc}N_\lf^{\beta_1}\cdot\prod_{\lf\in \Pc_1\cup\cdots \Pc_q\cup\Rc}N_\lf^{-8\varepsilon}\prod_{\lf\in G_1\cup\cdots\cup G_q}N_\lf^{4\varepsilon}\prod_{\pf\in\Yc_1\cup\cdots\cup\Yc_q}N_\pf^{-\delta^3}\cdot\Xc\Xc'.
\end{multline}  Here in the above:
\begin{itemize}
\item The product $\prod_i^{(\Os,n)}$ for $2\leq n\leq 4$, as well as the parameters $y_i$ and $z_i$, are defined in the same way as in \emph{Step 2} of the proof of Proposition \ref{overpair}.
\item The set $G_j=G\cap\Uc_j$, where $G$ is the union of all groups in \emph{Cases 1--4}, and $\Uc_j^\dagger=\Uc_j\backslash G_j$.
\item The set $\Rc=\Qc^\dagger\backslash \Dc$, where $\Qc^\dagger$ is the union of all $\Ac_i$'s with $|\Ac_i|=2$ (equivalently $\Qc^\dagger=\Qc\backslash G$), and $\Dc$ is the union of all $\Ac_i$'s with $|\Ac_i|=2$ and $\Ac_i\cap\Uc_1^\dagger\neq\varnothing$. The set $\Fc=(\Uc^\dagger\cup\Dc)\backslash\Uc_1^\dagger$ where $\Uc^\dagger=\Uc\backslash G$.
\item The factor $\Xc$ is such that (i) if $C\neq\varnothing$ and $\lf_{\mathrm{top}}\in\Uc_1^\dagger$ then $\Xc=N_{\lf_{\mathrm{top}}}^{-\beta_1}$, (ii) if $C=G_1=\varnothing$ (which implies $E=\varnothing$) then $\Xc=(N_{\lf_{\mathrm{bot}}})^{(d/2)-\beta_1}$, (iii) if $\Uc_1^\dagger=\varnothing$ then $\Xc=N_1^{-\varepsilon\delta}$, (iv) in other cases $\Xc=1$.
\item The factor $\Xc'$ is such that (i) if $\Uc_j^\dagger\neq\varnothing$ for some $2\leq j\leq q$, then $\Xc'=(\max N_\lf)^{-4p\varepsilon}$ where the maximum is taken over all $\lf\in\Uc_j^\dagger$ for all $2\leq j\leq q$, (ii) if $\Uc_j^\dagger=\varnothing$ for all $2\leq j\leq q$ then $\Xc'=(N_2\cdots N_q)^{-\varepsilon\delta}$.
\end{itemize}
In the last two points above regarding $\Xc$ and $\Xc'$, the powers $N_{\lf_{\mathrm{top}}}^{-\beta_1}$ and $(N_{\lf_{\mathrm{bot}}})^{(d/2)-\beta_1}$ in cases (i) and (ii) for $\Xc$ are obtained from the first and second lines in the definition (\ref{estimate02}) of $\Xc_{0,j}$ (which is the factor appearing in (\ref{boundmerge1})) with $j=1$. 
The powers $N_1^{-\varepsilon\delta}$ in case (iii) for $\Xc$, and $(N_2\cdots N_q)^{-\varepsilon\delta}$ in case (ii) for $\Xc'$ are obtained from the third line in the definition (\ref{estimate02}) of $\Xc_{0,j}$, which is only used here and not needed in other parts of the proof. 
Finally, the power $(\max N_\lf)^{-4p\varepsilon}$ in case (i) for $\Xc'$ is obtained from the last term of (\ref{outputbd2}).

Now, by the same arguments as in the proof of Proposition \ref{overpair} (in particular using the quantity $\Zf$ in (\ref{defZ})), we can simplify the right hand side of (\ref{Prop5.5:bound}) as follows. When $C\neq\varnothing$ we obtain that
\begin{equation}\label{Prop5.5:bound2}\Upsilon^{1/4}\cdot\|H\|_{kk_B\to k_C}\lesssim\tau^{-\theta}\prod_{j=1}^q\Xf_j\cdot\prod_{\lf\in(B\cup C)\backslash\{\lf_{\mathrm{top}}\}}N_\lf^{\beta_1}\prod_{\lf\in\Pc}N_\lf^{-8\varepsilon}\prod_{\lf\in E}N_\lf^{4\varepsilon}\prod_{\pf\in\Yc}N_\pf^{-\delta^3}\cdot M^{\delta^3}\Xc',
\end{equation} where we notice by definition that $\Xc'\leq M^{-\varepsilon\delta^2}\leq M^{-4\delta^3}$. This easily implies (\ref{boundmergenew1}). When $C=\varnothing$ we obtain that
\begin{equation}\label{Prop5.5:bound3}\Upsilon^{1/4}\cdot\|H\|_{kk_B}\lesssim\tau^{-\theta}\prod_{j=1}^q\Xf_j\cdot\prod_{\lf\in B}N_\lf^{\beta_1}\prod_{\lf\in\Pc}N_\lf^{-8\varepsilon}\prod_{\lf\in E}N_\lf^{4\varepsilon}\prod_{\pf\in\Yc}N_\pf^{-\delta^3}\cdot M^{\delta^3}\Xc\Xc';
\end{equation} in this case, if $E=\varnothing$, then we have $\Xc\leq (N_{\lf_{\mathrm{bot}}})^{(d/2)-\beta_1}$, which is acceptable as $\Xc_0=(N_{\lf_{\mathrm{bot}}})^{(d/2)-\beta_1}$ in (\ref{boundmergenew1}), and the bound $\Xc'\leq M^{-\varepsilon\delta^2}$ provides the needed gain  as above; if $E\neq\varnothing$ and $B\neq\varnothing$, then $\Xc\leq 1$ and $\Xc_0=1$ in (\ref{boundmergenew1}), so once again we can use the bound of $\Xc'$. Finally, if $B=\varnothing$, then in particular $\Uc_j^\dagger=\varnothing$ for each $1\leq j\leq q$. In this case $\Xc=N_1^{-\varepsilon\delta}$, $\Xc'=(N_2\cdots N_q)^{-\varepsilon\delta}$ and $\Xc_0=M^{-\varepsilon\delta}$ in (\ref{boundmergenew1}), so it suffices to prove that
\[M^{\delta^3}\Upsilon^{1/4}\prod_{j=1}^q N_j^{-\varepsilon\delta}\leq M^{-\delta^3}M^{-\varepsilon\delta}.\] This is obviously true if $N_j\ll M$ for each $j$, due to the definition (\ref{prop5.4:Upsilon}) of $\Upsilon$; if $N_j\sim M$ for some $j$, this is also true since the other $N_{j'}$ still satisfy $N_{j'}^{-\varepsilon\delta}\leq M^{-\varepsilon\delta^2}$. This completes the proof.
\end{proof}
\begin{prop}[A special case: Operator bounds]\label{overpair3} Let $3\leq q\leq p$ be odd and $1\leq r\leq q$. For $2\leq j\leq r$, assume that $\Psi_{k_j}^{(j)}=\Psi_{k_j}[\Sc_j,h^{(j)}]$ as in (\ref{defpsi}), where $\Sc_j$ is a regular plant with $|\Sc_j|\leq D$ and $N(\Sc_j)=N_j\leq M$ such that $N_\lf\geq M^\delta$ for each $\lf\in\Lc_j$, $h^{(j)}$ is an $\Sc_j$-tensor that is $\Bc_{M^{[\delta]}}$ measurable. Moreover we assume that $h^{(j)}$ either has type 0 and satisfies the assumptions of Proposition \ref{overpair} (1), or has type 1 and satisfies the assumptions of Proposition \ref{overpair} (2).

We also fix $\zeta_j\,(1\leq j\leq q)$ and $N_j\leq M/2\,(r+1\leq j\leq q)$, and assume that $\max(N_2,\cdots,N_q)=M$. Let $h=h_{kk_1\cdots k_q}(\lambda_{r+1},\cdots,\lambda_q)$ be a constant tensor (which does not depend on $\omega$) supported in the set \begin{equation}\label{operatorbdnew1}
\langle k_j\rangle\leq N_j\,(2\leq j\leq q),\quad\langle\lambda_j\rangle\leq 2M^{\kappa^2}\,(r+1\leq j\leq q),\quad k=\sum_{j=1}^q\zeta_jk_j;
\end{equation} also assume that $h$ \emph{can be written as a function of\footnote{We may also need to multiply this $h$ by functions $\mathbf{1}_{\langle k\rangle\geq M^2}$ or $\mathbf{1}_{\langle k_1\rangle\geq M^2}$, but they do not affect Proposition \ref{gausscont2} (which can be easily checked), so the proof below will proceed in the same way.} $k-\zeta_1k_1$, $|k|^2-\zeta_1|k_1|^2$, and $(k_2,\cdots,k_q,\lambda_{r+1},\cdots,\lambda_q)$}, and satisfies that
\begin{equation}\label{operatorbdnew2}|h|+|\partial_{\lambda_j}h|\lesssim \frac{\tau^{-\theta}}{\langle\Omega+
\zeta_{r+1}\lambda_{r+1}+\cdots+\zeta_q\lambda_q+\widetilde{\Xi}\rangle},\quad r+1\leq j\leq q,
\end{equation}where $\Omega=|k|^2-\sum_{j=1}^q\zeta_j|k_j|^2$, $\widetilde{\Xi}\in\Rb$ is fixed with $|\widetilde{\Xi}|\lesssim M^{\kappa^2}$, and that any pairing in $(k,k_1,\cdots,k_q)$ must be over-paired.
Now define 
\begin{equation}\label{operatorbdnew3}\Ms_{kk_1}=\sum_{k_2,\cdots,k_q}\int\mathrm{d}\lambda_{r+1}\cdots\mathrm{d}\lambda_q\cdot h_{kk_1\cdots k_q}(\lambda_{r+1},\cdots,\lambda_q)\prod_{j=2}^r(\Psi_{k_j}^{(j)})^{\zeta_j}\prod_{j=r+1}^q(\widehat{z_{N_j}})_{k_j}^{\zeta_j}(\lambda_j),
\end{equation} possibly with Fourier truncations on $z_{N_j}$ as in part (5) of Proposition \ref{mainprop}, then $\tau^{-1}M$-certainly we have
\begin{equation}\label{operatorbdnew4}\|\Ms_{kk_1}\|_{k\to k_1}\leq\prod_{j=2}^r\Xf_j\cdot\tau^{-\theta}M^{-\varepsilon^4}.
\end{equation}
\end{prop}
\begin{proof} The proof will be similar to the proof of Proposition \ref{overpair}, so we mainly focus on the parts where the two proofs are different.

Let $\Sc_1$ be the empty plant (each component being empty). For any $\Os$ as in Definition \ref{defmerge}, define 
\begin{equation*}\Sc=(\Lc,\Vc,\Yc)=\mathtt{Merge}(\Sc_1,\cdots,\Sc_r,\Bs,\Os),\end{equation*} and define
 \begin{multline}\label{Prop5.6:defH}
\Hc_{kk_1k_{\Uc}} (k_{\Vc}, \lambda_{\Vc})=\prod_{\lf,\lf'}^{(1)}\mathbf{1}_{k_\lf=k_{\lf'}}\prod_{\lf,\lf'}^{(2)}\mathbf{1}_{k_\lf\neq k_{\lf'}}\cdot\sum_{(k_2,\cdots,k_{r})}h_{kk_1\cdots k_q}(\lambda_{r+1},\cdots,\lambda_q)\\\times\sum_{k_{\Qc}}^{(3)}\prod_{\lf\in\Qc}\Delta_{N_\lf}\gamma_{k_\lf}\prod_{j=2}^{r}\big[h_{k_jk_{\Uc_j}}^{(j)}(k_{\Vc_j}, \lambda_{\Vc_j})\big]^{\zeta_j}.
\end{multline} Here in (\ref{Prop5.6:defH}), the set $\Qc$ (as well as $\Uc$ etc. below) is defined from $\Os$ in the same way as in Definition \ref{defmerge}; the products $\prod_{\lf,\lf'}^{(1)}$ and $\prod_{\lf,\lf'}^{(2)}$, and the sum $\sum_{k_{\Qc}}^{(3)}$, are defined as in (\ref{merge1}). By the same proof as Proposition \ref{formulas} (2), we can write $\Ms_{kk_1}$ as a linear combination (for different choices of $\Os$) 
of
 \begin{equation}\label{Prop5.6:M}
 \Ns_{kk_1}=\sum_{k_\Uc,k_\Vc}\int\mathrm{d}\lambda_\Vc\cdot \Hc_{kk_1 k_\Uc}(k_\Vc,\lambda_\Vc)\cdot\prod_{\lf\in\Uc}(f_{N_\lf})_{k_\lf}^{\zeta_\lf^*}\prod_{\ff\in\Vc}(\widehat{z_{N_\ff}})_{k_\ff}^{\zeta_\ff^*}(\lambda_\ff),
 \end{equation} 
 where $\zeta_\lf^*$ and $\zeta_\ff^*$ are defined as in Definition \ref{defmerge} when merging $(\Sc_1,\cdots,\Sc_r)$ via $(\Bs,\Os)$. If we assume $\left||k|^2-\zeta_1|k_1|^2\right|> M^{\kappa^3}$, then by (\ref{operatorbdnew2}) and (\ref{operatorbdnew1}), we have 
\[
|h|+|\partial_{\lambda_j} h|\lesssim \tau^{-\theta} M^{-\kappa^3}, \quad r+1\leq j\leq q,
\]
which easily implies (\ref{operatorbdnew4}) thanks to the dominant decay $M^{-\kappa^3}$. Therefore, below we will focus on the estimate for $\Ns$ with a fixed $\Os$ as in Definition \ref{defmerge}, and assume that $\left||k|^2-\zeta_1|k_1|^2\right|\leq M^{\kappa^3}$ (which will allow us to apply Proposition \ref{gausscont2}).

First, notice that $\Vc=\Vc_2\cup\cdots\cup\Vc_r\cup\{r+1,\cdots,q\}$; define $H^*=(H^*)_{kk_1}(k_\Vc,\lambda_\Vc)$ so that we have
\begin{equation}\label{Prop5.6:defH*}(H^*)_{kk_1}(k_\Vc,\lambda_\Vc)=\sum_{k_\Uc}\Hc_{kk_1 k_\Uc}(k_\Vc,\lambda_\Vc)\cdot\prod_{\lf\in\Uc}(f_{N_\lf})_{k_\lf}^{\zeta_\lf^*},
\end{equation}
\begin{equation}\label{Prop5.6:forH*}\Ns_{kk_1}=\sum_{k_\Vc}\int\mathrm{d}\lambda_\Vc\cdot (H^*)_{kk_1}(k_\Vc,\lambda_\Vc)\cdot\prod_{\ff\in\Vc}(\widehat{z_{N_\ff}})_{k_\ff}^{\zeta_\ff^*}(\lambda_\ff).
\end{equation} By the same argument as in \emph{Step 1} of the proof of Proposition \ref{overpair}, using Cauchy-Schwartz, we obtain that
\begin{equation}\label{Prop5.6:reduce1}\|\Ns_{kk_1}\|_{k\to k_1}\lesssim\|H^*\|_{X_\Vc^{-b_0}[k\to k_1]}\cdot\prod_{\ff\in\Vc}N_\ff^{-D_1},
\end{equation}
 so it suffices to control $\|H^*\|_{X_\Vc^{-b_0}[k\to k_1]}$.  For $2\leq j\leq r$, define $\Xf_j^*=\Xf_j^*(k_{\Vc_j},\lambda_{\Vc_j})$ as in the proof of Proposition \ref{overpair}. Using the inequality
\begin{equation*}\bigg(\sum_{k_\Vc}\int\mathrm{d}\lambda_{\Vc}\cdot\prod_{\ff\in\Vc}\langle\lambda_\ff\rangle^{-2b_0}\prod_{j=2}^r\Xf_j^*(k_{\Vc_j},\lambda_{\Vc_j})^2\bigg)^{1/2}\lesssim\prod_{j=2}^r\Xf_j\prod_{j=r+1}^q N_j^{d/2},\end{equation*} which is proved in the same way as (\ref{powersave}), it suffices to prove $\tau^{-1}M$-certainly that
\begin{equation}\label{Prop5.6:reduce2}
\prod_{j=r+1}^qN_j^{-d/2}\cdot\|(H^*)_{kk_1}(k_\Vc,\lambda_\Vc)\|_{k\to k_1}\lesssim\prod_{j=2}^r\Xf_j^*(k_{\Vc_j},\lambda_{\Vc_j})\cdot\tau^{-\theta}M^{-\varepsilon^4}\prod_{\ff\in\Vc_2\cup\cdots\cup\Vc_r}N_\ff^d
\end{equation} for any choice of $(k_\Vc,\lambda_\Vc)$. By a meshing argument as in the proof of Proposition \ref{overpair}, we may fix a single choice of $(k_\Vc,\lambda_\Vc)$ and view $h^{(j)}=h_{k_jk_{\Uc_j}}^{(j)}$ as depending only on $(k_j,k_{\Uc_j})$, $h=h_{kk_1\cdots k_r}$ as depending only on $(k,k_1,\cdots,k_r)$, $\Hc=\Hc_{kk_1k_\Uc}$ as depending only on $(k,k_1,k_\Uc)$, and $(H^*)_{kk_1}$ as depending only on $(k,k_1)$. Also note the extra factor $\prod_{j=r+1}^qN_j^{-d/2}$ in the norm bound (\ref{Prop5.6:reduce2}) that we want to prove for $(H^*)_{kk_1}$, which is analogous to \emph{Step 1} of the proof of Proposition \ref{overpair}, and can be exploited in exactly the same way.

Now, using (\ref{Prop5.6:defH*}), noticing that $\Hc$ is $\Bc_{M^{[\delta]}}$ measurable and $N_\lf\geq M^\delta$ for $\lf\in\Uc$, that $h$ (and hence $\Hc$) depends on $(k,k_1)$ only via the quantities $k-\zeta_1k_1$ and $|k|^2-\zeta_1|k_1|^2$ and is supported in $||k|^2-\zeta_1|k_1|^2|\leq M^{\kappa^3}$, and that no pairing occurs in $k_\Uc$ in view of the product $\prod_{\lf,\lf'}^{(2)}$ in the definition (\ref{Prop5.6:defH}) of $\Hc$, we can apply Proposition \ref{gausscont2}  to (\ref{Prop5.6:defH*}) with $(b, c, A)= (k, k_1, \Uc)$ and obtain that $\tau^{-1}M$-certainly,
\begin{equation}\label{Prop5.6:reduce3}\|H^*\|_{k\to k_1}\lesssim(\tau^{-1}M)^\theta\prod_{\lf\in\Uc}N_\lf^{-\alpha+\theta}\cdot\max_{(S,T)}\|\Hc\|_{kk_S\to k_1k_T},
\end{equation} where $(S,T)$ is any partition of $\Uc$. Then, repeating the same arguments in \emph{Steps 3--4} of the proof of Proposition \ref{overpair} (namely, first removing the over-pairings as in \emph{Step 3} after (\ref{step3}), then applying Proposition \ref{algorithm1} as in \emph{Step 4}---note that here we do not have the set $E$), we obtain (after omitting factors that are $\leq 1$) that
\begin{equation}\label{Prop5.6:bdH3}
\prod_{j=r+1}^qN_j^{-d/2}\cdot\|\Hc\|_{kk_S\to k_1k_T}
\lesssim  \prod_{j=2}^{r} \Xf_j^* \cdot \prod_{\ff \in \Vc_2\cup\cdots\cup\Vc_r} N_\ff^d \cdot\prod_{\lf\in \Uc}N_\lf^\beta\cdot M^{-2\varepsilon^4}.
\end{equation}
Plugging (\ref{Prop5.6:bdH3}) into (\ref{Prop5.6:reduce3}) we get (\ref{Prop5.6:reduce2}), as desired.
\end{proof}
\section{Proof of Proposition \ref{mainprop}}\label{mainproof1} In this section we apply Propositions \ref{algorithm1}--\ref{overpair3} to complete the inductive proof of Proposition \ref{mainprop}. Namely, assuming $\mathtt{Local}(M)$, we shall prove that $\mathtt{Local}(2M)$ holds $\tau^{-1}M$-certainly. Recall the choice of $M$ and the basic assumptions and facts listed in the beginning of Section \ref{mainproof}.
\subsection{The operator $\Vs^M$} We start by obtaining suitable bounds for the operator $\Vs^M$ (as well as $\Rs^M=\Vs^M+1$), which will follow from the corresponding bounds for $\Ls^M$, which in turn follow from the bounds for $\Ls^\zeta$ in part (5) of $\mathtt{Local}(M)$ in Proposition \ref{mainprop}.
\begin{prop}\label{linextra} Assume $\mathtt{Local}(M)$ is true. Then $\|\Ls^{M,\zeta}\|_{X^{1-b}\to X^{b}}\leq\tau^{(6\kappa)^{-1}}$ for $\zeta\in\{\pm\}$, so in particular $\Rs^M=(1-\Ls^M)^{-1}$ and $\|\Vs^{M,\zeta}\|_{X^{1-b}\to X^b}\leq \tau^{(7\kappa)^{-1}}$. For the kernel $\Ls^{M,\zeta}$ we also have
\begin{equation}\label{extrabd2}\int_{\Rb}\langle\lambda\rangle^{2(1-b)}\|\langle\lambda'\rangle^{-(1-b)}(\widehat{\Ls^{M,\zeta}})_{kk'}(\lambda,\lambda')\|_{\ell_{k'}^2L_{\lambda'}^2\to \ell_k^2}^2\,\mathrm{d}\lambda\leq\tau^{(3\kappa)^{-1}},
\end{equation} and the same bound (\ref{extrabd2}) holds also for the kernel $\Vs^{M,\zeta}$, with the power $(3\kappa)^{-1}$ on the right hand side replaced by $(7\kappa/2)^{-1}$. Moreover for any $\widetilde{N}$ we have
\begin{equation}\label{opervmbd2}\|\mathbf{1}_{\langle k'\rangle\leq \widetilde{N}}(\Vs^{M,\zeta})_{kk'}\|_{X^{b,-(1-b)}[kk']}\leq \tau^{(8\kappa)^{-1}}(\widetilde{N})^{\alpha_0}\cdot M^{C\delta},
\end{equation}
\begin{equation}\label{opervmbd3}\|(1+M^{-\delta}|k-\zeta k'|)^{\kappa^2}(\Vs^{M,\zeta})_{kk'}\|_{X^{b,-(1-b)}[k k']}\leq \tau^{(8\kappa)^{-1}}M^{\beta_1-\varepsilon}.
\end{equation}
\end{prop}
\begin{proof} \emph{Step 1: bounds for $\Ls^{M,\zeta}$.} For $L<M^\delta$, define the operators
\begin{equation}\label{deflml}(\Ls^{ML}w)_k(t)=-i\sum_{3\leq q\leq p}a_{pq}(m_M^*)^{(p-q)/2} \chi_\tau(t)\cdot\Ic_\chi\Pi_M\sum_{\mathrm{sym}}\Mc_q(w,v_{L}^\dagger,\cdots,v_{L}^\dagger)_k(t)
\end{equation} and $\widetilde{\Ls}^{ML}=\Ls^{ML}-\Ls^{M,L/2}$, then by definitions (\ref{deflml}) and (\ref{deflin}) we have 
\[\Ls^{M,\zeta}=\Ls^{M,M^{[\delta]},\zeta}=\sum_{L<M^\delta}\widetilde{\Ls}^{ML,\zeta},\] 
where the corresponding operators with $\zeta$ are defined as in Section \ref{notations}.
Moreover, each $\widetilde{\Ls}^{ML,\zeta}$  can be written as a superposition of at most $(\log L)^{C_\theta}$ operators of the form $\Ls^\zeta$, defined by (\ref{deflinoper}) with this fixed $\zeta$, where $N$ is replaced by $M$ and $\max(N_2,\cdots,N_{q})=L$. Therefore, to bound $\|\Ls^{M,\zeta}\|_{X^{1-b}\to X^{b}}$, it suffices to control the same norm for $\Ls^\zeta$ with a gain of a power of $L$.

Let the kernel of $\Ls^\zeta$ be $(\Ls^\zeta)_{kk'}(t,t')$, with Fourier transform $(\widehat{\Ls^\zeta})_{kk'}(\lambda,\lambda')$, then by (\ref{kernelfourier}),
\[(\widehat{\Ls^\zeta w})_k(\lambda)=\sum_{k'}\int\mathrm{d}\lambda'\cdot (\widehat{\Ls^\zeta})_{kk'}(\lambda,-\zeta\lambda')(\widehat{w})_{k'}(\lambda').\] For any $w$ with $\|w\|_{X^{b}}=1$, we can estimate
\[
\begin{split}\|\Ls^\zeta w\|_{X^{1-b}}^2&\leq\int_\Rb\langle \lambda\rangle^{2(1-b)}\,\mathrm{d}\lambda\cdot\bigg(\int_{\Rb}\|(\widehat{\Ls^\zeta})_{kk'}(\lambda,-\zeta\lambda')\|_{k\to k'}\cdot\|(\widehat{w})_{k'}(\lambda')\|_{\ell_{k'}^2}\,\mathrm{d}\lambda'\bigg)^2\\
&\leq\int_{\Rb^2}\langle \lambda\rangle^{2(1-b)}\langle\lambda'\rangle^{-2b}\|(\widehat{\Ls^\zeta})_{kk'}(\lambda,-\zeta\lambda')\|_{k\to k'}^2\,\mathrm{d}\lambda\mathrm{d}\lambda'\cdot\|w\|_{X^b}^2,
\end{split}\] so we have by (\ref{linbd}) that \begin{equation}\label{prop7.1:interpolate1}\|\Ls^\zeta\|_{X^b\to X^{1-b}}\leq\|\Ls^\zeta\|_{X^{1-b,-b}[k\to k']}\leq\tau^{(5\kappa)^{-1}}L^{-4\varepsilon\delta}.\end{equation} On the other hand, noticing that the $X^0$ and $X^1$ norms can be viewed as Sobolev $L^2$ and $H^1$ norms in the $t$ variable (for $\ell_k^2$ valued functions), and using the elementary inequalities leading to
\begin{equation}\label{elementary}\|\chi_\tau(t)\cdot\Ic_\chi v(t)\|_{H_t^1}\lesssim\|v(t)\|_{L_t^2}\end{equation} for both scalar and vector valued functions, we can deduce that
\begin{equation}\label{prop7.1:interpolate2}
\begin{split}\|\Ls^\zeta\|_{X^0\to X^1}&\lesssim\sup\big\{\|\Mc_q(y_{N_2}^*,\cdots,w,\cdots,y_{N_{q}}^*)\|_{X^0}:\|w\|_{X^0}=1\big\}\\
&\lesssim\prod_{j=2}^{q}\|(\widehat{y_{N_j}^*})_{k_j}(\lambda_j)\|_{\ell_{k_j}^1L_{\lambda_j}^1}\lesssim\tau^{-\theta}L^{dp},
\end{split}
\end{equation} where we have used $\mathtt{Local}(M)$ parts (3) and (5) to control the norms of $y_{N_j}^*$. Interpolating (\ref{prop7.1:interpolate1}) and (\ref{prop7.1:interpolate2}) gives $\|\Ls^\zeta\|_{X^{1-b}\to X^{b}}\lesssim\tau^{(6\kappa)^{-1}}L^{-3\varepsilon\delta}$, which implies the desired bound for $\Ls^{M,\zeta}$. In particular we also get that $\Rs^M=(1-\Ls^M)^{-1}$ and the corresponding bound for $\Vs^{M,\zeta}$. Note that the estimates for $\Ls^{\zeta}$ in this step \emph{do not require} that $L<M^\delta$.

\emph{Step 2: more bounds for $\Ls^\zeta$.} We need to prove (\ref{extrabd2}) for $\Ls^{M,\zeta}$. Clearly we may replace $\Ls^{M,\zeta}$ by $\Ls^{\zeta}$, provided we can prove (\ref{extrabd2}) with right hand side replaced by $\tau^{(3\kappa)^{-1}}L^{-6\varepsilon\delta}$. For the purpose of (\ref{opervmbd2})--(\ref{opervmbd3}) we will also prove an additional bound, which holds \emph{assuming} $L<M^\delta$, namely: \begin{equation}\label{extrabd0}\|\mathbf{1}_{\langle k'\rangle\leq \widetilde{N}}(\Ls^\zeta)_{kk'}\|_{X^{b,-(1-b)}[kk']}\leq \tau^{(6\kappa)^{-1}}\min(\widetilde{N},M)^{\alpha_0}M^{C\delta}.
\end{equation}

The proof of (\ref{extrabd0}) is straightforward. In fact, since $N_j\leq L<M^\delta$ for $2\leq j\leq q$, we may expand the functions $y_{N_j}^*$ using their Fourier transforms which satisfy the $\ell_{k_j}^1L_{\lambda_j}^1$ bounds as in \emph{Step 1}, and then reduce to fixed values of $k_j$ and $\lambda_j$ at a loss of $M^{C\delta}$. Using also Lemma \ref{localization} (note that $(\Ls^\zeta)_{kk'}(0,t')\equiv 0$), we may get rid of the $\chi_\tau$ factor at the price of replacing the $X^{b,-(1-b)}[kk']$ norm by the slightly larger $X^{b^+,-(1-b)}[kk']$ norm. With these reductions, let the resulting operator be $\Ls^{*,\zeta}$, then by Lemma \ref{duhamelform0} we have
\begin{equation}\label{auxkerest}|(\widehat{\Ls^{*,\zeta}})_{kk'}(\lambda,\lambda')|\lesssim\mathbf{1}_{\langle k\rangle\leq M}\mathbf{1}_{\langle k'\rangle\leq  \widetilde{N}}\mathbf{1}_{k-\zeta k'=k^*}\bigg(\frac{1}{\langle\lambda\rangle^3}+\frac{1}{\langle\lambda\pm(\lambda'\pm\Omega\pm\lambda^*)\rangle^3}\bigg)\frac{1}{\langle\lambda'\pm\Omega\pm\lambda^*\rangle},\end{equation} where $k^*$ and $\lambda^*$ are fixed, and $\Omega=|k|^2-\zeta|k'|^2$. The bound (\ref{extrabd0}) for $\Ls^{*,\zeta}$ then follows from (\ref{auxkerest}), elementary integral bounds and the fact that the number of choices for $(k,k')$ with $\langle k\rangle\leq M$ and $\langle k'\rangle\leq \widetilde{N}$, and the values of $k-\zeta k'=k^*$ and $|k|^2-\zeta |k'|^2=\Omega$ fixed, under the simplicity assumption\footnote{See Definition \ref{def:simple}. Here simplicity implies that, if $\zeta=+$ and $k=k'$, then $k$ must also equal some other $k_j$, which has already been fixed.}, is $\lesssim\min(\widetilde{N},M)^{d-1} \leq \min(\widetilde{N},M)^{2\alpha_0}$.

Now we prove (\ref{extrabd2}) for $\Ls^{\zeta}$ with right hand side replaced by $\tau^{(3\kappa)^{-1}}L^{-6\varepsilon\delta}$. Since
\[\|\langle\lambda'\rangle^{-b}(\widehat{\Ls^\zeta})_{kk'}(\lambda,\lambda')\|_{\ell_{k'}^2L_{\lambda'}^2\to \ell_k^2}^2\leq\int_\Rb\langle\lambda'\rangle^{-2b}\|(\widehat{\Ls^\zeta})_{kk'}(\lambda,\lambda')\|_{\ell_{k'}^2\to \ell_k^2}^2\,\mathrm{d}\lambda'\] and thanks to (\ref{linbd}), we know that the desired bound is true with $\langle\lambda'\rangle^{-(1-b)}$ replaced by $\langle\lambda'\rangle^{-b}$, and the right hand side replaced by $\tau^{(5\kappa/2)^{-1}}L^{-8\varepsilon\delta}$, in (\ref{extrabd2}). By interpolation, it then suffices to verify that
\begin{equation}\label{extrabd4}\int_\Rb\langle \lambda\rangle^{2(1-b)}\|(\widehat{\Ls^\zeta})_{kk'}(\lambda,\lambda')\|_{\ell_{k'}^2L_{\lambda'}^2\to \ell_k^2}^2\,\mathrm{d}\lambda\leq\tau^{-\theta}L^{2dp}.
\end{equation} Here we first use Lemma \ref{localization} (again since $(\Ls^\zeta)_{kk'}(0,t')\equiv 0$) to get rid of the $\chi_\tau$ factor in $\Ls^{\zeta}$, then apply the same arguments as in (\ref{prop7.1:interpolate2}) in \emph{Step 1} to control the spacetime $L^2$ norm of $\Mc_q(y_{N_2}^*,\cdots,w,\cdots,y_{N_{q}}^*)$, and reduce (\ref{extrabd4}) to the bound (viewing $w$ as an $\ell_k^2$ valued function)
\[|\Fc_t\Ic_\chi w(\lambda)|\lesssim\langle\lambda\rangle^{-1}\|w(\lambda')\|_{L_{\lambda'}^2},\] which easily follows from Lemma \ref{duhamelform0}.

\emph{Step 3: bounds for $\Vs^M$.} Now we can prove (\ref{extrabd2})--(\ref{opervmbd3}) for $\Vs^{M,\zeta}$. First look at (\ref{opervmbd2})--(\ref{opervmbd3}); note that $\Vs^M=\Ls^M+\Ls^M\Vs^M$, so in terms of kernels we have
\[(\widehat{\Vs^{M,\zeta}})_{kk'}(\lambda,\lambda')=(\widehat{\Ls^{M,\zeta}})_{kk'}(\lambda,\lambda')+\sum_{\substack{\iota_1,\iota_2\in\{\pm\}\\\iota_1\iota_2=\zeta}}\sum_{m}\int\mathrm{d}\mu\cdot(\widehat{\Ls^{M,\iota_1}})_{km}(\lambda,\mu)(\widehat{\Vs^{M,\iota_2}})_{m k'}^{\iota_2}(-\mu,\lambda').\] We may multiply by the truncation $\mathbf{1}_{\langle k'\rangle\leq \widetilde{N}}$ on both sides; then, by fixing $(k',\lambda')$ and applying the kernel $(\widehat{\Ls^{M,\iota_1}})_{km}(\lambda,\mu)$ to $(\widehat{\Vs^{M,\iota_2}})_{m k'}^{\iota_2}(-\mu,\lambda')$ as a function of $(m, \mu)$, we obtain that
\begin{multline}\sum_{\zeta\in\{\pm\}}\|\mathbf{1}_{\langle k'\rangle\leq \widetilde{N}}(\Vs^{M,\zeta})_{kk'}\|_{X^{b,-(1-b)}[kk']}\leq \sum_{\zeta\in\{\pm\}}\|\mathbf{1}_{\langle k'\rangle\leq \widetilde{N}}(\Ls^{M,\zeta})_{kk'}\|_{X^{b,-(1-b)}[kk']}\\+\sum_{\iota_1\in\{\pm\}}\|\Ls^{M,\iota_1}\|_{X^b\to X^b}\cdot\sum_{\iota_2\in\{\pm\}}\|\mathbf{1}_{\langle k'\rangle\leq  \widetilde{N}}\Vs_{kk'}^{M,\iota_2}\|_{X^{b,-(1-b)}[kk']}.\end{multline} Using (\ref{extrabd0}) for $\Ls^{M,\zeta}$, and the estimates for $\Ls^{M,\zeta}$ obtained in \emph{Step 1}, we get (\ref{opervmbd2}). The proof for (\ref{opervmbd3}) is similar, where we use Lemma \ref{weightedbd} (with $\kappa_1=\kappa^2$) to control the weighted norm of $\Vs^M$, noticing that $(\widehat{\Ls^{M,\iota_1}})_{km}(\lambda,\mu)$ is supported in $|k-\iota_1m|\lesssim M^\delta$.

Finally we prove (\ref{extrabd2}) for $\Vs^{M,\zeta}$. Since $\Vs^M=\Ls^M+\Vs^M\Ls^M$, similar to the above argument we can write \[(\widehat{\Vs^{M,\zeta}})_{kk'}(\lambda,\lambda')=(\widehat{\Ls^{M,\zeta}})_{kk'}(\lambda,\lambda')+\sum_{\substack{\iota_1,\iota_2\in\{\pm\}\\\iota_1\iota_2=\zeta}}\sum_{m}\int\mathrm{d}\mu\cdot(\widehat{\Vs^{M,\iota_1}})_{km}(\lambda,\mu)(\widehat{\Ls^{M,\iota_2}})_{m k'}^{\iota_2}(-\mu,\lambda').\] This implies, for any fixed $\lambda$, that
\begin{multline}\sum_{\zeta\in\{\pm\}}\|\langle\lambda'\rangle^{-(1-b)}(\widehat{\Vs^{M,\zeta}})_{kk'}(\lambda,\lambda')\|_{\ell_{k'}^2L_{\lambda'}^2\to \ell_k^2}\leq\sum_{\zeta\in\{\pm\}}\|\langle\lambda'\rangle^{-(1-b)}(\widehat{\Ls^{M,\zeta}})_{kk'}(\lambda,\lambda')\|_{\ell_{k'}^2L_{\lambda'}^2\to \ell_k^2}\\+\sum_{\iota_2\in\{\pm\}}\|\Ls^{M,\iota_2}\|_{X^{1-b}\to X^{1-b}}\cdot\sum_{\iota_1\in\{\pm\}}\|\langle\mu\rangle^{-(1-b)}(\widehat{\Vs^{M,\iota_1}})_{km}(\lambda,\mu)\|_{\ell_{m}^2L_{\mu}^2\to \ell_k^2}.
\end{multline} Using (\ref{extrabd2}) for $\Ls^{M,\zeta}$, and the estimates for $\Ls^{M,\zeta}$ obtained in \emph{Step 1}, we get (\ref{extrabd2}) for $\Vs^{M,\zeta}$.
\end{proof}
\subsection{The $h^{(\Sc,0)}$ tensors} In this section we prove part (1) of $\mathtt{Local}(2M)$.
\begin{prop}\label{induct1} Assume $\mathtt{Local}(M)$ is true. Then part (1) of $\mathtt{Local}(2M)$ is true. More precisely, $h^{(\Sc,0)}$ satisfies (\ref{support01}) and (\ref{estimate01}), for each plain regular plant $\Sc$ with $N(\Sc)=M$ and $|\Sc|\leq D$.
\end{prop}
\begin{proof} We induct in $|\Sc|$, using the inductive definition (\ref{deftensor1}). For the first term on the right hand of (\ref{deftensor1}), which corresponds to the mini-plant $\Sc=\Sc_M^+$, the desired bounds are obvious, so we just need to consider the second term, which is a multilinear expression of the input tensors $h^{(\Sc_j,0)}$. By induction hypothesis, each input tensor satisfies (\ref{support01}) and (\ref{estimate01}) associated with $\Sc_j$. Recall also that $\Vc_j=\Vc=\varnothing$ for $1\leq j\leq q$ when considering $h^{(\Sc,0)}$ tensors.

First, to prove (\ref{support01}) for $h^{(\Sc,0)}$, we notice that the sign of $\lf\in\Uc$ in $\Sc$ is given by $\zeta_\lf^*=\zeta_j\zeta_\lf$ where $\lf\in\Uc_j$ and $\zeta_\lf$ is the sign of $\lf$ in $\Sc_j$ (see Definition \ref{defmerge}). In the support of $\Sc$ we have
\begin{equation}\label{cencelnew}\sum_{\lf\in\Uc}\zeta_\lf^* k_\lf=\sum_{\lf\in\Wc}\zeta_\lf^* k_\lf=\sum_{j=1}^q\zeta_j\sum_{\lf\in\Uc_j}\zeta_\lf k_\lf=\sum_{j=1}^q\zeta_jk_j=k,\end{equation} where $\Wc=\Uc_1\cup\cdots\cup\Uc_q$, using the induction hypothesis, and the definition (\ref{basetensor}) of the tensor $h$ used in the merging process. Now let us prove (\ref{estimate01}) for $h^{(\Sc,0)}$.

\emph{Step 1: first reductions.} By definition of $\sum_{(a)}$ in (\ref{deftensor1}), we know $N_\lf\geq M^\delta$ for each $\lf\in\Lc_j$ and $1\leq j\leq q$, so we can omit the trimmings in (\ref{newH}), as they involve only the $\Yc$ sets which do not appear in the tensors. Applying Lemma \ref{localization}, we may remove the localization factor $\chi_\tau(t)$ on the right hand side of (\ref{deftensor1}) and gain a power $\tau^{8\kappa^{-1}}$ (which would overwhelm any possible $\tau^{-\theta}$ loss), provided we estimate this expression without $\chi_\tau$ in the stronger norm with the power $\langle\lambda\rangle^{2b}$ in (\ref{estimate01}) replaced by $\langle\lambda\rangle^{2b^+}$. By abusing notation we will still use $h^{(\Sc,0)}$ to denote the expression after these reductions; moreover, since $\sum_{\mathrm{sym}}$ and $\sum_{(a)}$ involve at most $(\log M)^{\kappa}$ terms, we may focus on one single term in the discussion below.

With the above reductions, and applying also Lemma \ref{duhamelform0} and the definition (\ref{basetensor}) of $h$, we can then take the Fourier transform in time and obtain that
\begin{equation}\label{fourierexpand}
\widehat{h^{(\Sc,0)}}(\lambda)=\int\mathrm{d}\lambda_1\cdots\mathrm{d}\lambda_q\cdot H(\lambda,\lambda_1,\cdots,\lambda_q),
\end{equation} where $H(\lambda,\lambda_1,\cdots,\lambda_q)=[H(\lambda,\lambda_1,\cdots,\lambda_q)]_{kk_\Uc}$ is a tensor (with $(\lambda,\lambda_1,\cdots,\lambda_q)$ being parameters) defined by
\begin{equation}\label{fourierexpand2}H(\lambda,\lambda_1,\cdots,\lambda_q)=\Upsilon\cdot\mathtt{Merge}(\widehat{h^{(\Sc_1,0)}}(\lambda_1),\cdots,\widehat{h^{(\Sc_q,0)}}(\lambda_q),\widetilde{h},\Bs,\Os).
\end{equation} In the above formula $(\Upsilon,\Bs,\Os)$ are as in Section \ref{constructansatz}, $\widetilde{h}=[\widetilde{h}(\lambda,\lambda_1,\cdots,\lambda_q)]_{kk_1\cdots k_q}$ is a function of $(k,k_1,\cdots,k_q)$ with parameters $(\lambda,\lambda_1,\cdots,\lambda_q)$, that is supported in the set $k=\sum_{j=1}^q\zeta_jk_j$, and satisfies the bound
\begin{equation}\label{boundbase1}|\widetilde{h}|\lesssim\frac{\tau^{-\theta}}{\langle\lambda\rangle\langle\lambda-\Omega-\zeta_1\lambda_1-\cdots-\zeta_q\lambda_q\rangle};\quad \Omega:=|k|^2-\sum_{j=1}^q\zeta_j|k_j|^2.
\end{equation} We shall separate two cases: the \emph{high modulation case} where $|\lambda|\geq M^{\sqrt{\kappa}}$, and the \emph{low modulation case} where $|\lambda|\leq M^{\sqrt{\kappa}}$.

\emph{Step 2: the high modulation case.} Assume $|\lambda|\geq M^{\sqrt{\kappa}}$. If we can estimate the norm in (\ref{estimate01}) for $h^{(\Sc,0)}$, but with $\langle\lambda\rangle^{2b^+}$ replaced by $\langle\lambda\rangle^{2}$, then this gain of power in $\lambda$ will overwhelm any possible loss coming from any summation of any $k_j$ and $k_\lf$ variables (the latter summation loses at most a power $M^{C\cdot D}$, while $\sqrt{\kappa}\gg D$). Because of this we can fix the values of $k$, $k_j$ and all $k_\lf$, and view $\widehat{h^{(\Sc,0)}}(\lambda)$ as a function of $\lambda$ only, and $\widehat{h^{(\Sc_j,0)}}(\lambda_j)$ as a function of $\lambda_j$ only; moreover by induction hypothesis (\ref{estimate01}) and H\"{o}lder, this function of $\lambda_j$ can be controlled in $L_{\lambda_j}^1$, so upon integrating in $\lambda_j$, we can also fix the value of $\lambda_j$, in which case $\widehat{h^{(\Sc,0)}}(\lambda)$ satisfies the bound
\[|\widehat{h^{(\Sc,0)}}(\lambda)|\lesssim\frac{\tau^{-\theta}}{\langle\lambda\rangle\langle\lambda-\Xi\rangle}\] due to (\ref{boundbase1}), where $\Xi$ is a fixed real number depending on the choices of $k$, $k_j$, $k_\lf$ and $\lambda_j$. Clearly this implies $\int_\Rb \langle \lambda\rangle^2|\widehat{h^{(\Sc,0)}}(\lambda)|^2\mathrm{d}\lambda\lesssim \tau^{-\theta}$ uniformly in all choices of the fixed parameters, so the desired estimate (\ref{estimate01}) follows.

\emph{Step 3: the low modulation case.} Assume $|\lambda|\leq M^{\sqrt{\kappa}}$. Recall the formula (\ref{fourierexpand}); we shall further decompose $h^{(\Sc,0)}$ into $h^{(\Sc,0,\Gamma)}$ and $h^{(\Sc_j,0)}$ into $h^{(\Sc_j,0,\Gamma_j)}$, where $\Gamma$ and $\Gamma_j$ are integers, as in part (1) of Proposition \ref{mainprop}. By induction hypothesis (\ref{estimate01}), if we define $\Xf_j=\Xf_j(\lambda_j,\Gamma_j)$ to be the smallest value such that $\widehat{h^{(\Sc_j,0,\Gamma_j)}}(\lambda_j)$ satisfies the type 0 bounds (\ref{supportmerge})--(\ref{boundmerge1}) in Proposition \ref{overpair}, then
\begin{equation}\label{integrability}\bigg(\sum_{\Gamma_j}\int_\Rb \Xf_j(\lambda_j,\Gamma_j)\,\mathrm{d}\lambda_j\bigg)^2\lesssim\int_\Rb\langle\lambda_j\rangle^{2b}\bigg(\sum_{\Gamma_j}\Xf(\lambda_j,\Gamma_j)\bigg)^{2}\,\mathrm{d}\lambda_j\lesssim 1.\end{equation} For fixed values of $(\lambda,\lambda_j)$ and $(\Gamma,\Gamma_j)$, if we replace $h^{(\Sc_j,0)}$ by $h^{(\Sc_j,0,\Gamma_j)}$ in (\ref{fourierexpand2}), restrict to the set $|k|^2-\sum_{\lf\in \Uc}\zeta_\lf^*|k_\lf|^2=\Gamma$, and denote the resulting tensor by $H^{(\Gamma,\Gamma_1,\cdots,\Gamma_q)}(\lambda,\lambda_1,\cdots,\lambda_q)$, then by similar arguments as in (\ref{cencelnew}) (but with $k_\lf$ replaced by $|k_\lf|^2$), we know that in this situation, we may further restrict $\widetilde{h}$ to the set $|k|^2-\sum_{j=1}^q\zeta_j|k_j|^2=\Gamma-\widetilde{\Gamma}$ in (\ref{fourierexpand2}), where $\widetilde{\Gamma}$ depends only on the fixed parameters $\Gamma_j$. Therefore we can apply Proposition \ref{overpair2}, also using (\ref{boundbase1}), to deduce that
\begin{multline}\label{atombd}\|H^{(\Gamma,\Gamma_1,\cdots,\Gamma_q)}(\lambda,\lambda_1,\cdots,\lambda_q)\|_{kk_B\to k_C}\\\lesssim\tau^{-\theta}M^{-\delta^3}\langle\lambda\rangle^{-1}\langle\lambda-\Gamma+\Xi\rangle^{-1}\prod_{j=1}^q\Xf_j(\lambda_j,\Gamma_j)\cdot\prod_{\lf\in B\cup C} N_\lf^{\beta_1}\prod_{\lf\in\Pc}N_\lf^{-8\varepsilon}\prod_{\lf\in E}N_\lf^{4\varepsilon}\prod_{\pf\in\Yc}N_\pf^{-\delta^3}\cdot\Xc_{0}\Xc_{1},
\end{multline} where $\Xc_0$ and $\Xc_1$ are as in (\ref{estimate02}), and $\Xi\in\Rb$ is a quantity depending only on the fixed parameters $\lambda_j$ and $\Gamma_j$. Note that here no meshing argument is required since Proposition \ref{overpair2} holds deterministically. Then, after summing in $\Gamma$, then taking the weighted $L^2$ norm in $\lambda$ within the set $|\lambda|\leq M^{\sqrt{\kappa}}$, then integrating in $\lambda_j$ and summing in $\Gamma_j$ using (\ref{integrability}), we deduce that
\[\int_\Rb\langle \lambda\rangle^{2b^+}\bigg(\sum_{\Gamma\in\Zb}\|\widehat{h^{(\Sc,0,\Gamma)}}(\lambda)\|_{kk_B\to k_C}\bigg)^2\,\mathrm{d}\lambda\lesssim\bigg(\tau^{-\theta}M^{-\delta^4}\prod_{\lf\in B\cup C}N_\lf^{\beta_1}\prod_{\lf\in\Pc}N_\lf^{-8\varepsilon}\prod_{\lf\in E}N_\lf^{4\varepsilon}\prod_{\pf\in\Yc}N_\pf^{-\delta^3}\cdot\Xc_0\Xc_1\bigg)^2,\] which implies (\ref{estimate01}) in view of Lemma \ref{localization}.
\end{proof}
\subsection{The $h^{(\Sc,1)}$ tensors} In this section we prove part (2) of $\mathtt{Local}(2M)$.
\begin{prop}\label{induct2} Assume $\mathtt{Local}(M)$ and part (1) of $\mathtt{Local}(2M)$ are true. Then $\tau^{-1}M$-certainly, part (2) of $\mathtt{Local}(2M)$ is true. More precisely, $h^{(\Sc,1)}$ satisfies (\ref{estimate11})--(\ref{estimate14}), for each regular plant $\Sc$ with $N(\Sc)=M$ and $|\Sc|\leq D$.
\end{prop}
\begin{proof} Again we proceed by induction on $|\Sc|$, using the inductive definition (\ref{deftensor2}). We first focus on the \emph{main case}, namely the estimates (\ref{estimate11})--(\ref{estimate13}) for the second line of (\ref{deftensor2}), assuming the second maximum of $N_j\,(1\leq j\leq q)$ is $\geq M^\delta$; then we will treat the remaining estimates. By induction hypothesis, if $n_j=0$ then $h^{(\Sc_j,n_j)}$ satisfies (\ref{support01}) and (\ref{estimate01}) associated with $\Sc_j$; if $n_j=1$ it satisfies the bounds (\ref{estimate11})--(\ref{estimate14}) associated with $\Sc_j$. In various steps below, we will abuse notation and refer to some components of $h^{(\Sc,1)}$ in (\ref{deftensor2}) still as $h^{(\Sc,1)}$ for simplicity.

\emph{Step 1: the main case.} We start with the second line of (\ref{deftensor2}), assuming the second maximum of $N_j\,(1\leq j\leq q)$ is $\geq M^\delta$. Here we will prove (\ref{estimate11})--(\ref{estimate13}) with all norms replaced by the stronger ones $X_\Vc^{b,-b_0}[\cdots]$. By Lemma \ref{localization}, we can get rid of the $\chi_\tau(t)$ localization with a gain of $\tau^{8\kappa^{-1}}$, as long as we estimate the expression without $\chi_\tau$ in the $X_\Vc^{b^+,-b_0}[\cdots]$ norms. Repeating the proof of Proposition \ref{induct1}, we can reduce to
\begin{equation}\label{prop7.3:hs1_1}
\widehat{h^{(\Sc,1)}}(\lambda)=\int\mathrm{d}\lambda_1\cdots\mathrm{d}\lambda_r\cdot H(\lambda,\lambda_1,\cdots,\lambda_r)
\end{equation} in the same way as (\ref{fourierexpand}), but instead of (\ref{fourierexpand2}) we have
\begin{equation}\label{newmergedef}H(\lambda,\lambda_1,\cdots,\lambda_r)=\Upsilon\cdot\mathtt{Trim}(\mathtt{Merge}(\mathtt{Trim}(\widehat{h^{(\Sc_1,n_1)}}(\lambda_1),M^\delta),\cdots,\mathtt{Trim}(\widehat{h^{(\Sc_r,n_r)}}(\lambda_r),M^\delta),\widetilde{h},\Bs,\Os),M^\delta),
\end{equation} where $\widetilde{h}=[\widetilde{h}(\lambda,\lambda_1,\cdots,\lambda_r)]_{kk_1\cdots k_q}(\lambda_{r+1},\cdots,\lambda_q)$ satisfies the same bound (\ref{boundbase1}), as does any $\lambda_j$ derivative of $\widetilde{h}$ for $r+1\leq j\leq q$.

Similar to the proof of Proposition \ref{induct1}, we shall consider two cases, the high modulation case where $\max(|\lambda|,|\lambda_1|,\cdots,|\lambda_r|)\geq M^{\sqrt{\kappa}}$, and the low modulation case where $\max(|\lambda|,|\lambda_1|,\cdots,|\lambda_r|)\leq M^{\sqrt{\kappa}}$. In the high-modulation case, we may again fix the values of $k$, $k_j$ and all $k_\lf$ and $k_\ff$; then $\widehat{h^{(\Sc_j,n_j)}}(\lambda_j)$ can be viewed as a function of $\lambda_j$ and $\lambda_{\Vc_j}$ only, and $\widehat{h^{(\Sc,1)}}(\lambda)$ can be viewed as a function of $\lambda$ and $\lambda_\Vc$ only, where recall $\Vc$ is associated to the plant $\Sc$ in (\ref{newpsi}). The trimming steps follow easily from Cauchy-Schwartz as in the proof of Proposition \ref{trimbd}, so we may omit them and consider only the merging step. We may take the $X_{\Vc_j}^{-b_0}$ norm for $\widehat{h^{(\Sc_j,n_j)}}(\lambda_j)$ and denote the result by $\hf_j(\lambda_j)$; similarly we may take the $X_\Vc^{-b_0}$ norm of $\widehat{h^{(\Sc,1)}}(\lambda)$ and denote it by $\hf(\lambda)$. Then each $\hf_j$ is bounded in a weighted $L^2$ space embedded in $L_{\lambda_j}^1$, so we may fix the value of $\lambda_j$ as in the proof of Proposition \ref{induct1} whenever wanted. Now if $|\lambda|=\max(|\lambda|,|\lambda_1|,\cdots,|\lambda_r|)$, then the same argument as in \emph{Step 2} of the proof of Proposition \ref{induct1} works and implies the desired bound (with significant decay) for $\int_\Rb \langle \lambda\rangle^{2b^+}|\hf(\lambda)|^2\mathrm{d}\lambda$ uniformly in all choices of the fixed parameters; instead, if (say) $|\lambda_1|$ is the maximum, then we can fix $\lambda_j\,(j\geq 2)$, also using the definition of $X_\Vc^{-b_0}$ norm, to get that
\[|\hf(\lambda)|^2\lesssim\tau^{-\theta}\int\bigg(\frac{1}{\langle\lambda\rangle}\int_\Rb\frac{1}{\langle\lambda\pm\lambda_1\pm\Xi\rangle}|\hf_1(\lambda_1)|\,\mathrm{d}\lambda_1\bigg)^2\cdot\prod_{j=r+1}^q\langle \lambda_j\rangle^{-2b_0}\,\mathrm{d}\lambda_{r+1}\cdots\mathrm{d}\lambda_q,\] where $\Xi$ is a real number depending on $(\lambda_{r+1},\cdots,\lambda_q)$, and the choices of the fixed variables. We can then fix $(\lambda_{r+1},\cdots,\lambda_q)$, estimate the $\lambda_1$ integral using Cauchy-Schwartz and the $L^2$ norm of $\hf_1(\lambda_1)$, save the $\langle\lambda_1\rangle^b$ weight to gain an $M^{\sqrt{\kappa}}$ power, and bound $\int_\Rb \langle \lambda\rangle^{2b^+}|\hf^*(\lambda)|^2\mathrm{d}\lambda$ (with significant decay) uniformly in $(\lambda_{r+1},\cdots,\lambda_q)$ and all choices of the fixed variables, where $\hf^*(\lambda)$ is the above integral in $\lambda_1$. This implies (\ref{estimate11})--(\ref{estimate12}).

As for (\ref{estimate13}), just notice that $\widetilde{h}$ is supported in $k=\sum_{j=1}^q\zeta_jk_j$, which implies that
\begin{equation}\label{maxweight}1+\frac{1}{M^{2\delta}}\bigg|k-\sum_{\lf\in\Uc}\zeta_\lf^* k_\lf-\ell\bigg|\lesssim\max_{1\leq j\leq r}\bigg(1+\frac{1}{M^{2\delta}}\bigg|k_j-\sum_{\lf\in\Uc_j}\zeta_\lf k_\lf-\ell_j\bigg|\bigg),\end{equation} where $\ell_j=\sum_{\ff\in\Vc_j}\zeta_\ff k_\ff$ and $\ell=\sum_{\ff\in\Vc}\zeta_\ff^* k_\ff$ (note that trimming at frequency $M^\delta$ or lower will not affect this inequality). This allows to control the weight in (\ref{estimate13}) for $h^{(\Sc,1)}$ by the weights in (\ref{estimate13}) for $h^{(\Sc_j,n_j)}$, so the $M^{\sqrt{\kappa}}$ power gain above also implies (\ref{estimate13}).

From now on we can restrict to the low modulation case. We first look at (\ref{estimate11})--(\ref{estimate12}); the proof of (\ref{estimate13}) requires slightly different arguments and is left to the end of this step. Recall the bounds (\ref{estimate11})--(\ref{estimate14}) for the norms $X_{\Vc_j}^{\widetilde{b},-b_0}[\cdots]$ where $1\leq j\leq r$ and $\widetilde{b}\in\{b,1-b\}$; since $|\lambda_j|\leq M^{\sqrt{\kappa}}$, we may replace $\widetilde{b}$ by $b$ in all these bounds, with a price of $M^{C/\sqrt{\kappa}}$ which in the end will be negligible as $\kappa\gg_{C_\delta}1$. Suppose we want to estimate the $X_\Vc^{b^+,-b_0}[kk_B\to k_C]$ norm of $h^{(\Sc,1)}$. If $C\neq\varnothing$, we shall rearrange the tensors such that $\max\{N_\lf:\lf\in C\cap \Uc_1\}\sim\max\{N_\lf:\lf\in C\}$; if $C=\varnothing$, we shall select all $1\leq j\leq q$ such that either $j>r$, or $n_j=1$ or $\min_{\lf\in\Lc_j}N_\lf<M^\delta$ (such $j$ exists by definition of $\sum_{(b)}$ in (\ref{deftensor2})), and by rearranging the tensors we may assume that the maximum of $N_j$ for such $j$ corresponds to\footnote{Here we have assumed $1\leq j\leq r$. If $r+1\leq j\leq q$, then this $j$ would correspond to the input function $z_{N_j}$ in (\ref{sums1}), which gains a big power $N_j^{-D_1}$. Thus the case where the maximum $N_j$ occurs at this $j$ will be strictly easier than the cases we actually treat in the proof.} $j=1$.

Next, for $1\leq j\leq r$, define $h^{(j)}=\mathtt{Trim}(h^{(\Sc_j,n_j)},R_j)$ where $R_1=M^\delta$ and $R_j=(N_*)^\delta$ for $j\geq 2$, with $N_*=\max(N_2,\cdots,N_q)$. Then, by the rearrangements we made above, it can be verified that the relevant parts (measurability, etc.) of the assumptions of Proposition \ref{overpair} are satisfied. Moreover, since $\partial_{\lambda_j}\widehat{h^{(\Sc_j,n_j)}}$ can be controlled using the fact that $h^{(\Sc_j,n_j)}$ is compactly supported in $t$, we can apply a meshing argument in $\lambda_j$ in the same way as in the proof of Proposition \ref{trimbd}. By combining the bounds (\ref{estimate11})--(\ref{estimate14}) for $h^{(\Sc_j,n_j)}$, Proposition \ref{trimbd}, and the above meshing argument in $\lambda_j$, we obtain that the following holds $\tau^{-1}M$-certainly:
\begin{enumerate}[(a)]
\item If $j$ is such that either $n_j=1$, or $n_j=0$ and $\min_{\lf\in\Lc_j}N_\lf<R_j$, define $\Xf_j=\Xf_j(\lambda_j)$ to be the smallest value such that the type 1 bounds (\ref{boundmerge2})--(\ref{boundmerge6}) in Proposition \ref{overpair} hold for $\widehat{h^{(j)}}(\lambda_j)$, then we have
\begin{equation}\label{integrability1}\bigg(\int_\Rb\Xf_j(\lambda_j)\,\mathrm{d}\lambda_j\bigg)^2\lesssim\int_\Rb\langle \lambda_j\rangle^{2b}\Xf_j(\lambda_j)^2\,\mathrm{d}\lambda_j\lesssim\big[\tau^{-\theta}M^{C/\sqrt{\kappa}}(1+R_j^CN_j^{-3\varepsilon})\big]^2.
\end{equation}
\item If $n_j=0$ and $\min_{\lf\in\Lc_j}N_\lf\geq R_j$ (in particular $h^{(j)}=h^{(\Sc_j,0)}$), we may decompose $h^{(j)}$ into $h^{(j,\Gamma_j)}=h^{(\Sc_j,0,\Gamma_j)}$ as in part (1) of Proposition \ref{mainprop}; define $\Xf_j=\Xf_j(\lambda_j,\Gamma_j)$ to be the smallest value such that the type 0 bounds (\ref{supportmerge})--(\ref{boundmerge1}) in Proposition \ref{overpair} hold for $\widehat{h^{(j,\Gamma_j)}}(\lambda_j)$, then we have (\ref{integrability}).
\end{enumerate} Moreover, since $\widetilde{h}$ satisfies (\ref{boundbase1}), we can apply a further meshing argument to replace it by some function (which we still denote by $\widetilde{h}$ for simplicity) that is supported in the big box $|\lambda|,|\lambda_j|\leq M^{\sqrt{\kappa}}$ and is constant on each small box of size (say) $\exp(-(\log M)^6)$. Let $\Xi:=\lambda-\Omega-\zeta_1\lambda_1-\cdots-\zeta_q\lambda_q$ (see (\ref{boundbase1})), we may also decompose $\widetilde{h}$ into $\widetilde{h}_{\Xi^*}$, which are restrictions of $\widetilde{h}$ to the set $\lfloor\Xi\rfloor=\Xi^*$ for $\Xi^*\in\Zb,|\Xi^*|\lesssim M^{2\sqrt{\kappa}}$.  Then $\widetilde{h}_{\Xi^*}$ satisfies the assumptions (\ref{mergesupp}) and (\ref{mergetensor}) in Proposition \ref{overpair}, with $\widetilde{\Gamma}$ in (\ref{mergesupp}) depending on $(\lambda,\lambda_1,\cdots,\lambda_r,\Xi^*)$, and the right hand side of (\ref{mergetensor}) multiplied by $\langle\lambda\rangle^{-1}\langle\Xi^*\rangle^{-1}$.

Let the tensor $H(\lambda,\lambda_1,\cdots,\lambda_r,\Gamma_j,\Xi^*)$ be defined as in (\ref{newmergedef}) but with $\widetilde{h}$ replaced by $\widetilde{h}_{\Xi^*}$, and $\mathtt{Trim}(\widehat{h^{(\Sc_j,n_j)}}(\lambda_j),M^\delta)$ replaced by $\widehat{h^{(j)}}(\lambda_j)$ or $\widehat{h^{(j,\Gamma_j)}}(\lambda_j)$ in case (a) or (b) above, and $\Os$ replaced by some $\Os^+$ containing $\Os$. By Proposition \ref{formulas} (3), $\widehat{h^{(\Sc,1)}}(\lambda)$ can be written as a linear combination of
\[\int\mathrm{d}\lambda_1\cdots\mathrm{d}\lambda_r\cdot\sum_{\Xi^*}\sum_{(\Gamma_j)}H(\lambda,\lambda_1,\cdots,\lambda_r,\Gamma_j,\Xi^*),\] where $\sum_{(\Gamma_j)}$ are present only for those $j$ in case (b) above. We now apply Proposition \ref{overpair} to conclude that the tensor $H(\lambda,\lambda_1,\cdots,\lambda_r,\Gamma_j,\Xi^*)$ satisfies (\ref{boundmerge7})--(\ref{boundmerge8}) with 
\begin{equation}\label{defnewy}\Yf=\sqrt{\Upsilon}\frac{1}{\langle\lambda\rangle\langle\Xi_*\rangle}\prod_{j=1}^r\Xf_j\cdot\tau^{-\theta}M^{\theta}(N_*)^{-\varepsilon^4},\end{equation} where $\Upsilon$ is as in Proposition \ref{overpair}, and $\Xf_j=\Xf_j(\lambda_j)$ or $\Xf_j(\lambda_j,\Gamma_j)$ in case (a) or (b) above; the meshing argument guarantees that the above holds for all values of $(\lambda,\lambda_j)$ after removing a single exceptional set of probability $\leq C_\theta e^{-(\tau^{-1}M)^\theta}$. Therefore, by taking the weighted $L^2$ norm in $\lambda$ within the set $|\lambda|\leq M^{\sqrt{\kappa}}$, then summing in $\Xi^*$ and $\Gamma_j$ and integrating in $\lambda_j$, we will obtain the bounds (\ref{estimate11})--(\ref{estimate12}) for this component of $h^{(\Sc,1)}$ under consideration, once we show that
\begin{equation}\label{auxterm}\sqrt{\Upsilon}\cdot\tau^{-\theta}M^{C/\sqrt{\kappa}}\prod_{j=1}^r(1+R_j^C N_j^{-3\varepsilon})\leq (N_*)^{\varepsilon^5}.\end{equation} But this is true since $R_j=(N_*)^{\delta}$ for $j\geq 2$, so any power $R_j^C$ for $j\geq 2$ will be negligible; moreover $R_1=M^\delta$, so either $N_*\gtrsim M$ and $R_1^C$ is also negligible, or $N_1\sim M$ and the $R_1^C$ loss is covered by the $N_1^{-3\varepsilon}$ gain, or $\max(N_1,N_*)\ll M$ and the $R_1^C$ loss is covered by the $\sqrt{\Upsilon}$ gain. Finally since $N_*\geq M^\delta$, the gain $(N_*)^{\varepsilon^4}$ in (\ref{defnewy}) will overwhelm the loss $M^{C/\sqrt{\kappa}}$ by our choice of $\kappa$. This proves (\ref{estimate11})--(\ref{estimate12}) in the main case.

Now we turn to the proof of (\ref{estimate13}). Starting with $H(\lambda,\lambda_1,\cdots,\lambda_r)$, we shall further decompose it by attaching smooth truncations supported in sets
\begin{equation}\label{maxweight2}1+\frac{1}{M^{2\delta}}\bigg|k-\sum_{\lf\in\Uc}\zeta_\lf^* k_\lf-\ell\bigg|\sim K,\quad 1+\frac{1}{M^{2\delta}}\bigg|k_j-\sum_{\lf\in\Uc_j}\zeta_\lf k_\lf-\ell_j\bigg|\sim K_j\,(1\leq j\leq r),\end{equation}
where $1\leq K,K_j\leq M$ are dyadic numbers, $\ell$ and $\ell_j$ as in (\ref{maxweight}). The number of terms in this decomposition is $\leq (\log M)^C$, so we only need to consider a single term. By (\ref{maxweight}) we know $K\lesssim\max(K_1,\cdots,K_r)$. By rearrangement we may assume $K\lesssim K_1$; in particular we may assume $n_1=1$ (otherwise $K\lesssim K_1\sim 1$, so (\ref{estimate13}) follows directly from (\ref{estimate12})). At this point we can repeat the arguments in the above proof of (\ref{estimate11})--(\ref{estimate12}) (namely trimming $h^{(\Sc_j,n_j)}$ at frequency $R_j$ with $R_1=M^\delta$ and $R_j=(N_*)^\delta\,(j\geq 2)$, decomposing $H(\lambda,\lambda_1,\cdots,\lambda_r)$ into $H(\lambda,\lambda_1,\cdots,\lambda_r,\Gamma_j,\Xi^*)$, defining $\Xf_j$ as above, etc.) and then apply Proposition \ref{overpair} to conclude that $H(\lambda,\lambda_1,\cdots,\lambda_r,\Gamma_j,\Xi^*)$ satisfies (\ref{boundmerge9}) with $\Yf$ defined as in (\ref{defnewy}). Here the assumptions of Proposition \ref{overpair} are satisfied, since $K\lesssim K_1$, and multiplying by any of the smooth truncations we introduced does not increase any of the $X_{\Vc_j}^{\widetilde{b},-b_0}[k_jk_{B_j}\to k_{C_j}]$ norms due to Lemma \ref{adjustment}. After obtaining (\ref{boundmerge9}), we can again repeat the arguments in the above proof (namely taking the weighted $L^2$ norm in $\lambda$ within the set $|\lambda|\leq M^{\sqrt{\kappa}}$, then summing in $\Xi^*$ and $\Gamma_j$ and integrating in $\lambda_j$) and deduce (\ref{estimate13}) for this component of $h^{(\Sc,1)}$ under consideration, using (\ref{auxterm}). This proves (\ref{estimate13}) and completes the main case.

\emph{Step 2: adding the $\Rb$-linear operators.} Now we prove the estimates (\ref{estimate11})--(\ref{estimate13}) for the second line of (\ref{deftensor2}) assuming the second maximum of $N_j\,(1\leq j\leq q)$ is $<M^\delta$ (and in particular the maximum of $N_j$ is $\leq N/2$), and the same estimates for the third line of (\ref{deftensor2}). First look at the second line of (\ref{deftensor2}); we may assume the maximum of $N_j$ is $N_1\geq M^\delta$ (and $L:=\max(N_2,\cdots,N_q)<M^\delta$), since the cases when the maximum of $N_j$ occurs at $r+1\leq j\leq q$, or when $N_1<M^\delta$ also, are much easier. With such assumptions we must have $\Os=\varnothing$, and the current term can be written as an $\Rb$-linear operator\footnote{When $\Sc$ is fixed, the sum in $(N_j,\Sc_j)$ etc. for $j\geq 2$ involve at most $(\log L)^{\kappa}$ terms, which is negligible in view of the $L^{\varepsilon\delta}$ gain we will obtain below. The sum in $(N_1,\Sc_1)$ involves at most $\kappa$ terms if $N_1\sim M$ and $\Sc_1=\Sc'$, and at most $(\log M)^{\kappa}$ terms otherwise; either way this is negligible in view of the gain from $\Upsilon$, and the gain of at least $M^{\delta^5}$ coming from trimming assuming $\Sc_1\neq\Sc'$, which is evident from the proof of Proposition \ref{trimbd}.} $\Ls^\zeta$ (for some $\zeta\in\{\pm\}$, see (\ref{deflinoper}) for definition, where $N$ should be replaced by $M$) applied to the tensor $h':=\mathtt{Trim}(h^{(\Sc_1,n_1)},M^\delta)$. More precisely, $\Sc':=\mathtt{Trim}(\Sc_1,M^\delta)$ has the same sets of leaves, pairings, blossoms and pasts as $\Sc$ (in particular the sets $\Lc,\Uc,\Vc$ etc. are common for both tensors), the only differences being that $N(\Sc')=N_1$ while $N(\Sc)=M$, and the sign $\zeta_\nf^*$ of $\nf\in\Lc\cup\Vc$ in $\Sc$ equals the sign $\zeta_\nf$ of $\nf$ in $\Sc'$ multiplied by $\zeta$; for the tensors we have
\begin{equation}\label{tensoroperator}(\widehat{h^{(\Sc,1)}})_{kk_\Uc}(\lambda,k_\Vc,\lambda_\Vc)=\Upsilon\cdot\sum_{k'}\int_\Rb\mathrm{d}\lambda'\cdot(\widehat{\Ls^\zeta})_{kk'}(\lambda,-\zeta\lambda')(\widehat{h'})_{k'k_\Uc}^\zeta(\lambda',k_\Vc,\lambda_\Vc).
\end{equation} 
To estimate the $X^{\widetilde{b},-b_0}[kk_B\to k_C]$ norms of $h^{(\Sc,1)}$ (including the weighted ones in (\ref{estimate13})), where $(B,C)$ is a subpartition of $\Uc$ and we denote $E:=\Uc\backslash(B\cup C)$, we will consider four cases.

(a) Assume $C\neq\varnothing$, then for any fixed $\lambda$, $(k_\Vc,\lambda_\Vc)$ and $k_E$, by (\ref{tensoroperator}) and a variant of Proposition \ref{bilineartensor} we have
\[\|(\widehat{h^{(\Sc,1)}})_{kk_\Uc}\|_{kk_B\to k_C}\leq\Upsilon\cdot\|\langle\lambda'\rangle^{1-b}(\widehat{h'})_{k'k_\Uc}(\lambda')\|_{\ell_{k'k_B}^2L_{\lambda'}^2\to \ell_{k_C}^2}\cdot\|\langle\lambda'\rangle^{-(1-b)}(\widehat{\Ls^\zeta})_{kk'}(\lambda,-\zeta\lambda')\|_{\ell_{k'}^2L_{\lambda'}^2\to\ell_k^2}.\] The third factor above is a function of $\lambda$ only, and we shall temporarily denote it by $G(\lambda)$; the second factor is bounded by\[\big\|\langle \lambda'\rangle^{1-b}\cdot\|(\widehat{h'})_{k'k_\Uc}(\lambda')\|_{k'k_B\to k_C}\big\|_{L_{\lambda'}^2},\] so upon taking supremum in $k_E$, and then taking the weighted $L^2$ norm in $\lambda$ and the weighted $\ell^2L^2$ norm in $(k_\Vc,\lambda_\Vc)$, we obtain that
\[\|h^{(\Sc,1)}\|_{X_\Vc^{1-b,-b_0}[kk_B\to k_C]}\leq \Upsilon\cdot\|h'\|_{X_\Vc^{1-b,-b_0}[k'k_B\to k_C]}\cdot\|\langle\lambda\rangle^{1-b}G(\lambda)\|_{L_\lambda^2}.\] By the proof of Proposition \ref{linextra}, the weighted norm of $G$ above is bounded by $\tau^{(6\kappa)^{-1}}L^{-3\varepsilon\delta}$, hence
\[\|h^{(\Sc,1)}\|_{X_\Vc^{1-b,-b_0}[kk_B\to k_C]}\lesssim\tau^{(6\kappa)^{-1}}L^{-3\varepsilon\delta}\cdot\Upsilon\cdot\|h'\|_{X_\Vc^{1-b,-b_0}[k'k_B\to k_C]}.\] By using the induction hypothesis for $h^{(\Sc_1,n_1)}$, Proposition \ref{trimbd}, and controlling the potential loss factor $1+M^{C\delta}N_1^{-3\varepsilon}$ occurring in Proposition \ref{trimbd} by the $\Upsilon$ factor, we deduce (\ref{estimate11}) for $h^{(\Sc,1)}$. Note that, if $\Sc_1=\Sc'$ then the above estimate has no loss. If $\Sc_1\neq\Sc'$, then applying Proposition \ref{trimbd} loses a factor $M^\theta$, but by examining the proof of Proposition \ref{trimbd} we see that we can also gain a small power of $N_1$ (which is $\leq N_1^{-\delta^4}$). This, in view of the $\Upsilon$ factor, is enough to cover this loss together with the potential log loss coming from summing over all plants; moreover the continuous variable $\lambda'$ can be handled by restricting to $|\lambda'|\leq M^{\kappa^2}$ and performing another meshing argument exploiting the above power gain. The same comment also applies in the other cases below.

(b) Assume $C=\varnothing$ and $E\neq\varnothing$, then the same argument as in case (a) is enough to control the norm $\|h^{(\Sc,1)}\|_{X_\Vc^{1-b,-b_0}[kk_B]}$; note that to prove (\ref{estimate12}) for $h^{(\Sc,1)}$ we need to gain a power $M^{-\varepsilon}$, which is provided by the corresponding power $N_1^{-\varepsilon}$ (or the better powers and the $\Xc_{1,1}$ factor in (\ref{estimate01}) corresponding to $\Sc_1$) from the induction hypothesis if $N_1\sim M$, and by the $\Upsilon$ factor if $N_1\ll M$. The weighted norm in (\ref{estimate13}) is bounded in the same way using the induction hypothesis, Proposition \ref{trimbd}, (the proof of) Proposition \ref{linextra}, and (a variant of) Lemma \ref{weightedbd}, using the fact that $(\Ls^\zeta)_{kk'}$ is supported in $|k-\zeta k'|\lesssim M^\delta$.

(c) Assume $C=E=\varnothing$, and either $n_1=1$, or $\min_{\lf\in\Lc_1}N_\lf<M^\delta$, then the norm in question, namely $X_\Vc^{b,-b_0}[kk_\Uc]$, is a weighted $\ell^2L^2$ norm in $(k_\Vc,\lambda_\Vc)$, while for fixed $(k_\Vc,\lambda_\Vc)$ it is an $\ell^2L^2$ norm in $(k,k_\Uc,\lambda)$ weighted by $\langle \lambda\rangle^b$. Therefore, by (\ref{tensoroperator}) we have
\[\|h^{(\Sc,1)}\|_{X_\Vc^{b,-b_0}[kk_\Uc]}\leq\Upsilon\cdot\|h'\|_{X_\Vc^{1-b,-b_0}[k'k_\Uc]}\cdot\|\Ls^\zeta\|_{X^{1-b}\to X^b}.\] Using the induction hypothesis, Proposition \ref{trimbd} and the bound $\|\Ls^\zeta\|_{X^{1-b}\to X^b}\leq\tau^{(6\kappa)^{-1}}L^{-3\varepsilon\delta}$, which also follows from the proof of Proposition \ref{linextra}, we can prove the desired estimates in the same way as in parts (a) and (b).

(d) Assume $C=E=\varnothing$, $n_1=0$, and that $\min_{\lf\in\Lc_1}N_\lf\geq M^\delta$ (in particular $\Vc=\varnothing$ and $\Lc=\Lc_1$). In this case we will use a different estimate. Still starting with (\ref{tensoroperator}), we have
\[\|h^{(\Sc,1)}\|_{X^b[kk_\Uc]}\leq\Upsilon\cdot\|(\Ls^{\zeta})_{kk'}\|_{X^{b,-(1-b)}[kk']}\cdot\|h^{(\Sc_1,0)}\|_{X^{1-b}[k'\to k_\Uc]}.\] Now let $\widetilde{N}=\max_{\lf\in\Uc}N_\lf$, then in the above formula we may assume $\langle k'\rangle\lesssim \widetilde{N}$. By the proof of Proposition \ref{linextra} we have
\[\|\mathbf{1}_{\langle k'\rangle\lesssim\widetilde{N}}\cdot(\Ls^{\zeta})_{kk'}\|_{X^{b,-(1-b)}[kk']}\lesssim\tau^{(6\kappa)^{-1}}(\widetilde{N})^{\alpha_0}M^{C\delta};\] combining with the induction hypothesis (in particular summing over the $\Gamma$ variable in (\ref{estimate01})) we obtain that
\begin{equation}\label{auxbd01}\|h^{(\Sc,1)}\|_{X^b[kk_\Uc]}\lesssim\Upsilon\cdot\tau^{(6\kappa)^{-1}}(\widetilde{N})^{\alpha_0}M^{C\delta}\cdot\prod_{\lf\in\Uc}N_\lf^{\beta_1}\prod_{\lf\in\Pc} N_\lf^{-8\varepsilon}\prod_{\pf\in\Yc}N_\pf^{-\delta^3}\cdot(\widetilde{N})^{-\beta_1}\cdot\Xc_{1,1},\end{equation} where $\Xc_{1,1}$ is the one in (\ref{estimate02}) corresponding to $\Sc_1$. To prove (\ref{estimate12}) we just need to gain an extra $M^{-2\varepsilon}$ power, which is provided either by the difference in the powers of $N_\lf$ for $\lf\in\Lc$ between (\ref{estimate12}) and (\ref{auxbd01}), or by the $\Xc_{1,1}$ and $\Upsilon$ factors. The proof of (\ref{estimate13}) is the same, as the weight in (\ref{estimate13}) is in fact bounded by $1$ using the induction hypothesis and the support condition for $(\Ls^\zeta)_{kk'}$.

Next we look at the third line of (\ref{deftensor2}). For simplicity, we will consider the case $\zeta=+$; the case $\zeta=-$ is analogous, with $\Sc$ replaced by $\overline{\Sc}$ in a few places. This term can be written as
\begin{equation}\label{tensoroperator2}(\widehat{h^{(\Sc,1)}})_{kk_\Uc}(\lambda,k_\Vc,\lambda_\Vc)=\sum_{k'}\int_\Rb\mathrm{d}\lambda'\cdot(\widehat{\Vs^{M,+}})_{kk'}(\lambda,-\lambda')(\widehat{\Hc})_{k'k_\Uc}(\lambda',k_\Vc,\lambda_\Vc).
\end{equation} 
The term (\ref{tensoroperator2}) is similar to (\ref{tensoroperator}), except that $\Ls^{+}$ is replaced by $\Vs^{M,+}$, and $h'$ is replaced by $\mathcal{H}$, which is either the second term on the right hand side of (\ref{deftensor1}) (if we take $\sum_{(a)}$ instead of $\sum_{(c)}$ in the third line of (\ref{deftensor2})), or the second line of (\ref{deftensor2}) (if we take $\sum_{(b)}$ instead of $\sum_{(c)}$). By what we proved in Proposition \ref{induct1} and the above arguments, we know that $\Hc$ is an $\Sc$-tensor, which either satisfies (\ref{support01}) and (\ref{estimate01}), or satisfies the bounds (\ref{estimate11})--(\ref{estimate13}). Moreover for the $\Rb$-linear operator $\Vs^{M,+}$ by Proposition \ref{linextra} we have
\begin{equation}\label{vmbds}
\begin{aligned}\|\Vs^{M,+}\|_{X^{1-b}\to X^b}&\leq\tau^{(7\kappa)^{-1}},\\
\big\|\langle\lambda\rangle^{1-b}\|\langle\lambda'\rangle^{-(1-b)}(\widehat{\Vs^{M,+}})_{kk'}(\lambda,\lambda')\|_{\ell_{k'}^2L_{\lambda'}^2\to\ell_k^2}\big\|_{L_\lambda^2}&\leq \tau^{(7\kappa)^{-1}},\\
\|\mathbf{1}_{\langle k'\rangle\leq \widetilde{N}}\Vs^{M,+}\|_{X^{b,-(1-b)}[kk']}&\leq \tau^{(8\kappa)^{-1}}(\widetilde{N})^{\alpha_0}\cdot M^{C\delta},\\
\|(1+M^{-\delta}|k- k'|)^{\kappa^2}\Vs^{M,+}\|_{X^{b,-(1-b)}[k k']}&\leq \tau^{(8\kappa)^{-1}}M^{\beta_1-\varepsilon},
\end{aligned}
\end{equation} which are similar to the bounds for $\Ls^+$ we have used above. The estimate for the third line of (\ref{deftensor2}) can then be deduced by considering cases (a)--(d), in the same way as above. We only mention a few important points: (i) going from $\Hc$ to $h^{(\Sc,1)}$ does not involve trimming, so in these proofs no meshing argument is needed, hence they do not require the derivative bound (\ref{estimate14}) for $\Hc$ (which has not been proved yet); (ii) the $\Upsilon$ factor is not needed, because $\Hc$ is an $\Sc$-tensor, therefore the $N_1$ in the above proof will be replaced by $M$; (iii) the fact that $\Ls^+$ is supported in $|k- k'|\leq M^\delta$ is replaced by the last bound in (\ref{vmbds}), which leads to the restriction $|k- k'|\leq M^{2\delta}$, since otherwise we gain a sufficiently high power of $M$. This allows us to apply Lemma \ref{weightedbd} in estimating the weighted norms in (\ref{estimate13}) in cases (b) and (c), and also bound the weight in (\ref{estimate13}) in case (d).

\emph{Step 3: remaining estimates.} Next we shall prove (\ref{estimate11})--(\ref{estimate13}) for the first line of (\ref{deftensor2}). In fact, this term can be treated in the same way as the third line of (\ref{deftensor2}), see \emph{Step 2} above, where the only difference is that $\Hc$ is replaced by the $\Sc_M^+$-tensor occurring as the first term of the right hand side of (\ref{deftensor1}). Since this tensor also satisfies (\ref{support01}) and (\ref{estimate01}) by Proposition \ref{induct1}, the same arguments as in \emph{Step 2} above suffice to estimate this term.

Finally we prove the derivative bound (\ref{estimate14}) for all terms in (\ref{deftensor2}). This is a very loose bound, so it can be proved by very loose estimates. Just notice that:
\begin{itemize}
\item The $\Rb$-linear operator $\Vs^{M,\zeta}$ commutes with $\partial_{\lambda_\ff}$ and increases the norm in consideration by at most a constant multiple;
\item If $h'=\mathtt{Trim}(h,R)$, where $h$ is an $\Sc$-tensor and $h'$ is an $\Sc'$-tensor, then $\partial_{\lambda_\ff}h'=\mathtt{Trim}(\partial_{\lambda_\ff}h,R)$ for any $\ff\in\Vc'$;
\item If $H=\mathtt{Merge}(h^{(1)},\cdots,h^{(r)},h,\Bs,\Os)$, where $h^{(j)}$ is an $\Sc_j$-tensor and $H$ is an $\Sc$-tensor, then $\partial_{\lambda_\ff}H=\mathtt{Merge}(h^{(1)},\cdots,\partial_{\lambda_\ff}h^{(j)},\cdots,h^{(r)},h,\Bs,\Os)$ for any $\ff\in\Vc_j$; in the same way we also have $\partial_{\lambda_j}H=\mathtt{Merge}(h^{(1)},\cdots,h^{(r)},\partial_{\lambda_j}h,\Bs,\Os)$ for any $r+1\leq j\leq q$.
\end{itemize} Therefore, in order to estimate $\partial_{\lambda_\ff}h^{(\Sc,1)}$ for $\ff\in\Vc$, we only need to consider the same $\mathtt{Trim}$-$\mathtt{Merge}$ combination, where \emph{one} of the inputs $h^{(\Sc_j,n_j)}$ is replaced by $\partial_{\Vc_j}h^{(\Sc_j,n_j)}$. This input is then bounded by the induction hypothesis, noting that either $N_j\leq M/2$ or $|\Sc_j|<|\Sc|$, and the other inputs are bounded trivially (say using part (3) of $\mathtt{Local}(M)$) by a power $M^{C\cdot\kappa}$. Therefore we get, without removing any exceptional set, that
\[\|\partial_{\lambda_\Vc}h_{kk_\Uc}^{(\Sc,1)}(t,k_\Vc,\lambda_\Vc)\|_{X_\Vc^{b,-b_0}[kk_\Uc]}\lesssim M^{C\cdot\kappa}+M^{C\cdot\kappa}\exp[(\log \widetilde{N})^5+|\widetilde{\Sc}|(\log \widetilde{N})^3],\] where $\widetilde{N}\leq M$, $|\widetilde{\Sc}|\leq D$, and either $\widetilde{N}\leq M/2$ or $|\widetilde{\Sc}|<|\Sc|$. Therefore (\ref{estimate14}) follows from the bound \[\exp[(\log M)^5+|\Sc|(\log M)^3]\geq \exp((\log M)^2)\cdot\exp[(\log \widetilde{N})^5+|\widetilde{\Sc}|(\log \widetilde{N})^3]\] under these assumptions. This completes the proof.
\end{proof}
\subsection{The remaining parts} In this section we prove parts (3)--(5) of $\mathtt{Local}(2M)$.
\begin{prop}\label{induct3} Assume $\mathtt{Local}(M)$ and parts (1) and (2) of $\mathtt{Local}(2M)$ are true. Then $\tau^{-1}M$-certainly, parts (3)--(5) of $\mathtt{Local}(2M)$ are true. More precisely, (\ref{psibd}) is true for each $n\in\{0,1\}$ and regular plant $\Sc$ with $N(\Sc)=M$ and $|\Sc|\leq D$, and the mapping that defines $z_M$ (see the right hand side of (\ref{eqnzm})) is a contraction mapping from the ball $\{z:\|z\|_{X^{b_0}}\leq M^{-D_1}\}$ to itself, and (\ref{linbd}) is true for the kernel $\Ls^{\zeta}$ defined by (\ref{deflinoper}), if $\max(N_2,\cdots,N_q)=M$.
\end{prop}
\begin{proof} First we shall prove (\ref{psibd}) for $\Psi^{(\Sc,n)}$, assuming either (\ref{support01}) and (\ref{estimate01}), or the estimates (\ref{estimate11})--(\ref{estimate14}) for $h^{(\Sc,n)}$. If $n=1$, by applying Cauchy-Schwartz (in the $(k_\Vc,\lambda_\Vc)$ variables), Lemma \ref{largedev0}, and a meshing argument in $\lambda$ and $\lambda_\Vc$ variables, we can get from (\ref{estimate11})--(\ref{estimate14}) (where we choose $C=E=\varnothing$) that $\tau^{-1}M$-certainly,
\[\|\Psi^{(\Sc,1)}\|_{X^{b_0}}\leq (\tau^{-1}M)^{\theta_0}M^{-\varepsilon}\prod_{\nf\in\Lc\cup\Vc\cup\Yc}N_\nf^{-\delta^3},\] which clearly implies (\ref{psibd}). Here for the meshing argument, just notice that the $\lambda$ and $\lambda_\Vc$ derivatives of $\widehat{h^{(\Sc,1)}}$ are bounded, and that the choice of $b_0$ in (\ref{psibd}) (compared to $b$) allows us to restrict $\lambda$ to the big box $|\lambda|\leq M^{\kappa^2}$, so we can apply the same arguments as in the proof of Proposition \ref{trimbd}.

Next, if $n=0$ and $\Sc\neq\Sc_M^+$ (which is the mini-tensor defined in Definition \ref{defstr}), then $|\Lc|\geq 3$. For each fixed $k\in\Zb^d$ with $\langle k\rangle \leq M$, using Lemma \ref{largedev0}, a meshing argument in $\lambda$ as above, and the simple inequality
\[\sup_k\|h_{kk_\Uc}\|_{k_\Uc}\leq\|h_{kk_\Uc}\|_{k\to k_\Uc}\] for any tensor $h=h_{kk_\Uc}$, we deduce from (\ref{estimate01}) (where we choose $B=E=\varnothing$ \emph{instead of} $C=E=\varnothing$) that $\tau^{-1}M$-certainly,
\[\bigg(\int_{\Rb}\langle\lambda\rangle^{2b_0}|\widehat{\Psi_k^{(\Sc,0)}}(\lambda)|^2\bigg)^{1/2}\leq(\tau^{-1}M)^{\theta_0}\Xc_1\cdot\prod_{\lf\in\Lc\backslash\{\lf_{\mathrm{top}}\}}N_\lf^{-4\varepsilon}\prod_{\pf\in\Yc}N_\pf^{-\delta^3}\cdot N_{\lf_{\mathrm{top}}}^{-\alpha+\theta},\] where $N_{\lf_{\mathrm{top}}}:=\max_{\lf\in\Uc}N_\lf$, and $\Xc_1$ is as in (\ref{estimate02}) but with $N$ replaced by $M$. We may assume that $\tau^{-1}M$-certainly the above holds for every $k$; since $\langle k\rangle\lesssim N_{\lf_{\mathrm{top}}}$ by (\ref{support01}), we conclude that
\[\|\Psi^{(\Sc,0)}\|_{X^{s',b_0}}\leq\tau^{-\theta_0}N_{\lf_{\mathrm{top}}}^{s'-s}M^{2\delta^3}\Xc_1\cdot\prod_{\lf\in\Lc\backslash\{\lf_{\mathrm{top}}\}}N_\lf^{-2\varepsilon}\cdot\prod_{\nf\in\Lc\cup\Yc}N_{\lf}^{-\delta^3},\] which implies (\ref{psibd}), noticing that
\[N_{\lf_{\mathrm{top}}}^{s'-s}M^{2\delta^3}\Xc_1\cdot\prod_{\lf\in\Lc\backslash\{\lf_{\mathrm{top}}\}}N_\lf^{-2\varepsilon}\leq M^{s'-s},\] which easily follows from the assumption $0<s-s'<\delta^2$ and the definition (\ref{estimate02}) of $\Xc_1$. Finally, if $n=0$ and $\Sc=\Sc_M^+$, then we simply have $\Psi_k^{(\Sc,0)}(t)=\chi(t)\cdot\mathbf{1}_{M/2<\langle k\rangle\leq M}(f_M)_k$, so (\ref{psibd}) follows from (\ref{constants2}).

Next we prove the contraction mapping part in the statement. We will only prove that the right hand side of (\ref{eqnzm}) maps the given ball to itself, since the contraction part follows in the same way. Suppose $\|z_M\|_{X^{b_0}}\leq M^{-D_1}$. The right hand side of (\ref{eqnzm}), which we shall denote by $z_{\mathrm{out}}$, contains three types of terms, which we shall analyze below. Like in the proof of Proposition \ref{induct2}, we will abuse notation and refer to some components of $z_{\mathrm{out}}$ on the right hand side of (\ref{eqnzm}) still as $z_{\mathrm{out}}$.

(1) Consider the terms on the right hand side of (\ref{eqnzm}) that contain no factor $z_M$ (that is, $N_j\leq M/2$ for all $r+1\leq j\leq q$). If some $z_{N_j}$ is replaced by $z_{N_j}^{\mathrm{hi}}$ whose Fourier transform is supported in $|\lambda_j|\geq M^{\kappa^2}$, then we can gain a power $M^\kappa$ by using the bound $\|z_{N_j}\|_{X^{b_0}}\leq 1$, which will overwhelm all loss and easily imply the desired estimate. Below we will assume each $z_{N_j}$ is replaced by $z_{N_j}^{\mathrm{lo}}$, so by definition of $\sum_{(d)}$ in (\ref{eqnzm}), the corresponding component (up to linear combination over different $\Os$) can be written in Fourier space as 
\[(\widehat{z_{\mathrm{out}}})_k(\lambda)=\Upsilon\cdot\sum_{\zeta\in\{\pm\}}\sum_{k'}\int\mathrm{d}\lambda'\cdot(\widehat{\Rs^{M,\zeta}})_{kk'}(\lambda,-\zeta\lambda')(\widehat{\Psi^{(\Sc,n)}})_{k'}^\zeta(\lambda'),\quad \Psi_{k'}^{(\Sc,n)}=\Psi_{k'}[\Sc,h^{(\Sc,n)}],\] where
\[\Sc=\mathtt{Trim}(\mathtt{Merge}(\mathtt{Trim}(\Sc_1,M^\delta),\cdots,\mathtt{Trim}(\Sc_r,M^\delta),\Bs,\Os),M^\delta)\] satisfies $|\Sc|>D$, and $h^{(\Sc,n)}$ is defined in the same way as the second term on the right hand side of (\ref{deftensor1}) (if $n=0$), or as the second line of (\ref{deftensor2}) (if $n=1$).

Note that we now have $D<|\Sc|\leq C\cdot D$; however in the proof of Propositions \ref{induct1} and \ref{induct2} we have not used the assumption $|\Sc|\leq D$, so the same proof also works in the current case and gives ($\tau^{-1}M$-certainly, i.e. after removing an exceptional set of probability $\leq C_\theta e^{-(\tau^{-1}M)^\theta}$) the bound (\ref{estimate01}) or the bounds (\ref{estimate11})--(\ref{estimate14}) for $h^{(\Sc,n)}$. Then, applying once more Lemma \ref{largedev0} and the meshing argument as before, we deduce that
\[\|\Psi^{(\Sc,n)}\|_{X^{b_0}}\leq M^{d/2}\prod_{\nf\in\Lc\cup\Vc\cup\Yc}N_\nf^{-\delta^3}\leq M^{d/2}M^{-\delta^4D}\leq M^{-2D_1},\] using the fact that $|\Lc|+|\Vc|+|\Yc|=|\Sc|>D$ and $N_\nf\geq M^\delta$ for each $\nf\in\Lc\cup\Vc\cup\Yc$. Using then the $X^{b_0}\to X^{b_0}$ norm bound for $\Rs^{M,\zeta}$ (which follows from the corresponding bound for $\Vs^{M,\zeta}$ proved in Proposition \ref{linextra}), we deduce the same bound for $z_{\mathrm{out}}$.

(2) Consider the terms on the right hand side of (\ref{eqnzm}) that contain at least two factors $z_M$ (that is, $N_j=M$ for at least two $r+1\leq j\leq q$). Then, these $z_M$ factors can be estimated in $X^{b_0}$ and lead to a gain of at least $M^{-2D_1}$. The other factors that are not $z_M$ can be bounded trivially using either the induction hypothesis or (\ref{psibd}) which we just proved and contributes at most an $M^{C}$ power. Since $D_1=\delta^5D$ and $D\gg_{C_\delta}1$, the power $M^{-2D_1}$ will be more than enough to close the estimate.

(3) It remains to consider the terms on the right hand side of (\ref{eqnzm}) that contains exactly one factor $z_M$ (that is, $N_j=M$ for exactly one $r+1\leq j\leq q$). By rearranging, we can write this term as $z_{\mathrm{out}}=\Rs^{M,\iota}\Ls^{\zeta}z_M$, where $\iota,\zeta\in\{\pm\}$, and $\Ls^{\zeta}$ is the $\Rb$-linear operator defined in (\ref{deflinoper}). Here in (\ref{deflinoper}), each $N_j\,(2\leq j\leq q)$ is assumed to be $\leq M$ instead of $<M$, but if $N_j=M$ then $y_{N_j}^*$ can only be one of the $\Psi[\Sc_j,h^{(\Sc_j,n_j)}]$ terms with $\Sc_j$ a regular plant of frequency $M$ and size at most $D$ (i.e. $y_{N_j}^*$ is not allowed to be $z_M$ or any Fourier truncation thereof). Our goal here is to prove that (\ref{linbd}) holds for such $\Ls^\zeta$. In fact, if (\ref{linbd}) holds, then by repeating \emph{Step 1} of the proof of Proposition \ref{linextra} we obtain that
\[\|\Ls^{\zeta}\|_{X^{b_0}\to X^{b_0}}\leq \tau^{(6\kappa)^{-1}}\big(\max_{2\leq j\leq q}N_j\big)^{-3\varepsilon\delta};\] summing over all possible choices of $\Ls^\zeta$ and using also the $X^{b_0}\to X^{b_0}$ bound of $\Rs^{M,\iota}$, we obtain
\[\|z_{\mathrm{out}}\|_{X^{b_0}}\leq\tau^{(8\kappa)^{-1}}\|z_M\|_{X^{b_0}}\leq\tau^{(8\kappa)^{-1}}M^{-D_1},\] which is acceptable. This means that, if we can prove (\ref{linbd}) for the $\Ls^\zeta$ as above, then part (4), i.e. the contraction mapping part of $\mathtt{Local}(2M)$ is true, and thus $z_M$, being the unique fixed point of a contraction mapping, does satisfy $\|z_M\|_{X^{b_0}}\leq M^{-D_1}$. Combining with Proposition \ref{linextra} and the construction in Section \ref{constructansatz}, we also obtain that $y_M$ defined by (\ref{ansatzsum}) solves (\ref{eqnyn}) with $N$ replaced by $M$. Finally, the bound (\ref{linbd}) for the $\Ls^\zeta$ as above also implies part (5) of $\mathtt{Local}(2M)$, since if any $y_{N_j}^*$ in (\ref{deflinoper}) is replaced by $z_M$ or its Fourier truncation, then the $M^{-D_1}$ decay will overwhelm any possible loss and immediately imply (\ref{linbd}).

In summary, we now only need to prove (\ref{linbd}) for $\Ls^\zeta$ as in (\ref{deflinoper}), where either $N_j<M$ or $N_j=M$ and $y_{N_j}^*=\Psi[\Sc_j,h^{(\Sc_j,n_j)}]$. We may assume $\max(N_2,\cdots,N_q)=M$ (otherwise use the induction hypothesis) and replace $z_{N_j}$ by $z_{N_j}^{\mathrm{lo}}$. Applying Lemma \ref{localization}, we can remove the $\chi_\tau$ factor in (\ref{deflinoper}) and gain a power $\tau^{\kappa^{-1}}$, which will overwhelm all possible $\tau^{-\theta}$ losses, provided we estimate the expression without $\chi_\tau$ in the stronger $X^{1-b_0,-b}[k\to k']$ norm.

By Proposition \ref{formulas} (1) we can write
\begin{equation}\label{fourierexpand3}(\widehat{\Ls^{\zeta}})_{kk'}(\lambda,\lambda')=\int[\Ms^{\zeta}(\lambda,\lambda',\lambda_2,\cdots,\lambda_r)]_{kk'}\,\mathrm{d}\lambda_2\cdots\mathrm{d}\lambda_r,
\end{equation} where for fixed values of $(\lambda,\lambda',\lambda_2,\cdots,\lambda_r)$, $\Ms^{\zeta}(\lambda,\lambda',\lambda_2,\cdots,\lambda_r)=[\Ms^{\zeta}(\lambda,\lambda',\lambda_2,\cdots,\lambda_r)]_{kk'}$ is the tensor $\Ms=\Ms_{kk_1}$ defined in (\ref{operatorbdnew3}), Proposition \ref{overpair3} (where we rename $k_1$ as $k'$). Here in (\ref{operatorbdnew3}), we assume that $\Psi_{k_j}^{(j)}=\Psi_{k_j}[\Sc_j',\widehat{h^{(j)}}(\lambda_j)]$ where $\Sc_j'=\mathtt{Trim}(\Sc_j,M^\delta)$ and $h^{(j)}=\mathtt{Trim}(h^{(\Sc_j,n_j)},M^\delta)$ for $2\leq j\leq r$; moreover, for fixed values of $(\lambda,\lambda',\lambda_2,\cdots,\lambda_r)$, the tensor $h=h(\lambda,\lambda',\lambda_2,\cdots,\lambda_r)=[h(\lambda,\lambda',\lambda_2,\cdots,\lambda_r)]_{kk'\cdots k_q}(\lambda_{r+1},\cdots,\lambda_q)$ satisfies (\ref{operatorbdnew1}) and \[|h|+|\partial_{\lambda_j}h|\lesssim \frac{\tau^{-\theta}}{\langle \lambda\rangle\langle\Omega+
\zeta_{r+1}\lambda_{r+1}+\cdots+\zeta_q\lambda_q+\widetilde{\Xi}\rangle},\,\,r+1\leq j\leq q;\quad \widetilde{\Xi}:=\zeta_2\lambda_2+\cdots +\zeta_r\lambda_r+\zeta\lambda'-\lambda.\] Moreover, $h$ is in fact a function of $(k-\zeta k',|k|^2-\zeta|k'|^2)$ and $(k_2,\cdots,k_q,\lambda_{r+1},\cdots,\lambda_q)$ only, in the same manner as in Proposition \ref{overpair3}.

Like before, we will separate the high modulation case $\max(|\lambda_2|,\cdots,|\lambda_r|)\geq M^{\sqrt{\kappa}}$, and the low modulation case $\max(|\lambda_2|,\cdots,|\lambda_r|)\leq M^{\sqrt{\kappa}}$. In the high modulation case we may assume (say) $|\lambda_2|=\max(|\lambda_2|,\cdots,|\lambda_r|)\geq M^{\sqrt{\kappa}}$, then as before, using the induction hypothesis and (\ref{psibd}), we can fix the values of $k_j\,(2\leq j\leq q)$ and $\lambda_j\,(3\leq j\leq q)$, and view $\Psi^{(2)}$ as a function of $\lambda_2$ only. Then we obtain, up to a loss of $M^{C}$, that
\[|(\widehat{\Ls^\zeta})_{kk'}(\lambda,\lambda')|\lesssim\frac{\tau^{-\theta}}{\langle\lambda\rangle}\mathbf{1}_{k-\zeta k'=k^*}\int_\Rb\frac{1}{\langle\lambda-\zeta\lambda'-|k|^2+\zeta|k'|^2-\zeta_2\lambda_2+\Xi\rangle}|\Psi^{(2)}(\lambda_2)|\,\mathrm{d}\lambda_2,\] where $k_*$ is a fixed $\Zb^d$ vector, and $\Xi$ is a fixed real number, depending on the choices of the fixed variables. Since $\Psi^{(2)}$ is bounded in $L_{\lambda_2}^2$ with the weight $\langle\lambda_2\rangle^{b_0}\geq M^{\sqrt{\kappa}/2}$, we can gain this $M^{\sqrt{\kappa}/2}$ power (which overwhelms all $M^{C}$ losses) and apply Cauchy-Schwartz, estimating $\Psi^{(2)}$ only in $L_{\lambda_2}^2$, to obtain that
\[\|(\widehat{\Ls^\zeta})_{kk'}(\lambda,\lambda')\|_{k\to k'}\lesssim\frac{\tau^{-\theta}}{\langle\lambda\rangle} M^{-\sqrt{\kappa}/4}\] uniformly in $(\lambda,\lambda')$, which is sufficient to prove (\ref{linbd}).

Now we can restrict to the low modulation case. With a loss of $M^{C/\sqrt{\kappa}}$, which will be negligible compare to the gain, we can replace the exponents $1-b$ and $\widetilde{b}$ in (\ref{estimate11})--(\ref{estimate14}) all by $b$. Therefore, by the same arguments as in the proof of Proposition \ref{induct2}, we obtain the followings:
\begin{enumerate}[(a)]
\item If $j$ is such that either $n_j=1$, or $n_j=0$ and $\min_{\lf\in\Lc_j}N_\lf<M^\delta$, define $\Xf_j=\Xf_j(\lambda_j)$ to be the smallest value such that the type 1 bounds (\ref{boundmerge2})--(\ref{boundmerge6}) in Proposition \ref{overpair} hold for $\widehat{h^{(j)}}(\lambda_j)$, then we have
\begin{equation}\label{integrability2}\bigg(\int_\Rb\Xf_j(\lambda_j)\,\mathrm{d}\lambda_j\bigg)^2\lesssim\int_\Rb\langle \lambda_j\rangle^{2b}\Xf_j(\lambda_j)^2\,\mathrm{d}\lambda_j\lesssim\tau^{-\theta}M^{C/\sqrt{\kappa}+C\delta}.
\end{equation}
\item If $n_j=0$ and $\min_{\lf\in\Lc_j}N_\lf\geq M^\delta$ (in particular $h^{(j)}=h^{(\Sc_j,0)}$), we may decompose $h^{(j)}$ into $h^{(j,\Gamma_j)}=h^{(\Sc_j,0,\Gamma_j)}$ as in part (1) of Proposition \ref{mainprop}; define $\Xf_j=\Xf_j(\lambda_j,\Gamma_j)$ to be the smallest value such that the type 0 bounds (\ref{supportmerge})--(\ref{boundmerge1}) in Proposition \ref{overpair} hold for $\widehat{h^{(j,\Gamma_j)}}(\lambda_j)$, then we have (\ref{integrability}).
\end{enumerate}

Moreover we may also assume $\max(|\lambda|,|\lambda'|)\leq M^{\kappa^2}$, otherwise we exploit the room coming from the exponents $1-b_0<1/2$ and $b>1/2$ to gain a power $M^{\kappa}$ that will overwhelm all possible losses. These assumptions allow us to apply Proposition \ref{overpair3} and get
\[\|[\Ms^{\zeta}(\lambda,\lambda',\lambda_2,\cdots,\lambda_r)]_{kk'}\|_{k\to k'}\leq\frac{1}{\langle\lambda\rangle}\prod_{j=2}^r\Xf_j\cdot M^{-\varepsilon^5},\] where $\Xf_j=\Xf_j(\lambda_j)$ or $\Xf_j(\lambda_j,\Gamma_j)$ in case (a) or (b) above. Since $\max(|\lambda_2|,\cdots,|\lambda_r|)\leq M^{\sqrt{\kappa}}$ and $\max(|\lambda|,|\lambda'|)\leq M^{\kappa^2}$, by a meshing argument we can remove a single exceptional set of probability $\leq C_\theta e^{-(\tau^{-1}M)^\theta}$ such that the above holds for all values of $(\lambda,\lambda',\lambda_2,\cdots,\lambda_r)$. Therefore, by taking the $L^2$ norm weighted by $\langle \lambda\rangle^{1-b_0}\langle\lambda'\rangle^{-b}$ in $(\lambda,\lambda')$, then summing in $\Gamma_j$ and integrating in $\lambda_j$, we obtain the bound (\ref{linbd}) for $\Ls^\zeta$. This completes the proof of parts (3)--(5) of $\mathtt{Local}(2M)$ and finishes the inductive proof of Proposition \ref{mainprop}.
\end{proof}
\section{Proof of the main results}\label{proofmain} In this section we prove our main theorems, Theorems \ref{main} and \ref{main2}. Theorem \ref{main} follows from Proposition \ref{mainprop} together with some arguments similar to those in Sections \ref{mainproof}--\ref{mainproof1}; Theorem \ref{main2} is easier and follows from simplified versions of these arguments.
\subsection{Proof of Theorem \ref{main}} First,  by Proposition \ref{mainprop}, after removing an exceptional set of probability at most $C_\theta e^{-\tau^{-\theta}}$, we may assume that $\mathtt{Local}(M)$ holds for all $M$. In particular, by (\ref{ansatzalt}), (\ref{psibd}) and part (4) of $\mathtt{Local}(M)$ in Proposition \ref{mainprop}, and in view of the fact that the number of plants $\Sc$ with frequency $N(\Sc)=N$ and size $|\Sc|\leq D$ is at most $(\log N)^{\kappa}$, we conclude that
\begin{equation}\label{ynbound}\|y_N\|_{X^{s',b_0}}\leq \tau^{-\theta} N^{(s'-s)/2}\end{equation} for each $s'<s$ and each $N$, thus \begin{equation}\label{limit1}\lim_{N\to\infty}v_N^\dagger=\lim_{N\to\infty}\sum_{N'\leq N}y_{N'}\qquad \textrm{exists in }X^{s',b_0}.\end{equation} Moreover, under all these $\mathtt{Local}(M)$ assumptions, $v_N^\dagger$  solves (\ref{reducedeqn3}), thus it must equal $v_N$ on $J=[-\tau,\tau]$. By definition $e^{it\Delta}v_N^\dagger$ also equals the solution $\widetilde{u_N}$  of (\ref{nlstruncgauged2}) on $J$.

Now, define
\begin{equation}\label{defbnt}B_N(t)=\frac{p+1}{2}\chi(t)\int_0^t\chi(t')\fint_{\Tb^d}W_N^{p-1}(e^{it'\Delta}v_N^\dagger)\,\mathrm{d}t',\end{equation} akin to the one in (\ref{gaugeinv}) but with the smooth cutoff $\chi$, and define $u_N^\dagger(t)=e^{it\Delta}v_N^\dagger(t)\cdot e^{-iB_N(t)}$, then $u_N^\dagger$ equals $u_N$, which is the solution to (\ref{nlstrunc}), on $J$. By analyzing the term $W_N^{p-1}$ as in \cite{DNY}, Proposition 2.2, and applying the same calculations as in Section \ref{condition1}, we can rewrite $B_N(t)$ as
\begin{equation}\label{formbnt}B_N(t)=\sum_{3\leq q\leq p}a_{pq}'(m_N^*)^{(p-q)/2}\cdot\Ic_\chi\Ab\Mc_{q-1}(v_N^\dagger,\cdots,v_N^\dagger)(t).\end{equation} In the above $q$ runs over odd integers, $a_{pq}'$ are constants, $m_N^*$ is defined as in Section \ref{condition1}, $\Ab$ is the projection onto frequency $k=0$, $\Mc_{q-1}$ is defined as in (\ref{multilin})--(\ref{defomega}) but with $q$ replaced by $q-1$. Note that instead of  the \emph{simplicity} condition as in Definition \ref{def:simple}, here the coefficients $c_{kk_1\cdots k_{q-1}}$ satisfy the slightly different \emph{input-simplicity} condition, namely that $c_{kk_1\cdots k_{q-1}}$ depends only on the set of pairings in $(k_1,\cdots,k_{q-1})$, and $c_{kk_1\cdots k_{q-1}}=0$ unless any pairing in $(k_1,\cdots,k_{q-1})$ is over-paired. However, in view of the projection $\Ab$ to $k=0$, this input-simplicity condition will imply the same tensor norm estimates in Section \ref{prelimtensor} that are proved under the simplicity condition.

Therefore, by decomposing $v_N^\dagger$ using (\ref{ansatzsum}) and repeating the proofs\footnote{If needed, we can always view an $\Rb$-multilinear operator of degree $q-1$ as one of degree $q$ by adding a trivial input function.} of Sections \ref{mainproof}--\ref{mainproof1}, after possibly removing another exceptional set of probability not exceeding $C_\theta e^{-\tau^{-\theta}}$, we conclude that $B_N(t)$ converges to some $B(t)$ in $H_t^{b_0}\hookrightarrow C_t^0$ as $N\to\infty$, where recall $b_0>1/2$ as in (\ref{defb}). Therefore $u_N$, which equals $u_N^\dagger=e^{it\Delta}v_N^\dagger\cdot e^{-iB_N}$ on $J$, converges in $C_t^0H_x^{s-}(J)$ as $N\to\infty$. This limit $u$ has the explicit expansion (which is valid on $J$)
\begin{multline}\label{expansionofu}
u_k(t)=e^{-i(|k|^2t+B(t))}\bigg[\sum_{n\in\{0,1\}}\sum_{|\Sc|\leq D}\sum_{k_\Uc}\bigg\{\sum_{k_\Vc}\int\mathrm{d}\lambda_\Vc\cdot h_{kk_\Uc}^{(\Sc,n)}(k_\Vc,\lambda_\Vc)\\\times\prod_{\ff\in\Vc}(\widehat{z_{N_\ff}})_{k_\ff}^{\zeta_\ff}(\lambda_\ff)\bigg\}\prod_{\lf\in\Uc}\frac{\Delta_{N_\lf}(k_\lf)}{\langle k_\lf\rangle^\alpha}\cdot g_{k_\lf}^{\zeta_\lf}(\omega)+z_k(t)\bigg],
\end{multline}where the sum is taken over all regular plants $\Sc$ with $|\Sc|\leq D$, the random tensors $h^{(\Sc,n)}$ and the functions $z_N$ are defined as in Section \ref{constructansatz}, and the remainder $z$ belongs to $C_t^0H_x^{D_1-1}(J)$.

It now remains to prove that the nonlinearity $W^p(u)$ defined by (\ref{wickpoly2}) exists as a spacetime distribution, and that $u$ solves (\ref{nls}) in the distributional sense. Define
\begin{equation}\label{defother}v^\dagger=\lim_{N\to\infty}v_N^\dagger=\sum_{N'}y_{N'},\quad\mathrm{and}\quad u^\dagger(t)=e^{it\Delta}v^{\dagger}(t)\cdot e^{-iB(t)},
\end{equation} then $u$ equals $u^\dagger$ on $J$, so it suffices to prove that
\begin{equation}\label{todolimit}\lim_{N\to\infty}W_N^p(\Pi_Nu^\dagger)=\lim_{N\to\infty}\Pi_NW_N^p(\Pi_Nu^\dagger)=\lim_{N\to\infty}W_N^p(u_N^\dagger)=\lim_{N\to\infty}\Pi_NW_N^p(u_N^\dagger)
\end{equation}in the sense of distributions. As $B_N\to B$, we may replace $u_N^\dagger$ by $e^{it\Delta}v_N^\dagger$ and $u^\dagger$ by $e^{it\Delta}v^\dagger$; then arguing as in Section \ref{condition1}, we can reduce to analyzing the terms
\begin{equation}\label{terms}\sum_{3\leq q \leq p}a_{pq}''(m_N^*)^{(p-q)/2}\,\Pi\Mc_q(w,\cdots,w)(t),
\end{equation} where $q$ is odd as before, $a_{pq}''$ are constants, $\Pi$ is either $1$ or $\Pi_N$, $w$ is either $v_N^\dagger$ or $\Pi_Nv^\dagger$, and $\Mc$ is as in (\ref{multilin})--(\ref{defomega}) but is input-simple instead of simple. Decomposing $w$ using (\ref{ansatzsum}), it then suffices to show that
\begin{equation}\label{terms2}\Phi:=\Mc_q(y_{N_1},\cdots,y_{N_q})\to 0,\quad \textrm{as }N_{\mathrm{max}}:=\max(N_1,\cdots,N_q)\to\infty,
\end{equation}in the sense of distributions. In fact we shall control the term $\Ic_\chi\Phi$ which, by Lemma \ref{duhamelform0}, satisfies \begin{equation}\label{duhamelPhi}\widehat{\Ic_\chi\Phi_k}(\lambda)=\int_\Rb\Ic(\lambda,\lambda')\widehat{\Phi_k}(\lambda')\,\mathrm{d}\lambda',\qquad |\Ic|+|\partial\Ic|\lesssim\bigg(\frac{1}{\langle\lambda\rangle^3}+\frac{1}{\langle\lambda-\lambda'\rangle^3}\bigg)\frac{1}{\langle\lambda'\rangle}.\end{equation}Note that, apart from simple modifications, $\Ic_\chi\Phi$ essentially has the same structure as $\Psi^{(\Sc,n)}$ in (\ref{psibd}) with the associated tensor $h^{(\Sc,n)}$ as in (\ref{deftensor1}) and (\ref{deftensor2}), so it can be estimated in the same way as in Sections \ref{mainproof}--\ref{mainproof1}. We only make two additional observations:
\begin{enumerate}[(a)]
\item The proofs of Sections \ref{mainproof}--\ref{mainproof1} do not depend on any cancellation in (\ref{duhamelPhi}), so the same arguments can be applied for the term $\Ic_\chi^{\mathrm{abs}}\Phi$ defined by
\begin{equation}\label{duhamelPhi2}\widehat{\Ic_\chi^{\mathrm{abs}}\Phi_k}(\lambda)=\int_\Rb\langle\lambda'\rangle^{-1}(\langle\lambda\rangle^{-3}+\langle\lambda-\lambda'\rangle^{-3})|\widehat{\Phi_k}(\lambda')|\,\mathrm{d}\lambda',\end{equation} leading to the control of $\Ic_\chi^{\mathrm{abs}}\Phi$ in the $X^{-d,b_0}$ norm (here the exponent $-d$ has to do with the potential power loss associated with fixing $k$, see observation (b) below), which in turn implies the control of $\Phi$ in the sense of distributions---for example, due to the trivial bound \[\|\Ic_\chi^{\mathrm{abs}}\Phi\|_{X^{-d,b_0}}\gtrsim\sum_{k}\int_{\Rb}\langle k\rangle^{-2d}\langle\lambda'\rangle^{-4}|\widehat{\Phi_k}(\lambda')|\,\mathrm{d}\lambda'.\]
\item The $\Rb$-multilinear operator $\Mc_q$ is input-simple instead of simple. However, in order to control the $X^{-d,b_0}$ norm of $\Ic_\chi^{\mathrm{abs}}\Phi$, we may fix the value of $k$ in $\widehat{\Phi_k}(\lambda)$ and it suffices to get a bound uniform in $k$  thanks to the exponent $-d$. Now once $k$ is fixed, the pairings between $k$ and any $k_j$ become unimportant (they no longer cause losses in any counting estimate as the paired and over-paired variables now have only one choice), so an input-simple $\Rb$-multilinear operator can be treated in the same way as a simple one, similar to the analysis of (\ref{formbnt}) above.
\end{enumerate} 

With the observations above, after possibly removing another exceptional set of probability not exceeding $C_\theta e^{-\tau^{-\theta}}$, the same proofs of Sections \ref{mainproof}--\ref{mainproof1} can be carried out to obtain, for example, that
\begin{equation}\label{nonlinbound}\|\Ic_\chi^{\mathrm{abs}}\Phi\|_{X^{-d,b_0}}\leq\tau^{-\theta}(N_{\mathrm{max}})^{-\delta^6}.
 \end{equation} 
 This proves (\ref{terms2}) and thus finishes the proof of Theorem \ref{main}.
 \subsection{Proof of Theorem \ref{main2}} Fix $\varepsilon>0$ small enough depending on $(d,p)$, $s-s_{pr}$ and $(p-1)(s-s_{pr})-\nu$ (note that this is different from (\ref{subcrit})). Let $(\delta,D,\kappa,\theta)$ etc. be defined as in Section \ref{notations}, we may assume $N\gg_{C_\theta}1$. Let
 \[B(T)=\frac{p+1}{2}\int_0^T\fint_{\Tb^d}|u_\mathrm{ho}|^{p-1}\,\mathrm{d}T',\] and repeat the gauging, conditioning and conjugating arguments as in Section \ref{reduct}, except that $\sigma_N$ is replaced by $0$ since we are dealing with the nonlinearity $|u_\mathrm{ho}|^{p-1}u_\mathrm{ho}$ instead of the Wick-ordered one. After also rescaling time, we can write $N$-certainly that
 \begin{equation}\label{formulauT}\widetilde{u}(T,x)=\sum_{k}w_k(N^{-\nu}T)e^{i(k\cdot x-|k|^2T-B(T))},\end{equation} where $w_k(t)$ is the solution to the system
 \begin{equation}\label{equationat}w_k(t)=f_k-iN^\nu\sum_{3\leq q\leq p}a_{pq}(m_\mathrm{ho})^{(p-q)/2}\int_0^t\Mc_q(w,\cdots,w)_k(t')\,\mathrm{d}t'
 \end{equation}
  similar to (\ref{reducedeqn}), where 
  \[m_\mathrm{ho}=\fint_{\Tb^d}|u_\mathrm{ho}|^2=N^{-2\alpha}\sum_{k}\phi^2\big(\frac{k}{N}\big)|g_k|^2,\]which is $N$-certainly bounded by $N^{d-2\alpha}$; like in (\ref{reducedeqn3}) we also consider the solution $w^\dagger$ to
 \begin{equation}\label{equationat2}w_k^\dagger(t)=\chi(t)f_k-iN^\nu\sum_{3\leq q\leq p}a_{pq}(m_\mathrm{ho})^{(p-q)/2}\cdot\Ic_\chi\Mc_q(w^\dagger,\cdots,w^\dagger)_k(t).
 \end{equation} Here in (\ref{equationat})--(\ref{equationat2}) we have $f_k=\gamma_k\cdot\eta_k(\omega)$ with constants $|\gamma_k|\leq N^{-\alpha+\theta}\langle k\rangle^\theta|\phi(k/N)|$, $a_{pq}$ are bounded constants and $|m_\mathrm{ho}|\leq N^{d-2\alpha}$ ($\alpha$ is defined as in (\ref{data1})), and the $\Rb$-multilinear expression
 \begin{equation}\label{multilinat}\Mc_q(w^{(1)},\cdots,w^{(q)})_k(t')=\sum_{\zeta_1k_1+\cdots+\zeta_qk_q=k}c_{kk_1\cdots k_q}\cdot e^{iN^\nu t'\Omega}\prod_{j=1}^q(w^{(j)})_{k_j}^{\zeta_j}(t'),
 \end{equation} with the signs $\zeta_j$ and coefficients $c_{kk_1\cdots k_q}$ as in (\ref{multilin}).
 
 Since the initial data $f_k$ is uniformly distributed in $k$, the analysis of (\ref{equationat}) will be significantly simpler than the arguments in Sections \ref{ansatz}--\ref{mainproof1}. More precisely, we will only need the tensors\footnote{Because we do not need to distinguish the low-frequency inputs as there is only one scale.} $h^{(\Sc,n)}$ with $n=0$ (which will be constant tensors, i.e. do not depend on $\omega$), and instead of the full plant structure, we will only need its tree part (which is called $\Lc$ before, see Definition \ref{defstr}), leaf pairings and signs of leaves. As such, we will define (in this proof only) $\Sc$ to be a set of leaves $\lf$ with possible pairings, together with the sign $\zeta_\lf\in\{\pm\}$ for each $\lf\in\Sc$. Let $\Pc$ (resp. $\Uc$) be the set of paired (resp. unpaired) leaves, we require that $\zeta_{\lf'}=-\zeta_\lf$ for any pair $(\lf,\lf')$, and that $\sum_{\lf\in\Lc}\zeta_\lf=1$. The $\Sc$ tensors $h=h_{kk_\Uc}$ are defined as in Definition \ref{deftensor} but instead of condition (1) we only assume\footnote{The choice of $N^{1+\theta}$ is because $\phi$ is not compactly supported and may have a Schwartz tail.} $\langle k_\lf\rangle\leq N^{1+\theta}$ for each $\lf\in\Sc$, and $\Psi_k=\Psi_k[\Sc,h]$ is defined as in (\ref{defpsi}). We do not need the $\mathtt{Trim}$ function, and $\mathtt{Merge}$ is defined in the same way as in Definition \ref{defmerge} (with only the tree part, and without the frequency parameters such as $N_\lf$ for leaves $\lf$; also the factors $\Delta_{N_\lf}\gamma_{k_\lf}$ in (\ref{merge1}) are replaced by $\gamma_{k_\lf}$).
 
Next, for any $\Sc$ with $|\Sc|\leq D$, define the $\Sc$ tensor $h^{(\Sc)}=h_{kk_\Uc}^{(\Sc)}(t)$ inductively by
\begin{equation}\label{defhsnew}h_{kk_\Uc}^{(\Sc)}(t)=\mathbf{1}_{|\Sc|=1}\cdot\mathbf{1}_{k=k_\lf}\mathbf{1}_{\langle k\rangle\leq N^{1+\theta}}+N^\nu\sum_{3\leq q\leq p}a_{pq}(m_{\mathrm{ho}})^{(p-q)/2}\sum_{(*)}\Ic_\chi H_{kk_\Uc}(t),
\end{equation} similar to (\ref{deftensor1}). Here the sum $\sum_{(*)}$ is taken over $\Bs=(\zeta_1,\cdots,\zeta_q)$, $\Sc_j$ as defined above, and $\Os$, such that $\Sc=\mathtt{Merge}(\Sc_1,\cdots,\Sc_q,\Bs,\Os)$; $H=H_{kk_\Uc}$ is defined by $H=\mathtt{Merge}(h^{(\Sc_1)},\cdots,h^{(\Sc_q)},h,\Bs,\Os)$ with \begin{equation}\label{basetensorat}h=h_{k_1\cdots k_q}(t')=\mathbf{1}_{k=\zeta_1k_1+\cdots +\zeta_qk_q}\cdot\mathbf{1}_{\langle k\rangle\leq N^{1+2\theta}}\prod_{j=1}^q\mathbf{1}_{\langle k_j\rangle\leq N^{1+2\theta}}\cdot c_{kk_1\cdots k_q}e^{iN^\nu t'\Omega}\end{equation} similar to (\ref{newpsi})--(\ref{basetensor}). By arguing similarly as in Propositions \ref{final}, \ref{algorithm2}, \ref{overpair2} and \ref{induct1}, we can prove inductively that $h_{kk_\Uc}^{(\Sc)}$ satisfies the support condition (\ref{support01}); moreover, for each subpartition $(B,C)$ of $\Uc$, we can prove that
\begin{equation}\label{tensorbdan}\int_\Rb\langle\lambda\rangle^{2b}\bigg(\sum_{\Gamma\in\Zb}\|\widehat{h_{kk_\Uc}^{(\Sc,\Gamma)}}(\lambda)\|_{kk_B\to k_C}\bigg)^2\,\mathrm{d}\lambda\leq \big( N^{(\alpha-2\varepsilon)|B\cup C|}N^{-\varepsilon|\Pc|}N^{\varepsilon|E|}\cdot \Xc_0\big)^2,
\end{equation} similar to (\ref{estimate01}). Here $\Pc$ is the set of paired leaves as defined above, $E=\Uc\backslash(B\cup C)$, and $h^{(\Sc,\Gamma)}$ is the restriction of $h^{(\Sc)}$ to the set (\ref{support02}) as in part (1) of Proposition \ref{mainprop}; moreover, $\Xc_0$ is defined to be $N^{-(\alpha-2\varepsilon)}$ if $C\neq\varnothing$, and $N^{\max(0,d/2-\alpha+2\varepsilon)}$ if $C=E=\varnothing$, and $N^{-\varepsilon\delta}$ if $B=C=\varnothing$, and $1$ otherwise, as in (\ref{estimate02}). 

We note that the proof of (\ref{tensorbdan}) is much easier than that of (\ref{estimate01}), as we do not need to apply the careful selection algorithm in Proposition \ref{algorithm1}. There are only two nontrivial differences. The first is due to the extra $N^\nu$ factor in  $e^{iN^\nu t'\Omega}$ in (\ref{basetensorat}), which actually helps us as $\nu>0$, since $N^\nu\Omega$ belonging to an interval of length $O(1)$ will force $\Omega$ to belong to an interval of length $O(1)$. The second is the extra factor $N^\nu$ on the right hand side of (\ref{defhsnew}), which gets cancelled by the $N^{-\alpha}$ decay of $\gamma_k$ and the $(m_\mathrm{ho})^{(p-q)/2}$ factor in (\ref{defhsnew}), in view of the inequality $\nu\leq (p-1)(\alpha-\alpha_0-10\varepsilon)$ ($\alpha_0$ is defined as in (\ref{subcrit})). In particular we have
\[(m_\mathrm{ho})^{(p-q)/2}N^\nu\cdot\prod_{j=2}^q N^{\alpha_0'+\theta}\leq \prod_{j=2}^q N^{\alpha-8\varepsilon},\quad \alpha_0':=\frac{d}{2}-\frac{1}{q-1},\] so the bound for the $h$ tensor we see when merging $h^{(\Sc_j)}$ tensors---which is the one on the right hand side of (\ref{finalbound}) with $p$ replaced by $q$---can be cancelled by the $N^{-\alpha}$ decay of $\gamma_k$ with extra gain of $N^{\varepsilon}$ powers, after being multiplied by $N^\nu$ and $(m_\mathrm{ho})^{(p-q)/2}$ factors.

Now, with (\ref{tensorbdan}) available, we can construct the solution $w^\dagger$ to (\ref{equationat2}) by the ansatz 
\begin{equation}\label{ansatzwd}w_k^\dagger(t)=\sum_{|\Sc|\leq D}\Psi_k[\Sc,h^{(\Sc)}(t)]+z_k^\dagger(t),\end{equation} where $z^\dagger$ satisfies the equation
\begin{equation}\label{equationforz}z_k^\dagger(t)=\chi(t)\mathbf{1}_{\langle k\rangle\geq N^{1+\theta}}\cdot f_k-i N^\nu\sum_{3\leq q \leq p}a_{pq}(m_*)^{(p-q)/2}\sum_{(v^{(1)},\cdots,v^{(q)})}\Ic_\chi\Mc_q(v^{(1)},\cdots,v^{(q)})_k(t),
\end{equation} with the sum taken over $(v^{(1)},\cdots,v^{(q)})$ such that each $v^{(j)}$ is either $z^\dagger$ or $\Psi_k[\Sc_j,h^{(\Sc)j)}]$ for some $\Sc_j$ with $|\Sc_j|\leq D$, and that either (i) at least one $v^{(j)}=z$, or (ii) $ v^{(j)}=\Psi_k[\Sc_j,h^{(\Sc_j)}]$ for each $j$, and $\Sc=\mathtt{Merge}(\Sc_1,\cdots,\Sc_r,\Bs,\Os)$ satisfies $|\Sc|>D$ (for any $\Os$).

Note that, due to the lack of projection $\Pi_N$ on the right hand side of (\ref{nlstrunc2}), the remainder $z^\dagger$ is not guaranteed to have compact support in $k$. Thus, instead of the $X^{b_0}$ norm as in Proposition \ref{mainprop}, we should control the $\widetilde{X}^{C_0,b_0}$ norm of $z^\dagger$, defined by
\[\|z^\dagger\|_{\widetilde{X}^{C_0,b_0}}^2:=\sum_{k}\int_\Rb\bigg(1+\frac{|k|}{N^{1+\theta}}\bigg)^{2C_0}\langle \lambda\rangle^{2b_0}|(\widehat{z^\dagger})_k(\lambda)|^2\,\mathrm{d}\lambda,\] where $C_0>0$ is a large absolute constant depending only on $(d,p)$ and $b_0$ is as in (\ref{defb}). Indeed we shall prove that $\|z^\dagger\|_{\widetilde{X}^{C_0,b_0}}\leq N^{-D_1}$ by a contraction mapping. Note that again we only need to prove that the right hand side of (\ref{equationforz}) satisfies this same inequality assuming that $z^\dagger$ does. This right hand side contains four types of terms, which are listed as follows\footnote{The reader may notice the similarity with the construction of $z_M$ in (\ref{eqnzm}).}:
\begin{itemize}
\item The first term on the right hand side of (\ref{equationforz}), which is acceptable because $|f_k|\leq|\phi(k/N)|$ where $\phi$ is Schwartz and $\langle k\rangle\geq N^{1+\theta}$.
\item The term where at least two $v^{(j)}$ equal $z^\dagger$, which is acceptable thanks to the decay $N^{-D_1}$ of $z^\dagger$ and the choice of large $C_0$.
\item The term where $v^{(j)}=\Psi_k[\Sc_j,h^{(\Sc_j)}]$ for each $j$. This term is acceptable because it can be written as a linear combination of $\Psi_k[\Sc,h^{(\Sc)}]$ for some $\Sc=\mathtt{Merge}(\Sc_1,\cdots,\Sc_q,\Bs,\Os)$ with $|\Sc|>D$, and the corresponding tensor $h^{(\Sc)}=\mathtt{Merge}(h^{(\Sc_1)},\cdots,h^{(\Sc_q)},h,\Bs,\Os)$, which can be shown to satisfy (\ref{tensorbdan}) by repeating the proofs above. By applying Lemma \ref{largedev0} again, this $\Psi_k$ term can be bounded by $N^{-D_1}$ in $X^{b_0}$, and hence in $\widetilde{X}^{C_0,b_0}$ because it is supported in $|k|\lesssim N^{1+\theta}$.
\item Finally, the term where exactly one $v^{(j)}$ equals $z^\dagger$. This term can be written as an $\Rb$-linear operator $\Ls^\zeta$ applied to $z^\dagger$, where this $\Ls^\zeta$ has similar form as the one in (\ref{deflinoper}). Now by repeating the same arguments as in Propositions \ref{final}, \ref{algorithm1}, \ref{overpair3}, \ref{linextra} and \ref{induct3}, we can bound the $X^{b_0}\to X^{b_0}$ norm of this operator by a negative power of $N$. As the kernel $(\Ls^\zeta)_{kk'}$ is supported in $|k-\zeta k'|\lesssim N^{1+\theta}$, by Lemma \ref{weightedbd}, the $\widetilde{X}^{C_0,b_0}\to \widetilde{X}^{C_0,b_0}$ norm of $\Ls^\zeta$ is also bounded by a negative power of $N$, so this term is also acceptable.
\end{itemize}

As such, we have closed the estimates for $z^\dagger$ and obtained the solution $w^\dagger$ in the form of (\ref{ansatzwd}). This means that the equation (\ref{equationat2}) is well-posed at least up to time $t=1$, so the equation (\ref{nlstrunc2}) is well-posed at least up to time $T=N^\nu$. Moreover (\ref{longtimecon}) easily follows from the bound for $z^\dagger$, as well as the bounds for nonlinear components $\Psi_k[\Sc,h^{(\Sc)}]$ with $|\Sc|>1$, which in turn follow from (\ref{tensorbdan}) and Lemma \ref{largedev0}. This completes the proof of Theorem \ref{main2}.
\section{Final remarks}\label{secfinrem} In this section we make some final remarks. These include a comparison with parabolic equations in Section \ref{comparison} and some future directions in Section \ref{furthercomment}. We also list some open problems.
\subsection{Comparison with parabolic equations}\label{comparison} The random data Schr\"{o}dinger equation
\begin{equation}\label{nlsend}(i\partial_t+\Delta)u=W^p(u),\quad u(0)=f(\omega),
\end{equation}
 is closely linked to, and fundamentally different from, the stochastic heat equation
\begin{equation}\label{parabend} (\partial_t-\Delta)u=\widetilde{W}^p(u)+\zeta,
\end{equation} if both are suitably renormalized. In this section we will explain their differences and connections.
\subsubsection{Difference in scaling}\label{comparison1} In Section \ref{heuristic} we explained the heuristics behind the \emph{probabilistic scaling} critical index $s_{pr}$ for (\ref{nlsend}). In fact the same philosophy can be applied to (\ref{parabend}), leading to the \emph{parabolic scaling} critical index $s_{pa}$ in Remark \ref{remscale}, which is the one appearing in works such as \cite{Hairer,GIP}. Note that $s_{pa}$ is strictly lower than $s_{pr}$.

Fix some value of $s$, and let $\alpha=s+d/2$. As in Section \ref{heuristic} we will make simplifications to (\ref{parabend}) by replacing the nonlinearity by $\Nc_{\mathrm{np}}$ defined in (\ref{nopairintro}) and neglecting the renormalization\footnote{Even with the $\Nc_{\mathrm{np}}$ nonlinearity, some renormalization may still be needed for higher order iterations, but not for the first nonlinear iteration which is discussed here. Also whether $u$ is real or complex valued, and whether $\Nc_{\mathrm{np}}$ contains complex conjugates, does not affect the scaling heuristics.}; we also set initial data $u(0)=0$. Similar to (\ref{randata}), assume the noise $\zeta$ (or its regularization) has the form
\[\zeta(t,x)=N^{-\alpha+1}\sum_{|k|\sim N}\partial_t{\beta_k}(t)\cdot e^{ik\cdot x},\] where $\beta_k(t)$ are independent Brownian motions. Let $\psi=(\partial_t-\Delta)^{-1}\zeta$ be the linear evolution of noise (which plays the same role as $e^{it\Delta}u(0)$ in Section \ref{heuristic}), then
\[\psi(t,x)=N^{-\alpha}\sum_{|k|\sim N}G_k(t)e^{ik\cdot x},\quad G_k(t):=N\int_0^t e^{-(t-t')|k|^2}\,\mathrm{d}\beta_k(t').\] For fixed $|t|\sim 1$ these $G_k(t)$ form a collection of independent Gaussian variables with $\Eb|G_k(t)|^2\sim 1$, hence $\psi$ is bounded in $C_t^0H_x^s$ (also in $C_t^0C_x^s$ by Khintchine's inequality), just as in Section \ref{heuristic}.

Now, plugging into (the simplified version of) (\ref{parabend}), we need to control the first nonlinear iteration
\begin{equation}u^{(1)}(t)=\int_0^t e^{(t-t')\Delta}\Nc_{\mathrm{np}}(\psi(t'))\,\mathrm{d}t',
\end{equation} where $|t|\sim 1$, in $H^s$ (or equivalently $C^s$). Similar to Section \ref{heuristic}, on the Fourier side we have
\begin{equation}\label{parabcase}u_k^{(1)}(t)\sim N^{-p\alpha}\sum_{\substack{k_j\in\Zb^d,|k_j|\sim N\\k_1-\cdots +k_p=k}}\int_0^t e^{-(t-t')|k|^2}G_{k_1}(t')\overline{G_{k_2}(t')}\cdots G_{k_p}(t')\,\mathrm{d}t'.\end{equation} Suppose $|k|\sim N$, using the square root cancellation in the sum in (\ref{parabcase}) (from independence, as in Section \ref{heuristic}) and the $N^{-2}$ gain from the $t'$ integral, we see that with high probability, the inner sum-integral has size $N^{(pd-d)/2-2}$, hence
\[\|u^{(1)}(t)\|_{H^s}\sim N^{-(p-1)s-2};\qquad \|u^{(1)}(t)\|_{H^s}\lesssim 1\Leftrightarrow s\geq  -\frac{2}{p-1}:=s_{pa}.\]
We make a few observations on the above heuristic calculation:

(a) It is no surprise that $s_{pa}=s_{cr}-d/2$. Indeed this makes $C^{s_{pa}}$ and $H^{s_{cr}}$ have equal scaling, and in the usual (deterministic) sense $H^{s_{cr}}$, thus also $C^{s_{pa}}$, is critical for the heat equation. The effect of randomness then comes through Khintchine's inequality, where a Gaussian random function which belongs to $H^{s_{pa}}$ must also belong to $C^{s_{pa}}$ which scaling-wise equals the critical space $H^{s_{cr}}$. This is essentially how $s_{pa}$ is calculated in \cite{Hairer}.

(b) The above argument does not work for Schr\"{o}dinger equations, because even though $H^{s_{cr}}$ and $C^{s_{pa}}$ are still scaling critical in the usual sense, the latter is not compatible with Schr\"{o}dinger flows. However, this does not tell us what \emph{is} the right notion of criticality for Schr\"{o}dinger.

(c) To exactly see the difference between $s_{pa}$ and $s_{pr}$, we have to compare the calculations in here and in Section \ref{heuristic}. Note that in (\ref{parabcase}) the $t'$ integral gains two derivatives $N^{-2}$; in comparison in (\ref{detersum}) and (\ref{randsum}) there is no derivative gain---since the Schr\"{o}dinger flow has no smoothing---only the denominator $\langle \Omega\rangle^{-1}$ which restricts to the submanifold $\Omega=0$. This is the fundamental difference between heat and Schr\"{o}dinger that eventually leads to different scalings in the random setting.

(d) More precisely, note that restricting to $\Omega=0$ reduces the number of dimensions by two. In the deterministic setting, this gains two derivatives $N^{-2}$ in the summation in (\ref{detersum}), which matches the two-derivative gain from heat, leading to the same criticality threshold $s_{cr}$; however in the random setting the summation in (\ref{randsum}) gets square rooted due to randomness, which means the $N^{-2}$ gain \emph{also} gets square rooted, leading to the different criticality thresholds $s_{pa}$ and $s_{pr}$.
\subsubsection{Necessary renormalizations} Another difference between our theory and the parabolic theories is that, in the latter more and more renormalization terms are needed when one gets close to criticality (for example with the $\Phi_{4-\delta}^4$ model \cite{BCCH,CMW}), while in the former we stay with Wick ordering in the full subcritical regime.

The main reason for this is the difference in the notions of scaling. For example, in the $\Phi_3^4$ setting where $(d,p)=(3,3)$, if solutions have regularity $C^{-1/2-}$ or equivalently $H^{-1/2-}$, then (\ref{parabend}) is still subcritical though needs a log correction $3C_2$ beyond Wick ordering (see (\ref{parab1})), but (\ref{nlsend}) is already \emph{critical} in the probabilistic scaling (comparable to the \emph{four dimensional} $\Phi_4^4$ problem which is critical in the parabolic scaling, due to the reason explained in Section \ref{comparison1}). Conversely, if (\ref{nlsend}) is subcritical, then a calculation shows that the $3C_2$ in (\ref{parab1}) will not appear as $\mathbb{E}(\<K*2_black>\cdot\<2_black>)$ is not divergent in the limit, so (\ref{parabend}) only needs Wick ordering.

More precisely, for (\ref{parabend}) there are two types of renormalization terms\footnote{This distinction may be artificial from the regularity structure perspective, but is convenient in comparison with the dispersive case here.}, namely those coming from the mass (which is just Wick ordering for $\Phi_3^4$ in (\ref{parab1})) and those not coming from the mass (such as the $\log$ term in (\ref{parab1}) and the further corrections described in \cite{BCCH}, Section 2.8.2 for $\Phi_{4-\delta}^4$). Now for (\ref{nlsend}) only the mass terms diverge (and need to be renormalized) in the probabilistically subcritical regime; moreover since the Schr\"{o}dinger equation \emph{conserves} mass, we can always replace the mass by the mass of initial data, which just leads to Wick ordering and no further renormalization is needed.

Given this difference, one might ask whether for (\ref{nlsend}) we can go strictly below $s_{pr}$ and down to $s_{pa}$ by including additional counterterms such as the ones in \cite{BCCH}. We believe the answer is \emph{no} due to the following reason. In all previous works, the counterterms in the renormalization process are needed because certain \emph{specific terms} in the formal expansion of the solution with respect to the random initial data or noise become unbounded when $\varepsilon\to 0$ (or $N\to\infty$ in the setting of Theorem \ref{main}); however if one considers (\ref{nlsend}) in the supercritical regime $s<s_{pr}$, then an extended version of the calculation in Section \ref{heuristic} shows that \emph{every} term in the formal expansion---not just those with a renormalizable structure---becomes unbounded (say with respect to the regularity of initial data) as $N\to\infty$. This is similar to\footnote{Put in another way, imagine one has all the input and output frequencies being the same in the nonlinearity. Then (\ref{parabend}) is locally well-posed, without any renormalization or special arguments, if and only if $s>s_{pa}$, while for (\ref{nlsend}) the same thing holds if and only if $s>s_{pr}$.} what happens to (\ref{parabend}) when $s<s_{pa}$. Therefore, it is at least highly improbable that the local-in-time problem for (\ref{nlsend}) can be solved perturbatively via a renormalization process similar to those in the theory of regularity structures \cite{Hairer,Hairer3,Hairer4, BCCH}.
\subsubsection{Invariant measures and quantum field theory}\label{invmeas} If the random initial data of (\ref{nlsend}) is given by (\ref{data0}) with $\alpha=1$, or if $\zeta$ in (\ref{parabend}) is the spacetime white noise, then both equations will possess the same formally invariant \emph{Gibbs measure}, which is the $\Phi_d^{p+1}$ measure in quantum field theory (up to real/complex distinction), formally defined by
\begin{equation}\label{measureend}\mathrm{d}\mu\sim\exp\bigg[\frac{-2}{p+1}\int_{\Tb^d}\widetilde{W}^{p+1}(u)\,\mathrm{d}x\bigg]\cdot\exp\bigg[-\int_{\Tb^d}|\nabla u|^2\,\mathrm{d}x\bigg]\prod_{x\in\Tb^d}\mathrm{d}x
\end{equation} for some renormalization $\widetilde{W}^{p+1}$ of $|u|^{p+1}$. The justification of the formal definition (\ref{measureend}) is a major problem in constructive quantum field theory, see \cite{GJ1,Simon,Nel2} and recently \cite{BaGu0}. It has been done in dimension $d\in\{1,2\}$ for any $p$, and in dimension $d=3$ for $p=3$. The other cases are not super-renormalizable in the sense of \cite{CMW}, and such constructions are either unknown or proved impossible \cite{Aiz,AiCo,Fro}.

The study of the dynamics of the measure (\ref{measureend}) under the flow of (\ref{parabend}), commonly known as \emph{stochastic quantization}, starts with \cite{PW}. The invariance of the $\Phi_d^{p+1}$ measure with $d\in\{1,2\}$ and any $p$ is proved in \cite{DaDe2}. Recent developments of parabolic theories has led to the resolution of the $\Phi_3^4$ case, with proof of invariance in \cite{MouWe3,AK}.

On the other hand, the Gibbs measure problem for (\ref{nlsend}) is harder, both conceptually and technically, due to lack of smoothing and (consequently) the different scalings as described above. The invariance of $\Phi_1^{p+1}$ (for any $p$) and $\Phi_2^4$ measures are proved by Bourgain \cite{Bourgain94,Bourgain} (see also \cite{LRS}). The $\Phi_2^{p+1}$ case for $p\geq 5$ is much more challenging and is resolved only in our recent work \cite{DNY}. This matches the results of the Gibbs measure problem for (\ref{nlsend}) with those of the measure construction problem, and of the stochastic quantization problem for (\ref{parabend}), \emph{except} in the $\Phi_3^4$ case $(d,p)=(3,3)$.

The Gibbs measure problem for $(d,p)=(3,3)$ has two main difficulties. First it is probabilistically critical. This is not as bad as supercritical cases which we believe---as mentioned above---cannot be renormalized (at least through a process similar to \cite{Hairer}), but still log divergences seem unavoidable in all aspects, even for short time. Second, the $\Phi_3^4$ measure is mutually singular with the reference Gaussian measure, as proved in \cite{BaGu}, thus $f(\omega)$ in (\ref{nlsend}) will not be given by the simple formula\footnote{This should be compared to the stochastic quantization problem for (\ref{parabend}) where the solution theory relies on the \emph{Gaussian noise} instead of the non-Gaussian measure, as observed in \cite{BaGu}.} (\ref{data0}). Therefore the hope is to somehow get rid of the log divergences by moving to the right measure, i.e. Gibbs instead of Gaussian, but then a local solution theory has to be developed without independence of Fourier coefficients.

{\textbf{Open problem 1.} Prove invariance of the Gibbs measure for (\ref{nlsend}) with suitable renormalizations, when $d=p=3$.
\subsection{Future directions}\label{furthercomment} Though in this paper we have restricted to Schr\"{o}dinger equations, our method can be applied to more general settings. In this last section we list some future directions.
\subsubsection{The stochastic setting and other dispersion relations} Consider (\ref{nlsend}), but with additive noise $\zeta$ instead of random data, for example
\[(i\partial_t+\Delta)u=W^p(u)+\zeta,\quad u(0)=0,\] see \cite{deD, FOW, FX, CM} for some previous works. Here the role of the linear evolution $e^{it\Delta}f(\omega)$ is played by $\psi:=(i\partial_t-\Delta)^{-1}\zeta$. Note that if we formally periodize the time, then $\psi$ will have the form
\[\psi(t,x)=\sum_{k,\lambda}a_{k,\lambda}g_{k,\lambda}(\omega)e^{i(k\cdot x-|k|^2t+\lambda t)},\] where $\lambda$ is the modulation variable, $a_{k,\lambda}$ is some fixed function of $(k,\lambda)$ and $g_{k,\lambda}$ are i.i.d. Gaussian random variables. Therefore, in addition to the $k$ variables, we should include also the $\lambda$ variables as input variables for our random tensor, which will then look like $h_{kk_A,\lambda\lambda_A}$ for some sets $A$. The counting estimate should be adjusted, which may lead to changes in the selection algorithm.

Similarly we may consider other dispersion relations, still in the semilinear setting. The main difference is again in the counting estimates: suppose the new dispersion relation is $\Lambda(k)$ for some function $\Lambda$, then we should look at the cardinality of sets
\[\{(k_1,\cdots,k_p):k_1-\cdots+k_p=k,\,\Lambda(k_1)-\cdots +\Lambda(k_p)=\Gamma+O(1)\}\] with fixed $k$ and $\Gamma$, perhaps with some additional linear relations between $k_j$ like those in Section \ref{counting}. Note that while \emph{parabolic equations are all alike, each dispersive equation is dispersive in its own way}. As a result, the above counting bound will depend on the exact form of $\Lambda$ (not just its homogeneity), and a selection algorithm is then designed to match the counting bound. In particular we will not have a general black-box argument working for all $\Lambda$, and the proof has to be done in a dispersion-specific way.
\subsubsection{Quasilinear problems} Recently there have been attempts to extend the existing parabolic theories to quasilinear equations \cite{FM,GH}. This is also of interest in the dispersive setting, especially in view of the recent results in low regularity deterministic local well-posedness \cite{ABZ,KRS}.

Of course, compared to parabolic equations, moving to the quasilinear (or even variable-coefficient semilinear) setting completely changes the methodology for dispersive equations. The $X^{s,b}$-based approaches become unavailable and dispersion has to be observed on the level of energy estimates or parametrices. In the deterministic setting, it is expected that the local well-posedness threshold is higher than $s_{cr}$, but the precise value is only known in some cases; in the random setting we also expect the threshold to be higher than $s_{pr}$, but are unable to decide or even guess the correct value. The method of random averaging operators can be applied to the quasilinear setting but may not achieve the same power as the semilinear version, and the quasilinear version of random tensor theory still needs to be explored.

\textbf{Open problem 2.} Build a random data theory for quasilinear dispersive (including wave) equations, and determine the threshold for almost-sure local well-posedness.
\subsubsection{Long-time propagation of randomness} It is natural to ask whether the short-time solutions for (\ref{nlsend}) constructed in Theorem \ref{main} can be extended to longer or infinite time; i.e. whether the randomness structure can be propagated beyond the perturbative regime. Such global-in-time extensions are immediate if an invariant Gibbs measure is available at the regularity we are considering, but as discussed in Section \ref{invmeas} this happens only in a few specific cases.

Note that the theory of random tensors, like the theory of regularity structures, is a short-time theory by nature; thus to get global results it has to be combined with separate global or large-scale techniques. In the context of (\ref{parabend}), the work \cite{MouWe3} combines the para-controlled calculus with energy estimates, and the more recent works \cite{MoWe,CMW} combines the regularity structures theory with the maximum principle. In the context of (\ref{nlsend}), the main global technique known is energy conservation, and the associated high-low and $I$-methods \cite{Bourgain98,CLO,CKSTT,CO}. Note that these require deterministic analysis at the $H^1$ (energy) level, so they need (deterministic) $H^1$ subcriticality, i.e. $s_{cr}<1$, to work.

In the $H^1$ supercritical ($s_{cr}>1$) case, another natural question is whether \emph{classical} solutions with random initial data (such as (\ref{data0}) with $\alpha$ suitably large, as opposed to low regularity solutions of Theorem \ref{main}) are almost surely global. This is also important from the PDE point of view, as it would mean that blowup for defocusing $H^1$ supercritical nonlinear Schr\"{o}dinger equations is \emph{non-generic} and unstable\footnote{For example, invariance of Gibbs measure and the associated almost-sure global well-posedness result would imply that there is no stable blowup mechanism at the regularity of the support of the Gibbs measure.}. Note that the blowup example in $\Rb^d$, recently constructed in \cite{MRRS}, is indeed non-generic.

\textbf{Open problem 3.} In the energy subcritical case, do the singular solutions constructed in Theorem \ref{main} extend to all time? In the energy supercritical case, does almost-sure \emph{global} well-posedness hold for random initial data of high regularity?

 \end{document}